\documentclass[11pt, reqno]{amsart}
\usepackage{fullpage}

\usepackage{amssymb,amsmath,amsthm,amscd,mathrsfs,graphicx, color}
\usepackage[cmtip,all,matrix,arrow,tips,curve]{xy}

\usepackage{hyperref}

\usepackage[usenames,dvipsnames]{xcolor}
\usepackage{xypic}
\hypersetup{colorlinks=true,citecolor=ForestGreen,linkcolor=Maroon,urlcolor=NavyBlue}
 \usepackage[normalem]{ulem}
\usepackage{mathtools}
\usepackage{mathpazo}

\numberwithin{equation}{section}
\newtheorem{teo}{Theorem}[section]
\newtheorem{pro}[teo]{Proposition}
\newtheorem{lem}[teo]{Lemma}
\newtheorem{cor}[teo]{Corollary}
\newtheorem{con}[teo]{Conjecture}

\newtheorem{que}[teo]{Question}

\newtheorem{teoalpha}{Theorem}

\theoremstyle{definition}
\newtheorem{dfn}[teo]{Definition}
\newtheorem{exa}[teo]{Example}

\theoremstyle{remark}
\newtheorem{rem}[teo]{Remark}

\DeclareMathOperator{\coniveau}{N}

\global\let\hom\undefined
\DeclareMathOperator{\hom}{Hom}

\newcommand{\powser}[1]{[\![#1]\!]}
\def\cris{{\rm cris}}
\def\sep{{\rm sep}}
\def\an{{\rm an}}
\def\univ{{\rm univ}}
\def\mmu{{\pmb\mu}}
\def\udot{^{\bullet}}
\def\proj{{\mathbb P}}
\def\cx{{\mathbb C}}
\def\rat{{\mathbb Q}}
\def\integ{{\mathbb Z}}
\def\ww{{\mathbb W}}
\def\dd{{\mathbb D}}
\def\iso{\cong}
\def\shim{{\mathsf Sh}}
\renewcommand{\bar}[1]{{\overline{#1}}}
\DeclareMathOperator{\gal}{Gal}
\DeclareMathOperator{\spec}{Spec}
\DeclareMathOperator{\spf}{Spf}
\DeclareMathOperator{\defo}{Def}
\DeclareMathOperator{\pic}{Pic}
\DeclareMathOperator{\Ab}{Ab}
\DeclareMathOperator{\A}{A}
\DeclareMathOperator{\chow}{CH}
\DeclareMathOperator{\hilb}{Hilb}

\AtBeginDocument{ 	\def\MR#1{}
}

\title{Decomposition of the diagonal, intermediate Jacobians,  and universal
 codimension-2 cycles in positive characteristic}

\author{Jeffrey D. Achter}
\address{Colorado State University, Department of Mathematics,
 Fort Collins, CO 80523,
 USA}
\email{j.achter@colostate.edu}

\author{Sebastian Casalaina-Martin }
\address{University of Colorado, Department of Mathematics,
 Boulder, CO 80309, USA }
\email{casa@math.colorado.edu}

\author{Charles Vial}
\address{Universit\"at Bielefeld, Germany}
\email{vial@math.uni-bielefeld.de}

\thanks{ The first and second authors were partially supported by grants from the NSA (H98230-16-1- 0046 and H98230-16-1-0053, respectively) and the Simons Foundation (637075 and 581058, respectively). The  third author was supported by the Deutsche Forschungsgemeinschaft (DFG, German Research Foundation) – Project-ID~491392403 -- TRR~358}

\date{March 19, 2025}

\begin{document}

 \begin{abstract}
  We consider the connections among algebraic cycles,
  abelian varieties, and stable rationality of smooth projective varieties in
  positive characteristic. Recently Voisin  constructed two new obstructions to stable rationality for rationally connected complex
  projective threefolds by giving necessary and sufficient conditions for the existence of a cohomological decomposition of the diagonal.    
  In this paper, we show how to extend these obstructions to rationally chain connected threefolds in positive characteristic  via  $\ell$-adic cohomological decomposition of the diagonal.  
  This requires extending results in Hodge theory regarding intermediate Jacobians and Abel--Jacobi maps to the setting of algebraic representatives.  
 For instance, we show that the algebraic representative for codimension-two cycle classes on a geometrically stably rational threefold admits a canonical auto-duality, which in characteristic zero agrees with the principal polarization on the intermediate Jacobian coming from Hodge theory. 
  As an application,
 we extend a result of Voisin, and show that 
 in characteristic greater than two, a desingularization of a very general quartic double solid  with  seven nodes 
 does not admit a universal codimension-two cycle class. In the process, we establish some results on the moduli space of nodal degree-four polarized K3 surfaces in positive characteristic.
   \\
   
  \noindent  ----------------------
   \\
   
  \noindent \textsc{R\'esum\'e.}
   Nous examinons les relations entre cycles alg\'ebriques, vari\'et\'es ab\'eliennes, et la propri\'et\'e de  rationalit\'e stable pour les vari\'et\'es projectives et lisses en caract\'eristique positive. 
   R\'ecemment, Voisin a exhib\'e deux nouvelles obstructions \`a la rationalit\'e stable pour les solides projectifs complexes rationnellement connexes en donnant des conditions n\'ecessaires et suffisantes \`a l'existence d'une d\'ecomposition cohomologique de la diagonale.
   Dans cet article, nous montrons comment \'etendre ces obstructions aux solides projectifs rationnellement connexes par cha\^ine en caract\'eristique positive en utilisant la cohomologie $\ell$-adique. 
   Pour cela, nous \'etendons des r\'esultats en th\'eorie de Hodge concernant les jacobiennes interm\'ediaires et les applications d'Abel--Jacobi au contexte des repr\'esentants alg\'ebriques.
   Par exemple, nous \'etablissons que le repr\'esentant alg\'ebrique pour les cycles de codimension deux sur un solide g\'eom\'etriquement rationnellement stable admet un isomorphisme canonique vers son dual qui co\"incide en caract\'eristique nulle avec la polarisation principale sur la jacobienne interm\'ediaire provenant de la th\'eorie de Hodge.
   Comme application, nous \'etendons un r\'esultat de Voisin et montrons qu'en caract\'eristique positive diff\'erente de deux une d\'esingularisation d'un solide quartique  tr\`es g\'en\'eral poss\'edant sept n\oe uds n'admet pas de classe de cycle universel de codimension deux. 
   En chemin, nous \'etablissons des r\'esultats concernant l'espace de modules des surfaces K3 nodales polaris\'ees de degr\'e quatre en caract\'eristique positive.
             \end{abstract}

             \maketitle

\newpage
\setcounter{tocdepth}{1}
{\tableofcontents}

    \newpage

\addtocontents{toc}{\protect\setcounter{tocdepth}{0}}
 \section*{Introduction}
\addtocontents{toc}{\protect\setcounter{tocdepth}{1}}

 In this paper we consider the connections among algebraic cycle classes,
 abelian varieties, and stable rationality of smooth projective varieties in
 positive characteristic.
 As motivation, recall that
 Clemens and Griffiths \cite{CG72} have shown that if a complex projective
 rationally connected threefold $X$ is rational, then the so-called minimal cohomology
 class
 \begin{equation}\label{E:CG72Rat}
 \frac{[\Theta_X]^{g-1}}{(g-1)!}\in H^{2g-2}(J^3(X),\mathbb Z)
 \end{equation}
 is an \emph{effective} algebraic cycle class,
 where   $g=\dim
 J^3(X)$, and $\Theta_X$ is the canonical polarization
         on the intermediate
 Jacobian $J^3(X)$ induced by the cup product on $H^3(X,\mathbb Z)$, which is principal as $h^{1,0}(X)=h^{3,0}(X)=0$.      For cubic threefolds, which are all unirational and therefore rationally connected, 
 Clemens and Griffiths showed, rephrasing via the Matsusaka--Ran criterion,  that
 $\frac{[\Theta_X]^{g-1}}{(g-1)!}$ is not an effective algebraic cycle class, and
 therefore that cubic threefolds are not rational.

 Recently Voisin \cite{voisinUniv} showed  that
 if a complex projective rationally connected  threefold $X$ is stably rational, then the minimal
 cohomology class \eqref{E:CG72Rat}
 is an algebraic cycle class (possibly \emph{not} effective), 
     and moreover,
 $J^3(X)$ admits a universal codimension-$2$ cycle class: there exists a cycle class $Z\in \operatorname{CH}^2(J^3(X)\times
 X)$, which is fiberwise algebraically trivial,  such  that the composition
 $$
 \xymatrix@R=.1em{
  \psi_{Z}:J^3(X)\ar[r] & \operatorname{A}^2(X)\ar[r]^{AJ}&  J^3(X)\\
  t \ar@{|->}[r]&  Z_t \ar@{|->}[r]& AJ(Z_t)
 }
 $$
 is the identity.   While it is not known whether there exist any principally
 polarized  abelian varieties $(A,\Theta)$ where the class
 $\frac{[\Theta]^{g-1}}{(g-1)!}$ is not algebraic ($g=\dim A$), 
  and thus it is
 unclear  whether this test for stable irrationality via minimal cohomology
 classes can  fail, the latter condition,  on universal codimension-$2$ cycle
 classes, has so far been more tractable.  For a  smooth projective threefold $X$
 obtained as the desingularization of a very general quartic double solid with
 $7$ nodes, Voisin showed \cite{voisinUniv} that  $J^3(X)$ does not admit a
 universal codimension-$2$ cycle class, and therefore, that such a unirational
 threefold is not stably rational.
 Recall that a nodal quartic double solid $X$ is obtained as a double cover $X\to
 \mathbb P^3$ branched along a nodal quartic surface, and that any such variety is
 unirational.

 A particularly interesting aspect of Voisin's example is that other standard
 tests of stable irrationality fail.  More precisely, building on previous work
 of Artin--Mumford \cite{ArtMum72} and Bloch--Srinivas \cite{BlSr83}, Voisin
 showed \cite{voisinUniv} that
 if a complex projective threefold $X$ is stably rational, then\,:
\begin{enumerate}
\item (Bloch--Srinivas)   $H^1(X,\mathbb Z)=0$\,;

\item (Bloch--Srinivas)  $H^{2i}(X,\mathbb Z)$ is  algebraic for all $i$\,;

\item (Bloch--Srinivas) The Abel--Jacobi map $AJ:\operatorname{A}^2(X)\to J^3(X)$ is surjective\,;
	
\item (Artin--Mumford) $\operatorname{Tors}H^\bullet (X,\mathbb Z)=0$\,;
		
\item   (Voisin) $J^3(X)$ admits a universal codimension-$2$ cycle class\,;
		
\item  (Voisin) $\frac{[\Theta_{X}]^{g-1}}{(g-1)!}\in H^{2g-2}(J^3(X),\mathbb Z)$ is an algebraic class,
	where $g = \dim J^3(X)$.
\end{enumerate}

We have omitted the \emph{a priori} weaker standard condition that $H^0(X,\Omega^i_X)=0$ for $i>0$ as it is implied by conditions (1)--(3).  
Also note that since the Albanese is a stable birational invariant, one gets (1) in characteristic $0$ without using the Bloch--Srinivas arguments\,; however, for reference in the positive characteristic case, we prefer to call this a Bloch--Srinivas condition.   We emphasize that any complex projective rationally connected threefold satisfies conditions (1)--(3)\,; for (1)  see \emph{e.g.}, \cite{BlSr83} or \cite[Cor.~10.18]{voisinII}, (2) is \cite[Thm.~2]{voisinIntHodgeUni}, and (3) is \cite[Thm.~1(i)]{BlSr83}.
In other words, (1)--(3) are obstructions to the rational connectivity of a threefold, while (4)--(6) are obstructions to the stable rationality of a rationally connected threefold.   

 Voisin's example is the first example of a unirational but stably irrational threefold
 satisfying (1)--(4) and (6) above, while only failing (5).
 Finding examples of unirational threefolds failing  (4) has been the
 typical method of establishing stable irrationality.  For instance, the first
 example of a unirational but stably irrational threefold was due to Artin and Mumford
 \cite{ArtMum72}, who  showed that there are threefolds  $X$ obtained as
 desingularizations of quartic double solids with $10$ nodes in special position,
 such that $\operatorname{Tors}H^4(X,\mathbb Z)\ne 0$ (\emph{i.e.}, (4) fails).  In that
 example (5) and (6) hold trivially since $g=0$.   It is not known if there are examples of unirational varieties failing (6), as again, it is unknown if this condition fails for any principally polarized abelian variety.
 
 The goal of this paper is to consider these types of questions in positive characteristic.   The Clemens--Griffiths results on rationality have been  investigated in this
 setting, for instance in  \cite{murre-cubic, beauville77} over algebraically closed fields, 
  and more recently in \cite{BenWittClGr,BW} over arbitrary fields.  In  this paper we focus on the topic of stable rationality, with an emphasis on 
 Voisin's conditions (5) and (6).   As a brief digression, we recall that  condition (4),  as well as  the condition that $H^0(X,\Omega_X^i)=0$ for $i>0$,   have been studied extensively in the literature in positive characteristic in the context of stable rationality.  
In condition  
  (4), one can for instance replace  Betti cohomology with
 $\ell$-adic cohomology,  and Artin--Mumford \cite{ArtMum72}  showed for example that over any
 algebraically closed field $k$ of characteristic not equal to $2$ there are
 threefolds  $X$ obtained as  desingularizations of quartic double solids with
 $10$ nodes in special position, such that $\operatorname{Tors}H^3(X,\mathbb
 Z_2)\ne 0$, \emph{i.e.}, (4) fails, and therefore that these give examples of
 unirational stably irrational threefolds over $k$.  Motivated by Voisin's degeneration techniques \cite{voisinUniv} (see also \cite[Thm.~V.5.14]{kollar}), 
 the condition on the Hodge numbers has been studied by  Totaro \cite{totaroHype}, who considered varieties $X$ obtained as desingularizations of hypersurfaces in positive characteristic with the property that $H^0(X,\Omega_X^{\dim X-1})\ne 0$\,; by degeneration to positive characteristic,  he gives examples of  rationally connected but stably irrational  hypersurfaces in characteristic $0$.

 \medskip 
Returning now to the focus of this paper, our goal  is to show that over algebraically closed fields of
 positive characteristic, there are examples of unirational but stably irrational
 threefolds satisfying (1)--(4), and (6) above, while failing obstruction (5)\,; \emph{i.e.}, examples of unirational threefolds with no universal codimension-$2$ cycle class.    We in fact study this question more generally over an arbitrary  perfect field.

 The first issue is to make sense of  conditions (3), (5), and (6)  over a perfect field
 $K$, since the conditions are defined in terms of the intermediate Jacobian and
 the Abel--Jacobi map,  which are inherently transcendental.  We take two
 approaches, one that works in characteristic $0$, and one that works in
 arbitrary characteristic.
 In the former case, where we may take $K\subseteq \mathbb C$, we have shown
 \cite{ACMVdmij}  that  $J^3_a(X^{\an})$, the image  of the Abel--Jacobi map
 $
 AJ:\operatorname{A}^2(X^{\an})\to J^3(X^{\an})
 $
 on algebraically trivial cycle classes, descends to a distinguished model
 $J^3_{a,X/K}$ over $K$ such that the Abel--Jacobi map is
 $\operatorname{Aut}(\mathbb C/K)$-equivariant.   It is easy to see that if
 $X^{\an}$ is a rationally connected threefold (or more generally, has  universally trivial \emph{rational} Chow group of zero cycles), in which case $J^3_a(X^{an})=J^3(X^{an})$, then the canonical principal polarization
 $\Theta_{X^{\an}}$ (see Remark~\ref{R:SRC-Hodge})
  descends to a principal
 polarization $\Theta_X$ on $J^3_{a,X/K}$ (Theorem~\ref{T:CanPol3-fold}\,; see also
 \cite[Prop.~2.5]{BenWittClGr}).

   Without the assumption that $\operatorname{char}(K)=0$, there are two replacements for the Abel--Jacobi map in condition (3) that both play a crucial role in our treatment. 
 To motivate this, we recall from \cite[Thm.~10.3]{murre83} that $T_\ell AJ:T_\ell \operatorname{A}^2(X)\to T_\ell J^3_a(X)$ is an isomorphism, and that the Abel--Jacobi map is surjective if and only if the $\ell$-adic Abel--Jacobi map $T_\ell AJ:T_\ell \operatorname{A}^2(X)\to T_\ell J^3(X)$, or equivalently $T_\ell AJ:T_\ell\operatorname{A}^2(X)\to H^3(X,\mathbb Z_\ell)_\tau$, is an isomorphism, where the subscript $\tau$ indicates the torsion-free quotient.  For technical reasons we find this formulation to be easier to work with in positive characteristic, and so we actually consider two replacements for the $\ell$-adic Abel--Jacobi map.    
 
 The first, which is defined on torsion cycles and takes values in the odd cohomology with torsion coefficients, is the \emph{Bloch map}~\cite{bloch79}\,; taking Tate modules defines the \emph{$\ell$-adic Bloch map}
 $T_\ell \lambda^2:T_\ell \operatorname{CH}^2(X_{\bar K})\longrightarrow H^3(X_{\bar K},\mathbb Z_\ell(2))_\tau$, 
 with values in  $\ell$-adic cohomology modulo torsion.
This map, which is in fact defined for cycles 
 of any codimension, was first considered by Suwa \cite{Suwa} in the case $\ell \neq p$ and by Gros--Suwa \cite{grossuwaAJ} in the case $\ell=p$.  We will focus on the restriction of this map to algebraically trivial cycle classes\,:
\begin{equation}\label{E:IntroBlMa}
T_\ell \lambda^2:T_\ell \operatorname{A}^2(X_{\bar K})\longrightarrow H^3(X_{\bar K},\mathbb Z_\ell(2))_\tau.
\end{equation}
 We recently studied the map further in \cite{ACMVBlochMap}, and will rely on the definitions and results presented there.

 The second replacement for the Abel--Jacobi map, which is defined on algebraically trivial cycle classes of codimension-$2$ and takes values in an abstract abelian variety, is the \emph{second algebraic representative}~\cite{murre83} (see \S \ref{S:PreAlgRep})\,:
\begin{equation}\label{E:IntroAlgRep}
\phi^2_{X_{\bar K}}:\operatorname{A}^2(X_{\bar K})\longrightarrow
 \operatorname{Ab}^2_{X_{\bar K}/\bar K}(\bar K).
\end{equation}
Building on work
 of Murre \cite{murre83} over an algebraically closed field, we  showed in \cite{ACMVdcg} that the algebraic
 representative $\operatorname{Ab}^2_{X_{\bar K}/\bar K}$ over $\bar K$
 admits a distinguished model $\operatorname{Ab}^2_{X/K}$ over $K$, 
 distinguished by the fact that the universal regular homomorphism
 \eqref{E:IntroAlgRep}
    is
 $\operatorname{Gal}(K)$-equivariant (see \S \ref{S:PreAlgRep}).
 Note that if $K\subseteq \mathbb C$, then $\operatorname{Ab}^2_{X/K}=J^3_{a,X/K}$\,; \emph{i.e.}, the algebraic representative agrees with the distinguished model of the algebraic intermediate Jacobian \cite{ACMVdcg}.   
Taking Tate modules in \eqref{E:IntroAlgRep}  yields a map 
\begin{equation}\label{E:IntroAlgRepTate}
T_\ell \phi^2_{X_{\bar K}}:T_\ell \operatorname{A}^2(X_{\bar K})\longrightarrow
T_\ell  \operatorname{Ab}^2_{X_{\bar K}/\bar K}.
\end{equation}

  While in characteristic zero both maps \eqref{E:IntroBlMa} and \eqref{E:IntroAlgRepTate} identify canonically with the $\ell$-adic Abel--Jacobi map (\cite[Prop.~3.7]{bloch79}, \cite{ACMVBlochMap}), it is not known if they agree in positive characteristic (\emph{i.e.}, after making some identification of the cohomology of the abelian variety $ \operatorname{Ab}^2_{X/ K}$ with that of~$X$).  However, in parallel with the characteristic $0$ case,  for a smooth projective geometrically rationally chain connected variety $X$ the maps  \eqref{E:IntroBlMa} and \eqref{E:IntroAlgRepTate} are known to be isomorphisms \cite[Prop.~2.3]{BenWittClGr}. 
  We also point out here that 
  in characteristic $0$, condition (3) is simply equivalent to $H^3(X,\rat)$ being supported on a divisor.

In positive characteristic, the replacement for condition (5) is given by the notion of a  \emph{universal codimension-$2$ cycle class} for $X$, which  is a cycle
 class $Z\in \operatorname{CH}^2(\operatorname{Ab}^2_{X/K} \times_{K} X)$, viewed as a family of cycles on $X$ parameterized by $\operatorname{Ab}^2_{X/K}$, which is fiberwise algebraically trivial, 
 such  that the composition
 $$
 \xymatrix@R=.1em{
  \psi_{ Z}:\operatorname{Ab}^2_{X/K}(\bar K) \ar[r] &
  \operatorname{A}^2(X_{\bar K})\ar[r]
 &
  \operatorname{Ab}^2_{X/K}(\bar K)\\
  t \ar@{|->}[r] & Z_t \ar@{|->}[r] &  {\phi^2_{X_{\bar K}}}(Z_t)
 }
 $$
 is the identity.

 Finally, we turn to condition (6)\,; \emph{i.e.}, to a  replacement for the principal polarization $\Theta_X$.   Recall that  in characteristic $0$ we have $\operatorname{Ab}^2_{X/K}=J^3_{a,X/K}$,  so that for a projective geometrically rationally connected threefold over a field of characteristic $0$, the algebraic representative $\operatorname{Ab}^2_{X/K}$ comes equipped with the principal polarization obtained via intersection in cohomology, as explained above.  
  While in positive characteristic we do not have a way to define a distinguished (principal) 
 polarization
   on $\operatorname{Ab}^2_{X/K}$ for every  rationally chain connected  threefold $X$,
 we can do something similar for a class of rationally chain connected threefolds that includes
 geometrically stably rational threefolds\,:

\begin{teoalpha}[Auto-duality of the algebraic representative]\label{T:Intro-CanPol0}
	Let $X$ be a smooth projective threefold over a
	perfect field $K$. 
\begin{enumerate}

\item If $X$ is  geometrically rationally chain connected, then there is a canonical purely inseparable symmetric $K$-isogeny 
 \begin{equation}\label{E:Intro-LambdaX}
\xymatrix{\Theta_X :\ \operatorname{Ab}^2_{X/K} \ar[r]^{} & {\widehat{\operatorname{Ab}}\,^2_{X/K}}}.
\end{equation}

\item  
If $V_\ell\lambda^2$ is an isomorphism for some prime $\ell \neq \operatorname{char}(K)$ (\emph{e.g.},  $X$ is geometrically uniruled),
 and  $X_{\bar K}$ admits a universal codimension-$2$ cycle class  $Z$, 
	  then the homomorphism of abelian varieties induced by the cycle class $-({}^tZ\circ Z) \in \chow^1(\operatorname{Ab}^2_{X_{\bar K}/\bar K} \times_{\bar K} \operatorname{Ab}^2_{X_{\bar K}/\bar K})$
	    descends to $K$ to give a canonical 
	   symmetric $K$-isogeny  
	  \begin{equation}\label{E:Intro-LambdaZ}
\xymatrix{\Theta_X :\ \operatorname{Ab}^2_{X/K} \ar[r]^{} & {\widehat{\operatorname{Ab}}\,^2_{X/K}}},
	\end{equation}
which is independent of the choice of universal cycle class $Z$. 
If moreover $X$ is geometrically rationally chain connected, then 
\eqref{E:Intro-LambdaZ} agrees with \eqref{E:Intro-LambdaX}.

\item If $X$ is geometrically stably rational, then $X_{\bar K}$ admits a universal codimension-$2$ cycle class $Z$, and the purely inseparable symmetric $K$-isogeny $\Theta_X$  is an isomorphism.
\end{enumerate}	

\end{teoalpha}

 Theorem~\ref{T:Intro-CanPol0} is proven in Theorem~\ref{T:main-pol} and Theorem \ref{T:CanPol3-fold}.   
The idea of the proof is as follows.  
 One chooses a miniversal codimension-$2$ cycle class  $Z \in  \operatorname{CH}^2(\operatorname{Ab}^2_{X_{\bar K}/\bar K} \times_{\bar K} X_{\bar K})$ of degree $N$ for some natural number $N$, \emph{i.e.}, $\psi_Z:{\operatorname{Ab}^2_{X_{\bar K}/\bar K}} \to {\operatorname{Ab}^2_{X_{\bar K}/\bar K}}$ is multiplication by~$N$\,;   such cycles are known to exist for any surjective regular homomorphism.  
 The cycle class ${}^{t}Z\circ Z \in \chow^1(\operatorname{Ab}^2_{X_{\bar K}/\bar K} \times_{\bar K} \operatorname{Ab}^2_{X_{\bar K}/\bar K})$ then defines, via the theory of Picard schemes, a  symmetric homomorphism $\Lambda_{Z}: (\operatorname{Ab}^2_{X/K})_{\bar K}\to (\widehat {\operatorname{Ab}}\,^2_{X/K})_{\bar K}$.  In Theorem \ref{T:CanPol3-fold} we show that $\Lambda_Z$  is surjective, descends to $K$,  and is independent
  of the choice  of a miniversal cycle class of degree $N$ (although it will depend on $N$).  We then show that $\Lambda_Z$ is divisible by $N^2$, giving a symmetric isogeny  $\Lambda_X$, which, when $X$ is geometrically rationally chain connected, we show is an isomorphism on Tate modules for all primes $l$, and is therefore a purely inseparable isogeny.    For reasons having to do with positivity, we take $\Theta_X=-\Lambda_X$ in Theorem \ref{T:Intro-CanPol0}.  
     When $X$ is assumed to be geometrically stably rational, we give a different proof of these facts, and obtain the stronger result, that $\Lambda_X$ is an isomorphism.  More precisely, 
  in Proposition~\ref{P:unicycle}, we show that there is a universal codimension-2 cycle class $Z$ over $\bar K$.  As before, the cycle class ${}^tZ\circ Z$ then defines     a  symmetric homomorphism $\Lambda_{X_{\bar K}}: (\operatorname{Ab}^2_{X/K})_{\bar K}\to (\widehat {\operatorname{Ab}}\,^2_{X/K})_{\bar K}$, which we show via a  
   similar argument  is surjective, descends to~$K$,  and is independent
  of the choice  of a universal cycle class, giving a symmetric isogeny  
    $\Lambda_X$.
However, to show that $\Lambda_X$ is an isomorphism, for which it suffices to show that $\Lambda_{X_{\bar K}}$ is an isomorphism,  we use a  crucial new ingredient: we  show that for a stably rational threefold over an algebraically closed field, the second algebraic representative is induced by a cycle class (the meaning of this is made precise in \S \ref{SS:RegHomCycle} and Proposition \ref{P:unicycle}).

  Recall that a symmetric isogeny from an abelian variety over a field $K$ to its dual is, after base change to the algebraic closure $\bar K$, induced by a symmetric line bundle, and is called a polarization if the line bundle is ample 
(see \S\ref{S:ThetaLB} for a review).
  We note that if  $\operatorname{char}(K)=0$, then $\Theta_X$ in Theorem~\ref{T:Intro-CanPol0} is the Hodge-theoretic principal polarization induced via  the intersection product in cohomology, as described above.  In fact, in positive characteristic as well, 
$\Theta_X$ is induced by the intersection product in the middle cohomology of $X$\,; the meaning of this is made precise in Definition \ref{D:DistMorph}.
When $X$ is a geometrically \emph{rational} threefold, 
Benoist--Wittenberg \cite[Cor.~2.8]{BenWittClGr} recently constructed a \emph{principal polarization}  $\Theta_X$ on $\operatorname{Ab}^2_{X/K}$, which agrees with $\Theta_X$ of Theorem \ref{T:Intro-CanPol0}\,; in fact, extending the result of Clemens--Griffiths,  they
 show that if $X$ is rational \emph{over $K$} then $(\operatorname{Ab}^2_{X/K},\Theta_X)$ is the product of principally
 polarized Jacobians of curves.
    For some rationally chain connected threefolds we can show that $\Theta_X$ is a polarization, so that 
Theorem \ref{T:Intro-CanPol0} in fact provides a partial answer to a question of  \cite[p.6]{BenWittClGr} 
  (see \S \ref{S:BWrat+que} and Corollary \ref{C:ThetaSpecMini}).

\medskip 
While  in the introduction  we have so far discussed the results  in the context of rationality, the results, as well as the techniques,  are in fact most naturally explained in terms of decomposition of the diagonal (see \S \ref{S:PreDecDiag}).  The basic implications we use are that stably rational (resp.~rationally chain connected) implies  universally trivial  integral (resp.~rational) Chow group of zero cycles, which implies strict integral (resp.~rational)  Chow decomposition of the diagonal (see Remarks~\ref{R:StabRatDec} and~\ref{R:RCDecDiag}).   In order to study the question of universal codimension-$2$ cycle classes, however, we must consider the yet weaker notion of \emph{cohomological} decomposition of the diagonal (see \S \ref{S:CohDD}).
    
 We can now state the following  theorem regarding strict cohomological $\mathbb Z$-decomposition of the diagonal with respect to $H^\bullet (-,\mathbb Z_\ell)$ (see Definition \ref{D:StrCoDec}), which generalizes
 \cite{voisinUniv} to algebraically closed fields of positive characteristic.  Recall that Voisin has shown that a smooth complex projective threefold admits a strict cohomological $\mathbb Z$-decomposition of the diagonal with respect to $H^\bullet (-,\mathbb Z)$ if and only if conditions (1)--(6) above hold (\cite[Thm.~1.7]{voisinUniv} and \cite[Thm.~4.1]{voisinCubicCH}). (Note that it is assumed in  \cite[Thm.~4.1]{voisinCubicCH}  that the threefold be rationally connected, but this is used only to ensure conditions (1)--(3) hold, which we have explained above hold for any complex projective rationally connected threefold.)
  We  give necessary and sufficient conditions over an algebraically closed field for the existence of a strict cohomological $\mathbb Z_\ell$-decomposition with respect to $H^\bullet (-,\mathbb Z_\ell)$   in Theorem \ref{T:ZZell-iff}, however here, we prefer to mention our result on strict cohomological $\mathbb Z$-decompositions\,:

  \begin{teoalpha}[Cohomological decomposition of the diagonal]\label{T:Intro-M-StabInv}
	Let  $X$ be a smooth projective threefold over an algebraically closed field $k$, and fix a prime number   $\ell\ne \operatorname{char}(k)$.
If 
$\Delta_{X}\in\chow^{3}(X\times_{k} X)$ admits a  strict cohomological \emph{$\mathbb Z$-decomposition}  with respect to $H^\bullet (- , \mathbb Z_\ell)$, then:

\begin{enumerate}
\item $H^1(X,\mathbb Z_\ell)=0$\,;

\item $H^{2i}(X,\mathbb Z_\ell(i))$  is $\mathbb Z$-algebraic for all $i$\,;

\item  The $\ell$-adic  Bloch map    $T_\ell \lambda^2:T_\ell \operatorname{A}^2(X)\to H^3(X,\mathbb Z_\ell(2))_\tau$ is an isomorphism\,;

\item[(3')] The $\ell$-adic map $T_\ell \phi^2_{X/k} :T_\ell \operatorname{A}^2(X) \to T_\ell \operatorname{Ab}^2_{X/k}$
 is an isomorphism\,;

\item  $\operatorname{Tors}H^\bullet (X,\mathbb Z_\ell )=0$\,;
		
\item    $\operatorname{Ab}^2_{X/k}$ admits a universal codimension-$2$ cycle class\,; 
		
\item   Assuming (3) and (5), and setting $\Theta_X:\operatorname{Ab}^2_{X/k}\to \widehat {\operatorname{Ab}}\,^2_{X/k}$ to be the symmetric isogeny of Theorem~\ref{T:Intro-CanPol0}(2), we have that $T_\ell\Theta_X$ is an isomorphism, and   $\frac{[\Theta_{X}]^{g-1}}{(g-1)!}\in H^{2g-2}(\operatorname{Ab}^2_{X/k},\mathbb Z_\ell(g-1))$ is a $\mathbb Z$-algebraic class,
	where $g = \dim \operatorname{Ab}^2_{X/k}$ and $[\Theta_X]$ is the first Chern class of the line bundle associated to $\Theta_X$.

\end{enumerate} 
As a partial converse, if (1)--(6) (including (3')) hold, then $\Delta_{X}\in\chow^{3}(X\times_{k} X)$ admits a  strict cohomological \emph{$\mathbb Z_\ell$-decomposition}  with respect to $H^\bullet (- , \mathbb Z_\ell)$.
 \end{teoalpha}

Theorem \ref{T:Intro-M-StabInv} is proven in  Theorem~\ref{T:intcohodec}, which is in fact stronger, addressing for instance the case of perfect fields, as well as cohomological decompositions of the diagonal  supported on curves, rather than points (see Remark \ref{R:Decomp21}).   
    We emphasize that we have omitted an assertion about the vanishing of Hodge numbers.  In characteristic $0$,  a strict cohomological $\mathbb Q$-decomposition of the diagonal implies that  $H^0(X,\Omega^i_X)=0$ for $i>0$ (see \emph{e.g.}~\cite[Thm.~4.4(iii)]{voisinAJ13})\,; however, in positive characteristic, we only know this holds under the stronger assumption of a strict  \emph{Chow} $\mathbb Z$-decomposition of the diagonal (see \emph{e.g.}, \cite[Lem.~2.2]{totaroHype}, and also Remark~\ref{R:SRC-Hodge}). 
   In addition, while conditions (1)--(6) in Theorem 
\ref{T:Intro-M-StabInv} are sufficient for the existence of a cohomological $\mathbb Z_\ell$-decomposition of the diagonal, they are not necessary.  This will follow from our Theorem \ref{T:Intro-QDS-1} below (and \cite{voisinUniv} over $\mathbb C$),  which establishes that the standard desingularization of a very general quartic double solid with exactly $7$ nodes does not admit a universal codimension-$2$ cycle class (\emph{i.e.}, (5) fails)\,; on the other hand, it is well-known that twice the class of the diagonal admits a strict Chow decomposition, and consequently the diagonal admits a strict cohomological $\mathbb Z_\ell$-decomposition for all $\ell \ne 2,\operatorname{char}(k)$ (see 
Remarks  
\ref{R:2DeltaCH} and \ref{R:T-2-S+N}). 
We reiterate that over an algebraically closed field, we give necessary and sufficient conditions for a cohomological $\mathbb Z_\ell$-decomposition of the diagonal  in Theorem \ref{T:ZZell-iff}.

Also, while we know that (1)--(3') hold for all geometrically rationally connected threefolds  in characteristic $0$, for geometrically rationally chain connected threefolds in positive characteristic we only know that (1) holds 
(Corollary \ref{C:vanCoh}),  and that (3) and (3') hold \cite[Prop.~2.3]{BenWittClGr}. In fact, we expect (3') may hold for all smooth projective varieties over any field, and for this reason have separated it from condition (3), although both (3) and (3') replace the condition (3) in the complex setting, namely the surjectivity of the Abel--Jacobi map.  In other words, in positive characteristic, (1), (3), and (3')  
 should still be viewed as an obstruction to a threefold being rationally chain connected, while in contrast to the characteristic zero case, (2) could potentially be an obstruction to the stable rationality of a rationally chain connected threefold. 
  However, 
in Corollary~\ref{C:RC->(0)} we show that if a geometrically rationally chain connected threefold lifts to characteristic $0$ to a geometrically rationally connected threefold with no torsion in cohomology, then conditions (1)--(3') hold 
(as well as condition (4)).

Regarding the proof of Theorem \ref{T:Intro-M-StabInv}, 
under the assumption of the cohomological decomposition of the diagonal,  (1), (2), and (4)  are now standard in the literature:  they  follow from the techniques in \cite{BlSr83} and \cite{voisinUniv} (we recall the proof in our setting in \S\ref{S:DD-Van},  \S\ref{S:DD-AlgCyc}, and \S \ref{S:DD-Tor}, respectively).   
In short, under the assumption of the cohomological decomposition of the diagonal, the main focus is on the conditions (3) and (3'), (5), and (6). 
Conditions (3) and (3') were investigated recently in \cite[Prop.~2.3]{BenWittClGr} in the context of a \emph{Chow} decomposition of the diagonal\,; in that setting (3') is essentially a consequence of \cite[Thm.~1(i)]{BlSr83} and  (3) is proven similarly.  
 The key addition here, in the context of cohomological decompositions,   is that we show that morphisms induced by families of cycle classes via universal regular homomorphisms  depend only on the cohomology class of the family of cycles (see Proposition \ref{P:vanish} and Corollary \ref{C:vanish}).  As Abel--Jacobi maps enjoy this property, we view this as a significant improvement on the theory of algebraic representatives in positive characteristic.  This addition also allows us to establish (5), which is an extension of a 
 result of Voisin
 \cite[Thm.~4.4(iii)]{voisinAJ13} \cite[Thm.~4.2]{voisinCubicCH}
 to the case of finite and algebraically closed fields (see Corollary~\ref{C:UnivCyc}).

 Condition (6)  follows Voisin's arguments in Hodge theory, as well as Mboro's work on
 cubic threefolds  for $\operatorname{char}(k)\ne 2$, but we note that there are several significant additions needed in our work.  
 First and foremost, one needs Theorem  \ref{T:Intro-CanPol0} to provide a  replacement for the principal polarization, which in Voisin's case comes from Hodge theory, and in Mboro's case comes from the theory of Prym varieties and fibrations in quadrics, which rules out the case $\operatorname{char}(k)=2$.     Note that $\Theta_X$ in Theorem \ref{T:Intro-CanPol0}(1) and (2) (and therefore in Theorem \ref{T:Intro-M-StabInv}) is not known to be a polarization, or even an isomorphism.  This is an important point in the sense that, unlike the cases considered by Voisin and Mboro, one does not automatically have condition (6) when $\dim \operatorname{Ab}^2_{X/k}\le 3$.   We will see this subtle point come into play later.  
 The second key addition is Proposition~\ref{P:commute}, which 
  is a technical
 point relating regular homomorphisms and actions of correspondences, which generalizes a classical result regarding
 Abel--Jacobi maps (see
 \cite[Thm.~12.17]{voisinI}), and plays a central role in Voisin's Hodge-theoretic arguments in \cite{voisinAJ13, voisinCubicCH}.  These techniques are used also in \cite{mboro}, and  for instance, Proposition~\ref{P:commute} applied to cubic
 threefolds provides a proof of  
  the assertion \cite[Lem.~3.3]{mboro}.  
  Along the way, we positively answer some cases of a conjecture of Gros--Suwa 
\cite[Conj.~III.4.1(iii)]{grossuwaAJ}
(see Lemma~\ref{L:factorM}). 
 \medskip
   
   In light of Theorem~\ref{T:Intro-M-StabInv}, we extend to positive characteristic Voisin's result that there exist
  unirational complex smooth projective varieties with no universal codimension-$2$
  cycle class.

 \begin{teoalpha}[Quartic double solids] \label{T:Intro-QDS-1} Let $k$ be an uncountable algebraically
  closed field with $\operatorname{char}(k)\ne 2$.
  Let $\widetilde X$ be the standard resolution of singularities of a  very
  general quartic double solid $X$ with exactly $n \le 9$ nodes (and no other singularities).
  Then  for  $\ell=2$:
  \begin{enumerate}
  \item[(A)] If $n\le 6$, then
 (1)--(4)  of Theorem~\ref{T:Intro-M-StabInv} hold for $\widetilde X$,
  and one or both of (5) and (6)
  fail.

  \item[(B)] If $7\le n\le 9$, then (1)--(4), and (6),  of Theorem~\ref{T:Intro-M-StabInv} hold for $\widetilde X$, while (5) fails.
 In other words, $\widetilde X$ does not admit  a universal codimension-$2$ cycle class.

  \end{enumerate}

 \end{teoalpha}

The precise notion of the meaning of a very general quartic double solid with $n\le 9$ nodes is given in \S \ref{S:QDS}\,; essentially it is a quartic double solid obtained from a quartic surface with exactly $n$ nodes, which corresponds to a very general point of the moduli of degree $4$ polarized K3 surfaces with exactly $n$ nodes, which we show is irreducible.  
 Note that one can conclude from
 Theorems~\ref{T:Intro-M-StabInv} and~\ref{T:Intro-QDS-1}, that  over an
 uncountable algebraically closed field $k$ with $\operatorname{char}(k)\ne 2$,  the
 standard resolution of singularities $\widetilde X$ of a  very general quartic
 double solid $X$ with at most $9$ nodes is not stably rational.

Our proof of
 Theorem~\ref{T:Intro-QDS-1}, which is given in \S \ref{S:Pf-QDS},  is similar to that in  \cite{voisinUniv}, but involves several key additions.  
First, as mentioned above, for every $n\le 9$ we show  that in the moduli space of polarized K3 surfaces of degree $4$, the discriminant locus corresponding to K3s with exactly $n\le 9$ nodes is irreducible (Proposition \ref{P:nnodesirred}), and that for a $10$-nodal quartic K3, the nodes can be deformed independently (Lemma \ref{L:smoothnodes}), so that the Artin--Mumford example is in the boundary of each of these components of the discriminant (Corollary \ref{C:degen10}).  
This result is slightly more general than what is proven in \cite{voisinUniv} in characteristic $0$ (see Remark \ref{R:VoisinK3}), and for instance, even in characteristic $0$,  gives a clean statement of Theorem \ref{T:Intro-QDS-1} for $n=8,9$ nodes (\emph{cf.}~\cite[p.210]{voisinUniv}).  The point is that while the locus of $n$-nodal quartic surfaces in the \emph{Hilbert scheme} of quartic surfaces  with exactly $n=6,7,8,9$ nodes is known to be reducible (see \cite[Rem.~1.2]{voisinUniv}), and for $n=6,7$ Voisin picks out a distinguished component containing the Artin--Mumford example, the locus in the \emph{moduli space of polarized K3 surfaces} is irreducible, and we are free to take very general points of of these irreducible components.
Next we show that for $n\le 9$ nodes, one can lift a nodal quartic surface, along with its nodes, to characteristic $0$ (Lemma \ref{L:smoothnodes}).  From this, we can use specialization from characteristic $0$ to show that conditions (1)--(4) hold (Corollary \ref{C:RC->(0)}). 

Having established that the Artin--Mumford examples are degenerations of our examples, the next step is to consider degenerations of decompositions of the diagonal.  Since we must use singular quartic double solids, as well as their resolutions, this requires us to work with 
$\ell$-adic \emph{homological} decompositions of the diagonal.  
This is discussed in \S \ref{S:ResSing-Def}, where we show that 
 existence of an $\ell$-adic homological decomposition of the diagonal is
 stable under specialization from the very general fiber
 (Theorem~\ref{T:DecDiagDeg}), as well as under resolution of singularities of
 nodes (Proposition~\ref{P:Dec-Res}).    
From our degeneration to the Artin--Mumford example, we can then conclude that the standard resolution of singularities of the very general quartic double solid with at most $9$ nodes does not admit a cohomological $\mathbb Z_2$-decomposition of the diagonal. 

 Therefore,  from  Theorem~\ref{T:Intro-M-StabInv}, we can conclude that condition (5) or (6) must fail.  
Turning now to condition (6), we assume that condition (5) holds, and let $\Theta_{\widetilde X}$ be the associated symmetric isogeny.  We note that unlike the case of characteristic $0$, where $\Theta_{\widetilde X}$ is known to be a principal polarization, in positive characteristic understanding the algebraicity of $[\Theta_{\widetilde X}]^{g-1}/(g-1)!$, even when $g\le 3$, is more subtle.  In addition, algebraic representatives need not be stable under specialization, so that even with a lift to characteristic $0$, there is no guarantee that $\Theta_{\widetilde X}$ is a principal polarization. 
To get around this issue, we first show (Corollary \ref{C:ThetaSpecMini}) that, due to the liftability of $\widetilde X$ to characteristic~$0$, the symmetric isogeny $\Theta_{\widetilde X}$ is a  polarization (although not necessarily principal). 
From this it follows that the polarized abelian variety $(\operatorname{Ab}^2_{\widetilde X/k},\Theta_{\widetilde X})$ admits an isogeny  to a principally polarized abelian variety, which has the same dimension, namely $g=10-n$, and therefore,  for dimension reasons,  must be a Jacobian of a curve if $n=7,8,9$.  Consequently, provided $n=7,8,9$, it follows that  $[\Theta_{\widetilde X}]^{g-1}/(g-1)!$ is $\mathbb Z$-algebraic, being the pull back under the isogeny of the class of the Abel--Jacobi embedded curve in its Jacobian (see Proposition~\ref{C:MinCohSpec}).  Therefore, for $n=7,8,9$, we must have had that condition (5) fails, since otherwise conditions (1)--(6) would hold and $\widetilde X$ would admit a cohomological $\mathbb Z_2$-decomposition of the diagonal, which we know is not the case.

 We note that once one has established that the very general quartic double solid with at most $9$ nodes degenerates to the Artin--Mumford example, then the conclusion regarding  stable irrationality
follows also from the  degeneration and resolution of singularities results of
 \cite[Thm.~1.12]{CTP16}, \cite[Prop.~8, Thm.~9]{HKT16},
 \cite[Thm.~2.3]{totaroHype} (and for  $\operatorname{char}(k)=0$,
 from  the degeneration results of \cite[Thm.~2.1]{voisinUniv},
 \cite[Thm.~4.2.11]{NicShin},  \cite[Thm.~1]{KontTsch}).
 The \emph{irrationality} of desingularizations of \emph{all} quartic double solids with exactly $n\le 1$ nodes was established via the Clemens--Griffiths criterion in the case $n=0$ (over $\mathbb C$) in \cite{voisin88}, and for $n=1$ in \cite[Thm.~4.9]{beauville77}\,; recall that the Clemens--Griffiths criterion was extended to threefolds over an algebraically closed field in  \cite[Thm.~p.63]{murre-cubic} and \cite[Prop.~4.6]{beauville77}, and to threefolds over arbitrary fields in  \cite[Thm.~2.7]{BenWittClGr} and \cite[Thm.~C]{BW}.

 \subsection*{Outline} 
  The paper is split into three parts. Part~\ref{P:Chow} focuses on applications of Chow decompositions of the diagonal to the second algebraic representative. There we start by reviewing the theory of algebraic representatives and fix the notation for decompositions of the diagonal, both of which will be used throughout the paper. We then proceed to prove Theorem~\ref{T:Intro-CanPol0} under the hypothesis of stable rationality, and draw some consequences. The main objective of Part~\ref{P:coho} is the proof of Theorem~\ref{T:Intro-M-StabInv}. We proceed by first proving that the existence of a strict cohomological of the diagonal of a threefold implies conditions (1)--(6) of Theorem~\ref{T:Intro-M-StabInv}, and
 then conclude this part in \S \ref{S:proof2} by establishing that, conversely, conditions (1)--(6) ensure the existence of a strict $\integ_{\ell}$-decomposition of the diagonal.  (Where possible, we also study $p$-adic decompositions in positive characteristic~$p$.)
   Along the way  we complete the proof of Theorem~\ref{T:Intro-CanPol0} in \S \ref{S:DiagPolCoh}.   
    In Theorem \ref{T:ZZell-iff} we give necessary and sufficient conditions for the existence of a strict cohomological $\mathbb Z_\ell$-decomposition of the diagonal.  
  Finally, in Part~\ref{P:positive} we prove Theorem~\ref{T:Intro-QDS-1}.

\subsection*{Acknowledgments} 
We are grateful to Olivier Wittenberg for useful comments on a preliminary version of this paper, which led in particular to improvements to Theorem~\ref{T:Intro-CanPol0} in the case of geometrically rationally chain connected threefolds.

 \subsection{Conventions} \label{conventions} 

  A \emph{variety} over a field is
 a geometrically reduced separated scheme of finite type over that field.
 For  a scheme~$X$ of finite type over a field~$K$, we denote by
 $\operatorname{CH}^n(X)$ the Chow group of codimension-$n$ cycle
 classes on $X$, and by $\operatorname{A}^n(X)$ the group of algebraically
 trivial cycle classes. Unless explicitly  stated otherwise, $\mathcal H^\bullet$ denotes $\ell$-adic cohomology with coefficient ring $R_{\mathcal H}=\integ_{\ell}$ for some $\ell \neq \operatorname{char}(K)$. 
   For a smooth
 projective variety $X$ over a field $K$, we denote the cycle class map by
 $[-]:\operatorname{CH}^n(X)\to \mathcal H^{2n}(X)$.
 For a commutative ring $R$, and a scheme $X$ of finite type over a
 field $K$, we denote by 
 $\operatorname{CH}^\bullet(X)_R:=\operatorname{CH}^\bullet(X)\otimes_{\mathbb
  Z} R$ the Chow group with coefficients in $R$.

 The symbol $\ell$
 always denotes a rational prime (\emph{i.e.}, a natural number that is a prime)  invertible in the base field, while
 $l$ is allowed to be \emph{any} rational prime, including the
 characteristic of the base field.
 
 If $G$ is an abelian group, then $G_\tau$ denotes the quotient of $G$
 by its torsion subgroup, $G[\ell^\infty]$ denotes its $\ell$-primary torsion
 and $G_\rat$ denotes $G\otimes_{\mathbb Z} \rat$.

 Given a field $K$ with algebraic closure $\bar K$ and separable closure denoted $K^{\sep}$,
  together with an $\operatorname{Aut}(\bar K/K)=\gal(K^{\sep}/K)$-module $M$, we denote $T_l M$
 the Tate module $\varprojlim M\otimes_{\mathbb Z} \integ/l^n
 \integ$.
 As usual, we denote $\integ_{\ell}(1)$ the Tate module $\varprojlim \mmu_{\ell^n}$, where $\mmu_{\ell^n}$ is the group of $\ell^n$-th roots of unity. Given a $\integ_{\ell}$-module $M$, we denote $M^\vee := \hom(M, \integ_\ell)$ and $M(n) := M\otimes_{\integ_{\ell}} \integ_{\ell}(1)^{\otimes n}$ its $n$-th Tate twist, where for $n<0$ we have $\integ_{\ell}(1)^{\otimes n} := (\integ_{\ell}(1)^\vee)^{\otimes -n}$.
 
 If $X$ is a smooth projective variety over a field $K$ and if $l$ is a prime, we will denote by
 $$T_l\lambda^n : T_l \operatorname{CH}^n(X_{\bar K}) \to H^{2n-1}(X_{\bar K},\integ_l(n))_\tau$$ the $l$-adic Bloch map defined by Suwa~\cite{Suwa} in case $l$ is invertible in $K$ and by Gros--Suwa~\cite{grossuwaAJ} otherwise\,; see also \cite{ACMVBlochMap}. Abusing notation, we will also denote by  $T_l\lambda^n: T_l \operatorname{A}^n(X_{\bar K}) \to H^{2n-1}(X_{\bar K},\integ_l(n))_\tau$ the restriction of the above map to $T_l\operatorname{A}^n(X_{\bar K})$.

 For $K$ of positive characteristic $p$, see \cite[\S I.3.1]{grossuwaAJ} for details on $H^j(X_{\bar K}, \integ_p)$.  We let $\ww(K)$ denote the ring of Witt vectors over $K$, and $\mathbb B(K)$ its field of fractions.

Let $\mathcal H^\bullet$ be a Weil cohomology theory.
For $X$ smooth projective and geometrically connected  over a field $K$ of pure dimension $d$, the intersection product $\mathcal H^k (X)_\tau \times \mathcal H^{2d-k}(X)_\tau \to \mathcal H^{2d}(X)$ provides a canonical identification
 $$ \mathcal H^{i}(X)_\tau(d) \stackrel{\cup}{=} \mathcal H^{2d-i}({X})_\tau^\vee. $$

\newpage 

\part{Chow decomposition of the diagonal and algebraic representatives}\label{P:Chow}

 \section{Preliminaries on algebraic representatives}
 \label{S:PreAlgRep}

In this section we review the notion of an algebraic representative.  In positive characteristic, this takes the role of the intermediate Jacobian.

 \subsection{Galois-equivariant regular homomorphisms and algebraic representatives}

 We start by reviewing the definition of a regular homomorphism and of an algebraic representative
 (\emph{i.e.},  \cite[Def.~1.6.1]{murre83} or \cite[2.5]{samuelequivalence}), as well as
 the notion of  a Galois-equivariant algebraic representative
 (\cite[Def.~4.2]{ACMVdcg}).

 Let $X$ be a smooth projective variety over an algebraically closed field $k$
 and let $n$ be a nonnegative integer.
 For a smooth separated scheme $T$ of finite type over $k$, we denote
 $$\mathscr A^n_{X/k}(T) := \{ Z\in \operatorname{CH}^n(T\times _k X)\ |  \ \forall t\in T(k), \ \mbox{
 the Gysin fiber $Z_t$ is algebraically trivial}\}$$ 
and for all $Z \in \mathscr A^n_{X/K}(T)$
  we denote by 
 $$w_{Z}:T(k)\to
 \operatorname{A}^n(X)$$
  the map defined by $w_Z(t)=Z_t$.  Given an abelian
 variety $A/k$, a \emph{regular homomorphism} (in codimension~$n$)
 $$\xymatrix{\phi:\operatorname{A}^n(X)\ar[r] & A(k)}$$ is a homomorphism of
 groups such that for every $Z\in \mathscr A^n_{X/k}(T)$, the
 composition
 $$
 \xymatrix{
  T(k) \ar[r]^{w_Z}& \operatorname{A}^n(X) \ar[r]^\phi& A(k)
 }
 $$
 is induced by a morphism of varieties 
 $$\psi_Z: T\to A.$$ 
 
  An \emph{algebraic
 representative} (in codimension $n$) is a regular homomorphism
 $$\phi^n_{X/k}:\operatorname{A}^n(X)\to \operatorname{Ab}^n_{X/k}(k)$$ that is
 initial among all regular homomorphisms (in codimension~$n$)\,; in particular if it exists then it is unique up to unique isomorphism. For $n=1$, the
 algebraic representative is given by
 $(\operatorname{Pic}^0_{X/k})_{\operatorname{red}}$ together with the
 Abel--Jacobi map.  For $n=d_X$, the algebraic representative is given by the
 Albanese variety and the Albanese map.  For $n=2$, it is a result of Murre
 \cite[Thm.~A]{murre83} that there exists an algebraic representative.\medskip

 We now review the extension in \cite{ACMVdcg} to the case of a smooth projective
 variety $X$ over a perfect field $K$.  Given an abelian variety $A/K$,  we say
 that a regular homomorphism $\phi:\operatorname{A}^n(X_{\bar K})\to A(\bar K)$
 is Galois-equivariant if it is equivariant with respect to the natural
 actions of $\operatorname{Gal}(K)$.
 We say an algebraic representative $\phi^n_{X_{\bar K}/\bar
  K}:\operatorname{A}^n(X_{\bar K})\to \operatorname{Ab}^n_{X_{\bar K}/\bar
  K}(\bar K)$ is Galois-equivariant if $ \operatorname{Ab}^n_{X_{\bar K}/\bar K}$
 descends to an abelian variety $\operatorname{Ab}^n_{X/K}$ defined over $K$ in
 such a way that  $\phi^n_{X_{\bar K}/\bar K}$ is a Galois-equivariant regular
 homomorphism.
 We show \cite[Thm.~4.4]{ACMVdcg} that if $X_{\bar K}$ admits an algebraic
 representative in codimension~$n$, $(\operatorname{Ab}^n_{X_{\bar K}/\bar
  K},\phi^n_{X_{\bar K}/\bar K})$,   then $\operatorname{Ab}^n_{X_{\bar K}/\bar
  K}$ descends uniquely to an abelian variety, denoted $\operatorname{Ab}^n_{X/ K}$, over $K$ making $\phi^n_{X_{\bar K}/\bar K}$ Galois-equivariant. 
   We also show that these are
 stable under Galois base change of field, as well as under algebraically closed
 base change of field.

 Importantly for our purposes, we show that for any Galois-equivariant regular
 homomorphism $\phi:\operatorname{A}^n(X_{\bar K}) \to A(\bar K)$, any  smooth
 separated scheme $T$ of finite type over $K$,  and any cycle class $Z\in
 \operatorname{CH}^n(T\times _K X)$ such that for every $t\in T(\bar K)$ the
 Gysin fiber $Z_t$ is algebraically trivial,
 the induced map $\psi_{Z_{\bar K}}:T_{\bar K}\to A_{\bar K}$ descends to a
 morphism $\psi_Z:T\to A$ of $K$-schemes.

\subsection{The functorial approach}
In \cite{ACMVfunctor} we have translated
the notion of a Galois-equivariant regular homomorphism into a functorial
language, which greatly clarifies many of the arguments in \cite{ACMVdcg},
allowing us to extend some of those results, and also discuss algebraic
representatives in families.  As we believe this is the correct language to use
going forward, we will use this notation in this paper.  Here we briefly review
the definition, referring the reader to \cite{ACMVfunctor} for details.  Over a
perfect field (which is the setting here), the functorial approach is entirely
equivalent to the notion of a Galois-equivariant regular homomorphism, and the
reader is free to simply interchange the notation throughout.
   
Fix a field $K$.
We start by defining the category of spaces that provide parameter spaces for
our cycles.
Specifically, we define
$$
\mathsf {Sm}/K
$$
to be the category with \emph{objects being smooth separated schemes of finite
	type over $K$}, and with \emph{morphisms being morphisms of $K$-schemes}.
Note that every morphism $t:T'\to T$ in $\mathsf {Sm}/K$ is lci  in the sense of
\cite[B.7.6]{fulton} (see \cite[B.7.3]{fulton}),
so that there is a refined Gysin pull-back $t^!$
\cite[\S 6.6]{fulton}.
The \emph{functor of codimension-$n$ algebraically trivial cycle classes on
	$X$ over $K$} is the contravariant functor
$$\mathscr {A}^n_{X/K}
:\mathsf {Sm}/K \longrightarrow \mathsf {AbGp}
$$
to the category of abelian groups $\mathsf {AbGp}$ given by families of
algebraically
trivial cycles on $X/K$.
Precisely, given $T$ in $\mathsf {Sm}/K$, we take  $\mathscr
{A}^i_{X/K}(T)$ to be the group of cycle classes $Z \in
\operatorname{CH}^i(T\times_K X)$ such that
 $Z_t \in \operatorname{CH}^i(X_{K^s})$ is algebraically trivial for some (equivalently, for any) separably closed point $t: \spec K^s \to T$\,; see \cite[\S 1.1]{ACMVfunctor}.
        The functor is defined on  morphisms $t:T'\to T$ in
$\mathsf {Sm}/K$  via  the refined Gysin pullback $t^!$ for lci
morphisms.

Let $A/K$ be an abelian variety, viewed via Yoneda as
the contravariant representable functor $\operatorname{Hom}(-,A):\mathsf
{Sm}/K\to \mathsf {AbGp}$.
A \emph{regular homomorphism in codimension~$n$ from $\mathscr A^n_{X/K}$ to
	$A/K$} is a natural transformation of functors
$$\Phi:\mathscr A^n_{X/K}\to A.$$
Here we parse the definition.  Given $T$ in $\mathsf {Sm}/K$, we obtain
$\Phi(T):\mathscr A^n_{X/K}(T) \to A(T)$\,; in other words, given a cycle class
$Z\in \mathscr A^n_{X/K}(T)$, \emph{i.e.}, a family of algebraically trivial cycle
classes on $X$ parameterized by $T$, we obtain a $K$-morphism $\Phi(T)(Z):T\to
A$. The regular homomorphism $\Phi$ is said to be \emph{surjective} if it is surjective on $K^{\sep}$-points, \emph{i.e.}, if 
$$\phi :=\Phi(K^{\sep}) : \mathscr A^n_{X/K}(K^{\sep})= \operatorname{A}^n(X_{K^{\sep}}) \to A(K^{\sep})$$ is surjective.
An  \emph{algebraic representative in codimension~$n$}  consists of an abelian
variety  $\operatorname{Ab}^n_{X/K}$ over~$K$ together with   a natural
transformation of  functors
$$\Phi^n_{X/K}:\mathscr {A}^n_{X/K} \to \operatorname{Ab}^n_{X/K}$$
over $\mathsf {Sm}/K$
that is initial among all regular homomorphisms $ \Phi:\mathscr
{A}^n_{X/K} \to A$.
    An algebraic representative, if it exists, is a surjective regular homomorphism~\cite[Prop.~5.1]{ACMVfunctor} and it is unique up to unique isomorphism.

\begin{rem}[Connection with Galois-equivariant regular homomorphisms]
	If $K$ is perfect, then regular homomorphisms and algebraic
        representatives in this sense are equivalent to
        Galois-equivariant regular homomorphisms and algebraic
        representatives over $\bar K$ (see \cite{ACMVfunctor}).  One
        translates the notation as follows.  Given a regular
        homomorphism $\Phi:\mathscr A^n_{X/K}\to A$, then
        $\phi=\Phi(\bar K):\mathscr A^n_{X/K}(\bar
        K)=\operatorname{A}^n(X_{\bar K})\to A(\bar K)$ is a
        Galois-equivariant regular homomorphism.  Conversely, given a
        Galois-equivariant regular homomorphism $\phi:
        \operatorname{A}^n(X_{\bar K})\to A(\bar K)$, then we define a
        regular homomorphism $\Phi:\mathscr A^n_{X/K}\to A$ as
        follows.  For $T$ in $\mathsf {Sm}/K$ and $Z\in \mathscr
        A^n_{X/K}(T)$, it is shown in \cite{ACMVfunctor} that there is
        a morphism $\Phi(T)(Z)=\psi_Z:T\to A$ of $K$-schemes
        determined by the map of $\bar K$-points given by $t\mapsto
        \phi(Z_t)$. The assignment on morphisms $T'\to T$ is made in
        the obvious way.
\end{rem}

\begin{rem}\label{R:SepBC/D}
We have shown in \cite[Thm.~1]{ACMVfunctor} that algebraic representatives satisfy base change and descent along separable field extensions.  More precisely, for a smooth projective variety $X$ over a  field $K$ and  a (not necessarily algebraic) separable field extension  $\Omega/K$, 
  an algebraic representative
  $\Phi^i_{X_\Omega} : \mathscr A^i_{X_\Omega/\Omega} \to
  \mathrm{Ab}^i_{X_\Omega/\Omega}$ exists if and only if an algebraic
  representative $\Phi^i_{X} : \mathscr A^i_{X/K} \to \mathrm{Ab}^i_{X/K}$
  exists.
  If this is the case, we have in addition that there is a canonical isomorphism $\mathrm{Ab}^i_{X_\Omega/\Omega} \stackrel{\sim}{\longrightarrow}
   (\mathrm{Ab}^i_{X/K})_\Omega$,  
and   $\Phi^i_{X_\Omega}(\Omega) : \mathrm{A}^i(X_\Omega) \to
   \mathrm{Ab}^i_{X_\Omega/\Omega}(\Omega)$ is
   $\mathrm{Aut}(\Omega/K)$-equivariant, relative to the above identification.  We will typically use this in the case where $\Omega/K$ is an extension of a perfect field $K$ by an algebraically closed field $\Omega$.
\end{rem}

\subsection{Miniversal and universal cycles}\label{SS:mini-universalCycle}
 Let $X$ be a smooth projective variety over a field $K$ and let $\Phi:\mathscr A^n_{X/K}\to A$ be a regular homomorphism. A \emph{miniversal cycle class} for $\Phi$ is a cycle $Z\in \mathscr{A}^n_{X/K}(A)$ such that the homomorphism $\psi_Z := \Phi(A)(Z) : A \to A$ is given by multiplication by $r$ for some natural number~$r$, which we call the \emph{degree} of $Z$. A miniversal cycle class is called \emph{universal} if $\psi_Z := \Phi(A)(Z) : A \to A$ is the identity morphism, \emph{i.e.}, if it is miniversal of degree one. In the case where  $\Phi$ is an algebraic representative for codimension-$n$ cycles on $X$, we call a universal cycle class for $\Phi$ a universal cycle in codimension-$n$ for $X$ (or for $\operatorname{Ab}^2_{X/K}$).

If $K$ is algebraically closed, it is a classical and crucial fact~\cite[1.6.2 \& 1.6.3]{murre83} that a miniversal cycle class exists  if and only if $\Phi$ is surjective\,; this also holds without any restrictions on the field~$K$ by~\cite[Lem.~4.7]{ACMVfunctor}. In particular, since an algebraic representative is always a surjective regular homomorphism  \cite[Prop.~5.1]{ACMVfunctor}, it always admits a miniversal cycle class. However, the existence of a universal cycle class is restrictive: Voisin~\cite{voisinUniv} established that the standard desingularization of the very general complex double quartic solid with 7 nodes does not admit a universal cycle class  in codimension~2. One of the main results of this paper, Theorem~\ref{T:Intro-QDS-1}, consists in extending Voisin's result to the positive characteristic case.

Nonetheless, recall \cite[\S 7.1]{ACMVfunctor} that  if $X$ is a smooth projective variety over a field $K$, then its first algebraic representative exists and it coincides with the reduced Picard scheme $(\mathrm{Pic}^0_X)_{\mathrm{red}}$\,; in addition if $X$ possesses a zero-cycle of degree-1 (\emph{e.g.}~if $K$ is finite or separably closed), then $X$ admits a universal cycle class in codimension~1.

\subsection{Regular homomorphisms and torsion}\label{SS:RegTors}
The following lemma is crucial.

\begin{lem}[{\cite[Prop. 11, Lem. p.259]{beauville83fourier}}]\label{L:commute}
	Let $A$ be an abelian variety  over $K$. The map $A(K^{\sep}) \to  \operatorname{A}_0(A_{K^{\sep}}),\ a \mapsto [a]-[0]$ is an isomorphism on torsion. In particular,  for any integer $N>1$, it sends $N$-torsion to $N$-torsion. \qed
\end{lem}

It admits the following consequence, which will be used in the proofs of Theorem~\ref{T:main-pol} and Proposition~\ref{P:vanish}.
\begin{lem}\label{L:ablambda}
 	Let $X$ be a smooth projective variety over a field $K$. If $Z \in \mathscr{A}^n_{X/K}(B)$ is a family of algebraically trivial cycles on $X$ parameterized by an abelian variety $B$ over $K$ with $Z_0 =0 \in \operatorname{A}^n(X)$, then $w_Z : B(K^{\sep}) \to \operatorname{A}^n(X_{K^{\sep}}), b\mapsto Z_b$ is a homomorphism on torsion. In particular, if $\Phi : \mathscr{A}^n_{X/K} \to A$ is a regular homomorphism, then for any prime $l$ we have 	
	 $T_l\psi_Z = T_l \Phi(K^{\sep}) \circ T_l w_Z : T_l B \to T_l A$.
\end{lem}
\begin{proof} That  $w_Z : B(K^{\sep}) \to \operatorname{A}^n(X_{K^{\sep}}), b\mapsto Z_b$ is a homomorphism on torsion follows simply from the fact that it factors through $B(K^{\sep}) \to  \operatorname{A}_0(B_{K^{\sep}}),\ b \mapsto [b]-[0]$ and from Lemma~\ref{L:commute}. 
	
	Note that given any  $Z \in \mathscr{A}^n_{X/K}(T)$ for any smooth variety $T$ over $K$, we do have $\Phi(T)(Z) =: \psi_Z = \Phi \circ w_Z$, where $w_Z : T \to \mathscr{A}^n_{X/K}$ is seen as a natural transformation. The assertion about Tate modules when $T$ is an abelian variety uses the above-established fact that in that case $w_Z(K^{\sep}) : B(K^{\sep}) \to \mathscr{A}^n_{X/K}(K^{\sep})$ is a homomorphism on torsion.
 \end{proof}

\begin{pro}\label{P:phizero}
	Let $X$ be a smooth projective variety over a field $K\subseteq \mathbb C$, and let $\Phi_{X/K}^n : \mathscr{A}^n_{X/K} \to \operatorname{Ab}^n_{X/K}$ be the algebraic representative of $X$ for $n=1,2$ or $\dim X$. Then $(\phi^n_X)_{\mathrm{tors}} : \operatorname{A}^n(X_{\bar K})_{\mathrm{tors}} \to  \operatorname{Ab}^n_{X/K}(\bar K)_{\mathrm{tors}}$ is an isomorphism. In particular  $T_\ell \phi^n_X : T_\ell \operatorname{A}^n(X_{\bar K}) \to T_\ell \operatorname{Ab}^n_{X/K}$ is an isomorphism for all primes $\ell$.
\end{pro}
\begin{proof} For $n=1$ this is a result of Bloch and for $n=\dim X$ this is a result of Roitman (see \emph{e.g.}, \cite{ACMVBlochMap} for references).  
For $n=2$, this is a direct result of \cite{murre83} over $\mathbb C$ and  \cite{ACMVfunctor} that algebraic representatives are stable under base change of field from $K$ to an algebraically closed field $\Omega$ containing $K$.
\end{proof}

In contrast, if $\operatorname{char}(K)>0$,  it is not known whether $(\Phi^2_{X/K})_{\mathrm{tors}} : \operatorname{A}^2(X_{K^{\sep}})_{\mathrm{tors}} \to \operatorname{Ab}^2_{X/K}(K^{\sep})_{\mathrm{tors}}$
is injective. The following proposition shows however that $(\Phi^2_{X/K})_{\mathrm{tors}}$ is surjective. In order to deduce that $T_\ell \Phi^2_{X/K}$ is surjective, one needs some more assumptions\,:

\begin{pro}\label{P:surj}
	Let $X$ be a smooth projective variety over a perfect field $K$ and let $\Phi : \mathscr{A}^n_{X/K} \to A$ be a surjective regular homomorphism. Denote $\phi:=\Phi(\bar K) : \operatorname{A}^n(X_{\bar K}) \to A(\bar K)$. Then
	\begin{enumerate}
\item $\phi_{\mathrm{tors}} : \operatorname{A}^n(X_{\bar K})_{\mathrm{tors}} \to A(\bar K)_{\mathrm{tors}}$ is surjective.
\item $\phi[l^\infty] : \operatorname{A}^n(X_{\bar K})[l^\infty] \to A(\bar K)[l^\infty]$ is surjective for all primes $l$.   

\item  $T_l\phi : T_l \operatorname{A}^n(X_{\bar K}) \to T_l A(\bar K)$ has finite cokernel for all primes $l$, and if $\Phi_{\bar K}$ admits a miniversal cycle class of degree $r$ coprime to $l$, then $T_l\phi$ is surjective.
	\end{enumerate}
\end{pro}
\begin{proof}
  Parts (1) and (2) are \cite[Lem.~3.2, Rem.~3.3]{ACMVdmij}\,; part (3) follows immediately from the definition of miniversality, and the fact that all surjective regular homomorphisms admit a miniversal cycle class of some degree.
       \end{proof}

\subsection{Regular homomorphisms induced by cycle classes}
\label{SS:RegHomCycle}
One of the difficulties in working with regular homomorphisms is that it is hard to construct non-trivial examples.  In this subsection we 
explain (\S \ref{SSS:Z_R_H_Def})  that given any abelian variety $A/K$ and any cycle class
 $
\widehat Z\in \operatorname{CH}^{d_X+1-n}( X\times_K \widehat A)
$
there is an 
induced
     a regular homomorphism
\begin{equation}\label{E:RegHomZ}
\widehat Z_* : \mathscr{A}^n_{X/K} \longrightarrow A.
\end{equation}
 While there is no guarantee that such  a regular homomorphism will be nonzero (any such regular homomorphism would be zero in the case  $n=1$ if  $H^1(X,\mathcal O_X)=0$, since in this case the algebraic representative $(\operatorname{Pic}^0_{X/K})_{\operatorname{red}}$ is trivial),  we nevertheless find this to be quite helpful in constructing regular homomorphisms. 
Conversely, regular homomorphisms induced by cycle classes enjoy properties that we exploit in \S \ref{S:self-duality}.

\subsubsection{}\label{SSS:Z_R_H_Def}
Let  $X$ and $Y$ be smooth projective varieties over a field $K$. Then any cycle class
 $$
Z\in \operatorname{CH}^{d_X+1-n}( X\times_K Y)
$$
induces a  regular homomorphism
\begin{equation}\label{E:RegHomZgen}
\xymatrix{
\Phi_{Z}=Z_*:\mathscr{A}^n_{X/K} \ar[r]^<>(0.5){Z_*}&  \mathscr{A}^1_{Y/K} \ar[r]^<>(0.5){AJ} &  (\operatorname{Pic}^0_{Y/K})_{\mathrm{red}},
}
\end{equation}
where $ AJ:\mathscr{A}^1_{Y/K} \to (\operatorname{Pic}^0_{Y/K})_{\mathrm{red}}$ is the Abel--Jacobi map to the first algebraic representative, and $Z_*:\mathscr A^n_{X/K}\to \mathscr A^1_{Y/K}$ is the canonical natural transformation induced by the correspondence $Z$.  
In other words, given  any $\Gamma\in \mathscr A^n_{X/K}(T)$, we have $
Z\circ  \Gamma \in \operatorname{CH}^1(T\times_K  Y)=\operatorname{A}^1(T\times_K Y)$, and  the theory of regular homomorphisms  (\emph{i.e.}, the Abel--Jacobi map) provides a morphism $T\to (\operatorname{Pic}^0_{Y/K})_{\operatorname{red}}$.  Note that on points, this sends a point $t$  of $T$ to the point corresponding to the line bundle associated to the divisor $(Z\circ \Gamma)_t=Z_*(\Gamma_t)$ on $Y$.  
 We observe that our notation overloads the use of $Z_*$, since it has two meanings in \eqref{E:RegHomZgen}\,; some motivation for this is that the Abel--Jacobi map is an isomorphism when evaluated on $\bar K$-points, and the meaning of $Z_*$ should always be clear from the context.
 For $Z\in \operatorname{CH}^{d_X+1-n}( Y\times_K X)$, we  denote by $Z^*$ the regular homomorphism $({}^tZ)_*$ induced by the transpose~${}^tZ$. 

If we apply this construction to the case $Y=\widehat A$ for an abelian variety $A/K$ (and switch notation $Z=\widehat Z$), we obtain the regular homomorphism \eqref{E:RegHomZ}.

 \begin{que}\label{Q:cycle}
	Given a regular homomorphism $\Phi:\mathscr A^n_{X/K}\to A$, when is $\Phi$ induced by a cycle\,?  In other words, when is there a cycle class $\widehat Z\in \operatorname{CH}^{d_X+1-n}( X\times _K \widehat A)$ with $\Phi=\widehat Z_*$\,?
\end{que}

If $X$ admits a zero-cycle of degree $1$, the question has a positive answer for $n=\dim X$\,:

\begin{lem}\label{L:AlbIndZ}
Let $X$ be a smooth projective variety over a field $K$ admitting a zero-cycle of degree $1$. 
  The algebraic representative $\Phi^{d_X}_{X/K}: \mathscr A^{d_X}_{X/K}\to \operatorname{Alb}_{X/K}$ is induced by the universal codimension-$1$ cycle class $Z\in \operatorname{CH}^1((\operatorname{Pic}^0_{X/K})_{\operatorname{red}}\times_KX)$\,; \emph{i.e.}, $\Phi^{d_X}_{X/K}=Z^*:={}^tZ_*$.  
    \end{lem}

\begin{proof}
This is explained in \cite[Thm.~7.9(i) and Rem.~7.4]{ACMVfunctor}.
\end{proof}

We will see in Proposition~\ref{P:unicycle} that the question also has a positive answer for codimension-2 cycles on stably rational threefolds over an algebraically closed field\,; this is a crucial step towards establishing Theorem~\ref{T:Intro-CanPol0}.

\subsubsection{}\label{SSS:alb}
  The main point of the following discussion is to prove Lemma \ref{L:RHCorB}, regarding regular homomorphisms induced by correspondences.
We start by recalling some basic facts concerning the Picard scheme. Let $X$ and $Y$ be smooth projective varieties over a field
$K$ admitting $K$-points $x_0$ and $y_0$ respectively. Then there is a canonical isomorphism
     \begin{equation}\label{eq:canon}
\frac{\operatorname{CH}^1(X\times_K Y)}{p_X^*\operatorname{CH}^1(X) + p_Y^*\operatorname{CH}^1(Y)} = \operatorname{Hom}(\operatorname{Alb}_{X/K},\operatorname{Pic}^0_{Y/K})\ \ (
=\operatorname{Hom}(\operatorname{Alb}_{Y/K},\operatorname{Pic}^0_{X/K})),
\end{equation}
which is induced from the homomorphism sending a line-bundle $\mathcal L$ on $X\times_K Y$ with trivial restriction on $\{x_0\}\times_K Y$ to the unique homomorphism $\operatorname{Alb}_{X/K}\to \operatorname{Pic}^0_{Y/K}$ induced by the morphism $X\to \operatorname{Pic}^0_{Y/K}$ given on points by $x \mapsto \mathcal{L}|_{\{x\} \times_K Y}$. 
Note that the homomorphism $\operatorname{Alb}_{X/K}\to \operatorname{Pic}^0_{Y/K}$  above is the one induced from the regular homomorphism $\mathscr{A}^{d_X}_{X/K} \to \operatorname{Pic}^0_{Y/K}$ by the universal property of the albanese map, considered as an algebraic representative for zero-cycles on $X$.
Note also that the second equality of \eqref{eq:canon} is simply given by taking the dual homomorphism.
Finally, we note that 
$\operatorname{Hom}(\operatorname{Alb}_{X/K},\operatorname{Pic}^0_{Y/K})=\operatorname{Hom}(\operatorname{Alb}_{X/K},(\operatorname{Pic}^0_{Y/K})_{\operatorname{red}})$, since $\operatorname{Alb}_{X/K}$, being an abelian variety, is reduced.

Now consider the case where $X=B$ is an abelian variety of dimension $g$. 
If $$Z \in \operatorname{CH}^1(B\times_KY ),$$ then identifying $B=\operatorname{Alb}_{B/K}$, the homomorphism $B \to (\operatorname{Pic}^0_{Y/K})_{\operatorname{red}}$ induced by~\eqref{eq:canon} coincides with the homomorphism $\Phi(B)(\Delta_B-(B\times \{0\})):B\to (\operatorname{Pic}^0_{Y/K})_{\operatorname{red}}$ where $\Phi$ is the regular homomorphism 
$$\Phi = Z_* : \mathscr{A}^g_{B/K} \to \mathscr{A}^1_{Y/K} \to  (\operatorname{Pic}^0_Y)_{\operatorname{red}}$$
of \eqref{E:RegHomZgen} with $n=g$, 
and $$\Delta_B \in \operatorname{CH}^{g}(B\times_K B)$$ is the  family of dimension-$0$ cycles given by the diagonal.
 Indeed, it is enough to check that these two homomorphisms agree on $K^{\sep}$-points. We have a commutative diagram
$$\xymatrix{
	B(K^{\sep}) \ar[r] \ar[rd]_{\operatorname{id}} & \operatorname{A}_0(B_{K^{\sep}}) \ar[d]^{\operatorname{alb}} \ar[r]^{Z_*} & \operatorname{A}^1(Y_{K^{\sep}}) \ar[d]^{AJ}  \\
	& B(K^{\sep}) \ar[r] & \operatorname{Pic}^0_Y(K^{\sep})
}$$
where the left triangle commutes because $\mathrm{alb}([x]-[0]) = x$ by definition of the albanese map, and where the bottom horizontal arrow is the homomorphism induced by the fact that $\operatorname{alb}$ is an algebraic representative for $\mathscr{A}^g_{B/K}$.  The latter coincides, by construction, with the homomorphism induced by the canonical isomorphism  \eqref{eq:canon}.  

A special case we will use often is\,:

\begin{lem}\label{L:RHCorB}
Let $X$ be a smooth projective variety over a field $K$, and let $A$ and $B$ be abelian varieties over $K$.  Given correspondences  $Z\in \operatorname{CH}^{n}( B\times _K  X)$ and $\widehat Z\in \operatorname{CH}^{d_X+1-n}( X\times _K \widehat A)$, the regular homomorphism
$$
\xymatrix{
\Phi: \mathscr A^{d_B}_{B/K} \ar[r]^<>(0.5){Z_*}&  \mathscr A^n_{X/K} \ar[r]^<>(0.5){\widehat Z_*}& A
}
$$
when evaluated at the abelian variety $B$ and the cycle $\Delta_B-(B\times_K \{0\})\in \mathscr A^{d_B}_{B/K}(B)$ gives a homomorphism
\begin{equation}\label{E:RHCorB0}
\Phi(B)(\Delta_B-(B\times_K \{0\})): B\longrightarrow A,
\end{equation}
which agrees with the homomorphism induced from the correspondence $\widehat Z\circ  Z\in \operatorname{CH}^1(B\times_KA)$ via \eqref{eq:canon}, 
and which on $\bar K$ points factors as 
\begin{equation}\label{E:L:RHCorB}
\xymatrix{
B(\bar K)\ar[r]^<>(0.5){w_Z}& A^n(X_{\bar K}) \ar[r]^<>(0.5){\widehat Z_*} & A(\bar K).
}
\end{equation}
Moreover, the dual homomorphism $\widehat A\to \widehat B$ to \eqref{E:RHCorB0} is induced by ${}^tZ\circ {}^t\widehat Z$.  
\end{lem}

\begin{proof}
Everything except the assertion \eqref{E:L:RHCorB}  follows from the discussion above.  For \eqref{E:L:RHCorB}, we simply note that from the definitions, there is a factorization 
$$
\xymatrix{
B(\bar K)\ar[r]^{w_Z} \ar[rd]_{w_{\Delta_B-(B\times_K \{0\})}}& \operatorname{A}^n(X_{\bar K}) \ar[r]^{\widehat Z_*} & A(\bar K).\\
& \operatorname{A}^{d_B}(B_\bar K) \ar[u]_{Z_*}
}
$$

\end{proof}

 \section{Preliminaries on Chow decomposition of the diagonal}
 \label{S:PreDecDiag}

 In this section we recall the definition of a Chow
 decomposition of the diagonal, and the connection with universal $\operatorname{CH}_0$-triviality.  The purpose is primarily to fix terminology, since the terminology is somewhat fluid in the literature.  In particular, we  consider here decompositions of the diagonal that are slightly more general than those equivalent to universal $\operatorname{CH}_0$-triviality, and we wish, as well,  to keep track of the exact multiple of the diagonal that may admit a decomposition.\medskip

  Recall that for a scheme $X$ of finite type over a field $K$, we say that a
 cycle class $Z \in \operatorname{CH}^n(X)$ is supported on a closed subscheme
 $W\subseteq X$ if it is in the kernel of the restriction map
 $\operatorname{CH}^n(X)\to \operatorname{CH}^n(X\smallsetminus W)$.
 Equivalently, from the exact  sequence
 \begin{equation}\label{E:ChowSupp}
  \operatorname{CH}^n(W)\to  \operatorname{CH}^n(X)\to
 \operatorname{CH}^n(X\smallsetminus W)\to 0,
 \end{equation}
 we can say that $Z$ is \emph{supported on $W$} if and only if it is the push forward of
 a cycle class on $W$.

 \begin{dfn}[Chow decomposition of a cycle class]
  Let $R$ be a commutative ring.
  Let $X$ be a  smooth projective variety
  over a  field $K$,
  and
  let
  $$
  \xymatrix{j_i:W_i\ar@{^{(}->}[r]^{\quad \ne}& X}, \ \ i=1,2
  $$
  be reduced closed subschemes not containing any component of $X$.
  An  \emph{$R$-decomposition   of  type $(W_1,W_2)$ of  a cycle  class $Z\in
   \operatorname{CH}^{d_X}(X\times_KX)_R$ }
  is an equality
  \begin{equation}\label{E:DefDcp}
  Z =Z_1+Z_2\in \operatorname{CH}^{d_X}(X\times _K X)_R
  \end{equation}
  where $Z_1\in \operatorname{CH}^{d_X}(X\times _K X)_R$ is supported on
  $W_1\times_KX$ and $Z_2\in \operatorname{CH}^{d_X}(X\times _K X)_R$ is supported
  on $X\times_K W_2$.   When $R=\mathbb Z$,  we call this a  \emph{decomposition
   of  type $(W_1,W_2)$}.

We say that $Z\in
\operatorname{CH}^{d_X}(X\times_KX)_R$ has an \emph{$R$-decomposition of type $(d_1,d_2)$} if it admits an $R$-decomposition of type $(W_1,W_2)$ with $\dim W_1 \leq d_1$ and $\dim W_2 \leq d_2$.
    \end{dfn}

 \begin{rem}\label{R:dW1dW2}
We note here for convenience that if $X$ is smooth and projective, a decomposition of a multiple $N\Delta_X$  of the diagonal must have  $d_{1}+d_{2}\ge d_X-1$. Indeed, if
  $d_{1}+d_{2}< d_X-1$, then a short argument using \eqref{E:Zj*diagH}, below, would imply that
  $N\cdot \mathcal{H}^2(X) = 0$, giving a contradiction as the class of any ample line bundle on $X$ is not torsion. 
 \end{rem}

 In what follows, let
 $$
 \operatorname{pr}_i:X\times_K X\to X,  \ \ i=1,2
 $$
 be the respective projection maps.

 \begin{exa}
  For projective space $\mathbb P^r_K$, the diagonal class is  $\Delta_{\mathbb
   P^r_K}=\sum_{i=0}^{r-1}\operatorname{pr}_1^*[H]^i\times
  \operatorname{pr}_2^*[H]^{r-i}\in \operatorname{CH}^r(\mathbb P^r_K\times
  _K\mathbb P^r_K)$, where $[H]$ is the class of a hyperplane in $\mathbb P^r_K$.
  Thus for any non-negative integers $d_{1},d_{2}$ with
  $d_{1}+d_{2}=r-1$, the cycle class
  $\Delta_{\mathbb P^r_K}\in \operatorname{CH}^r(\mathbb P^r_K\times _K\mathbb
  P^r_K)$  has a decomposition of type $(W_1,W_2)$ with
  $W_i\subseteq \mathbb P^n_K$  a  linear space of dimension $d_{i}$, $i=1,2$.
 \end{exa}

 In many situations, we will want to specify a more restricted type of
 decomposition, which is common in the literature due to its connection with
 universal $\operatorname{CH}_0$-triviality (see Remark~\ref{R:CH0triviality} below). 

 \begin{dfn}[Strict  decomposition of a cycle  class]
    A \emph{strict $R$-decomposition  of  a cycle class $Z\in
   \operatorname{CH}^{d_X}(X\times_KX)_R$} is an $R$-decomposition of type $(d_X-1,0)$. In other words, it
  is an equality 
$  Z =Z_1+Z_2\in \operatorname{CH}^{d_X}(X\times _K X)_R$ as in \eqref{E:DefDcp} 
where $Z_1$is supported on $D\times_K
X$ for some codimension-$1$ subvariety $D\subseteq X$ and $Z_2$ is supported
on $X\times_K W_2$ for some $0$-dimensional subvariety $W_2\subseteq X$.   When $R=\mathbb Z$,  we call this a  \emph{strict decomposition}.
 \end{dfn}

 \begin{rem}[Strict decomposition of the diagonal]  
 	\label{R:ChThEqDef}
 If for some integer $N>0$ we have $N \Delta_X =Z_1+Z_2\in \operatorname{CH}^{d_X}(X\times _K X)_R$ is a strict $R$-decomposition  of  $N$ times the diagonal, then the image of the degree map $\chow_0(X)_R \to R$ contains $NR$ and $Z_2 = \operatorname{pr}_2^*\alpha$ for any zero-cycle $\alpha \in \chow_0(X)$ of degree~$N$. Indeed, by definition of a strict decomposition, we must have $Z_2 = \operatorname{pr}_2^*\alpha$ for some zero-cycle $\alpha \in \chow_0(X)_R$. Letting $\Delta_X$ act on zero-cycles on $X$, and since $Z_1$ acts trivially on zero-cycles on $X$, we find that for all $\beta \in \chow_0(X)_R$ we have $N\beta = (\Delta_X)_* \beta = (\operatorname{pr}_2^*\alpha)_* \beta = \deg(\beta)\alpha$. 
 It follows that $\deg(\alpha)=N$ and that any zero-cycle of degree~$N$ is rationally equivalent to $\alpha$.
 In particular, in the situation where $R=\integ$ and  $X(K)\ne \emptyset$, if $\Delta_X$ admits a strict decomposition, then $Z_2 = X\times_K x$ for any $K$-point $x\in X(K)$.
                   \end{rem}

 \begin{rem}[Universal $\operatorname{CH}_0$-triviality] \label{R:CH0triviality}
  A proper variety $X$ over a field $K$ is said to be \emph{universally
   $(\operatorname{CH}_0)_R$-trivial} if, for any field extension $L/K$,  the degree map
  $\operatorname{CH}_0(X_L )_R \to   R$ is an isomorphism.  When $R=\mathbb Z$, we simply say \emph{universally
  	$\operatorname{CH}_0$-trivial}.
   It follows classically from \cite{BlSr83} that a smooth
  proper variety $X$ over a field $K$ is universally
  $(\operatorname{CH}_0)_R$-trivial if and only if the class of the diagonal admits a strict  $R$-decomposition.
   \end{rem}

 \begin{rem}[Stably rational varieties and decomposition of the diagonal]\label{R:StabRatDec}
  It is well-known (see, \emph{e.g.}, \cite[Ex.~16.1.11]{fulton}) that universal $\operatorname{CH}_0$-triviality is a stable
  birational invariant for smooth proper varieties.  Thus a stably rational proper variety is universally
  $\operatorname{CH}_0$-trivial.  As a consequence, for a stably rational
  smooth projective variety $X$, we have that $\Delta_X\in
  \operatorname{CH}^{d_X}(X\times _KX)$ admits a strict decomposition.
 \end{rem}

 \begin{rem}[Rationally chain connected varieties and decomposition of the
  diagonal]\label{R:RCDecDiag}
  Let $X/K$ be a smooth projective rationally chain connected variety 
  over a field $K$.
  From say \cite[Thm.~IV.3.13]{kollar}, we have that $X_{\Omega}$ is  $\chow_0$-trivial for every algebraically closed field $\Omega/K$, and therefore that 
 $X$ is universally $(\chow_0)_\rat$-trivial. It follows that some nonzero integer multiple of the diagonal $N\Delta_X\in \operatorname{CH}^{d_X}(X\times_K X)$ admits
  a strict decomposition.
       \end{rem}

\begin{rem}[Unirational varieties and decomposition of the diagonal]\label{R:UniRDecDiag}
Let $X$ be a smooth projective  unirational variety over a field $K$.  As $X$ is rationally chain connected,
  it follows that some nonzero integer multiple of the diagonal $N\Delta_X\in \operatorname{CH}^{\dim X}(X\times_K X)$ admits
  a strict decomposition.
In fact, if  $\mathbb P^n\dashrightarrow X$ is a dominant rational map of degree $N$, 
 then $\operatorname{CH}_0(X)_{\mathbb Z[1/N]}$ is universally trivial, so that 
  $\Delta_{X}\in
  \operatorname{CH}^{d_X}( X\times _K  X)$ admits a strict
  $\mathbb Z[\frac{1}{N}]$-decomposition (Remark \ref{R:CH0triviality}).
However, if 
either $\operatorname{char}(K)=0$, or $K$ is perfect and $d_X\le 3$, then we obtain the stronger result that $N\Delta_X\in \operatorname{CH}^{d_X}(X\times_K X)$ admits
  a strict decomposition\,; this is well-known, and we direct the reader to the \emph{proof} of \cite[Prop.~2.2]{mboro}.
\end{rem}

 \begin{rem}[Uniruled varieties and decomposition of the
  diagonal]\label{R:UniDecDiag}
  Let $X$ be a smooth projective geometrically uniruled variety over a
  field $K$.  It is clear that $\chow_0(X)_\rat$ is universally supported on a subvariety $W_2$ of dimension
  $d_X-1$, in the sense that the push-forward map $\chow_0((W_2)_L)_\rat \to \chow_0(X_L)_\rat$ is surjective for all field extensions $L/K$. By \cite{BlSr83}, we see that $\Delta_X\in
  \operatorname{CH}^{\dim X}(X\times_K X)_\rat$  admits  a $\rat$-decomposition of type
  $(W_1,W_2)$ with $\dim W_1 \leq d_X-1$. Clearing denominators,  there is some nonzero integer multiple $N\Delta_X\in
  \operatorname{CH}^{\dim X}(X\times_K X)$ that admits  a decomposition of type
  $(W_1,W_2)$.
       \end{rem}

\section{Factoring correspondences with given support} \label{S:FactChow}
 
A key tool we will use in what follows is the fact that a decomposition of the
 diagonal, viewed from the perspective of correspondences, gives a factorization
 of the identity into maps involving lower dimensional varieties.  Here we
 consider the situation in the various cohomology theories.  The key point is
 that correspondences require intersection theory, and therefore we prefer to
 work on smooth spaces.

 \subsection{A factorization lemma for correspondences}
 \label{SS:FactLemma}

Let $X$ be a scheme of finite type over a field $K$. We say that $X$ has dimension $\leq d$ if all irreducible  components of $X$ have dimension $\leq d$, while we say $X$ has pure dimension $d$ if all irreducible  components of $X$ have dimension $d$.

\begin{lem}\label{L:factorcorres}
	Let  $X$ and $Y$ be connected smooth proper varieties over a field $K$ of characteristic exponent~$p$, and let
	$j:W\hookrightarrow X$ be  a closed subscheme of dimension $\leq n$.  If $Z\in
	\operatorname{CH}_c(X\times_KY)$ is a cycle class of dimension $c$ supported on $W\times_K
	Y$, then $Z$, seen as a correspondence from $X$ to $Y$ with $\integ[\frac{1}{p}]$-coefficients, factors through a scheme $\widetilde W$ which is smooth proper and of finite type over a finite purely inseparable extension of  $K$ and of pure dimension $n$.
	More precisely, there exist a nonnegative integer $e$ and  correspondences $\widetilde{Z} \in \operatorname{CH}_c( \widetilde W\times_K Y)$ and $\tilde \gamma \in \operatorname{CH}^{\dim X}(X \times_K \widetilde{W} )$
	such that $$p^e\ Z = \widetilde{Z}\circ \tilde \gamma \quad \mbox{in}\ \operatorname{CH}_c(X\times_KY).$$
	In addition, assuming $K$ is perfect and that resolution of singularities holds in dimensions $\leq n$ over $K$, the correspondence $Z$ factors as above with $e=0$ and with $\widetilde W$ smooth proper of finite type over $K$ of pure dimension $n$.
\end{lem}
\begin{proof} It is clearly sufficient to establish the lemma for a cycle class $Z$ that is the class of an integral closed subscheme of $X\times_K Y$ that, by abuse of notation, we still denote $Z$. Replacing $W$ with the (scheme-theoretic) image of $Z$ inside $X$ via the first projection, 
 we may assume that $W$ is integral and that $Z$ dominates~$W$.
	Let us consider an alteration $\tau:  W' \to W$ (\emph{i.e.}, $\tau$ proper and generically finite) of $W$ with $ W'$ smooth and proper over a finite purely inseparable extension $L$ of $K$. By  \cite[Thm.~4.3.1]{temkin17} we may in fact choose $\tau$ of degree $p^e$ for some nonnegative integer~$e$, while we can choose $\tau$ to be of degree $1$ if $K$ is perfect and resolution of singularities holds in dimensions $\leq n$ (which is the case for $n=3$ in positive characteristic  \cite[Thm.~p.1893]{CPRes2} and for any $n$ in zero characteristic).
	Since the push-pull along the field extension $L/K$ is multiplication by $[L:K]$, which in our case is a power of $p$, we may by pulling back $Z$ to $X_L\times_L Y_L = (X\times_KY)_L$ assume that $L=K$.
	We now define
	$\tilde j:=j\circ \tau \circ p_{W'}:\widetilde W\to
	X$, where $p_{W'} : \widetilde W := \mathbb{P}_L^{n-\dim W} \times_L W' \to W'$ is the natural projection, and we set $\tilde \gamma$ to be the transpose of the class of the graph of $\tilde j$.    We claim that there is a cycle class  $\widetilde Z \in \operatorname{CH}_c(\widetilde W\times_K
	Y)$  such that
	\begin{equation}\label{E:MZpush}
		(\tilde j\times
		\operatorname{Id}_Y)_*\widetilde Z=p^e\ Z.
	\end{equation}
	Together with the identity $(\tilde j\times
	\operatorname{Id}_Y)_*\widetilde Z =  \widetilde{Z}\circ \tilde \gamma $ (see \cite[16.1.1]{fulton}),
 	 	we obtain the sought-after cycles $\tilde \gamma$ and $\widetilde Z$.
	To establish the claim,  consider the fibered product diagram
	$$
	\xymatrix{
		W'\times_K Y\ar[r]^{\tau \times \operatorname{Id}_Y} & W\times_KY \\
		W'^\circ\times_K Y  \ar[r]^{\tau\times \operatorname{Id}_Y} \ar@{^(->}[u]& W^\circ \times_K Y \ar@{^(->}[u]
	}
	$$
	where $W^\circ$ is the smooth locus of $W$, and $W'^\circ$ is the pre-image of $W^0$ under $\tau$.  Since $W^\circ$, $W'^\circ$, and $W'$ are all smooth, all of the morphisms in the diagram admit lci factorizations in the sense of Fulton (see \cite[Note, p.439]{fulton})\,;
	 	thus all the morphisms admit refined Gysin pull-backs.  We further restrict $W^\circ$ (and $W'^\circ$) so that $W'^\circ\to W^\circ$ is finite and flat (it is generically finite and generically \'etale).
	With this set-up, we set $Z'$ to be the closure in $W'\times _KY $ of the flat pull back of $Z$ along the composition $W'^\circ\times_K Y\to W^\circ \times_K Y\hookrightarrow W\times_KY$, where $Z$ is considered as a cycle on $W\times_K Y$\,; the assertion  $(\tilde j\times
	\operatorname{Id}_Y)_*Z'=p^eZ$ then follows from functoriality of pull-back, and the fact that the push-forward of a pull-back along a finite flat morphism is multiplication by the degree.
	Indeed, we are assuming that $Z$ dominates $W$. Since $W'^\circ\times_K Y\to W^\circ \times_K Y\hookrightarrow W\times_KY$ is a composition of flat morphisms, we may use flat pull-back.  Taking the closure in $W'\times_KY$ we obtain a class in $W'\times_KY$.  Now by definition of the push-forward, and the fact that $Z$ dominates $W$,  we can compute this on $W'^\circ \times_KY\to W^\circ \times_KY$ giving the result.

	Finally, we define $\widetilde Z := \{0\}\times Z' \in \operatorname{CH}_c((\mathbb{P}_L^{n-\dim W} \times_L W') \times_K Y)$ and we clearly have $(p_{W'} \times \operatorname{Id}_Y)_* \widetilde Z = Z'$ and hence $(\tilde j \times \operatorname{Id}_Y)_* \widetilde Z = p^eZ$.
\end{proof}

 \subsection{Factorizations of morphisms induced by correspondences}

 With the set-up of Lemma~\ref{L:factorcorres} and its proof, we can factor the action of the correspondence $p^eZ$ in various settings as follows.

 \subsubsection{Chow groups}
 The correspondence $p^eZ^*:\operatorname{CH}^n(X)\to \operatorname{CH}^n(X)$,
 $\alpha\mapsto p^e(\operatorname{pr}_1)_*(\operatorname{pr}_2^*\alpha \cdot  Z)$,
 and the correspondence  $p^eZ_*:\operatorname{CH}^n(X)\to \operatorname{CH}^n(X)$,
 $\alpha\mapsto p^e(\operatorname{pr}_2)_*(\operatorname{pr}_1^*\alpha \cdot  Z)$,
 factor, respectively, as\,:
 \begin{equation}\label{E:Zj*diag}
 \xymatrix@R=1em{
  &\operatorname{CH}^{d_{\widetilde W}-d_X+n}(\widetilde W) \ar[rd]^<>(0.5){\tilde
   j_*}&&\operatorname{CH}^n(X) \ar[rd]_<>(0.5){\widetilde Z_*} \ar[rr]^{p^eZ_*}
  &&\operatorname{CH}^n(X) \\
  \operatorname{CH}^n(X) \ar[ru]^<>(0.5){\widetilde Z^*} \ar[rr]^{p^eZ^*}
  &&\operatorname{CH}^n(X)& &\operatorname{CH}^{n}(\widetilde W)
  \ar[ru]_<>(0.5){\tilde j^*}&\\
 }
 \end{equation}

 \subsubsection{Algebraic representatives} We
 obtain  factorizations similar to \eqref{E:Zj*diag}  if we consider
 algebraically trivial cycle classes.
 This induces diagrams of  algebraic representatives, if they exist\,:

 \begin{equation}\label{E:Z2*di-Ab}
 \xymatrix@R=.7em{
  &\mathscr A^{d_{\widetilde  W}-d_X+n}_{\widetilde W/K} \ar[rd]^{\tilde j_*}
  \ar@{-}[d]&&\mathscr A^n_{X/K} \ar[rr]^{p^eZ_*} \ar[dd]_{\Phi^n_{X/K}}
  \ar[rd]^<>(0.5){\tilde j^*}&&\mathscr A^n_{X/K} \ar[dd]^{\Phi^n_{X/K}}\\
     \mathscr A^n_{X/K} \ar[ru]^{\widetilde Z^*} \ar[rr]^<>(0.3){p^eZ^*}
  \ar[dd]_{\Phi^n_{X/K}}&\ar[d]^<>(0.3){\Phi^{d_{\widetilde  W}-d_X+n}_{\widetilde
    W/K}}&\mathscr A^n_{X/K}\ar[dd]^{\Phi^n_{X/K}}&&\mathscr A^n_{\widetilde W/K}
  \ar[ru]^{\widetilde Z_*} \ar[dd]^<>(0.2){\Phi^n_{\widetilde W/K}}&\\
     &\operatorname{Ab}^{d_{\widetilde  W}-d_X+n}_{\widetilde
   W/K}\ar@{-->}[rd]^{\tilde j_*}&&\operatorname{Ab}^n_{X/K} \ar@{--}[r]^{p^eZ_*}
  \ar@{-->}[rd]_{\tilde j^*}&\ar@{-->}[r]&\operatorname{Ab}^n_{X/K}\\
     \operatorname{Ab}^n_{X/K} \ar@{-->}[ru]^{\widetilde Z^*}
  \ar@{-->}[rr]^{p^eZ^*}&&\operatorname{Ab}^n_{X/K}&&\operatorname{Ab}^n_{\widetilde
   W/K} \ar@{-->}[ru]_{\widetilde Z_*}&\\
 }
 \end{equation}
 where the dashed arrows are induced by the universal property of the algebraic
 representative.

 \subsubsection{Cohomology groups}
 In the same situation at the start of \S\ref{S:FactChow}, fix a Weil cohomology theory $\mathcal H^\bullet$ with coefficient ring $R_{\mathcal H}$, and a ring homomorphism $R\to R_{\mathcal H}$.
   Then the correspondences $p^eZ^*:\mathcal H^n(X)\to \mathcal H^n(X)$  and
 $p^eZ_*:{\mathcal H}^n(X)\to {\mathcal H}^n(X)$ factor, respectively, as\,:

 \begin{equation}\label{E:Zj*diagH}
 \xymatrix@R=1em{
  &\mathcal H^{2(d_{\widetilde W}-d_X)+n}(\widetilde W)(d_{\widetilde W}-d_X)
  \ar[rd]^<>(0.5){\tilde j_*}&& {\mathcal H}^n(X) \ar[rd]_<>(0.5){\tilde j^*}
  \ar[rr]^{p^eZ_*}
  &&{\mathcal H}^n(X) \\
  \mathcal H^n(X) \ar[ru]^<>(0.5){\widetilde Z^*} \ar[rr]^{p^eZ^*}
  &&\mathcal H^n(X) &&{\mathcal H}^{n}(\widetilde W) \ar[ru]_<>(0.5){\widetilde
   Z_*}&\\
 }
 \end{equation}
 Note that in the case where $n$ is even, these diagrams are  compatible with the
 cycle class maps and the diagrams \eqref{E:Zj*diag}.

 \subsubsection{Abel--Jacobi maps}
 Consider now the case where $K\subseteq \mathbb C$.  Since correspondences
 induce morphisms of integral Hodge structures, we obtain  factorizations similar
 to \eqref{E:Z2*di-Ab}. Here $\chow^*(-)_{\operatorname{hom}}$ denotes the kernel of the cycle class map $\chow^*(-) \to H^{2*}(-,\integ(*))$.
 \begin{equation}\label{E:Z2*di-AbCC}
 \xymatrix@R=.7em@C=.001em{
  &\operatorname{CH}^{d_{\widetilde  W}-d_X+n}(\widetilde
  W^{\an})_{\operatorname{hom}}\ar[rd]^<>(0.8){\tilde j_*}
  \ar@{-}[d]&&\operatorname{CH}^n(\operatorname{X}^{\an})_{\operatorname{hom}}
  \ar[rr]^{p^eZ_*} \ar[dd]_{AJ} \ar[rd]^<>(0.8){\tilde
   j^*}&&\operatorname{CH}^n(X^{\an})_{\operatorname{hom}} \ar[dd]^{AJ}\\
     \operatorname{CH}^n(X^{\an})_{\operatorname{hom}}  \ar[ru]^<>(0.3){\widetilde Z^*}
  \ar[rr]^<>(0.3){p^eZ^*} \ar[dd]_{AJ}&\ar[d]^<>(0.3){
   AJ}&\operatorname{CH}^n(X^{\an})_{\operatorname{hom}}\ar[dd]^{AJ}&&\operatorname{CH}^n(\widetilde
  W^{\an})_{\operatorname{hom}}  \ar[ru]^<>(0.2){\widetilde Z_*} \ar[dd]^<>(0.2){AJ}&\\
     &J^{2(d_{\widetilde  W}-d_X+n)-1}(\widetilde W^{\an}) \ar@{->}[rd]^<>(0.8){\tilde  j_*}&&J^{2n-1}(X^{\an}) \ar@{-}[r]^{\qquad p^eZ_*} \ar@{->}[rd]_{\tilde
   j^*}&\ar@{->}[r]&J^{2n-1}(X^{\an})\\
     J^{2n-1}(X^{\an})\ar@{->}[ru]^<>(0.3){\widetilde Z^*}
  \ar@{->}[rr]^{p^eZ^*}&&J^{2n-1}(X^{\an})&&J^{2n-1}(\widetilde W^{\an})
  \ar@{->}[ru]_{\widetilde Z_*}&\\
 }
 \end{equation}

 In \cite{ACMVdmij}, we have defined a
 distinguished model $J^{2n-1}_{a,X/K}$ of the image of the Abel--Jacobi map on
 algebraically trivial cycle classes $AJ:\operatorname{A}^n(X^{\an})\to
 J^{2n-1}(X^{\an})$.  From the functoriality statement of
 \cite[Prop.~5.1]{ACMVdmij}, we  obtain commutative diagrams\,:

 \begin{equation}\label{E:Z2*di-AbDM}
 \xymatrix@R=.7em{
  &\mathscr A^{d_{\widetilde  W}-d_X+n}_{\widetilde W/K} \ar[rd]^{\tilde j_*}
  \ar@{-}[d]&&\mathscr A^n_{X/K} \ar[rr]^{p^eZ_*} \ar[dd]_{AJ}
  \ar[rd]^<>(0.5){\tilde j^*}&&\mathscr A^n_{X/K} \ar[dd]^{AJ}\\
     \mathscr A^n_{X/K} \ar[ru]^{\widetilde Z^*} \ar[rr]^<>(0.3){p^eZ^*}
  \ar[dd]_{AJ}&\ar[d]^<>(0.3){AJ}&\mathscr A^n_{X/K}\ar[dd]^{AJ}&&\mathscr
  A^n_{\widetilde W/K} \ar[ru]^{\widetilde Z_*} \ar[dd]^<>(0.2){AJ}&\\
     &J^{2(d_{\widetilde  W}-d_X+n)-1}_{a,\widetilde W/K}\ar@{->}[rd]^{\tilde
   j_*}&&J^{2n-1}_{a,X/K} \ar@{-}[r]^{p^eZ_*} \ar@{->}[rd]_{\tilde
   j^*}&\ar@{->}[r]&J^{2n-1}_{a,X/K}\\
     J^{2n-1}_{a,X/K} \ar@{->}[ru]^{\widetilde Z^*}
  \ar@{->}[rr]^{p^eZ^*}&&J^{2n-1}_{a,X/K}&&J^{2n-1}_{a,\widetilde
   W/K}\ar@{->}[ru]_{\widetilde Z_*}&\\
 }
 \end{equation}
  Note that the morphisms of complex tori in \eqref{E:Z2*di-AbCC} and
 \eqref{E:Z2*di-AbDM} are induced by the morphisms in cohomology
 \eqref{E:Zj*diagH} with $\mathcal H^\bullet (-)=H^\bullet ((-)^{\an},\mathbb Z)$.

\subsection{Chow decomposition of the diagonal and existence of algebraic representatives}
As already hinted at in \cite[\S 6.2]{ACMVdcg}, the existence of certain Chow decompositions of the diagonal implies
 the existence of algebraic representatives.

\begin{pro}[Existence of  algebraic representatives]
	\label{T:ExistAlgRep}
	 	Let $X$ be a smooth projective variety over a perfect field $K$ and let $n$ be a positive integer.
	 	Assume that $\Delta_X \in \chow^{d_X}(X\times_K X)_\rat$ admits a  decomposition of type $(d_1,d_2)$ with
		$d_1\le d_X-(n-1)$ and $d_2\le n-1$.
	Then  there is an algebraic
	representative $\Phi^n_{X/K}:\mathscr A^n_{X/K}\to
	\operatorname{Ab}^n_{X/K}$.
\end{pro}

\begin{proof}
 	By Lemma~\ref{L:factorcorres}, we may write $\Delta_X = \widetilde{Z}_1\circ \tilde{\gamma}_1 + {}^t\tilde \gamma_2 \circ {}^t\widetilde{Z}_2$, with $\widetilde{Z}_i \in \operatorname{CH}_{d_X}(\widetilde{W}_i \times_K X)_\rat$ and $\tilde{\gamma}_i \in \operatorname{CH}^{d_X}(X\times_K \widetilde W_i)_\rat$ for some smooth projective varieties $\widetilde{W}_i$ of pure dimension $d_i$ with $d_1 = d_X-n+1$ and $d_2 = n-1$. 
		By \eqref{E:Zj*diag} with rational coefficients, this decomposition provides a surjective homomorphism $\operatorname{A}^1((\widetilde{W}_1)_{\bar K}) \to \operatorname{A}^n(X_{\bar K})$ induced by a correspondence, and we conclude the existence of the algebraic representative  by Saito's criterion (\emph{e.g.}, \cite[Prop.~5.3]{ACMVfunctor}), as in the proof of \cite[Prop.~6.7]{ACMVdcg}.
\end{proof}

\section{Chow decomposition and self-duality of the algebraic representative}
\label{S:self-duality}

The aim of this section is to prove the auto-duality statement of Theorem~\ref{T:Intro-CanPol0} in the case where the threefold is assumed to be geometrically stably rational\,; this is Theorem \ref{T:main-pol}(2).  We start with a motivic result (Theorem \ref{T:motive}), which allows us to show that threefolds admitting a strict decomposition of the diagonal have universal regular homomorphisms for codimension-$2$ cycle classes that are themselves induced by a cycle class (Proposition \ref{P:unicycle}).  With this we prove Theorem \ref{T:main-pol}. 
  
\subsection{A motivic statement} 
To start with, we consider a commutative unital ring $R$ and we consider the category $\mathfrak{M}_{K,R}$ of pure Chow motives over~$K$ with $R$-coefficients as described in \cite[\S 4]{andre}.  Denote by $\mathfrak{h}(X)_R$ the Chow motive of $X$ with $R$-coefficients.
Here is a general proposition, which is a version of \cite[Thm.~2.1]{VialCK} with $R$-coefficients.

\begin{pro}\label{effective}
	Let $X$ be a smooth projective variety of pure dimension $d$ over a perfect field $K$, and let $p \in \operatorname{Hom}_{\mathfrak{M}_{K,R}}(\mathfrak{h}(X)_R, \mathfrak{h}(X)_R):= \operatorname{CH}^d(X\times_KX)_R$ be an idempotent correspondence with the property  that $p_*\operatorname{CH}_0(X_{L})_R
	= 0$ for all field extensions $L/K$. 
	Assume either that the exponential characteristic of $K$ is invertible in $R$ or that resolution of singularities holds in dimensions  $\leq d-1$. Then there exist a smooth projective variety $Y$ over $K$ of
	dimension $d -1$ and an idempotent $q \in \operatorname{CH}^{d-1}(Y \times_K
	Y)_R$ such that $(X,p,0) \simeq (Y,q,-1)$ in ${\mathfrak{M}_{K,R}}$.
\end{pro}

\begin{proof}
	 	Let us denote $X_i$ the connected components of $X$ and $p_{ij} \in \mathrm{Hom}_{\mathfrak{M}_{K,R}}(\mathfrak{h}(X_i)_R,\mathfrak{h}(X_j)_R)$ the $(i,j)$-component of $p$. By assumption, if $\eta_i$ denotes the
	generic point of $X_i$, then we have $p_*[\eta_i]=0$ and hence $(p_{ij})_*[\eta_i] =0$ for all $j$.  But $(p_{ij})_*[\eta_i]$
	is the restriction of $p_{ij} \in \operatorname{CH}^d(X_i \times_K X_j)_R$ to $\varinjlim \operatorname{CH}_d(U
	\times_K X_j)_R = \operatorname{CH}_0((X_j)_{k(\eta_i)})_R$, where the limit is taken over all open
	subsets $U$ of $X_i$. Therefore, by the localization exact sequence
	for Chow groups, there exist for all $j$ a proper closed subset $D_{ij} \subset X_i$
	and a correspondence $\gamma_{ij} \in \operatorname{CH}_d(D_{ij} \times_K X_j)_R$ such that
	$\gamma_{ij}$ maps to $p_{ij}$ via the inclusion $D_{ij} \times_K X_j \to X_i \times_K X_j$.
	In other words, $p_{ij}$ is supported on $D_{ij} \times_K X_j$. By Lemma~\ref{L:factorcorres},
	 	 	 	 	 	 	 	 	we get  a factorization $p_{ij} = r_{ij} \circ s_{ij}$,
	where $r_{ij}\in \operatorname{CH}_d(Y_{ij} \times_K X_j)_R$ and $s_{ij}
	 	\in \operatorname{CH}^d(X_i\times_K Y_{ij})_R$ and where $Y_{ij}$ is smooth projective over $K$ of pure dimension $d-1$.
	If we now consider the correspondences $r = \sqcup r_{ij}$ and $s = \sqcup s_{ij}$, we have
	 	$p = r\circ s$. Moreover,  the correspondence $q:=s \circ r \circ s
	\circ r = s \circ p \circ r \in \operatorname{CH}^{d-1}(Y \times_K Y)_R$ is an idempotent, and  $p
	\circ r \circ q \circ s \circ p = p$ as well as $q \circ s \circ p
	\circ r \circ q = q$.  These last two equalities exactly mean that
	$p \circ r \circ q$, seen as a morphism of Chow motives with $R$-coefficients from
	$(Y,q,-1)$ to $(X,p,0)$, is an isomorphism, with inverse $q \circ s
	\circ p$.
\end{proof}

\begin{teo}\label{T:motive}
	Let $X$ be a (geometrically) connected smooth projective threefold over a perfect field $K$. Assume that $\operatorname{CH}_0(X)_R$ is universally spanned by a degree-1 zero-cycle $x\in \operatorname{CH}_0(X)_R$. Set
	$$\Gamma := \Delta_X - x\times_K X - X\times_K x \quad \in \operatorname{CH}^3(X\times_KX)_R.$$
	Then there exist a smooth projective curve $C$ over $K$, an idempotent correspondence $p \in \operatorname{CH}^1(C\times_K C)_R$ and an isomorphism of Chow motives with $R$-coefficients
	$$(X,\Gamma) \simeq (C,p,-1).$$
	Concretely, there exist a smooth projective curve $C$ over $K$ and correspondences $\alpha\in \operatorname{CH}^2(X\times_KC)_R$ and $\beta \in \operatorname{CH}^2(C\times_KX)_R$ such that $\Gamma = \beta \circ \alpha$.
\end{teo}
\begin{proof} First note that $\Gamma$ does not depend on the choice of the degree-1 zero-cycle $x$.
	 	That $X$ admits a decomposition of the diagonal implies not only that $\Gamma_* \operatorname{CH}_0(X_L)_R = 0$ but also that $\Gamma^* \operatorname{CH}_0(X_L)_R = 0$ for all field extensions $L/K$. Since resolution of singularities holds for surfaces over perfect fields,  we may apply Proposition~\ref{effective} and obtain a smooth projective surface $S$ together with an idempotent $q \in \operatorname{CH}^2(S\times_K S)_R$ such that the Chow motive with $R$-coefficients $(S,q,-1)$ is isomorphic to $(X,\Gamma,0)$. Since $\Gamma^* \operatorname{CH}_0(X_L)_R = 0$ for all field extensions $L/K$, we find that $q^*\operatorname{CH}_0(S_L)_R = 0$ for all field extensions~$L/K$. Applying Proposition~\ref{effective} to the motive $(S,{}^tq,0)$, we obtain a smooth projective curve $C$ and an idempotent $p \in \operatorname{CH}^1(C\times_K C)_R$ such that $(S,{}^tq,0)$ is isomorphic to $(C,{}^tp,-1)$. Dualizing we get that $(S,q,-1)$ is isomorphic to $(C,p,-1)$, thereby concluding the proof.
\end{proof}

\subsection{Proof of auto-duality in Theorem~\ref{T:Intro-CanPol0}} As a first consequence of Theorem~\ref{T:motive}, one obtains information on the second algebraic representative for threefolds admitting a decomposition of the diagonal. The key feature of the following proposition is that for such threefolds the second algebraic representative is induced by an algebraic cycle.
\begin{pro}\label{P:unicycle}
	Let  $X$ be a smooth projective threefold over a field $K$ that is either finite or algebraically closed. Assume that there exists a natural number $N$ such that $N\Delta_X$ admits a strict decomposition. Then the second algebraic representative
	$$\Phi^2_{X/K} : \mathscr{A}^2_{X/K} \to \operatorname{Ab}^2_{X/K}$$ admits a degree-$N$ miniversal cycle $Z\in \operatorname{A}^2(\operatorname{Ab}^2_X \times_K X)$. Moreover, $\phi^2_{X_{\bar K}} := \Phi^2_{X/K}(\bar K) : \operatorname{A}^2(X_{\bar K}) \to \operatorname{Ab}^2_{X/K}(\bar K)$ is an isomorphism
 	and there is a nonnegative integer $d$ such that
 	 $N^d \Phi^2_{X/K}$ 
 	is induced by a cycle $\widehat{Z} \in \operatorname{CH}^2({\widehat{\operatorname{Ab}}\,^2_X}\times_K X)$, \emph{i.e.}, with the notation of \S \ref{SS:RegHomCycle}
 	$$N^d \Phi^2_{X/K} = \widehat{Z}^* :  \mathscr{A}^2_{X/K} \to \operatorname{Ab}^2_{X/K}.$$
\end{pro}
\begin{proof} First, we note that more generally, $\phi^2_{X_{\bar K}}$ is an isomorphism under the weaker hypothesis that $X$ is a smooth proper variety of any dimension whose diagonal $\Delta_X\in \operatorname{CH}^{d_X}(X\times_K X)\otimes \rat$ admits a $\rat$-decomposition of type $(d_X-1,1)$\,; see Proposition~\ref{P:BlSr}. Likewise, $\Phi^2_{X/K}$ admits a degree-$N$ miniversal cycle under the weaker hypothesis that $X$ is a smooth proper variety of dimension $\leq 4$ such that $N\Delta_X\in \operatorname{CH}^{d_X}(X\times_K X)$ admits a decomposition of type $(d_X-1,1)$\,; see Theorem~\ref{T:UnivCyc}.
In fact, in the aforementioned two results, it is enough to assume the existence of a \emph{cohomological} decomposition (we discuss this notion in \S \ref{S:CohDD}).
	 Hence the key feature of Proposition~\ref{P:unicycle}, \emph{i.e.}, requiring the hypothesis of a strict integral (Chow) decomposition of $N$ times the diagonal,  is that we establish that  $N^d\Phi^2_{X/K}$ is induced by a cycle $\widehat Z$ for some nonnegative integer $d$.

	It remains to prove that there is a nonnegative integer $d$ such that
	$N^d \Phi^2_{X/K}$ 
	is induced by a cycle.
	Via Theorem~\ref{T:motive} with $R=\integ[1/N]$,
	the proof reduces to the case of codimension-1 cycles on curves. Indeed, with the notation of Theorem~\ref{T:motive} and after clearing out denominators, we have a commutative diagram
	$$\xymatrix{\mathscr{A}^2_{X/K} \ar[d]^{\Phi^2_{X/K}} \ar[r]^{\alpha} & 	\mathscr{A}^1_{C/K}  \ar[d]^{AJ}_{\simeq} \ar[r]^{\beta}   & 	\mathscr{A}^2_{X/K} \ar[d]^{\Phi^2_{X/K}}\\
		\operatorname{Ab}^2_{X/K} \ar[r]^f &
		\operatorname{Pic}^0_{C/K} \ar[r]^g
		& \operatorname{Ab}^2_{X/K}
	}$$
	where $\alpha \in \operatorname{CH}^2(X\times_K C)$ and $\beta \in \operatorname{CH}^2(C\times_K X)$ are \emph{integral} correspondences such that $\alpha \circ \beta = N^dp$ and $\beta \circ \alpha = N^d\Gamma$ for some nonnegative integer $d$, and where $f$ and $g$ are the (unique) homomorphisms induced by the universal property of the algebraic representatives.
	Since the integral correspondences $N^d (x\times_K X)$ and $N^d (X\times_K x)$ act as zero on $\mathscr{A}^2_{X/K}$, the integral correspondence $N^d\Gamma$ acts as multiplication by $N^d$ on~$\mathscr{A}^2_{X/K}$. 
	Thus, if  $\mathrm{alb}_C : C \to \operatorname{Pic}^0_C$ denotes the $K$-morphism $c \mapsto \mathcal{O}_C(c-c_0)$ for any choice of zero-cycle $c_0$ of degree $1$ on $C$ (which exists if $K$ is finite or algebraically closed) and if $\mathcal{P}_{\operatorname{Ab}^2_X}$ denotes the universal line-bundle on $\operatorname{Ab}^2_{X/K}\times_K \widehat{\operatorname{Ab}}\,^2_{X/K}$, then,  by Lemma~\ref{L:AlbIndZ},  $N^d \Phi^2_{X/K}$ is induced by the correspondence $\mathcal{P}_{\operatorname{Ab}^2_X} \circ g \circ \mathrm{alb}_C \circ \alpha$.
\end{proof}

We are now in a position to prove the auto-duality statement of Theorem~\ref{T:Intro-CanPol0}, which we formulate in the more precise form of Theorem \ref{T:main-pol}, below.  The basic observation is that when $X$ is a threefold, then given any miniversal cycle $Z\in \mathscr A^2(\operatorname{Ab}^2_{X/K}\times_K X)$, the symmetric correspondence ${}^tZ\circ Z\in \operatorname{CH}^1(\operatorname{Ab}^2_{X/K}\times_K {\operatorname{Ab}}^2_{X/K})$ induces a morphism 
$\operatorname{Ab}^2_{X/K}\to \widehat{\operatorname{Ab}}\,^2_{X/K}$ that is symmetric.    
We return to investigating such morphisms later, in \S \ref{S:DiagPolCoh}, in more generality.  Here we focus on the case of threefolds under the assumption that the diagonal admits a strict (Chow) decomposition, so that we can use Proposition \ref{P:unicycle}, which implies the stronger 
 conclusion that the given morphism is an isomorphism.

\begin{teo}[Auto-duality]\label{T:main-pol}
 	Let  $X$ be a smooth projective threefold over a perfect field $K$ and let $\Omega/K$ be an algebraically closed field extension. 
 	\begin{enumerate}
\item Let $Z\in \operatorname{CH}^2(\operatorname{Ab}^2_{X_{\Omega}/\Omega} \times_{\Omega} X_{\Omega})$ be a miniversal cycle over $\Omega$, of degree, say, $r$. Let $N$ be a natural number. Assume that $N\Phi^2_{X_\Omega/\Omega}$ is induced by a cycle $\widehat{Z} \in \operatorname{CH}^2({\widehat{\operatorname{Ab}}\,^2_{X_\Omega}}\times_\Omega X_\Omega)$.
Then the symmetric $\Omega$-homomorphism
$$ \operatorname{Ab}^2_{X_{\Omega}/\Omega} \longrightarrow  {\widehat{\operatorname{Ab}}\,^2_{X_{\Omega}/\Omega}}
$$ induced by
${}^tZ \circ Z \in \operatorname{CH}^1(
\operatorname{Ab}^2_{X_{\Omega}/\Omega} \times_{\Omega}
\operatorname{Ab}^2_{X_{\Omega}/\Omega})$ is an isogeny with kernel
contained in the torsion subscheme $\Ab^2_{X_\Omega/\Omega}[Nr^2]$.  In
particular, the degree of this isogeny divides $(Nr^2)^{2g}$
where $g$ is the dimension of $
\operatorname{Ab}^2_{X_{\Omega}/\Omega}$.

\item 
Assume that $\Delta_{X_{\Omega}}$ admits a strict decomposition.
Let  $Z\in \operatorname{CH}^2(\operatorname{Ab}^2_{X_{\Omega}/\Omega} \times_{\Omega} X_{\Omega})$ be a universal cycle over $\Omega$, the existence of which is provided by Proposition~\ref{P:unicycle}. Then the symmetric $\Omega$-homomorphism 
$$ \operatorname{Ab}^2_{X_{\Omega}/\Omega} \longrightarrow  {\widehat{\operatorname{Ab}}\,^2_{X_{\Omega}/\Omega}}
$$ induced by
${}^tZ \circ Z \in \operatorname{CH}^1( \operatorname{Ab}^2_{X_{\Omega}/\Omega} \times_{\Omega} \operatorname{Ab}^2_{X_{\Omega}/\Omega})$ is an isomorphism that descends to a symmetric $K$-isomorphism 
$$
\Lambda_X : \ \operatorname{Ab}^2_{X/K} \longrightarrow {\widehat{\operatorname{Ab}}\,^2_{X/K}}
$$
independent of the choice of the universal cycle $Z$. 
 	\end{enumerate}
 \end{teo}

\begin{proof} 
From Chow rigidity, and descent for regular homomorphisms along separable field extensions \cite[Thm.~1]{ACMVfunctor} (see Remark \ref{R:SepBC/D}), we may immediately reduce to the case $\Omega = \bar K$.

(1)	 We have on $\bar K$-points a commutative diagram
	\begin{equation}\label{E:stardiag}
	\xymatrix@R=1.5em{
	{\widehat{\operatorname{Ab}}\,^2_{X_{\bar K}/\bar K}}(\bar K) \ar[rd]^{w_{\widehat Z}} \ar[rr]^{Nr} & & {\widehat{\operatorname{Ab}}\,^2_{X_{\bar K}/\bar K}}(\bar K) \\
	& \operatorname{A}^2(X_{\bar K}) \ar[ur]^{Z^*} \ar[dr]^{\phi^2_{X_{\bar K}}} & \\
	{\operatorname{Ab}^2_{X_{\bar K}/\bar K}}(\bar K) \ar[ru]^{w_{Z}} \ar[rr]^{r} & & {\operatorname{Ab}^2_{X_{\bar K}/\bar K}}(\bar K) \ar@{-->}[uu]_F
}
	\end{equation}
	where $N\phi^2_{X_{\bar K}}=\widehat Z^*$ and where $F : \operatorname{Ab}^2_{X_{\bar K}/\bar K} \to {\widehat{\operatorname{Ab}}\,^2_{X_{\bar K}}}$ is the homomorphism induced by the universal property of $\Phi^2_{X/K}$. The bottom horizontal arrow is by definition $\psi_Z := \phi^2_{X_{\bar K}}\circ w_Z$ and is multiplication by~$r$ because $Z$ is a degree-$r$ miniversal cycle for $\Phi^2_{X/K}$.
	The top horizontal arrow $Z^*\circ w_{\widehat Z}$ is multiplication by~$Nr$. Indeed, by  Lemma~\ref{L:RHCorB}, it is induced by the cycle ${}^tZ \circ \widehat Z$ and is the dual of the homomorphism  $\operatorname{Ab}^2_{X_{\bar K}/\bar K} \to {\operatorname{Ab}^2_{X_{\bar K}/\bar K}}$ induced by the transpose ${}^t\widehat Z \circ Z$, which in turn is nothing but $N\psi_Z$, as it coincides with $\widehat Z^*\circ w_Z = N\phi^2_{X_{\bar K}} \circ w_Z$. In particular, applying Lemma~\ref{L:RHCorB} to the composition $Z^*\circ w_Z$, and using the bottom horizontal arrow, we see that $rF$ is induced by~${}^tZ\circ Z$.

	Now, since $F\circ \psi_{\widehat Z} =Nr\cdot  \mathrm{id} : {{\widehat{\operatorname{Ab}}\,^2_{X_{\bar K}}}} \to {\widehat{\operatorname{Ab}}\,^2_{X_{\bar K}}}$, and since an abelian variety and its dual have the same dimension, it is clear that $F$, and hence $rF$, is an isogeny. 
	In addition, we find that the kernel of $F$ is contained in the torsion subscheme $\Ab^2_{X_\Omega/\Omega}[Nr]$ and hence that the kernel of $rF$  is contained in the torsion subscheme $\Ab^2_{X_\Omega/\Omega}[Nr^2]$.

(2) By Proposition~\ref{P:unicycle}, $\Phi^2_{X_{\bar K}/\bar K}$ admits a universal cycle $Z\in \operatorname{CH}^2(\operatorname{Ab}^2_{X_{{\bar K}}/{\bar K}} \times_{{\bar K}} X_{{\bar K}})$ and is induced by a cycle $\widehat Z\in \operatorname{CH}^2(\widehat{\operatorname{Ab}}\,^2_{X_{{\bar K}}/{\bar K}} \times_{{\bar K}} X_{{\bar K}})$. By point (1), the symmetric $\bar K$-homomorphism 
$\operatorname{Ab}^2_{X_{{\bar K}}/{\bar K}} \longrightarrow  {\widehat{\operatorname{Ab}}\,^2_{X_{{\bar K}}/{\bar K}}}
$ induced by
${}^tZ \circ Z$ is an isomorphism and it coincides with the $\bar K$-homomorphism $F$ of~\eqref{E:stardiag}.
We now proceed to show that the isomorphism $F$ descends to a homomorphism $\Lambda_X$ over $K$ and is independent of the choice of a universal cycle $Z$ for $\Phi^2_{X/K}$.  The starting point is that $\Phi^2_{X/K} : \operatorname{A}^2(X_{\bar K}) \to \operatorname{Ab}^2_{X_{\bar K}/\bar K}(\bar K)$ is an isomorphism that is $\operatorname{Gal}(K)$-equivariant by \cite{ACMVdcg}, so that its inverse $w_Z$ is also $\operatorname{Gal}(K)$-equivariant.
	In order to show that $Z^*\circ w_Z$ is $\operatorname{Gal}(K)$-equivariant and independent of~$Z$, it suffices to show that the induced map on Tate modules $Z^*\circ w_Z : T_\ell \operatorname{Ab}^2_{X_{\bar K}/\bar K} \to T_\ell {\widehat{\operatorname{Ab}}\,^2_{X_{\bar K}/\bar K}}$ is $\operatorname{Gal}(K)$-equivariant and independent of $Z$, for some prime $\ell \neq \operatorname{char} K$. 
	However, the isomorphism $Z^* : T_\ell \operatorname{A}^2(X_{\bar K}) \to T_\ell {\widehat{\operatorname{Ab}}\,^2_{X_{\bar K}}}$ is the dual of the $\operatorname{Gal}(K)$-equivariant
	 isomorphism $(T_\ell\Phi^2_{X/K})^{-1}  : T_\ell {\operatorname{Ab}^2_{X_{\bar K}/\bar K}} \to T_\ell \operatorname{A}^2(X_{\bar K})$, where  $T_\ell \operatorname{A}^2(X_{\bar K})$ 
	 is identified with its dual via the $\operatorname{Gal}(K)$-equivariant isomorphism 
	$T_\ell \lambda^2 : T_\ell \operatorname{A}^2(X_{\bar K}) \to H^3(X_{\bar K},\integ_\ell(2))$ provided by \cite[Prop.~5.2]{ACMVBlochMap} (see also Proposition~\ref{P:lambdacoho} below),  and the $\operatorname{Gal}(K)$-equivariant perfect pairing given by the intersection product $H^3(X_{\bar K},\integ_{\ell}(2)) \times H^3(X_{\bar K},\integ_{\ell}(2)) \to \integ_{\ell}(1)$, via the following commutative diagram
	\begin{equation}\label{E:tZcircZ}
	\xymatrix{T_\ell \operatorname{Ab}^2_{X/K} \ar@/^2pc/[rr]^{w_Z= (\Phi^2_{X/K})^{-1}} \ar[r] \ar[rd]_{\operatorname{id}} & T_\ell\operatorname{A}_0(\operatorname{Ab}^2_{X_{\bar K}/\bar K}) \ar[d]_{T_\ell\lambda^0}^{\simeq} \ar[r]^{Z_* } & T_\ell\operatorname{A}^2(X_{\bar K}) \ar[r]^{Z^*} \ar[d]_{T_\ell\lambda^2}^{\simeq} & T_\ell \operatorname{A}^1(\operatorname{Ab}^2_{X_{\bar K}/\bar K}) \ar[d]_{T_\ell\lambda^1}^{\simeq} \\
		& T_\ell\operatorname{Ab}^2_{X/K} \ar[r]^{Z_*} & H^3(X_{\bar K},\integ_\ell(2)) \ar[r]^{Z^*} & T_\ell {\widehat{\operatorname{Ab}}\,^2_{X/K}}
	}
	\end{equation}
	where the left triangle commutes thanks to \S \ref{SSS:alb} together with Lemma~\ref{L:ablambda} and the fact that the Bloch map $\lambda^0$ coincides with the albanese map on $\ell$-primary torsion~\cite[Prop.~3.9]{bloch79}.
\end{proof}

\begin{rem} In Theorem~\ref{T:main-pol}(2),
note that although the homomorphism  $ \operatorname{Ab}^2_{X_{\bar K}/\bar K} \to {\widehat{\operatorname{Ab}}\,^2_{X_{\bar K}/\bar K}}$ descends to $K$, the cycle $Z\in \operatorname{A}^2(\operatorname{Ab}^2_{X_{\bar K}/\bar K} \times_{\bar K} X_{\bar K})$ might not be defined over $K$.
\end{rem}

\begin{rem}
Note that the proof of Theorem~\ref{T:main-pol}(2) realizes $\operatorname{Ab}^2_{X_{\bar K}/\bar K}$ as a direct factor of  $\operatorname{Pic}^0_{C/\bar K}$ for some smooth projective curve $C$ over $\bar K$. Although $\operatorname{Pic}^0_{C/\bar K}$ is principally polarizable, this is not enough to conclude that $\operatorname{Ab}^2_{X_{\bar K}/\bar K}$ is principally polarizable and hence is isomorphic to its dual. Indeed, any abelian variety is the direct summand of a principally polarizable abelian variety\,; this can be seen for instance by Zarhin's trick which states that, for any abelian variety $A$, the abelian variety $(A\times_K \widehat A)^4$ is principally polarizable, while $A$ need not be principally polarizable.
\end{rem}

  As an application of Theorem~\ref{T:main-pol}(2), we can refine \cite[Thm.~15 \& 6.4]{ACMVBlochMap} in the case of stably rational threefolds:

\begin{cor}\label{C:MazurCor}
	 	Let $X$ be a smooth projective stably rational threefold
	over a field $K$ that is either finite or algebraically closed. Then there  exist correspondences $Z\in \operatorname{CH}^{2}(\operatorname{Ab}^2_X \times _K
	X)$ and $Z' \in \operatorname{CH}_{2}(X\times_K \operatorname{Ab}^2_X )$
	inducing for all primes $l$
	inverse isomorphisms
	$$
	\begin{CD}
	Z_*:\ T_l \operatorname{Ab}^2_X @>\cong>> H^{3}(X_{\bar K}, \mathbb Z_l(2))  \quad \text{and} \quad Z'_*:\ H^{3}(X_{\bar K}, \mathbb Z_l(2))  @>\cong>> T_l\operatorname{Ab}^2_X
	\end{CD}
	$$
	of $\operatorname{Gal}(K)$-modules.
\end{cor}

\begin{proof} By Proposition~\ref{P:unicycle}, let $Z \in \operatorname{CH}^{2}(\operatorname{Ab}^2_X \times _K
	X)$ be a universal codimension-2 cycle for $X$.
	By \cite[Thm.~6.4]{ACMVBlochMap}
	 	and the fact that resolution of singularities holds for surfaces over perfect fields, $Z$ induces for all primes $l$ an isomorphism $Z_*:\ T_l \operatorname{Ab}^2_X \stackrel{\simeq}{\longrightarrow} H^{3}(X_{\bar K}, \mathbb Z_l(2))$. 
	 	With $\Theta : \operatorname{Ab}^2_X \to {\widehat{\operatorname{Ab}}\,^2_X}$ the canonical $K$-isomorphism induced by $Z^*\circ Z_*$ provided by Theorem~\ref{T:main-pol}, we get that $Z'_* := \Theta^{-1}\circ Z^* $ provides the inverse to $Z_*$.
\end{proof}

\section{Specialization and polarization on the algebraic representative}

       Given an abelian variety $A$ over a field $K$, recall that a symmetric isogeny $\Theta : A\to \widehat A$ is called a \emph{polarization} if there exists an ample symmetric line-bundle $L$ on $A_{\bar K}$ such that  $\Theta_{\bar{K}} : A_{\bar K} \to \widehat{A}_{\bar K}$ is given by $a \mapsto t_a^*L\otimes L^{-1}$, where $t_a : A_{\bar K} \to A_{\bar K}$ is the translation-by-$a$ morphism\,; see also~\S \ref{S:ThetaLB}.
 In characteristic zero, by Hodge theory, the isomorphism $\Theta_X=-\Lambda_X$ in Theorem \ref{T:main-pol} agrees with the canonical principal polarization induced by the cup-product on $H^3(X_{\cx},\integ)$ (Theorem \ref{T:CanPol3-fold} and Remark~\ref{R:T:main-polHdg}).
Although there are abelian varieties $A$ that are isomorphic to their dual but are not principally polarizable (we thank Bas Edixhoven for explaining an example to us), we are led to make the following conjecture\,:

\begin{con}[Canonical polarization]\label{conjecture}
	Let  $X$ be a smooth projective threefold over a perfect field $K$ and let $\Omega/K$ be an algebraically closed field extension. Assume that $\Delta_{X_{\Omega}}$ admits a strict decomposition (\emph{i.e.}, $X_\Omega$ is universally $\operatorname{CH}_0$-trivial).
	Then the canonical symmetric $K$-isomorphism 
	$\Theta_X: \operatorname{Ab}^2_{X/K} \to {\widehat{\operatorname{Ab}}\,^2_{X/K}}$, where $\Theta_X=-\Lambda_X$ (Theorem \ref{T:main-pol}) 	is a principal polarization.
\end{con}

With this as motivation, we now establish some results regarding specialization and Chow decomposition. 
For that purpose, let $\mathcal X \to S$ be a smooth projective morphism, where $S$ is the spectrum of a DVR with generic point $\eta$ and closed point $s$. We denote  $\bar \eta$ and $\bar s$ algebraic closures of $\eta$ and $s$, respectively, and we denote $\mathcal X_\eta$ and $\mathcal X_s$ the generic fiber and the closed fiber of $\mathcal X \to S$, respectively. \medskip

We start with the following basic result about polarizations\,:

\begin{lem}[Polarizations and specializations for algebraic representatives] \label{L:PolSpec}
  If $\Theta_\eta:\operatorname{Ab}^2_{\mathcal X_\eta/\eta}\to \widehat{\operatorname{Ab}}\,^2_{\mathcal X_\eta/\eta}$ is a degree-$d$ isogeny, then $\operatorname{Ab}^2_{\mathcal X_{ \eta}/\eta}$ and $\Theta_{\mathcal X_\eta}$ extend to a degree-$d$ isogeny $\Theta_{{\mathcal X}}:\operatorname{Ab}^2_{{\mathcal X}/S}\to \widehat{\operatorname{Ab}}\,^2_{{\mathcal X}/S}$ of abelian $S$-schemes, and the following are equivalent: $\Theta_{{\mathcal X}}$ is a polarization,   $\Theta_{{\mathcal X}_\eta}$ is a polarization, $\Theta_{{\mathcal X}}|_s$ is a polarization.  
 \end{lem}

\begin{proof}
The fact that $\operatorname{Ab}^2_{{\mathcal X}_{ \eta}/\eta}$ extends to an abelian scheme $\operatorname{Ab}^2_{{\mathcal X}/S}$  is  \cite[Thm.~8.3]{ACMVfunctor}, and essentially follows directly from the Ogg--N\'eron--Shafarevich criterion, and the fact that the second Bloch map is injective (\cite[Thm.~8.3]{ACMVfunctor} 
makes the stronger assertion that this extension is the relative algebraic representative for $\mathcal X/S$, which we do not use here).   The fact that $\Theta_{{\mathcal X}_\eta}$ then extends to an isogeny  (resp.~isomorphism) $\Theta_{{\mathcal X}}$ over $S$ is standard (see \emph{e.g.}, \cite[Prop.~4.5]{ACMVfunctor} and \cite[Prop.~6.1]{mumfordGIT}).  Regarding polarizations, recall that  by \cite[p.121]{mumfordGIT}, we have that $2\Theta_{{\mathcal X}}$ is induced by a line bundle $\mathcal L$ on $\operatorname{Ab}^2_{{\mathcal X}/S}$.  Thus, for questions of polarizations,  if suffices to establish the ampleness of the fibers of $\mathcal L$.  
On the one hand, since relative ampleness is an open property over the base, we have $\mathcal L$ is ample  if and only if $\mathcal L|_s$ is ample.  On the other hand, assuming $\mathcal L_\eta$ is ample, it suffices to show that $\mathcal L|_s$ is ample.  In general, relative ampleness is not a closed condition, however, for abelian varieties, a nondegenerate line bundle is ample if and only if it is effective \cite{mumfordAV}.  Therefore, taking $D_\eta$ to be an effective divisor realizing the line bundle $\mathcal L_\eta$, we may take the closure of $D_\eta$ to obtain an effective divisor $D$ over $S$ realizing $\mathcal L$.  
Thus $\mathcal L_s$ is effective and nondegenerate, and therefore  ample.
\end{proof}

\begin{pro}\label{P:PolChar0}
 Assume both the generic point $\eta$ and the closed point $s$ of $S$ have perfect residue fields and assume that the diagonal $\Delta_{{\mathcal X}_{\bar {\eta}}}$ has a strict decomposition (\emph{e.g.}, $\mathcal X_\eta$ is geometrically stably rational).  
(By specialization, $\Delta_{{\mathcal X}_{\bar {s}}}$ also has a strict decomposition  \cite[Thm.~2.1]{voisinUniv}, \cite[Thm.~1.12]{CTP16}.) Denote by  $\Theta_{{\mathcal X}_s} : \operatorname{Ab}^2_{{\mathcal X}_{ s}/s} \to \widehat{\operatorname{Ab}}\,^2_{{\mathcal X}_{ s}/s}$ and $\Theta_{{\mathcal X}_\eta} : \operatorname{Ab}^2_{{\mathcal X}_{\eta}/\eta} \to \widehat{\operatorname{Ab}}\,^2_{{\mathcal X}_{\eta}/\eta}$ the negatives of the canonical isomorphisms provided by Theorem~\ref{T:main-pol}(2).

Then $\operatorname{Ab}^2_{{\mathcal X}_{ \eta}/\eta}$ and $\Theta_{{\mathcal X}_\eta}$ extend to an isomorphism $\Theta_{{\mathcal X}}:\operatorname{Ab}^2_{{\mathcal X}/S}\to \widehat{\operatorname{Ab}}\,^2_{{\mathcal X}/S}$ of abelian $S$-schemes, and there is a canonical isomorphism
\begin{equation}\label{E:PolChar0}
 (\operatorname{Ab}^2_{{\mathcal X}_{ s}/s},\Theta_{{\mathcal X}_s}) \cong (	\operatorname{Ab}^2_{{\mathcal X}/S}|_s,\Theta_{{\mathcal X}}|_s).
\end{equation}
In particular,  $\Theta_{{\mathcal X}_\eta}$ is a polarization (and therefore a principal polarization) if and only if $\Theta_{{\mathcal X}_s}$ is.
           \end{pro}

\begin{proof} 
From Lemma \ref{L:PolSpec}, it suffices to establish the canonical isomorphism 
\eqref{E:PolChar0}.

By  Proposition~\ref{P:unicycle}, $\phi^2_{{\mathcal X}_{\bar \eta}}$ and $\phi^2_{{\mathcal X}_{\bar s}}$ are isomorphisms, and, by \cite[Prop.~5.2]{ACMVBlochMap} (see also Proposition~\ref{P:lambdacoho} below), the maps $T_\ell\lambda^2_{{\mathcal X}_{\bar \eta}} : T_\ell\operatorname{A}^2({\mathcal X}_{\bar \eta}) \to H^3({\mathcal X}_{\bar \eta},\integ_{\ell}(2))$ and $T_\ell\lambda^2_{{\mathcal X}_{\bar s}} : T_\ell\operatorname{A}^2({\mathcal X}_{\bar s}) \to H^3({\mathcal X}_{\bar s},\integ_{\ell}(2))$ are isomorphisms for all primes $\ell$. Choose a prime $\ell$ invertible in the function fields $\kappa(s)$ and $\kappa(\eta)$. 
	On the one hand, we obtain a Galois-equivariant isomorphism $T_\ell\lambda^2_{{\mathcal X}_{\bar \eta}}  \circ (T_\ell\phi^2_{{\mathcal X}_{\eta}})^{-1} : T_\ell \operatorname{Ab}^2_{{\mathcal X}_{\eta}/\eta} \to 
	H^3({\mathcal X}_{\bar \eta},\integ_{\ell}(2))$, showing that $\operatorname{Ab}^2_{{\mathcal X}_{ \eta}/\eta}$ has good reduction. 	On the other hand,
	since the specialization map $H^3({\mathcal X}_{\bar \eta},\integ_{\ell}(2)) \to H^3({\mathcal X}_{\bar s},\integ_{\ell}(2))$ is an isomorphism (by smooth proper base change), it follows that  $\dim  {\operatorname{Ab}^2_{{\mathcal X}_{s}/s}} = \dim \operatorname{Ab}^2_{{\mathcal X}_{\eta}/\eta}$.

	  Let $Z \in \operatorname{A}^2(\operatorname{Ab}^2_{{\mathcal X}_{\bar \eta}/\bar \eta} \times_{\bar \eta} {\mathcal X}_{\bar \eta})$ be a universal cycle for $\operatorname{Ab}^2_{{\mathcal X}_{\bar \eta}/\bar \eta}$ and $\widehat Z$ be a cycle inducing $\Phi^2_{{\mathcal X}_{\bar \eta}/\bar \eta}$. 
	We denote $Z_{\bar s}$ and $\widehat{Z}_{\bar s}$ their specializations. 
	Let also 
	$\zeta$ be a universal cycle for $\operatorname{Ab}^2_{{\mathcal X}_{\bar s}}$ and 
	$\hat \zeta$ be a cycle inducing~$\Phi^2_{{\mathcal X}_{\bar s}}$. 
	 The key point is that since $\phi_{\bar \eta}$ is induced by a cycle $\widehat Z$, its specialization defines a regular homomorphism, namely the one induced by the specialization of the cycle $\widehat Z$. Therefore, 	we have a commutative diagram
			$$\xymatrix@R=1em{
		\operatorname{Ab}^2_{{\mathcal X}/S}|_s(\bar s) \ar[rd]^{w_{Z_{\bar s}}} \ar[rr]^=  \ar@/_3.0pc/[rrdd]^{\psi_{Z_{\bar s}}}
		& & 	\operatorname{Ab}^2_{{\mathcal X}/S}|_s(\bar s) \\
		& \operatorname{A}^2({\mathcal X}_{\bar s}) \ar[ur]^{\widehat Z_{\bar s}^*} \ar[dr]^{\phi^2_{{\mathcal X}_{\bar s}} =\hat \zeta^*} & \\
 		 & & {\operatorname{Ab}^2_{{\mathcal X}_{s}/s}}(\bar s) \ar@{-->}[uu]_f
	}$$
where the upper triangle is the specialization of the corresponding triangle over $\eta$, where $f$ is the homomorphism induced by the universal property of $\Phi^2_{{\mathcal X}_s}$, and where $\psi_{Z_{\bar s}} : 	\operatorname{Ab}^2_{{\mathcal X}/S}|_{\bar s} \to {\operatorname{Ab}^2_{{\mathcal X}_{\bar s}/\bar s}}$ is the homomorphism induced by the cycle $Z_{\bar s}$. Hence $f : \operatorname{Ab}^2_{{\mathcal X}_{s}/s} \to 	\operatorname{Ab}^2_{{\mathcal X}/S}|_s$ is surjective. We already established that  $\dim  \operatorname{Ab}^2_{{\mathcal X}_{s}/s} = \dim \operatorname{Ab}^2_{{\mathcal X}/S}$, and we can thus conclude that $f$ is an isomorphism (with inverse $\psi_{Z_{\bar s}}$).

We now check that the isomorphism $\psi_{Z_{\bar s}}$ is canonical, \emph{i.e.}, it does not depend on the choice of a universal cycle $Z$ for $\operatorname{Ab}^2_{{\mathcal X}_{\bar \eta}/\bar \eta}$. Choose a prime $\ell$ invertible in the function fields $\kappa(s)$ and $\kappa(\eta)$. It is enough to check that $T_\ell \psi_{Z_{\bar s}}$ is canonical and, since on $\bar s$-points we have $\psi_{Z_{\bar s}} = \phi^2_{{\mathcal X}_s} \circ w_{Z_\bar s}$, by Lemma~\ref{L:ablambda}, it suffices to check that $T_\ell w_{Z_\bar s}$ is canonical.
We have a commutative diagram
\begin{equation}\label{E:indep}
\xymatrix{T_\ell \operatorname{Ab}^2_{{\mathcal X}_\eta/\eta} \ar[rr]^{T_\ell w_Z} \ar[d]_{{\mathrm{sp}}}^\cong 
	&& T_\ell \operatorname{A}^2({\mathcal X}_{\bar \eta}) \ar[d]_{{\mathrm{sp}}} \ar[rr]^{T_\ell\lambda^2}_\cong 
	&& H^3({\mathcal X}_{\bar \eta}, \integ_{\ell}(2)) \ar[d]_{{\mathrm{sp}}}^\cong \\
	T_\ell \operatorname{Ab}^2_{{\mathcal X}_s/s} \ar[rr]^{T_\ell w_{Z_{\bar s}}} 
	&& T_\ell \operatorname{A}^2({\mathcal X}_{\bar s}) \ar[rr]^{T_\ell\lambda^2}_\cong 
	&& H^3({\mathcal X}_{\bar s}, \integ_{\ell}(2)).}
\end{equation}
The commutativity of the right square shows that the middle specialization map is an isomorphism. We can conclude that $T_\ell w_{Z_\bar s}$ is canonical by noting that $T_\ell w_Z = (T_\ell \phi^2_{{\mathcal X}_\eta})^{-1}$ is canonical.

Finally, the isomorphism $f : \operatorname{Ab}^2_{{\mathcal X}_{ s}/s} \stackrel{\sim}{\longrightarrow} 	\operatorname{Ab}^2_{{\mathcal X}/S}|_s$ satisfies $\Theta_{{\mathcal X}_s} = f^\vee \circ \Theta_{{\mathcal X}_{\eta}}|_s \circ f$\,; \emph{i.e.}, it
is in fact an isomorphism $f: (\operatorname{Ab}^2_{{\mathcal X}_{ s}},\Theta_{{\mathcal X}_s}) \stackrel{\sim}{\longrightarrow} (\operatorname{Ab}^2_{{\mathcal X}/S}|_s,\Theta_{{\mathcal X}_S}|_s)$. This follows from the commutativity of the outer square of the diagram

	$$\xymatrix@R=1.5em{
	\operatorname{Ab}^2_{{\mathcal X}/S}|_s(\bar s) \ar[rd]^{w_{Z_{\bar s}}} \ar[rr]^{-\Theta_{{\mathcal X}_{S}}|_s} \ar@<.5ex>[dd]^{\psi_{Z_{\bar s}}}
	& & 	\widehat{\operatorname{Ab}}\,^2_{{\mathcal X}/S}|_s(\bar s) \ar@<.5ex>[dd]^{f^\vee} \\
	& \operatorname{A}^2({\mathcal X}_{\bar s}) \ar[ur]^{Z_{\bar s}^*} \ar[dr]^{\zeta^*} & \\
			\operatorname{Ab}^2_{{\mathcal X}_{s}/s}(\bar s) \ar[ru]^{w_{\zeta}} \ar[rr]^{-\Theta_{{\mathcal X}_s}} \ar@<.5ex>[uu]^f
	& & \widehat{\operatorname{Ab}}\,^2_{{\mathcal X}_{s}/s}(\bar s) \ar@<.5ex>[uu]^{\psi_{Z_{\bar s}}^\vee}
}$$
This whole diagram is in fact commutative: by duality it suffices to check that $-\Theta_{{\mathcal X}_s} \circ \psi_{Z_{\bar s}}$ is induced by ${}^t\zeta \circ Z_{\bar s} \in \chow^1(	\operatorname{Ab}^2_{{\mathcal X}/S}|_{\bar s} \times_{\bar s} 	{\operatorname{Ab}^2_{{\mathcal X}_{\bar s}/\bar s}})$. By construction, $-\Theta_{{\mathcal X}_s}$ is such that the regular homomorphism $\zeta^* : \mathscr{A}^2_{{\mathcal X}_{\bar s}/\bar s} \to \widehat{\operatorname{Ab}}\,^2_{{\mathcal X}_{\bar s}}$ is equal to $(-\Theta_{{\mathcal X}_s})_{\bar s}\circ \phi^2_{{\mathcal X}_{\bar s}}$, and it follows that the homomorphism induced by ${}^t\zeta \circ Z_{\bar s}$ coincides with the homomorphism  $(-\Theta_{{\mathcal X}_s})_{\bar s}\circ \phi^2_{{\mathcal X}_{\bar s}}(	\operatorname{Ab}^2_{{\mathcal X}/S}|_{\bar s})(Z_{\bar s}) = (-\Theta_{{\mathcal X}_s})_{\bar s}\circ \psi_{Z_{\bar s}}$.
\end{proof}

\begin{rem} As already mentioned,
in characteristic $0$,  the isomorphism $\Theta_X=-\Lambda_X$ 
 in Theorem~\ref{T:main-pol} agrees with the canonical principal polarization induced by the cup-product on $H^3(X_{\cx},\integ)$ (Theorem~\ref{T:CanPol3-fold} and Remark \ref{R:T:main-polHdg}).  In particular, Proposition \ref{P:PolChar0} implies that 
for a smooth projective geometrically stably rational threefold over a perfect field $K$ that lifts to a smooth projective geometrically stably rational threefold in characteristic zero, the isomorphism $\Theta_X$ is a principal polarization.
Later, we will strengthen this to show that we only need to assume that $X$ lifts to a geometrically \emph{rationally chain connected} threefold  (Corollary \ref{C:ThetaSpecMini}).  We point out, however, that while we have a stronger hypothesis in Proposition~\ref{P:PolChar0},  we also 
get the stronger conclusion that $(\operatorname{Ab}^2_{\mathcal X/S})_s\cong \operatorname{Ab}^2_{\mathcal X_s/s}$, and moreover, that 
the principal polarization on the generic fiber extends to a principal polarization on the special fiber, which agrees with the canonical auto-duality of Theorem \ref{T:main-pol}.   
In Corollary \ref{C:ThetaSpecMini} we will only obtain that 
$(\operatorname{Ab}^2_{\mathcal X/S})_s$ is isogenous to $\operatorname{Ab}^2_{\mathcal X_s/s}$.  
    \end{rem}

\newpage 

\part{Cohomological decomposition of the diagonal and algebraic representatives}\label{P:coho}

\section{Preliminaries on cohomological decomposition of the diagonal}\label{S:CohDD}
 
In this section, we deviate slightly from our Conventions \ref{conventions} and fix an arbitrary Weil
cohomology theory $\mathcal H^\bullet$ and denote 
$R_{\mathcal H}$ its coefficient ring. If $X$ is a smooth projective variety over a field $K$,
the cohomology class of a cycle class $Z\in \chow^i(X)$ will be denoted $[Z] \in \mathcal H^{2i}(X)(i)$.
 The following definition parallels the definition of Chow decomposition of the diagonal.

\begin{dfn}[Cohomological decomposition of a cycle class]\label{D:CoDec}
	Let $R\to R_{\mathcal H}$ be a 
	homomorphism of commutative rings.
	Let $X$ be a  smooth projective variety
	over a  field $K$,
	and
	let
	$$
	\xymatrix{j_i:W_i\ar@{^{(}->}[r]^{\quad \ne}& X}, \ \ i=1,2
	$$
	be reduced closed subschemes not equal to $X$.
	A \emph{cohomological $R$-decomposition   of  type $(W_1,W_2)$ of  a cycle class $Z\in
		\operatorname{CH}^{d_X}(X\times_KX)_R$} with respect to $\mathcal H^\bullet$
	is an equality
	\begin{equation}\label{E:DefDcp2}
	[Z] =[Z_1]+[Z_2]\in \mathcal H^{2d_X}(X\times_K X)(d_X),
	\end{equation}
	where $Z_1\in \operatorname{CH}^{d_X}(X\times _K X)_R$ is supported on
	$W_1\times_KX$ and $Z_2\in \operatorname{CH}^{d_X}(X\times _K X)_R$ is supported
	on $X\times_K W_2$ (see \eqref{E:ChowSupp} for the support of a cycle).   When $R=\mathbb Z$,  we call this a  \emph{cohomological decomposition
		of  type $(W_1,W_2)$} with respect to $\mathcal H^\bullet$. 	We say that $Z\in
	\operatorname{CH}^{d_X}(X\times_KX)_R$ has a \emph{cohomological $R$-decomposition of type $(d_1,d_2)$} with respect to $\mathcal H^\bullet$ if it admits a cohomological $R$-decomposition of type $(W_1,W_2)$ with $\dim W_1 \leq d_1$ and $\dim W_2 \leq d_2$.
\end{dfn}

 Beware that these definitions  depend \emph{a priori} on the choice of the Weil cohomology theory $\mathcal H ^\bullet$.
We emphasize that  the cycle classes $Z_1,Z_2$ in the definition   are
cycle classes with $R$-coefficients, and the statement about support of  $Z_1,Z_2$ is in terms of the Chow group (not the cohomology group).   Clearly, by applying the cycle class map, we see that if a cycle class $Z$ has an $R$-decomposition of type $(W_1,W_2)$, then it has a cohomological $R$-decomposition of type $(W_1,W_2)$ (with respect to any Weil cohomology theory).
We will primarily be interested in the case where $R=\integ$ and where $Z=N\Delta_X$ is a multiple of the diagonal in $X\times_KX$.

Note that with the notations of Lemma~\ref{L:factorcorres},  if a cycle $Z\in
\operatorname{CH}^{d_X}(X\times_KX)_R$ is cohomologically equivalent to a cycle supported on $W\times_K X$, then the maps  $p^eZ^*:\mathcal H^n(X)\to \mathcal H^n(X)$  and
$p^eZ_*:{\mathcal H}^n(X)\to {\mathcal H}^n(X)$ factor, respectively, as\,:

\begin{equation}\label{E:Zj*diagH2}
\xymatrix@R=1em{
	&\mathcal H^{2(d_{\widetilde W}-d_X)+n}(\widetilde W)(d_{\widetilde W}-d_X)
	\ar[rd]^<>(0.5){\tilde j_*}&& {\mathcal H}^n(X) \ar[rd]_<>(0.5){\tilde j^*}
	\ar[rr]^{p^eZ_*}
	&&{\mathcal H}^n(X) \\
	\mathcal H^n(X) \ar[ru]^<>(0.5){\widetilde Z^*} \ar[rr]^{p^eZ^*}
	&&\mathcal H^n(X) &&{\mathcal H}^{n}(\widetilde W) \ar[ru]_<>(0.5){\widetilde
		Z_*}&\\
}
\end{equation}

\begin{dfn}[Strict cohomological decomposition of a cycle  class]\label{D:StrCoDec}
	A \emph{strict cohomological $R$-decomposition  of  a cycle class $Z\in
		\operatorname{CH}^{d_X}(X\times_KX)_R$} is a cohomological $R$-decomposition of type $(d_X-1,0)$. In other words, it
	is an equality as in \eqref{E:DefDcp2} such that
	$Z_1 \in \operatorname{CH}^{d_X}(X\times _K X)_R$ is supported on $D\times_K
	X$ for some codimension-$1$ subvariety $D\subseteq X$, and
	$Z_2\in \operatorname{CH}^{d_X}(X\times _K X)_R$
	for some zero-cycle class $\alpha\in \operatorname{CH}_0(X)_R$.
	When $R=\mathbb Z$, we call this a  \emph{strict cohomological  decomposition}.
\end{dfn}

\begin{rem}[Strict cohomological decomposition of the diagonal]
As in Remark \ref{R:ChThEqDef},
 if $N [\Delta_X] =[Z_1]+[Z_2]\in \mathcal H^{2d_X}(X\times _K X)(d_X)$ is a strict cohomological $R$-decomposition, then $[Z_2] = \operatorname{pr}_2^*[\alpha]$ for any zero-cycle $\alpha \in \chow_0(X)$ of degree~$N$.
In particular, in the situation where $R=\integ$ and  $X(K)\ne \emptyset$, if $\Delta_X$ admits a strict cohomological decomposition, then $[Z_2] = [X\times_K x]$ for any $K$-point $x\in X(K)$.
\end{rem}

\begin{rem} 
	In the case where $K\subseteq \mathbb C$ and $R=\mathbb Z$ (resp.~$\mathbb Q$),  the comparison isomorphisms in
	cohomology imply that if  $Z\in \operatorname{CH}^{d}(X_1\times_{K} X_2)_R$  has a  cohomological 
	$R$-decomposition of type $(W_1,W_2)$  (resp.~a strict  cohomological $R$-decomposition) with respect to $H^\bullet ((-)^{an},\mathbb Z)$ (resp. $H^\bullet ((-)^{an},\mathbb Q)$)
	then for each prime number $\ell$   it has a cohomological  $R$-decomposition of type $(W_1,W_2)$ (resp.~a strict 
	$R$-decomposition) with respect to  $H^\bullet ((-)_{\bar K},\mathbb Z_\ell)$ (resp.~$H^\bullet ((-)_{\bar K},\mathbb Q_\ell)$).
\end{rem}

\section{Cohomological decomposition, torsion, algebraicity, and the Bloch map}

We now proceed to
recall some of the basic results concerning decomposition of the diagonal, which
essentially go back to Bloch--Srinivas, and have recently been strengthened by
Voisin.  
The main addition in this section is to explain how to modify   these
results  to hold in  the case of varieties over arbitrary perfect 
fields.
\medskip

A projective variety $W$ over a perfect field $K$ admits an alteration $\widetilde{W}\to W$ of degree $M$ invertible in $R_{\mathcal H}$ with $\widetilde{W}$ smooth projective over $K$ in the following situations\,:
 \begin{itemize}
	\item  If 
	$\mathcal H\udot$ is $\ell$-adic \'etale cohomology $\mathcal H\udot(X) =
	H\udot(X_{\bar K},\integ_\ell)$ with $\ell$ invertible in $K$. This is Gabber's $\ell'$-alteration theorem.  A strengthening, due to Temkin~\cite{temkin17}, shows that a projective variety $W$ over $K$ admits an alteration $\widetilde{W}\to W$ as above of degree some power of the characteristic exponent of $K$.
	\item If $W$ admits a resolution of singularities (in which case $M$ can be chosen to be equal to $1$). This holds unconditionally if $\operatorname{char}(K) = 0$ or if $\dim W \leq 3$~\cite{CPRes2}.
\end{itemize}

We fix a ring homomorphism $R\to R_{\mathcal H}$.

 \subsection{Cohomological decomposition of the diagonal and torsion in cohomology}\label{S:DD-Tor}

Each of Betti, $\ell$-adic and crystalline cohomology
 has the property that $\mathcal H^0$, $\mathcal H^1$ and $\mathcal
 H^{2\dim X}$ are
 torsion-free for proper smooth varieties $X$, and vanish in degrees
 greater than $2\dim X$.  (In degree $1$ this follows, for
 instance, from the identification of $\mathcal H^1(X)$ with $\mathcal
 H^1(\operatorname{Alb}_{X/K})$, and the known calculation for abelian
 varieties.  See \cite[II.3.11.2]{illusiedRW} for the
 case of crystalline cohomology.)

 \begin{pro}[{\cite[Thm.~4.4(i)]{voisinAJ13}}]\label{P:V-AJ4.4}
  Let $K$ be a perfect field, and assume $\mathcal H^i$ is torsion-free for
  $i=0,1$.  Let $X/K$ be a
  smooth projective variety.  Assume that for
  some $N$  in $R$ the multiple
  $N \Delta_{X_{\bar K}}\in \operatorname{CH}^{d_X}(X_{\bar K}\times_{\bar K}X_{\bar K})_R$ admits a cohomological 
  $R$-decomposition of type $(W_1,W_2)$ with $\dim W_1\le
 d_X- 1$ and  $\dim W_2\le
  1$.  Let $\widetilde W_1\to W_1$ be an alteration
  of degree $M$ invertible
   in $R_{\mathcal H}$ with $\widetilde W_1$ smooth projective over $\bar K$.  
  \begin{enumerate}

   \item  If $\mathcal H^{i-2}(\widetilde W_1)$ is torsion-free (\emph{e.g.}, $i\le 3$),
   then torsion in $\mathcal H^i (X)$ is killed by multiplication by $N$.

   \item If $\mathcal H^i(\widetilde W_1)$ is torsion-free (\emph{e.g.}, $i>2d_X-3$), then
   torsion in $\mathcal H^i (X)$ is killed by multiplication by~$N$.
  \end{enumerate}
 \end{pro}

 \begin{proof}
  This follows directly from the proof \cite[Thm.~4.4(i)]{voisinAJ13}\,; \emph{i.e.}, from
  the factorization of correspondences in cohomology \eqref{E:Zj*diagH2}.
    \end{proof}

    As before, if  $N$ is a natural number whose image is zero in the coefficient ring of the cohomology
  theory, the conclusions of Proposition \ref{P:V-AJ4.4} are trivial.   At the
  opposite extreme, if the image of $N$ is a unit in $R_{\mathcal H}$, then under the
  hypotheses we may conclude that $\mathcal H^i(X)$ is torsion-free.

 \begin{cor}\label{C:V-AJ4.4} Let $K$ be a perfect field, and assume
  $\mathcal H^i(-)$ is torsion-free for $i=0,1$.  Let $X/K$ be a smooth
  projective variety. If  $\Delta_{X_{\bar K}}\in \operatorname{CH}^{d_X}(X_{\bar K}\times_{\bar K}X_{\bar K})$
  admits a cohomological  $R$-decomposition  of type $(W_1,W_2)$  with  $\dim W_1\le
  d_X- 1$ and $\dim W_2\le
  1$, and with   $W_1$ admitting a  resolution of singularities, then
  $\mathcal H^i (X)$ is torsion-free for  $i\le 3$ and $i \ge  2d_X-2$.

  In particular, if $X/K$ is a smooth projective threefold such that
  $\Delta_{X_{\bar K}}\in 
\operatorname{CH}^3(X_{\bar K}\times_{\bar K}X_{\bar K})$ admits a cohomological $R$-decomposition of
  type $(W_1,W_2)$ with $\dim W_2\le 1$, then $\mathcal H^i (X)$ is
  torsion-free for all $i$.
 \end{cor}

 \begin{proof}
  This is immediate from the proposition, using the fact that there is resolution
  of singularities for surfaces over perfect fields.
 \end{proof}
 
  \begin{rem}\label{R:Tors3fold} For a smooth projective threefold $X$
  over a field $K$, and $\mathcal H^\bullet$ denoting Betti or $\ell$-adic
  cohomology ($\ell\ne \operatorname{char}(K)$) recall that $\mathcal H^0(X)$,
  $\mathcal H^1(X)$,
  and $\mathcal H^6(X)$ are torsion-free.  Moreover,
  $\operatorname{Tors}\mathcal H^2(X)\cong \operatorname{Tors}\mathcal
  H^5(X)$ and $\operatorname{Tors}\mathcal H^3(X)\cong
  \operatorname{Tors}\mathcal H^4(X)$.
  If $X$ is unirational over $\mathbb C$, a result of Serre \cite{serrePi1}
  implies that the fundamental group of $X$ is trivial, so that $\mathcal
  H_1(X)=\mathcal H^5(X)=0$.
 \end{rem}

 \begin{cor}
  Let $X$ be a smooth projective threefold over a perfect field $K$.  If the class of the diagonal $\Delta_{X_{\bar K}}\in \operatorname{CH}^3(X_{\bar K}\times_{\bar K}X_{\bar K})$ admits a cohomological $\ww(K)$-decomposition of type $(W_1,W_2)$  with respect to crystalline cohomology $H^\bullet_{\operatorname{cris}}(-/\ww(K))$, 
    with  $\dim W_1\le
  d_X- 1$ and $\dim W_2 \le 1$, then
  the Picard scheme $\pic^0_{X/K}$ is reduced.
 \end{cor}

 \begin{proof} Note that $W_1$, like any surface over a perfect field, admits a resolution of singularities.
  By Corollary~\ref{C:V-AJ4.4}, we have that
  $H^2_{{\rm cris}}(X/\ww(K))$ is torsion-free, and so (\emph{e.g.}, \cite[Prop.~II.5.16]{illusiedRW}) $\pic^0_{X/K}$
  is reduced.
 \end{proof}

 For threefolds, it is often convenient to also consider the following situation:

 \begin{pro}\label{P:DD3fold}
  Let $X$ be a smooth projective threefold over a  perfect field $K$,
  and let $\mathcal H^\bullet$ denote  Betti, crystalline or  $\ell$-adic
  cohomology
  (with $\ell\ne\operatorname{char}K$).
  Assume that for some $N\in R$  the multiple
  $N\Delta_{X_{\bar K}}\in \operatorname{CH}^{d_X}(X_{\bar K}\times_{\bar K}X_{\bar K})_R$  admits a cohomological  $R$-decomposition  of type $(W_1,W_2)$.   Let $\widetilde W_2\to W_2$ be a
  resolution of singularities. 

  \begin{enumerate}
   \item For $i\le 3$, if  $\mathcal H^i(\widetilde W_2)$ is torsion-free (\emph{e.g.},
   $\widetilde W_2$ is rational), then torsion in $\mathcal H^i(X)$ is killed by
   multiplication by $N$.

   \item For $i>3$, if $\mathcal H^{i-2}(\widetilde W_2)$ is torsion-free (\emph{e.g.},
   $\widetilde W_2$ is rational), then $\mathcal H^i(X)$ is killed by
   multiplication by $N$.
  \end{enumerate}

 \end{pro}

 \begin{proof}
  The arguments are the same as for Proposition \ref{P:V-AJ4.4}.
    \end{proof}

 \subsection{Cohomological decomposition of the diagonal and vanishing cohomology}\label{S:DD-Van}
We now consider some results on vanishing of cohomology.  The main take-away for our applications is Corollary~\ref{C:vanCoh} for threefolds.  

\begin{pro}\label{P:vanCoh}
 Let $X$ be a
  smooth projective variety over a perfect field $K$.  Assume that for
  some $N\in R$  the multiple
  $N\Delta_{X_{\bar K}}\in \operatorname{CH}^{d_X}(X_{\bar K}\times_{\bar K}X_{\bar K})_R$ admits a cohomological 
  $R$-decomposition of type $(W_1,W_2)$ with $\dim W_1\le
 d_X- 1$ and  $\dim W_2=
  0$.  Let $\widetilde W_1\to W_1$ be an alteration
  of degree $M$ invertible
   in $R$ with $\widetilde W_1$ smooth projective over $\bar K$.  
  \begin{enumerate}

   \item  If $i\ge 1$ and $\mathcal H^{i-2}(\widetilde W_1)=0$  (\emph{e.g.}, $i=1$),
   then $\mathcal H^i (X)$ is killed by multiplication by $N$.

   \item If  $i\ne 2d_X$ and $\mathcal H^i(\widetilde W_1)=0$ (\emph{e.g.}, $i=2d_X-1$), then
    $\mathcal H^i (X)$ is killed by multiplication by $N$.
  \end{enumerate}
 \end{pro}

 \begin{proof}
  This follows directly from the proof \cite[Thm.~4.4(i)]{voisinAJ13}\,; \emph{i.e.}, from
  the factorization of correspondences in cohomology \eqref{E:Zj*diagH2}.
   \end{proof}

    Note that if the image of  $N$ is zero in the coefficient ring of the cohomology
  theory, the conclusions of Proposition \ref{P:V-AJ4.4} are trivial.   At the
  opposite extreme, if the image of  $N$ is a unit in  $R_{\mathcal H}$, then under the
  hypotheses we may conclude that $\mathcal H^i(X)=0$.

 \begin{cor}\label{C:vanCoh} 
  Let $X$ be a
  smooth projective variety over a perfect field $K$.  Assume that that for
  some $N\in R$  the multiple 
  $N \Delta_{X}\in \operatorname{CH}^{d_X}(X_{\bar K}\times_{\bar K}X_{\bar K})_R$ admits a cohomological 
  $R$-decomposition of type $(W_1,W_2)$ with $\dim W_1\le
 d_X- 1$ and  $\dim W_2=
  0$, with  $W_1$ admitting a  resolution of singularities.   Then
  $\mathcal H^1 (X)$ and $\mathcal H^{2d_X-1}(X)$ are killed by multiplication by $N$.  If $N=1$, then $\mathcal H^1(X)=\mathcal H^{2d_X-1}(X)=0$.

  In particular, if $X/K$ is a smooth projective threefold such that
  $N \Delta_{X_{\bar K}}\in \operatorname{CH}^3(X_{\bar K}\times_{\bar K}X_{\bar K})_R$ admits a strict cohomological $R$-decomposition (\emph{e.g.}, $X$ is geometrically rationally chain connected), then $\mathcal H^1 (X)=0$ and $\mathcal H^5(X)$ is killed by multiplication by $N$.  
If $N=1$ (e.g., $X$ is geometrically stably rational), then  $\mathcal H^1 (X)=\mathcal H^5(X)=0$.
   \end{cor}

 \begin{proof}
  This is immediate from the proposition.  
  The key point is that $\mathcal{H}^1(X)$ and $\mathcal{H}^{2d_X-1}(X)$ model
  the cohomology of the Abelian varieties $(\mathrm{Pic}^0_{X_K/K})_{\mathrm{red}}$ and $\mathrm{Alb}_{X_K/K}$, respectively, and that $\mathcal{H}^1(X)$ is torsion while torsion in $\mathcal{H}^{2d_X-1}(X)$ is killed by multiplication by $N$ under our assumption
  on the decomposition of the diagonal by virtue of Corollary~\ref{C:V-AJ4.4}).  
  The assertion for threefolds follows as there is resolution
  of singularities for surfaces over perfect fields.  For the case where $X$ is assumed to be geometrically rationally chain connected, see Remark~\ref{R:RCDecDiag}.
 \end{proof}

\begin{rem}[Vanishing Hodge numbers]\label{R:SRC-Hodge}
Let $X$ be a smooth projective variety over a field $K$.   If $X$ is geometrically  rationally chain connected and $K$ is perfect, the previous corollary implies that $\mathcal H^1(X)=0$. In characteristic $0$, this implies the vanishing of $H^0(X,\Omega_X)$ and $H^1(X,\mathcal O_X)$\,; here we recall some further results on vanishing Hodge numbers.  
First, if $X$ is geometrically separably rationally connected  over a field $K$,  then    $H^0(X,
	\Omega_{X}^{\otimes i})=0$ for all $i>0$ \cite[Cor.~IV.3.8]{kollar}, and 
	$H^1(X,\mathcal O_X)=0$ \cite[Thm.~p.872]{gounelas}. 
More generally, if $\Delta_{X_{\bar K}}$ admits a strict \emph{Chow} decomposition, then 	\cite[Lem.~2.2]{totaroHype} implies that $H^0(X,\Omega_X^i)=0$ for $i>0$. 
If $K\subseteq \mathbb C$,  the 
	\emph{proof} of  \cite[Thm.~3.13]{voisinDiag}, or equivalently of 	\cite[Thm.~10.17, p.294]{voisinII}, shows that if  for some natural number~$N$  the  multiple $N\Delta_{X_\cx}\in
	\chow^{d_X}(X_\cx\times_\cx X_\cx)$   admits a cohomological decomposition of type
	$(d_1,d_2)$ with respect to $H^\bullet((-)^{an},\mathbb Z)$, then $H^0(X,\Omega_X^i)=0$ for all $i>d_2$.
\end{rem}

 \subsection{Cohomological decomposition of the diagonal and algebraic cycle classes}\label{S:DD-AlgCyc}
 
\begin{dfn}[Algebraic cohomology classes]
	 \label{D:algebraic}
	Let $X$ be a smooth projective variety over a perfect field $K$. Let $R\to R_{\mathcal H}$ be a ring homomorphism. We say that a class $\alpha \in \mathcal H^{2i}(X)(i)$ is $R$-\emph{algebraic} if it is in the image of the cycle class map 
 		$$[-]:\operatorname{CH}^i(X)_R\to \mathcal H^{2i}(X)(i).$$ 
We say that $\mathcal H^{2i}(X)$ is $R$-algebraic if the map above is surjective.  We say that  
$\alpha \in \mathcal H^{2i}(X)(i)$ is \emph{algebraic}  (resp.~$\mathcal H^{2i}(X)$ is \emph{algebraic}) 
if it is $R_{\mathcal H}$-algebraic.  In other words, $\mathcal H^{2i}(X)$  is algebraic if 
$\mathcal
	H^{2i}(X)(i)$ is spanned over $R_{\mathcal H}$ by the image of the cycle
	class map
	$[-]:\operatorname{CH}^i(X)\to \mathcal H^{2i}(X)(i)$.  
   \end{dfn}
    Note that we require algebraic cohomology classes to be the classes of algebraic cycles on $X$, rather than on $X_{\bar K}$. This leads to some subtleties. 
  For instance, even
 for $i=0$, one may have $\mathcal H^{0}(X)$ is not algebraic: consider $X$ connected but not geometrically connected.
 Likewise, even
 for $i=d_X$, one may have $\mathcal H^{2d_X}(X)$ is not algebraic:  taking $X$ to be a smooth conic over $\mathbb Q$ with no
 $\mathbb Q$-points, $H^2(X_{\bar {\mathbb Q}},\mathbb Z_2)$ is not
spanned by zero-cycles on $X$, as there is no zero-cycle of odd degree defined over
 $\mathbb Q$.  Note however that, if $K$ is a finite field, then any variety over $K$ admits a zero-cycle of degree-$1$ and hence $\mathcal H^{2d_X}(X)$ is indeed algebraic in that case.

 \begin{pro}\label{P:AlgCyc}
  Let $X$ be a smooth projective variety over a perfect field
  $K$.
   Assume that for some natural number $N$  the multiple $N \Delta_{X}\in \chow^{d_X}(X\times_KX)_R$  admits
   a cohomological $R$-decomposition  of type
$(W_1,W_2)$  with $\dim W_1\le d_X - 1$ and $\dim W_2\le 1$.
  Let $\widetilde W_1\to W_1$ be an alteration of degree $M$
  invertible in $R_{\mathcal H}$ with $\widetilde{W}_1$ smooth projective over $K$.
   \begin{enumerate}

   \item  If $\mathcal H^{2i-2}(\widetilde W_1)$ is $R$-algebraic (\emph{e.g.}, $2i\le 2$),
   then  $N\cdot \mathcal H^{2i} (X)$ is $R$-algebraic.

   \item If $\mathcal H^{2i}(\widetilde W_1)$ is $R$-algebraic  (\emph{e.g.}, $2i\geq2d_X-2$),
   then   $N\cdot \mathcal H^{2i} (X)$ is $R$-algebraic.
  \end{enumerate}

 \end{pro}

 \begin{proof}
  This follows directly from the proof of \cite[Thm.~1(iv)]{BlSr83}
  or  \cite[Thm.~4.4(ii)]{voisinAJ13}\,; \emph{i.e.}, from  the factorization of
  correspondences in cohomology \eqref{E:Zj*diagH}.
    \end{proof}

   Note that if the image of $N$ is  zero in the coefficient ring of the cohomology theory, the
  conclusions of Proposition \ref{P:AlgCyc} are trivial.

 \begin{cor}\label{C:AlgCyc}
    Let $X$ be a smooth projective variety over an algebraically closed 
         field  $K$. If  $\Delta_X\in \chow^{d_X}(X\times_KX)$  admits a cohomological $R$-decomposition  of type
  $(W_1,W_2)$  with $\dim W_1\le d_X - 1$ and $\dim W_2\le 1$, and with $W_1$ admitting a  resolution of
  singularities, then
  $\mathcal H^{2i} (X)$ is $R$-algebraic for  $2i\le 2$ and $2i \ge 2d_X-2$.
 
  In particular, 
     for any smooth
  projective \emph{threefold} $X$ over an algebraically closed field $K$ such that
 $\Delta_X\in \chow^{d_X}(X\times_KX)$    admits a cohomological  $R$-decomposition  of type $(2,1)$, we have $\mathcal
  H^{2i} (X)$ is $R$-algebraic for all  $i$.
 \end{cor}

 \begin{proof}
  This is immediate from Proposition~\ref{P:AlgCyc}.
 \end{proof}

\subsection[Cohomological decomposition of the diagonal and the Bloch map]{Cohomological decomposition of the diagonal and  the Bloch map}

We observe here that the assumption of \cite[Prop.~5.2]{ACMVBlochMap} involving a Chow decomposition can be weakened  to a  \emph{cohomological} decomposition\,:

\begin{pro}\label{P:lambdacoho}
    Let $X$ be a smooth projective variety over a perfect field $K$ of characteristic exponent~$p$.  Fix a natural number $N$ and a prime $\ell$.  Either let $R_{\mathcal H} = \integ_\ell$ and assume $\ell\nmid  pN$, or let $R_{\mathcal H} = \rat_\ell$ and assume $\ell\not = p$. Fix a ring homomorphism $R \to R_{\mathcal H}$.  Assume that $N\Delta_X\in \chow^{d_X}(X\times_KX)$ admits an \emph{$R$-cohomological} decomposition of type ($d_X-1,2)$ with respect to $H^\bullet(-,R_{\mathcal H})$.

  Then the inclusion $V_\ell \operatorname{A}^2(X_{\bar K})\hookrightarrow V_\ell \operatorname{CH}^2(X_{\bar K})$ is an isomorphism of $\gal(K)$-modules, and the second $\ell$-adic Bloch map
  $$ V_\ell \lambda^2: V_\ell \operatorname{CH}^2(X_{\bar K}) \longrightarrow H^3(X_{\bar K},\rat_\ell(2))$$ is an isomorphism of Galois modules.

  Moreover, if $R_{\mathcal H} = \integ_\ell$, then
  		$$T_\ell \lambda^2 : T_\ell \operatorname{CH}^2(X_{\bar K}) \longrightarrow
		H^3(X_{\bar K},\integ_\ell(2))_\tau$$
	        is an isomorphism of $\operatorname{Gal}(K)$-modules\,; and if if $\dim W_1\le d_X-1$ and $\dim W_2\le 1$, then $H^3(X_{\bar K},\integ_\ell)$ is
                torsion-free.
\end{pro}

\begin{proof}  
 	 The proof is exactly the same as that of \cite[Prop.~5.2]{ACMVBlochMap}, where the assertion is made only for a Chow decomposition of the diagonal. 
 For convenience, we include the proof here for a cohomological $R$-decomposition with respect to $H^\bullet (-,\mathbb Z_\ell)$\,; the case of $H^\bullet (-,\mathbb Q_\ell)$ is identical. 
    Let $\ell$ be a prime.
	That $T_\ell \lambda^2$ is an injective morphism of  $\operatorname{Gal}(K)$-modules was reviewed in \cite{ACMVBlochMap}.
	Using Lemma~\ref{L:factorcorres},
	one decomposes $Np^e\Delta_X^*=p^e Z_1^*+p^e Z_2^*$, with factorizations $p^e Z_1^*=(\tilde j_1)_*\circ  \widetilde Z_1^*$ and $p^e Z_2^*= (\widetilde Z_2)_*\circ \tilde j_2^*$. (For use in Lemma \ref{L:lambdacohop} below, we note that we can set $e=0$ in case $\dim X \leq 4$ due to resolution of singularities \cite{CPRes2} in dimensions $\leq 3$.)  We obtain by the naturality of the $\ell$-adic Bloch map  a commutative diagram
		\begin{equation}\label{E:diagfact2}
	\xymatrix{ T_\ell\operatorname{CH}^2(X_{\bar K}) \ar[r]^{\widetilde Z_1^*\oplus \tilde j_2^*\qquad \qquad} \ar[d]^{T_\ell\lambda^2} & T_\ell\operatorname{CH}^1((\widetilde{W}_1)_{\bar K}) \oplus  T_\ell \operatorname{CH}^2((\widetilde{W}_2)_{\bar K}) \ar[r]^{\qquad \qquad (\tilde j_1)_*+(\widetilde Z_2)_*} \ar[d]_{\simeq}^{T_\ell\lambda^1_{\widetilde{W}_1}\oplus T_\ell \lambda^2_{\widetilde{W}_2}}&  
		T_\ell	\operatorname{CH}^2(X_{\bar K}) \ar[d]^{T_\ell\lambda^2} \\
		H^3(X_{\bar K},\integ_\ell(2))_\tau  \ar[r]^{\widetilde Z_1^*\oplus \tilde j_2^*\qquad \qquad \quad} & H^1((\widetilde{W}_1)_{\bar K}, \integ_\ell(1))\oplus H^3((\widetilde{W}_2)_{\bar K}, \integ_\ell(2))_\tau \ar[r]^{\qquad \qquad \quad  (\tilde j_1)_*+(\widetilde Z_2)_*}   & 	H^3(X_{\bar K},\integ_\ell(2))_\tau.
	}
	\end{equation}
      	The middle vertical arrow is an isomorphism by Kummer theory and Rojtman's theorem (see \cite{ACMVBlochMap}), while the composition of the (bottom) horizontal arrows is multiplication by $Np^e$.
 	 A diagram chase then establishes the surjectivity of $T_\ell \lambda^2$. 
	
Finally, in case $\dim W_2\le 1$, that 	$H^3(X_{\bar K},\integ_\ell(2))$ is torsion-free follows simply from the factorization of the multiplication by $Np^e$ map 
$$ \xymatrix{
	H^3(X_{\bar K},\integ_\ell(2))  \ar[r]^{\widetilde Z_1^* \quad } & H^1((\widetilde{W}_1)_{\bar K}, \integ_\ell(1)) \ar[r]^{\quad   (\tilde j_1)_*}   & 	H^3(X_{\bar K},\integ_\ell(2))
}$$
and the fact that  $H^1((\widetilde{W}_1)_{\bar K}, \integ_\ell(1))$ is torsion-free.
\end{proof}

For the sake of completeness, we record an analogous statement about $p$-torsion in Chow.  Note that in positive characteristic $p$, crystalline cohomology, unlike $H^\bullet(-,\rat_p)$, is a Weil cohomology.

\begin{lem}\label{L:lambdacohop}   Let $X$ be a smooth projective variety over a perfect field $K$ of characteristic $p>0$.  Let $R_{\mathcal H}$ be either $\ww(K)$ or $\mathbb B(K)$, and fix a ring homomorphism $R \to R_{\mathcal H}$. Assume that for some natural number~$N$, we have that $N\Delta_X \in \operatorname{CH}^{d_X}(X\times_KX)$ admits an $R$-cohomological decomposition of type $(d_X-1,2)$ with respect to $H^\bullet_\cris(-/R_{\mathcal H})$.

  Then the inclusion $V_p\A^2(X_{\bar K}) \hookrightarrow V_p \operatorname{CH}^2(X_{\bar K})$ is an isomorphism of $\gal(K)$-modules,
and  the second $p$-adic Bloch map 
		$$V_p \lambda^2 : V_p \operatorname{CH}^2(X_{\bar K}) \longrightarrow
H^3(X_{\bar K},\rat_p(2))$$
is an isomorphism of $\gal(K)$-modules.

Moreover, if $R_{\mathcal H} = \ww(K)$, if $p\nmid N$, and if resolution of singularities holds in dimension $< d_X$, then $T_p \A^2(X_{\bar K}) \to T_p \operatorname{CH}^2(X_{\bar K})$ is an isomorphism, and the second $p$-adic Bloch map
		$$T_p \lambda^2 : T_p \operatorname{CH}^2(X_{\bar K}) \longrightarrow
		H^3(X_{\bar K},\integ_p(2))_\tau$$
	        is an isomorphism of $\operatorname{Gal}(K)$-modules.
                If in addition we have  $\dim W_1 \le d_X-1$ and $\dim W_2 \le 1$, then $H^3_\cris(X/\ww(K))$ and $H^3(X_{\bar K},\integ_p(2))$ are torsion-free.
\end{lem}

\begin{proof}
  The proof strategy is identical to that of Proposition \ref{P:lambdacoho}, bearing in mind that a decomposition in crystalline cohomology induces maps on $p$-adic cohomology, and that the necessary properties of the $p$-adic Bloch map are secured by Gros and Suwa \cite{grossuwaAJ}.
\end{proof}

We can similarly give a small improvement on \cite[Prop.~2.3(ii)]{BenWittClGr} regarding the Bloch map, which shows that under the stronger assumption that the decomposition is of type $(d_X-1,1)$, one gets the stronger conclusion that the $\ell$-adic Bloch map is an isomorphism at all primes.  
	 
\begin{pro}[{\cite[Prop.~2.3(ii)]{BenWittClGr}}]\label{P:lambdacohoBW}
		Let $X$ be a smooth projective variety over a perfect field $K$ of characteristic exponent~$p$.
		Fix a natural number $N$ and a  prime $l$\,; let $\mathcal H^\bullet(-)$ denote $H^\bullet(-,\integ_l)$ if $l\ne p$ and $H^\bullet_\cris(-/\ww(K))$ if $l=p$.  

Suppose that $N\Delta_{X}\in
		\operatorname{CH}^{d_X}(X\times_KX)$  admits a  
$\mathbb Z$-cohomological decomposition 
		of type $(d_X-1,1)$ with respect to $\mathcal H^\bullet$.  
Then the second Bloch map
	$$ \lambda^2 : \operatorname{A}^2(X_{\bar K})[l^\infty] \longrightarrow
		H^3(X_{\bar K},\integ_l(2))\otimes_{\mathbb Z_l}\mathbb Q_l/\mathbb Z_l$$	
			 is an isomorphism of $\operatorname{Gal}(K)$-modules.	
Taking Tate modules,  the second $l$-adic Bloch map
		$$T_l \lambda^2 : T_l \operatorname{A}^2(X_{\bar K}) \longrightarrow
		H^3(X_{\bar K},\integ_l(2))_\tau$$
	 is also an isomorphism of $\operatorname{Gal}(K)$-modules.
\end{pro}

\begin{proof}
  We adapt the proof of \cite[Prop.~2.3(ii)]{BenWittClGr} to the setting of cohomological decomposition of the diagonal, and verify in the process that  argument of Benoist--Wittenberg works for $p$-adic cohomology, as well.   For $X$ smooth and projective, we have a diagram with exact row (see \emph{e.g.}, \cite[A.16]{ACMVBlochMap} and \cite[(3.33)]{grossuwaAJ})
  \[
  \xymatrix{
    && \operatorname{CH}^2(X_{\bar K})[l^\infty]\ar@{^{(}->}[d]^{\lambda^2_l} \ar@{-->}[rd] \\
    0 \ar[r] & H^3(X_{\bar K},\integ_l(2))\otimes_{\mathbb Z_l} \rat_l/\integ_l \ar[r] & H^3(X_{\bar K},\rat_l/\integ_l(2))   \ar[r] & H^4(X_{\bar K},\integ_l(2))
  }
  \]
  where the dashed arrow is, up to sign, the cycle class map (\cite[Cor.~4]{CTSS83}, \cite[Prop.~III.1.16 and Prop.~III.1.21]{grossuwaAJ}).  Since algebraically trivial cycles are homologically trivial, it follows that the image of $\A^2(X_{\bar K})[l^\infty]$ under $\lambda^2_l$ is contained in $H^3(X_{\bar K},\mathbb Z_l)\otimes_{\mathbb Z_l}\mathbb Q_l/\mathbb Z_{l}\subseteq H^3(X_{\bar K},\mathbb Q_l/\mathbb Z_l)$.  In particular, the cokernel of $\lambda^2_l : \A^2(X_{\bar K})[l^\infty] \to H^3(X_{\bar K},\integ_l(2))\otimes_{\integ_l} \rat_l/\integ_l$ is divisible.

  Now suppose $N\Delta_X$ admits a cohomological decomposition of type $(d_X-1,1)$.  Arguing as in the proof of Proposition \ref{P:lambdacoho} and Lemma \ref{L:lambdacohop}, 
  $\operatorname{coker}\lambda^2_l$ is annihilated by $Np^e$.  Therefore, this cokernel is trivial, and $\lambda^2_l : \A^2(X_{\bar K})[l^\infty] \to H^3(X_{\bar K},\integ_l(2))\otimes_{\integ_l} \rat_l/\integ_l$ is an isomorphism. 
  Taking the Tate module of this isomorphism gives the assertion for the $l$-adic Bloch map, 
since $H^3(X_{\bar K},\mathbb Z_l)$ is a finitely generated $\mathbb Z_l$-module, and so it is elementary to check that $T_l( H^3(X_{\bar K},\mathbb Z_l)\otimes_{\mathbb Z_l}\mathbb Q_l/\mathbb Z_{l})\cong H^3(X_{\bar K},\mathbb Z_l)_\tau$. 
\end{proof}

\section{Cohomological correspondences and algebraic representatives}
 \label{S:URH-CCFam}

Recall that for a complex projective manifold $X$, given a cycle class $Z\in \operatorname{CH}^n(T\times X)$ parameterized by a complex projective manifold $T$ with fiber-wise homologically trivial cycle classes, the Abel--Jacobi map induces a holomorphic map $T\to J^{2n-1}(X)$, $t\mapsto AJ(Z_t)$, which, by construction, only depends on the cohomology class of $Z$ in $H^{2n}(T\times X,\mathbb Z)$.    The main goal of this section is to show (Corollary \ref{C:vanish}) that the same holds for morphisms induced by algebraic representatives.  
In fact, this follows from Proposition \ref{P:vanish}, which will allow us in many situations to ``lift'' to rational equivalence the action of a cohomological decomposition on algebraic representatives. Recall that, given a smooth projective variety $X$ over a field $K$, an algebraic representative $\Phi^i_X: \mathscr A^i_{X/K} \to \operatorname{Ab}^i_{X/K}$ exists for $i \in \{1,2,d_X\}$\,; see \cite{ACMVfunctor}.

\begin{pro}\label{P:vanish}
 	Let $X$ and $Y$ be smooth proper varieties of dimension $d_X$ and $d_Y$ respectively over a field~$K$ and let $\gamma \in \operatorname{CH}^{d_X+j-i}(X\times_K Y)$ be a correspondence. Assume that $j\in \{1,2,d_Y\}$ and that an algebraic representative $\Phi^i_X: \mathscr A^i_{X/K} \to \operatorname{Ab}^i_{X/K}$ exists.  
	Let $f:\operatorname{Ab}^i_{X/K} \to \operatorname{Ab}^j_{Y/K}$ be the morphism induced by $\gamma$ and the universal property of the algebraic representative via the commutative diagram 
	\begin{equation}\label{E:diagvanish}
\xymatrix@R=1.5em{
	\mathscr A^i_{X/K}  \ar[r]^{\gamma_*} \ar[d]_{\Phi_X^i}   &  		\mathscr A^j_{Y/K} \ar[d]^{\Phi^j_Y}  \\
		\operatorname{Ab}^i_{X/K} \ar[r]^{f}  & \operatorname{Ab}^j_{Y/K}
	}
\end{equation}
	If $[\gamma] = 0 \in H^{2(d_X+j-i)}(X_{K^{\sep}}\times_{K^{\sep}} Y_{K^{\sep}}, \rat_\ell(d_X+j-i))$ for some $\ell \neq \operatorname{char}(K)$, then $f=0$.
	In other words, $f$   depends only on the cohomology class of~$\gamma$.

 Moreover,	if $[\gamma] = 0 \in H^{2(d_X+j-i)}(X_{K^{\sep}}\times_{K^{\sep}} Y_{K^{\sep}}, \mathbb Z/\ell^{\nu+1}\mathbb Z(d_X+j-i))$ for some $\ell \neq \operatorname{char}(K)$ and some natural number $\nu$, then $f[\ell^\nu]:\operatorname{Ab}^i_{X/K}[\ell^\nu]\to \operatorname{Ab}^j_{Y/K}[\ell^\nu]$ is the zero map.
 \end{pro}
\begin{proof}
	 	 	 	 	 	 	 	 	First note that, since $\Phi^j_Y \circ \gamma_* : \mathscr {A}^i_{X/K} \to \operatorname{Ab}^j_{Y/K}$ defines a regular homomorphism,
	the universal property of the algebraic representative $\operatorname{Ab}^i_{X/K}$ provides a $K$-homomorphism $f : \operatorname{Ab}^i_{X/K} \to \operatorname{Ab}^j_{Y/K}$ giving the commutative 
	diagram \eqref{E:diagvanish}.
     
	Let $Z$ be a miniversal cycle (see \S \ref{SS:mini-universalCycle}) for $\Phi^i_X$ and let us denote $r$ the natural number such that the $K$-homomorphism $\Phi^i_X(\operatorname{Ab}^i_{X/K})(Z)$ is given by multiplication by $r$. Recalling that $\Phi^i_X(\operatorname{Ab}^i_{X/K})(Z)$ is simply given by $\phi^i_{X_{\bar K}}\circ w_Z$, we then have
	\begin{equation}\label{eq:commutephii}
	\phi^j_{Y_{\bar K}} \circ \gamma_* \circ w_Z = f \circ (\phi^i_{X_{\bar K}} \circ w_Z) = r\ f.
	\end{equation}
	Now, if $[\gamma] = 0 \in H^{2(d_X+j-i)}(X_{K^{\sep}}\times_{K^{\sep}} Y_{K^{\sep}}, \rat_\ell(d_X+j-i))$, 
	clearly $[\gamma \circ Z]$ also vanishes. By naturality of the $\ell$-adic Bloch maps, we hence get a commutative diagram
	$$\xymatrix{V_\ell \operatorname{Ab}^i_{X/K} \ar[r] \ar[rd]_{\operatorname{id}} \ar@/^2pc/[rr]^{\gamma_*\circ w_Z}  & V_\ell  \operatorname{A}_0((\operatorname{Ab}^i_{X/K})_{K^{\sep}}) \ar[r]^{\qquad \gamma_*\circ Z_*} \ar[d]_{V_\ell \lambda^0} & V_\ell \operatorname{A}^j(Y_{K^{\sep}}) \ar[d]^{V_\ell \lambda^j} \\
		&	V_\ell \operatorname{Ab}^i_{X/K} \ar[r]^{\gamma_*\circ Z_* = 0 \qquad } & H^{2j-1}(Y_{K^{\sep}},\rat_\ell(j)).
	}$$
	Note that the left triangle commutes thanks to \S \ref{SSS:alb} together with Lemma~\ref{L:ablambda} and the fact that the Bloch map $\lambda^0$ coincides with the albanese map on $\ell$-primary torsion~\cite[Prop.~3.9]{bloch79}.
	Recalling that $V_\ell \lambda^j$ is injective for $j=1,2,d_Y$, we conclude that $\gamma_*\circ w_Z : V_\ell \operatorname{Ab}^i_{X/K} \to  V_\ell \operatorname{A}^j(Y_{K^{\sep}})$
	 	 	is zero, and hence in view of \eqref{eq:commutephii} that $r \cdot V_\ell f = 0$\,; \emph{i.e.}, $f = 0$.

The same argument works in the case $[\gamma] = 0 \in H^{2(d_X+j-i)}(X_{K^{\sep}}\times_{K^{\sep}} Y_{K^{\sep}}, \mathbb Z/\ell^{\nu+1}\mathbb Z(d_X+j-i))$, using that the finite level Bloch maps  $\lambda^j[\ell^\nu]$ are injective for $j=1,2,d_Y$ (see \cite[App.~A, Prop.~A.27]{ACMVBlochMap}), and replacing the Tate modules with  $\ell^\nu$-torsion in the diagram above.
\end{proof}

From Proposition \ref{P:vanish} we can show that morphisms induced by projective families of cycles and universal regular homomorphisms depend only on the cohomology class of the family of cycles.  

\begin{cor}\label{C:vanish}
Let $X$ and $T$ be a smooth projective varieties over a field $K$, with $T(K)\ne \emptyset$, and let $Z\in \mathscr A^i_{X/K}(T)$ be a family of algebraically trivial cycle classes.    Assume that there exists  an algebraic representative $\Phi^i_{X/K}: \mathscr A^i_{X/K}\to \operatorname{Ab}^i_{X/K}$ (\emph{e.g.}, $i\in \{1,2,d_X\}$), and let 
$$
\psi_Z : T\longrightarrow  \operatorname{Ab}^i_{X/K}
$$
be the associated morphism.  If  $[Z]=0\in H^{2i}((T \times_KX)_{\bar K},\mathbb Q_\ell(i))$, then $\psi_Z=0$\,; in other words, $\psi_Z$ 
depends only on the cohomology class  of $Z$.  
\end{cor}

\begin{proof} Fix $t_0\in T(K)$.  
Set $\Gamma = \Delta_T - (T\times_K t_0)$ and  $Z'=Z+(T\times_K Z_{t_0})$.  We have the commutative diagram
$$
\xymatrix@C=4.5em{
T\ar[d]_{\operatorname{alb}_{t_0}} \ar[r]_<>(0.85){w_{\Gamma}}^{t\mapsto t-t_0}  \ar@/^.3pc/@{-->}[rrd]^<>(0.8){\psi_{Z}} & \operatorname{A}^{d_T}(T_{\bar K}) \ar[r]^{Z'_*} \ar[d]_<>(0.7){\phi^{d_T}_{T_{\bar K}}}& \operatorname{A}^i(X_{\bar K}) \ar[d]^{\phi^i_{X_{\bar K}}}\\
\operatorname{Alb}_{T/K} \ar@{=}[r]&\operatorname{Ab}^{d_T}_{X/K} \ar[r]_f & \operatorname{Ab}^i_{X/K}
}
$$
By Proposition \ref{P:vanish}, $f$ depends only on the cohomology class $[Z']$, and therefore, by commutativity, we see that $\psi_{Z}$ depends only on the cohomology class $[Z']$.  
Now since the cohomology class $[Z_{t_0}]$ depends only on the cohomology class $[Z]$, we see that the cohomology class $[Z']$ depends only on the cohomology class $[Z]$.
\end{proof}

\begin{rem} In both Proposition \ref{P:vanish} and Corollary \ref{C:vanish}, if 
 $K\subseteq \mathbb C$, then by the comparison theorems in cohomology, one is free to replace $\ell$-adic cohomology with Betti cohomology. 
 \end{rem}

\section[Decomposition of the diagonal and universal codimension-$2$ cycle classes]{Cohomological decomposition and universal codimension-$2$ cycle classes}
\label{S:ddalgrep}

This section contains our new results regarding cohomological decomposition of the
diagonal and
universal codimension-$2$ cycle classes.  \medskip

For a surjective regular homomorphism $\Phi:\mathscr
A^n_{X/K}\to A$, a \emph{universal} (resp.~\emph{miniversal})
\emph{codimension-$n$ cycle class}  (for
$\Phi$) is a cycle class $Z\in \mathscr A^n_{X/K}(A)$ such
that the induced morphism $\psi_{Z}=\Phi(A)(Z):A\to
A$  is the
identity (resp.~$N$ times the identity for some natural number $N$). 
It is a basic fact that there always exist miniversal cycle classes
\cite[Lem.~4.9]{ACMVdcg}.   
Theorem~\ref{T:UnivCyc} relates the number $N$ in the case of
algebraic representatives to cohomological decompositions of $N$ times the diagonal.

We start by recalling the well-known story for $n=1$, coming from the
identification $\Ab^1_{X/K} =
(\pic^0_{X/K})_{\operatorname{red}}$.
For brevity, we will call  a universal codimension-$1$ cycle class for
$\operatorname{Ab}^1_{X/K}$ a universal divisor.
There exists a universal  divisor if
each component of $X$ admits a $K$-point (\emph{e.g.}, if $K$ is separably closed).
In general, there need not exist a universal divisor\,; in fact, one can exhibit
curves over fields of characteristic $0$ with no universal divisor.  However,
any smooth projective variety $X$ over a finite field $K$ (with or without a
$K$-point) does admit a universal divisor.  We refer the reader to
\cite[\S 7.1]{ACMVfunctor} for more details.

The albanese morphism provides an algebraic representative in codimension $n=\dim X$, and we have the identification $\operatorname{Ab}^n_{X/K} = \mathrm{Alb}_{X/K}$.  
Voisin~\cite[Cor.~0.14]{voisin2023geometric} showed that a universal zero-cycle does not necessarily exist even for $K$ algebraically closed, thereby answering the question raised in \cite[Rem.~7.11]{ACMVfunctor}.

We now investigate the connection between universal codimension-$2$ cycle
classes and cohomological decompositions of the diagonal.

\begin{teo}[Miniversal cycle classes on algebraic representatives]
	\label{T:UnivCyc}
 	Let $X$ be a smooth projective variety over a perfect field $K$ of characteristic exponent $p$. Fix a
	positive integer $n$ and a prime number $\ell\ne p$.
 	 Assume:
	\begin{itemize}
\item For some natural
number $N$ (resp.~some natural number $N$ coprime to $\ell$) the multiple $N\Delta_X \in \chow^{d_X}(X\times_K X)$ of the diagonal admits a cohomological decomposition (resp.\ cohomological $\mathbb Z_\ell$-decomposition) 
  of type $(W_1,W_2)$ with
 $d_{W_1}\le d_X-(n-1)$ and $d_{W_2}\le n-1$\,;
\item $W_1$ admits a smooth
projective alteration of degree $p^e$ admitting a
universal divisor (\emph{e.g.}, $K$ is finite or algebraically closed).
	\end{itemize}
If  there exists an algebraic
	  representative $\Phi^n_{X/K}:\mathscr A^n_{X/K}\to
	  \operatorname{Ab}^n_{X/K}$ (\emph{e.g.}, $n\in \{1,2,d_X\}$), then  there exists
	a family of algebraically trivial cycle
	classes $Z\in\mathscr A^n_{X/K}(\operatorname{Ab}^n_{X/K})$
	such that the induced
	morphism of abelian varieties over $K$, $$\psi_{Z}=\Phi^n_{X/K}(\operatorname{Ab}^n_{X/K})(
	Z):\operatorname{Ab}^n_{X/K}\to \operatorname{Ab}^n_{X/K},$$ 
  	is
	multiplication by $Np^e$ (resp.~multiplication by $r$ for some natural number $r$ coprime to $\ell$).  
\end{teo}

\begin{rem}\label{R:UnivCyc}
	In Theorem~\ref{T:UnivCyc}, the constraints on $d_{W_1}$ and $d_{W_2}$ are perhaps more restrictive than they first appear.  In fact, we can have $d_{W_1}= d_X-(n-1)$ and $d_{W_2}= n-1$ or $n-2$, or we can have  $d_{W_1}= d_X-n$ and $d_{W_2}= n-1$, since from Remark~\ref{R:dW1dW2}, we know that
	$d_{W_1}+d_{W_2}\ge d_X-1$.  At the same time, since $0\le d_{W_1},d_{W_2}\le d_X-1$, this also puts the constraints $2\le n\le d_X$.
	Moreover,  in the case $d_{W_1}=d_X-n$, it will follow from the proof of Theorem \ref{T:UnivCyc}  that  $\Phi^n_{X/K} = 0$.
   \end{rem}

\begin{proof}
  	 	Assuming a cohomological decomposition, 
	the proof of \cite[Thm.~4.4(iii)]{voisinAJ13}  applies here\,; \emph{i.e.}, this follows
	from  the factorization of correspondences \eqref{E:Z2*di-Ab} and Proposition \ref{P:vanish}.
  	More
	precisely, using Lemma~\ref{L:factorcorres}, one decomposes $Np^e\Delta_X^*=p^e Z_1^*+p^e Z_2^*$, with factorizations $p^e Z_1^*=(\tilde j_1)_*\circ  \widetilde Z_1^*$ and $p^e Z_2^*= (\widetilde Z_2)_*\circ \tilde j_2^*$.
	Since $n>\dim \widetilde W_2$, it follows that $\tilde j_2^*=0$ on Chow groups,
	and therefore, by diagram \eqref{E:Z2*di-Ab}, this also holds for the induced morphism of
	abelian varieties. We therefore obtain a commutative diagram
	\begin{equation}\label{E:ThmUnivCyc}
\xymatrix@R=1.5em{\mathscr{A}^n_{X/ K} \ar[d]^{\Phi^n_X} \ar[rr]^{\widetilde Z_1^*} & &	\mathscr{A}^1_{\widetilde W_1/ K}  \ar[d]^{\Phi^1_{\widetilde W_1}}_{\simeq} \ar[rr]^{(\tilde j_1)_*}   & &	\mathscr{A}^n_{X/K} \ \ar[d]^{\Phi^n_X}\\
	\operatorname{Ab}^n_X \ar[rr]^{\widetilde Z_1^*}  &&
(\operatorname{Pic}^0_{\widetilde
	W_1/K})_{\operatorname{red}} \ar[rr]^{(\tilde j_1)_*} 
	& &\operatorname{Ab}^n_X,
}
	\end{equation}
	where the composition of the horizontal arrows is given by multiplication by $Np^e$ and where we have denoted  ${\widetilde Z_1^*} :	\operatorname{Ab}^n_{X/K} \to
	(\operatorname{Pic}^0_{\widetilde W_1})_{\mathrm{red}}$ and  $(\tilde j_1)_* :		(\operatorname{Pic}^0_{\widetilde W_1})_{\mathrm{red}} \to \operatorname{Ab}^n_{X/K}$ the $K$-homomorphisms induced by the correspondences ${\widetilde Z_1^*}$ and $(\tilde j_1)_*$.  Here we are using Proposition \ref{P:vanish}.  

      	We have assumed  the existence of a
	universal divisor
	$\widetilde {\mathcal D} \in
	\mathscr A^1_{\widetilde W_1/K}((\operatorname{Pic}^0_{\widetilde
		W_1/K})_{\operatorname{red}})$, meaning that the
	induced morphism: $$\psi_{\widetilde{\mathcal D}}=\Phi^1_{\widetilde
		W_1/K}((\operatorname{Pic}^0_{\widetilde
		W_1/K})_{\operatorname{red}})(\widetilde {\mathcal D}) : \ (\operatorname{Pic}^0_{\widetilde
		W_1/K})_{\operatorname{red}}\to (\operatorname{Pic}^0_{\widetilde
		W_1/K})_{\operatorname{red}}$$ is the identity. 
            Consider the cycle class
		$$
	Z:=(\tilde j_1)_*\circ
\widetilde {\mathcal
		D} \circ \widetilde Z_1^* \quad \in \mathscr A^n_{X/K}(\operatorname{Ab}^n_{X/K}).
	$$
	It is then clear from \eqref{E:ThmUnivCyc} that the associated homomorphism $\psi_{Z}=\Phi^n_{X/K}(\operatorname{Ab}^n_{X/K})(Z):\operatorname{Ab}^n_{X/K}\to \operatorname{Ab}^n_{X/K}$ is given by
	$Np^e\operatorname{Id}_{\operatorname{Ab}^n_{X/K}}$.

The case where one assumes a cohomological $\mathbb Z_\ell$-decomposition is similar.  
For simplicity, since we are assuming that $\ell\nmid N$,  $N$ is invertible in $\mathbb Z_\ell$, and so we may and do assume $N=1$.
We assume that we have $Z_1,Z_2$, as in the definition of the $\mathbb Z_\ell$-decomposition of the diagonal, observing that they are given as cycles with $\mathbb Z_\ell$-coefficients.  We have by definition that $[\Delta_X]=[Z_1]+[Z_2]$ in $H^{2d_X}((X\times _KX)_{\bar K},\mathbb Z_\ell(d_X))$.
Now let $Z_1'$ and $Z_2'$ be integral cycles which $\ell$-adically approximate $Z_1$ and $Z_2$.  (Concretely, if $Z_i = \sum a_{ij} Z_j$ with $a_{ij} \in \integ_\ell$, choose $a'_{ij} \in \integ$ with $a'_{ij} \equiv a_{ij} \bmod \ell^2$, and let $Z'_i = \sum a_{ij}' Z_i$.)
 We therefore have the equality $[\Delta_X]=[Z_1']+[Z_2']$ in $H^{2d_X}((X\times _KX)_{\bar K},\mathbb Z/\ell^2\mathbb Z (d_X))$.
The rest of the proof goes through identically, so that we find a cycle $Z$ such that
 that the associated homomorphism $\psi_{Z}[\ell]=\Phi^n_{X/K}(\operatorname{Ab}^n_{X/K})(Z)[\ell]:\operatorname{Ab}^n_{X/K}[\ell]\to \operatorname{Ab}^n_{X/K}[\ell]$ is given by
	$Np^e\operatorname{Id}_{\operatorname{Ab}^n_{X/K}}$.  Note that the cycle $Z$ depended on our choice to approximate to order $\ell^2$, as well as our choice of lift\,; truncating at higher order would also work, although it typically yields a different cycle $Z$. In any case, since $\psi_Z$ is an injection on $\ell$-torsion, it is an injection on $\ell$-power torsion, and therefore an isomorphism on $\ell$-power torsion.  
	
It follows that $\psi_Z$ is an isogeny, of degree coprime to $\ell$.  Arguing as in \cite[Lem.~4.9]{ACMVdcg} or \cite[Lem.~4.7]{ACMVfunctor}, one can find a cycle $Z'\in \mathscr A^n_{X/K}(\operatorname{Ab}^n_{X/K})$ so that $\psi_{Z'}$ is multiplication by $r$ for some natural number $r$ with $\ell \nmid r$.  
            	 \end{proof}

\begin{cor}[Universal codimension-$2$ cycle classes] \label{C:UnivCyc}
                             Suppose $X$ is a smooth projective variety of dimension $\leq 4$ over a field $K$ that is either finite or algebraically closed. If $\Delta_X$ admits a cohomological 
 decomposition (resp.\ cohomological $\mathbb Z_\ell$-decomposition) of type
		$(d_X-1,1)$, then
		there exists a universal codimension-$2$ cycle class  (resp.~a miniversal codimension-$2$ cycle class of degree  coprime to $\ell$).  
\end{cor}

\begin{proof}
This follows immediately from Theorem  \ref{T:UnivCyc} and the fact  \cite{CPRes2} that resolution of singularities holds for varieties  of dimension $\leq 3$ over perfect fields.  
  \end{proof}

We now consider the case where $X$ is a smooth projective variety over a field
$K\subseteq \mathbb C$, and we consider the surjective  regular homomorphism
$AJ:\mathscr A^n_{X/K}\to J^{2n+1}_{a,X/K}$ to the distinguished model,  given
by the Abel--Jacobi map, as defined in  \cite{ACMVdmij}.

\begin{teo}[Mini- and universal cycle classes on distinguished models]
	\label{T:UnivCycDM}
	Let $X$ be a smooth projective variety over a field $K\subseteq \mathbb C$, and
	fix a positive integer $n$.  Assume further that for some natural number~$N$ the
	multiple $N\Delta_{X}\in \chow^{d_X}(X\times_K X)$ admits a cohomological decomposition 
	  of type $(W_1,W_2)$ with $d_{W_2}\le n-1$,
	$d_{W_1}\le d_X+1-n$, and that $W_1$ admits a resolution of singularities with
	each component admitting a $K$-point.  Then there exists  a family of
	algebraically trivial cycle classes
	$Z\in \mathscr A^n_{X/K}(J^{2n-1}_{a,X/K})$ such that the induced morphism
	of abelian varieties over $K$, $$\psi_{Z}:J^{2n-1}_{a,X/K}\to
	J^{2n-1}_{a,X/K},$$  is  multiplication by $N$.
\end{teo}

\begin{proof} 
 	The proof of
	Theorem~\ref{T:UnivCycDM}  is identical to that of
	Theorem~\ref{T:UnivCyc}, except one uses diagram \eqref{E:Z2*di-AbDM} rather
	than  \eqref{E:Z2*di-Ab} and the fact that the action of a correspondence on intermediate Jacobians only depends on its cohomology class.  	
 	 	 	 	 	 \end{proof}

\begin{rem}
	In the cases $n=1,2$ or $d_X$, we have
	$J^{2n-1}_{a,X/K}=\operatorname{Ab}^n_{X/K}$, and the Abel--Jacobi map is the
	universal regular homomorphism $\Phi^n_{X/K}$, so that in those cases
	Theorems~\ref{T:UnivCycDM} and~\ref{T:UnivCyc} are closely related.  The benefit
	of Theorem~\ref{T:UnivCycDM} is that in characteristic $0$, we always have the
	distinguished model of the image of the Abel--Jacobi map for any $n$, whereas
	algebraic representatives, in any characteristic,  are  known to exist in
	general only for $n=1,2,d_X$.
\end{rem}

 \section{The standard assumption}
 \label{S:StandAs}

 We now discuss 
 an isomorphism, \eqref{E:StAs1} 
  below, that
 plays a central role in our presentation moving forward (Theorem~\ref{T:CanPol} and consequently Theorem~\ref{T:Intro-CanPol0}(1)--(2)).  
 It makes it possible, in positive characteristic, to relate the Tate module of the second algebraic representative of a smooth projective variety $X$ to the third cohomology group of~$X$. 
 Since we expect such an isomorphism to hold true in general (and it does in characteristic zero and in any characteristic for geometrically rationally chain connected varieties\,; see Proposition~\ref{P:BlSrTors}), we call this the
 standard assumption (Definition~\ref{D:StAs}).
 In order to slightly streamline the presentation, we recall our convention that
 $\ell$ always denotes a rational prime (\emph{i.e.}, a prime number in $\mathbb Z$)
 invertible in
 the base field (which we denote by $K$), while $l$ is allowed to be the
 (positive)
 characteristic of the base field.

 \subsection{Cohomological decomposition of the diagonal and Bloch--Srinivas' result on
  algebraic representatives} \label{SS:decBSalgrep}
  A  fundamental result due to  Bloch--Srinivas \cite[Thm.~1(i)]{BlSr83} states
 that over an algebraically closed field of characteristic $0$, a decomposition
 of the diagonal implies
 the algebraic representative in codimension-$2$ is  isomorphic to the group of
 algebraically trivial codimension-$2$ cycle classes.
 Their argument carries over to smooth projective varieties over perfect fields
 (this was also recently observed in \cite[Prop.~2.3]{BenWittClGr}). Here we improve their result by only assuming the existence of a \emph{cohomological} decomposition of the diagonal.

 \begin{pro} \label{P:BlSr}
  Let $X$ be a smooth projective variety over   a perfect field $K$ of characteristic exponent~$p$ and let $\ell$ be a prime $\neq p$.
    Suppose that,
 for some natural number $N$, the multiple $N\Delta_{X_{\bar K}}\in
  \operatorname{CH}^{d_X}(X_{\bar K}\times_{\bar K} X_{\bar K})$
  admits a \emph{cohomological} decomposition (resp.\ cohomological $\mathbb Z_\ell$-decomposition) 
   of type $(d_X-1,1)$ with respect to $H^\bullet (-,\mathbb Z_\ell)$.  Then the
  universal regular homomorphism  $\Phi^2_{X/K}: \mathscr A^2_{X/K} \to
  \operatorname{Ab}^2_{X/K}$ is an isomorphism of
  $\operatorname{Gal}(K)$-modules
   on
  $\bar K$-points\,; \emph{i.e.}, $$\Phi^2_{X/K}(\bar K):\operatorname{A}^2(X_{\bar K})
  \stackrel{\sim}{\longrightarrow}  \operatorname{Ab}^2_{X/K}(\bar K)$$
 (resp.~is an isomorphism on $\ell$-power torsion if $\ell \nmid N$\,; \emph{i.e.}, $\Phi^2_{X/K}(\bar K)[\ell^\infty]$ is an isomorphism).  
 \end{pro}

 \begin{proof}
  We combine the proof in \cite[Thm.~1(i)]{BlSr83}, where it is assumed that  $N\Delta_{X_{\bar K}}\in
  \operatorname{CH}^{d_X}(X_{\bar K}\times_{\bar K} X_{\bar K})$
  admits a decomposition of type $(d_X-1,1)$,
    with Proposition~\ref{P:vanish}.

  First, it suffices to prove the result after base change to the algebraic
  closure $\bar K$, so we assume that $K=\bar K$ to make the notation more
  streamlined.
  We only need to show that $
  \phi^2_{X}=\Phi^2_{X/K}(K):\operatorname{A}^2(X_{}) \to
  \operatorname{Ab}^2_{X/K}(K)$ is injective, since the map to the algebraic
  representative is always surjective on points over an algebraically closed
  field.

  We are given that
  $$
  N[\Delta_{X}]=[Z_1]+[Z_2]\in H^{2d_X}(X\times_KX,\integ_{\ell}(d_X))
  $$
  with $Z_1\in \chow^{d_X}(X\times_KX)$ supported on $W_1\times _KX$ where $W_1$ is a divisor, and $Z_2\in \chow^{d_X}(X\times_KX)$ is
  supported on $X\times_K W_2$ with $\dim W_2\le 1$.  Since $K=\bar K$ is perfect,
  we have a smooth projective alteration of $W_1$ of some degree a power of $p$, say $p^e$.  We
  then consider the correspondence  $Np^e\Delta_X^*$.  
  By Proposition~\ref{P:vanish}, $Np^e\Delta_X^*$ acts as $p^eZ_1^* + p^eZ_2^*$ on $\operatorname{Ab}^2_{X/K}$.
   Now we use our factorization
  \eqref{E:Z2*di-Ab}.  
  As $\dim W_2\le 1$, we have
  $p^eZ_2^*=0$ on $\operatorname{A}^2(X)$, so that $Np^e\Delta_X^*=p^eZ_1^*$ on $\operatorname{Ab}^2_{X/K}$.  Under the identification $\operatorname{A}^1(\widetilde
  W_1)=(\operatorname{Pic}^0(\widetilde W_1))_{\operatorname{red}}(K)$ we obtain a
  map  $r:=(\tilde j_1)_*\circ \widetilde Z_1^*:\operatorname{Ab}^2_{X/K}(K)\to
  \operatorname{A}^2(X)$, which is a homomorphism (via a diagram chase) with
  kernel contained in the $Np^e$-torsion.
    The inverse of the isomorphism $\operatorname{Ab}^2_{X/K}(K)/(\ker
  r)\stackrel{\sim}{\to} \operatorname{A}^2(X)$ is a regular homomorphism (use
  that for any smooth variety $T$ we have $\operatorname{A}^2(T\times_KX)$ is
  divisible), and then a diagram chase using the universal property of the
  algebraic representative shows that $\Phi^2_{X/K}$ is injective, completing the
  proof.
     
In the case of a $\mathbb Z_\ell$-cohomological decomposition, we make the same modifications as in Theorem \ref{T:UnivCyc}.  
From  Proposition \ref{P:surj}, we know that $\Phi^2_{X/K}(\bar K)[\ell^\infty]$ is surjective, so we only need to show it is injective.  As before, we may set $N=1$. We assume that we have $Z_1,Z_2$, as in the definition of the $\mathbb Z_\ell$-decomposition of the diagonal, observing that they are given as cycles with $\mathbb Z_\ell$-coefficients.  We have by definition that $[\Delta_X]=[Z_1]+[Z_2]$ in $H^{2d_X}((X\times _KX)_{\bar K},\mathbb Z_\ell(d_X))$.  As in the proof of Theorem~\ref{T:UnivCyc}, choose integral cycles $Z'_1$ and $Z'_2$ which respectively approximate $Z_1$ and $Z_2$ modulo $\ell^2$.
We have the equality $[\Delta_X]=[Z_1']+[Z_2']$ in $H^{2d_X}((X\times _KX)_{\bar K},\mathbb Z/\ell^2\mathbb Z (d_X))$.
The rest of the proof goes through identically, so that we find that $\Phi^2_{X/K}(\bar K)$ is an injection on $\ell$-torsion.  This implies it is an injection on $\ell$-power torsion, and we are done.
 \end{proof}

 \begin{rem}
  In Proposition \ref{P:BlSr},   the natural transformation $\Phi^2_{X/K}:\mathscr A^2_{X/K}\to
  \operatorname{Ab}^2_{X/K}$ need not be an isomorphism of functors  \cite[Rem.~5.2]{ACMVfunctor}.
                   \end{rem}

We have the following application\,:

 \begin{pro}\label{P:BlSrTors}
  Let $X$ be a smooth projective variety over a perfect field $K$.
   Assume one of the following:
  \begin{enumerate}
   \item $\operatorname{char}(K)=0$, or,

   \item for some natural number $N$,  the multiple $N \Delta_{X_{\bar K}}\in
   \operatorname{CH}^{d_X}(X_{\bar K}\times_{\bar K} X_{\bar K})$
   admits a cohomological decomposition (resp.\ cohomological $\mathbb Z_\ell$-decomposition) of type $(d_X - 1,1)$
    with respect to $H^\bullet (-,\mathbb Z_\ell)$ for some prime $\ell\ne\operatorname{char}(K)$.
  \end{enumerate}
  Then for all primes $l$ (resp.~for $l=\ell$), the morphisms
  \begin{align} \label{E:PhiTorsMod}
  \phi^2_{X_{\bar K}}[l^\infty]: &\operatorname{A}^2(X_{\bar
   K})[l^\infty]\stackrel{\sim}{\longrightarrow} \operatorname{Ab}^2_{X/K}(\bar
  K)[l^\infty]\\ \label{E:PhiTateMod}
  T_l \phi^2_{X_{\bar K}}:& T_l \operatorname{A}^2(X_{\bar
   K})\stackrel{\sim}{\longrightarrow} T_l \operatorname{Ab}^2_{X/K}
  \end{align}
  are isomorphisms of $\operatorname{Gal}(K)$-modules.
   \end{pro}

 \begin{proof}  
  Under the  hypothesis (1), 
  applying $T_l$ to both sides  of \eqref{E:PhiTorsMod}, we see that the morphism
  \eqref{E:PhiTateMod} follows from \eqref{E:PhiTorsMod}.  For
  \eqref{E:PhiTorsMod},  using Lecomte's rigidity theorem \cite{lecomte86},
  we immediately reduce to the case  $K=\mathbb C$, in which case this is
  \cite[Thm.~10.3]{murre83}.  We note that there is a small gap in Murre's proof\,;
  it is not obvious that for a surjective regular homomorphism, the induced
  morphism on $\ell$-torsion is surjective.  This uses Lemma~\ref{L:commute} and the existence of miniversal cycles, and is explained in
  \cite[Lem.~3.2(b) \& Rem.~3.3]{ACMVdmij}.
  
  Under hypothesis (2), the conclusions  follow immediately from
  Proposition~\ref{P:BlSr}.
 \end{proof}

\begin{rem}
If in Proposition \ref{P:BlSrTors} the decomposition is of type $(d_X-1,2)$, and $\ell\nmid Np$, then \eqref{E:PhiTorsMod} and \eqref{E:PhiTateMod} are isomorphisms.  This follows by combining Proposition~\ref{P:vanish} with the arguments for \cite[Prop.~3.8(3)]{ACMVBlochMap} (which addresses the case of a Chow decomposition).
\end{rem}

 \subsection{The standard assumption}
  To start our discussion, we observe that given a regular homomorphism
 $$\xymatrix{\Phi:\mathscr A^n_{X/K}\ar[r]& A},$$
 if  the natural map
 \begin{equation}\label{E:RegTate}
 \xymatrix{
 T_l \Phi(\bar K):T_l \operatorname{A}^n(X_{\bar
  K})\ar[r]^{\qquad \quad \ \simeq}& T_l A}
 \end{equation}
 is an isomorphism, then there is a canonical  morphism $\iota$ of
 $\operatorname{Gal}(K)$-modules\,:
 \begin{equation}\label{E:IotaCanA}
 \xymatrix{
  T_l  A(\bar K) \ar@{->}[rr]^<>(0.5){T_l \Phi(\bar K)^{-1}}
  \ar@/_1.5pc/[rrrr]_{\iota}& &T_l \operatorname{A}^n(X_{\bar K})\ar@{^(->}[r] &T_l
  \operatorname{CH}^n(X_{\bar K}) \ar@{->}[r]^<>(0.5){T_l\lambda^n}&H^{2n-1}(X_{\bar
   K},\mathbb Z_l(n))_\tau.
 }
 \end{equation}
 When $K=\mathbb C$, $n$ is $1$, $2$, or $d_X$,  and $A=\operatorname{Ab}^n_{X/K}$, this
 agrees with the canonical inclusion coming from Hodge theory (as the Bloch map
 agrees with the Abel--Jacobi map on torsion~\cite[Prop.~3.7]{bloch79}).
\medskip

An isomorphism of the type \eqref{E:RegTate}  in the case
 $n=2$  turns out to be quite central to our treatment and deserves to be singled out\,:

 \begin{dfn}[The standard assumption] \label{D:StAs}
  We say  a smooth projective variety  $X$ over a field $K$ and a prime  $l$
  satisfy the \emph{standard assumption}, or that $X$ satisfies the
  standard assumption at $l$, if the homomorphism 
    \begin{equation} \label{E:StAs1}
  \xymatrix@C=3em{
\phi^2_{X_{\bar K}}[l^\infty]: \operatorname{A}^2(X_{\bar
   K})[l^\infty]  \ar@{->}[r] &  \operatorname{Ab}^2_{X/K}[l^\infty] \\
}
  \end{equation}
 is an isomorphism.
   \end{dfn}

 If the standard assumption holds, then, by taking Tate modules, 
  \begin{equation} \label{E:StAs1Tate}
  \xymatrix@C=3em{
  T_l \phi^2_{X_{\bar K}}: T_l \operatorname{A}^2(X_{\bar
   K})  \ar@{->}[r] & T_l \operatorname{Ab}^2_{X/K} \\
}
  \end{equation}
      is an isomorphism, as well, 
   in which case  we will denote
  \begin{equation}\label{E:iota}
\xymatrix@C=3em{
	\iota :  \ T_l \operatorname{Ab}^2_{X/K}   \ar@{->}[r]^{(T_l \phi^2_{X_{\bar K}})^{-1}}_\sim &  T_l \operatorname{A}^2(X_{\bar
		K})
	\ar@{->}[r]^{T_l\lambda^2 \quad } &H^{3}(X_{\bar K},\mathbb Z_l(n))_\tau\\
}
  \end{equation} the composition, and similarly with $\rat_l$-coefficients.

 \begin{rem}\label{R:StAsHolds}
As explained in Proposition \ref{P:BlSrTors}, the standard assumption holds unconditionally if $\operatorname{char}(K)=0$, and  holds in positive characteristic for those smooth projective varieties $X$ whose diagonal admit a positive multiple with a cohomological decomposition of type $(d_X-1,1)$\,; 
\emph{e.g.}, geometrically rationally chain connected varieties.  We in fact expect that the standard assumption holds unconditionally\,;  to establish this 
for almost all primes,
 it would suffice to show the standard assumption holds for varieties over finite fields (see \cite[Lem.~4.3]{ACMVBlochMap}).
\end{rem}

 \section{Algebraic representatives and cohomological actions of correspondences}
 \label{S:DiagComm}
  
It is a basic fact in Hodge theory that for a family of homologically trivial cycle classes, the normal function defined via the Abel--Jacobi map induces on  cohomology the same morphism as the family of cycle classes viewed as a correspondence (see \eqref{E:commdiagHdg}).  As this fact is 
quite useful in characteristic $0$, the purpose of this section is to explain how to interpret this fact in positive characteristic, which we do in terms of the commutativity of a certain diagram
 (see \eqref{E:RegHomCorr}). The main results are Proposition~\ref{P:commute}, as well as 
 Proposition~\ref{P:commute} together with its consequence, Corollary~\ref{C:commute-2}, regarding the case of codimension-$2$ cycle classes, which essentially says that the diagram is commutative for geometrically   rationally chain connected varieties (see Remark \ref{R:co2RatCon}).
 
 \subsection{Defining the commutative diagram}\label{SS:commdiag}

 For motivation, consider the situation where $X$ and $T$ are complex projective
 manifolds, and $Z\in \operatorname{CH}^n(T\times X)$ is a cycle class that is
 fiberwise algebraically trivial, {\emph{i.e.}}, $Z\in \mathscr{A}_{X/\cx}(T)$.
     Via the Abel--Jacobi map, we obtain a morphism
 $$
 \xymatrix{
 \psi_Z:T\ar[r]& J^{2n-1}(X).} $$
 It is a standard fact (see \emph{e.g.}, \cite[Thm.~12.17]{voisinI}) that $(\psi_Z)_*$
 and the correspondence $Z_*$ agree on $H^{2d_T-1}(T,\mathbb Z)$ in the sense
 that there is a commutative diagram
 \begin{equation}\label{E:commdiagHdg}
        \xymatrix@C=1em@R=1em{
	H_1(T,\mathbb Z)  \ar[rrr]^<>(0.5){(\psi_Z)_* } \ar[rrrd]_{Z_*} &&& H_1(J^{2n-1}(X),\mathbb
	Z) \ar@{=}[rrr]& &&H^{2n-1}(X,\mathbb Z(n))_\tau\\
	&&&H^{2n-1}(X,\mathbb Z(n))\ar@{->>}[rrru]&\\
}
 \end{equation}
 We note that $\psi_Z$ has image contained in the image $J_a^{2n-1}(X)$ of the restriction of the Abel--Jacobi map to algebraically trivial cycles $AJ : \operatorname{A}^n(X) \to J^{2n-1}(X)$. If $X$ admits $K\subseteq \cx$ as a field of definition, the subtorus $J_a^{2n-1}(X)\subseteq J^{2n-1}(X)$ was shown in \cite{ACMVdmij} to admit a distinguished model $J^{2n-1}_{a,X/K} $ over $K$, in such a way that $AJ : \operatorname{A}^n(X) \to J_a^{2n-1}(X)$ is $\operatorname{Aut}(\cx/K)$-equivariant. Since the Bloch map agrees with the Abel--Jacobi map on torsion~\cite[Prop.~3.7]{bloch79}, we obtain 
 for smooth projective varieties $T,X$ over a
 field $K\subseteq \mathbb C$ and $Z\in \mathscr A^n_{X/K}(T)$ a commutative diagram\,:
 $$\xymatrix@C=1em@R=1em{  &&& T_\ell\operatorname{A}^n(X_{\bar K}) \ar[dl]_{T_\ell AJ} \ar[dr]^{T_\ell\lambda^n} \\
  H^{2d_T-1}(T_{\bar K},\mathbb Z_\ell(d_T)) \ar[rrd]_{Z_*}
  \ar[rr]^<>(0.5){(\psi_Z)_* }& & T_\ell J^{2n-1}_{a,X/K} \ar@{^(->}[rr] &&
  H^{2n-1}(X_{\bar K},\mathbb Z_\ell(n))_\tau \\
  &&H^{2n-1}(X_{\bar K},\mathbb Z_\ell(n)) \ar@{->>}[rru]& \\
 }$$

As the Abel--Jacobi map is complex in nature,
 it is not immediately clear how to generalize this statement to
 varieties over  fields of positive characteristic. 
    However, as algebraic representatives provide a replacement for the Abel--Jacobi maps in positive characteristic, let us assume that there exists an algebraic representative
 $\Phi_{X/K}^n:\mathscr A^n_{X/K}\to \operatorname{Ab}^n_{X/K}$, which is always the case if $n=1,2,d_X$.
 While in general it is unclear whether there is a canonical map $T_l
 \operatorname{Ab}^n_{X/K} \to H^{2n-1}(X_{\bar K},\mathbb Z_l(n))_\tau$, if we assume that 
$$  T_l \phi^n_{X_{\bar K}}: T_l
 \operatorname{A}^n(X_{\bar K})\stackrel{\sim}{\longrightarrow} T_l
 \operatorname{Ab}^n_{X/K}$$  is an isomorphism,
       then we obtain a diagram
 \begin{equation}\label{E:RegHomCorr}
 \xymatrix@C=.5em@R=1em{
  H^{2d_T-1}(T_{\bar K},\mathbb Z_l(d_T)) \ar[rrrd]_{Z_*}
  \ar[rr]^<>(0.5){(\psi_Z)_* } 
   & & T_l {\Ab}^{n}_{X/K}
  \ar[rr]^{(T_l \phi_X^n)^{-1}}_{\simeq} && T_l\A^n(X_{\bar K})
  \ar@{->}[rr]^<>(0.5){T_l\lambda^n} && H^{2n-1}(X_{\bar K},\integ_l(n))_\tau \\
  &&&H^{2n-1}(X_{\bar K},\mathbb Z_{l}(n))\ar@{->>}[rrru]& \\
 }
 \end{equation}
       Note that if $n=1$ or $n=d_X$, then $T_l \phi^n_{X_{\bar K}}$ is an isomorphism for all primes $l$ (for $n=1$ this is Kummer theory,
  while for $n=d_X$ this is Rojtman's theorem\,; see~\cite{bloch79, grossuwaAJ,milne82}, and \cite[Appendix]{ACMVBlochMap} for a review). Recall also from Remark~\ref{R:StAsHolds} that $T_l \phi^n_{X_{\bar K}}$ is an isomorphism for $n=2$ in characteristic zero and for geometrically rationally chain connected varieties in positive characteristic.
  \medskip

 The question is then the following:
 \begin{que}\label{Q:CorrRegHom}
 Assuming $  T_l \phi^n_{X_{\bar K}}: T_l
 \operatorname{A}^n(X_{\bar K})\stackrel{\sim}{\longrightarrow} T_l
 \operatorname{Ab}^n_{X/K}$ is an isomorphism, under what conditions is diagram \eqref{E:RegHomCorr} commutative\,?
 \end{que}

 In Lemma~\ref{L:cod1} below, we establish the commutativity of \eqref{E:RegHomCorr} in the case $n=1$ and in Corollary~\ref{C:commute-d} below, we establish the commutativity of \eqref{E:RegHomCorr} in the case $n=d_X$.
If $n=2$ and $K\subseteq \cx$, the algebraic representative $\Phi_{X/K}^2$ coincides with the Abel--Jacobi map after base-change to $\cx$. As such, under the above conditions, the diagram \eqref{E:RegHomCorr} commutes. The commutativity of \eqref{E:RegHomCorr} for $n=2$ and for perfect fields $K$ of positive characteristic will be established in Corollary~\ref{C:commute-2} below, assuming that $T_l \phi^2_{X_{\bar K}}$ is an isomorphism and that $V_l\lambda^2$ is surjective.

 \begin{rem}\label{R:CorrRegHom}
  Note that since $H^{2n-1}(X_{\bar K},\mathbb Z_l(n))_\tau$ is by definition
  torsion-free, it suffices to show that \eqref{E:RegHomCorr} is commutative with
  $\mathbb Q_l$-coefficients.
  Indeed, suppose $M$ and $N$ are $\mathbb Z_l$-modules, and $f,g:M\to N$ are
  $\mathbb Z_l$-module homomorphisms  that agree after tensoring with $\mathbb
  Q_l$.  If $N$ is torsion-free, then $f=g$:  given $m\in M$, we have that
  $r(f(m)-g(m))=0$ for some $r\in \mathbb Z_l$\,; but since $N$ is torsion-free,
  this implies that $f(m)=g(m)$.
 \end{rem}

 \subsection{Diagram \eqref{E:RegHomCorr} in the case $n=1$}

In case $\operatorname{char}(K)=0$, as explained in \S \ref{SS:commdiag}, the diagram~\eqref{E:RegHomCorr} with $n=1$ commutes. We now observe that it also commutes with $n=1$ for any perfect field $K$.
 We start by recalling the following fact\,:

\begin{pro}\label{P:coniveaucurve}
	Let $X$ be a smooth projective variety over a perfect field $K$, and let $n$ be a natural number.
Then
there exist a smooth projective, geometrically integral,  curve $C$ over $K$ admitting a $K$-point
and a
correspondence $\gamma \in \operatorname{CH}^n(X\times_K C)$
such that for all primes $\ell\ne \operatorname{char}(K)$  
 $$\coniveau^{n-1}H^{2n-1}(X_{\bar K},\mathbb Q_\ell(n)) = \operatorname{Im}\big( \gamma_* :H^{1}(C_{\bar K},\mathbb
Q_\ell (1)) \longrightarrow H^{2n-1}(X_{\bar K},\mathbb
Q_\ell (n))\big).$$
Here, $\coniveau^\bullet$ is the geometric coniveau filtration.
  \end{pro}
\begin{proof}
This is \cite[Prop.~1.1]{ACMVdmij} which actually holds over any perfect field $K$.
\end{proof}

 \begin{lem}[Diagram \eqref{E:RegHomCorr} with $n=1$]\label{L:cod1}
  Let $X$ be a smooth projective variety over a perfect field $K$ and let $l$ be a prime
  number.
  The map
  $ T_l \phi^1_{X_{\bar K}}: T_l \operatorname{A}^1(X_{\bar
   K})\stackrel{\sim}{\longrightarrow} T_l \operatorname{Ab}^1_{X/K}$
  is an isomorphism, and for any smooth projective variety $T$ and any $Z\in
  \mathscr A^1_{X/K}(T)$, the diagram \eqref{E:RegHomCorr} commutes in the case
  $n=1$. In particular, if $T=B$ is an abelian variety, then $T_l\lambda^1 \circ T_lw_Z = Z_*: T_l B \to H^1(X_{\bar K},\integ_l(1))$.
 \end{lem}

 \begin{proof} 
 Using the fact that $\operatorname{Ab}^1_{X/K}=(\operatorname{Pic}^0_{X/K})_{\operatorname{red}}$ (\emph{e.g.}, \cite[Rem.~7.2]{ACMVfunctor}), it follows from the definition of the Picard functor that $\phi^1_{X_{\bar K}}:  \operatorname{A}^1(X_{\bar
 		K})\stackrel{\sim}{\longrightarrow} \operatorname{Ab}^1_{X/K}(\bar K)$ is an isomorphism. Taking Tate modules gives 
that the map  $ T_l \phi^1_{X_{\bar K}}: T_l \operatorname{A}^1(X_{\bar
 		K})\stackrel{\sim}{\longrightarrow} T_l \operatorname{Ab}^1_{X/K}$ is an isomorphism.  	
 	  		
 		We now proceed to establish the commutativity of~\eqref{E:RegHomCorr} in the case
 		$n=1$.
 	We first consider the case where $T$ is a curve.
  The key point is then to use the identification, valid even when
  $l = \operatorname{char}(K)$ (\emph{e.g.}, \cite[Cor.~3.3]{grossuwaAJ}), that
  $H^1(X_{\bar K},\mmu_{l^\nu})=\operatorname{Pic}^0_{X_{\bar K}/\bar
   K}(\bar K)[l^\nu]$, obtained via identifying
  $\mmu_{l^\nu}$-torsors over $X$ with \'etale covers, and then with
  torsion line bundles (and similarly for $T$).  Using that $d_T=1$ (so that
  $H^{2d_T-1}(T_{\bar K},\mathbb Z_l(d_T))=H^1(T_{\bar K},\mathbb
  Z_l(1))$),  the commutativity follows from the definitions of the maps.

  For the general case, it is slightly more convenient to use
  Remark~\ref{R:CorrRegHom}, and prove commutativity with $\mathbb
  Q_l$-coefficients.
  
  Initially, suppose $l = \ell \ne \operatorname{char}(K)$.
  To start, we use Proposition~\ref{P:coniveaucurve}:
  there exist  a smooth projective curve $C$ over $K$ and a
  correspondence $\gamma \in \operatorname{CH}^1(C\times_K T)$
  such that $H^{2d_T-1}(T_{\bar K}, \mathbb Q_\ell(d_T)) = \gamma_*H^1(C_{\bar K},
  \mathbb Q_\ell(1))$.
    We then consider the diagram
  \begin{equation}
  \label{E:cod1}
  \xymatrix@C=.5em@R=1em{
   H^1(C_{\bar K})(1)\ar@{->>}[rr]^<>(0.5){\gamma_*} \ar@/_1pc/[rrrrrd]_{(Z\circ
    \gamma)_*} \ar@/^2pc/[rrrr]^{(\psi_{Z\circ \gamma})_*}&&  H^{2d_T-1}(T_{\bar
    K})(d_T) \ar[rrrd]_{Z_*} \ar[rr]^<>(0.5){(\psi_Z)_* }& & V_l {\Ab}^{1}_{X/K}
   \ar[rr]^{(V_l \phi_X^1)^{-1}}_{\simeq} && V_l\A^1(X_{\bar K})
   \ar@{->}[rr]^<>(0.5){\lambda_X^1} && H^{1}(X_{\bar K})(1) \\
   && &&&H^{1}(X_{\bar K})(1)\ar@{=}[rrru]& \\
  }
  \end{equation}
  where we are using $\mathbb Q_\ell$-coefficients.
  A diagram chase then completes the proof of commutativity.

  We now consider the case where $l = p = \operatorname{char}(K)>0$.
     The same $\gamma$ and $C$ used before yield a surjection of
  $F$-isocrystals $\gamma_*:H^1_\cris(C/\mathbb B(\bar K))(1) \to
  H^{2d_{T}-1}_\cris(T/\mathbb B(\bar K))(d_T)$.  Upon taking Frobenius invariants, we
  obtain
  a surjection $\gamma_{*,p}:H^1(C_{\bar K},\rat_p(1)) \to
  H^{2d_T-1}(X_{\bar K},\rat_p(d_T))$.  The same chase of
  \eqref{E:cod1} establishes the commutativity.
  
  Finally, that $T_l\lambda^1 \circ T_lw_Z = Z_*$ in case $T$ is an abelian variety follows from the commutativity of~\eqref{E:RegHomCorr} and Lemma~\ref{L:ablambda}.
 \end{proof}

 \subsection{Diagram \eqref{E:RegHomCorr} in the case $n>1$}\label{SS:diagcommute2}
  We start with the following observation which answers positively the conjecture \cite[Conj.~III.4.1(iii)]{grossuwaAJ} in case $2n-1\leq d_X$ (and in case $2n-1>d_X$ provided $X$ satisfies the Lefschetz standard conjecture)\,:

 \begin{lem}[{\cite[Prop.~6.1 \& Rem.~6.2]{ACMVBlochMap}}]
  \label{L:factorM}
  Let $X$ be a smooth projective variety over a perfect field~$K$.
Let $n$ be a natural number and let  $\ell_0 \not = \operatorname{char}(K)$ be a prime.
Suppose that
  \begin{itemize}
  	\item $
  		V_{\ell_0}\lambda^m : V_{\ell_0}\operatorname{A}^m(X_{\bar K})
  		\longrightarrow H^{2m-1}(X_{\bar K},\mathbb Q_{\ell_0}(m))
  	$ is surjective for $m:=\min \{n, d_X-n+1\}$.
  \end{itemize} 
Then  $
V_{l}\lambda^n : V_{l}\operatorname{A}^n(X_{\bar K})
\longrightarrow H^{2n-1}(X_{\bar K},\mathbb Q_{l}(n))
$ is surjective for all primes $l$. 

 In addition,
   there exist an abelian variety $A$ over $K$,  and
  correspondences $\Gamma \in \operatorname{CH}^{d_X+1-n}(X\times_K
  \widehat{A})$
  and $\Gamma' \in \operatorname{CH}^{n}(
  A\times_K X)$, which  induce for all primes $l$
  isomorphisms of $\operatorname{Gal}(K)$-modules
  \begin{equation}\label{E:Mid}
  \xymatrix{
 \Gamma_{*}  : H^{2n-1}(X_{\bar K},
   \mathbb Q_l(n)) \ar[r]^{\qquad \qquad  \sim} & V_l A & \text{and} &  \Gamma'_{*}  :  V_l A\ar[r]^{\sim\quad\quad } & H^{2n-1}(X_{\bar K},
   \mathbb Q_\ell(n)).
  }
  \end{equation}
      \end{lem}

 \begin{proof} This is essentially
   \cite[Prop.~6.1]{ACMVBlochMap} and we provide a proof here for convenience.
  We start with the given  ${\ell_0} \ne\operatorname{char}(K)$. We can of course assume $0\le n \le d$, so that $0\le m\le d/2$. Since $V_{\ell_0}\lambda^m$ is surjective, and since $\operatorname{im}(V_{\ell_0}\lambda^m) = \coniveau^{m-1}H^{2m-1}(X_{\bar K},\rat_{{\ell_0}}(m))$ (see \cite[Prop.~5.2]{Suwa} or \cite[Prop.~2.1]{ACMVBlochMap}), there exist by
  Proposition~\ref{P:coniveaucurve} a smooth projective, geometrically integral,  curve $C$ over $K$ admitting a $K$-point
  and a
  correspondence $\gamma \in \operatorname{CH}^m(X\times_K C)$
  inducing a surjection
   $$
  \xymatrix{
   H^1(C_{\bar K},\mathbb Q_{\ell_0}(1)) \ar@{->>}[r]^<>(0.5){\gamma^*}&  H^{2m-1}(X_{\bar
    K},\mathbb Q_{\ell_0}(m)).
  }
  $$
  Taking a hyperplane section $H_X$ and dualizing the surjection above we obtain
  an injection
  \[
  \xymatrix{
   H^{2m-1}(X_{\bar K},\mathbb Q_{\ell_0}(m)) \ar[r]^<>(0.5){H_X^{d-2m+1}}_<>(0.5){\sim} &
   H^{2d-2m+1}(X_{\bar K},\mathbb Q_{\ell_0}(d-m+1)) \ar@{^(->}[r]^<>(0.5){\gamma_*}  &
   H^1(C_{\bar K},\mathbb Q_{\ell_0}(1)). \\
  }
  \]
  The correspondence $\gamma_* \circ H_X^{d-2m+1} \circ \gamma^* \in \chow^1(C\times_K C)_\rat$ induces, via the choice of a $K$-point on $C$, an element of $f\in \operatorname{End}(\mathrm{Alb}_C)\otimes \rat$. 
  Using the semi-simplicity of the category of abelian varieties over $K$ up to isogeny, taking $A$ to be the image of $f$, and clearing  denominators, we obtain correspondences  $\Gamma \in \operatorname{CH}^{d_X+1-m}(X\times_K
  \widehat{A})$
  and $\Gamma' \in \operatorname{CH}^{m}(
  A\times_K X)$ such that $\Gamma\circ \Gamma' \in \chow^1(A\times \widehat{A})$ induces an isogeny $A\to A$ and such that
 $$   \xymatrix{
  	\Gamma_{*}  : H^{2m-1}(X_{\bar K},
  	\mathbb Q_{\ell_0}(m)) \ar[r]^{\qquad \qquad  \sim} & V_{\ell_0} A & \text{and} &  \Gamma'_{*}  :  V_{\ell_0} A\ar[r]^{\sim\qquad } & H^{2m-1}(X_{\bar K},
  	\mathbb Q_{\ell_0}(m))
  }
  $$
  are isomorphisms.
  In case $m=d-n+1$, by dualizing, we also obtain isomorphisms
   $$   \xymatrix{
  	^{t}\Gamma'_{*}  : H^{2n-1}(X_{\bar K},
  	\mathbb Q_{\ell_0}(n)) \ar[r]^{\qquad \qquad  \sim} & V_{\ell_0} \widehat A & \text{and} &  {}^t \Gamma_{*}  :  V_{\ell_0} \widehat A\ar[r]^{\sim\qquad } & H^{2n-1}(X_{\bar K},
  	\mathbb Q_{\ell_0}(n)).
  }
  $$
  We conclude that the homomorphisms~\eqref{E:Mid} are isomorphisms for all primes $\ell\ne p$ using the invariance of the $\ell$-adic Betti numbers for all $\ell\neq p$.
     Since crystalline and $\ell$-adic Betti numbers coincide, $\Gamma$ induces an isomorphism of isocrystals $H^{2n-1}_{\operatorname{cris}}(X/\mathbb B(K))(n) \iso H^1_{\operatorname{cris}}(A/\mathbb B(K))$\,; taking Frobenius invariants yields \eqref{E:Mid} with $l=p$. 
     
     Finally, from the right-hand side isomorphism of~\eqref{E:Mid}, we see that for all primes $l$ we have $H^{2n-1}(X_{\bar K},\rat_l(n)) = {\coniveau}^{n-1} H^{2n-1}(X_{\bar K},\rat_l(n))$. We conclude, by \cite[Prop.~2.1]{ACMVBlochMap} again, that $V_l \lambda^n$ is surjective.
   \end{proof}

 The main result of this section is the following proposition providing an answer
 to Question~\ref{Q:CorrRegHom} under certain conditions.

 \begin{pro} \label{P:commute}
       Let $X$ be a smooth projective variety  
  over a perfect field~$K$. Suppose that $X$ admits an algebraic representative $\Phi^n_X : \mathscr{A}^n_{X/K} \to \operatorname{Ab}^n_{X/K}$ in codimension $n$ and suppose that there exists a prime $\ell_0\neq \operatorname{char}(K)$ 
   such that
  $V_{\ell_0}\lambda^m : V_{\ell_0}\operatorname{A}^m(X_{\bar K})
  \longrightarrow H^{2m-1}(X_{\bar K},\mathbb Q_{\ell_0}(m))$ is surjective for $m:=\min \{n, d-n+1\}$.

Let $l$ be any prime such that $  T_l \phi^n_{X_{\bar K}}: T_l
 \operatorname{A}^n(X_{\bar K})\stackrel{\sim}{\longrightarrow} T_l
 \operatorname{Ab}^n_{X/K}$ is an isomorphism.    Then for any smooth projective variety
  $T$ and any $Z\in \mathscr A^n_{X/K}(T)$, the diagram
  \eqref{E:RegHomCorr} commutes.
 \end{pro}

 \begin{proof}
  As mentioned in Remark~\ref{R:CorrRegHom}, it suffices to prove the commutativity of
  \eqref{E:RegHomCorr}  with $\mathbb Q_l$-coefficients.
  The proof consists in reducing to the case of codimension-1 cycles by showing
  that the diagram~\eqref{E:RegHomCorr} with $\rat_l$-coefficients is the direct
  summand of a similar diagram with $\Ab^1$ in place of $\Ab^n$, in which case the
  commutativity is proven in Lemma~\ref{L:cod1}.
 
  Using the given $\ell_0$, choose  $A$, $\Gamma$ and $\Gamma'$ as in Lemma~\ref{L:factorM}\,; recall that these objects induce the isomorphisms \eqref{E:Mid} for \emph{all} $l$.
  By the universal property of the algebraic representatives, the correspondence
  $\Gamma$ induces a $K$-morphism $f : \Ab^n_X \to \Ab^1_A = A^\vee$ and the correspondence $\Gamma'$ induces a $K$-morphism $g :   \Ab^1_A \to \Ab^n_X$, making the
  following diagrams commute\,:
  $$\xymatrix{\A^n(X_{\bar K}) \ar[r]^{\phi^n_{X_{\bar K}}} \ar[d]^{\Gamma_*} &
   \Ab^n_X(\bar K) \ar[d]^{f} & & \A^1(A_{\bar K}) \ar[r]^{\phi^1_{A_{\bar K}}} \ar[d]^{\Gamma'_*} &
   \Ab^1_A(\bar K)  \ar[d]^{g} \\
   \A^1(A_{\bar K}) \ar[r]^{\phi^1_{A_{\bar K}}} &  \Ab^1_A(\bar K) & &    \A^n(X_{\bar K}) \ar[r]^{\phi^n_{X_{\bar K}}} &  \Ab^n_X(\bar K).}
  $$
  Given a regular homomorphism $\Phi:\mathscr A^i_{X/K}\to D$ and a smooth separated
  variety $T$  over $K$,
  a correspondence $\Theta \in \mathscr {A}^i_{X/K}(T)$ induces a $K$-morphism
  $\psi_{\Theta} : T\to D$, which itself induces a morphism $(\psi_{\Theta})_* :
  H^{2d_T - 1}(T_{\bar K},\integ_l(d_T)) \to T_l D$.
        We obtain a commutative diagram
 
  \begin{equation}
  \label{E:commute}
  \xymatrix@C=1em @R=1.5em{H^{2d_T-1}(T_{\bar K},\rat_l(d_T)) \ar[rr]^{\qquad \psi_{Z,*}}
   \ar[rrd]_{\qquad \psi_{\Gamma\circ Z,*}} 
  \ar[rrdd]_{\qquad   \psi_{\Gamma'\circ \Gamma\circ Z,*}}
  && V_l \Ab^n_{X/K} \ar[d]^{V_l f}
   \ar[rr]^{(V_l \phi_{X_{\bar K}}^n)^{-1}}_{\simeq} && V_l\A^n(X_{\bar K}) \ar[d]^{\Gamma_*}
   \ar@{->>}[rr]^{\lambda_X^n \quad } && H^{2n-1}(X_{\bar K},\rat_l(n)) \ar[d]^{\Gamma_*} \\
   && V_l \Ab^1_{A/K} \ar[d]^{V_l g} \ar[rr]^{(V_l \phi_{A_{\bar K}}^1)^{-1}}_{\simeq} &&
   V_l\A^1(A_{\bar K}) \ar[d]^{\Gamma'_*} \ar[rr]^{\lambda_A^1\quad }_{\simeq\quad } &&
   H^1(A_{\bar K},\rat_l(1))  \ar[d]^{\Gamma'_*} \\
   && V_l \Ab^n_{X/K} \ar[rr]^{(V_l \phi_{X_{\bar K}}^n)^{-1}}_{\simeq} && V_l\A^n(X_{\bar K})
   \ar@{->>}[rr]^{\lambda_X^n \quad} && H^{2n-1}(X_{\bar K},\rat_l(n)).
  }
  \end{equation}
  The right squares commute thanks to the naturality of the Bloch map (see
  \cite{grossuwaAJ} for the case $l=p$), the middle squares commute by
  construction of $f$ and $g$ and the left part of the diagram
  commutes by the definition of regular homomorphisms.
  The commutativity of~\eqref{E:commute} yields $\Gamma'_*\circ \Gamma_*\circ
  \lambda^n_X \circ (V_l\phi_X^2)^{-1} \circ \psi_{Z_*} = \Gamma'_* \circ \lambda^1_A \circ (V_l \phi_A^1)^{-1} \circ \psi_{\Gamma\circ Z,*} = \Gamma'_*\circ \Gamma_* \circ Z_* $,
    \emph{i.e.},
  that the diagram~\eqref{E:RegHomCorr} commutes after composing with $\Gamma'_*\circ \Gamma_*$. Since
  $\Gamma'_*\circ \Gamma_*: H^{2n-1}(X_{\bar K},\rat_l(n)) \to H^{2n-1}(X_{\bar K},\rat_l(n))$ is an isomorphism thanks to Lemma~\ref{L:factorM}, we conclude the diagram~\eqref{E:RegHomCorr}  commutes with $\mathbb Q_l$-coefficients.
 \end{proof}

We obtain unconditionally the commutativity of the diagram~\eqref{E:RegHomCorr} in case $n=
\dim X$\,:

\begin{cor} [Diagram \eqref{E:RegHomCorr} with
	$n=\dim X$] \label{C:commute-d} 	Let $X$ be a smooth projective variety of dimension~$d$
	over a perfect field~$K$.
	Then, for any smooth projective variety
	$T$, any $Z\in \mathscr A^d_{X/K}(T)$ and any prime $l$, the diagram
	\eqref{E:RegHomCorr} with $n=d$  commutes.
\end{cor}
\begin{proof} 
	Recall that an algebraic representative for codimension-$d$ cycles on $X_{\bar K}$ is given by the albanese morphism. The latter is an isomorphism on torsion by Rojtman's theorem and so $T_l \phi^d_{X_{\bar K}}$ is an isomorphism. On the other hand $T_l \lambda^1$ is an isomorphism by Kummer theory. The assumptions of Proposition~\ref{P:commute} are met and we can conclude. 
\end{proof}

Finally, since this will be important to us, we state explicitly Proposition~\ref{P:commute} in the case $n=2$\,:

 \begin{cor} [Diagram \eqref{E:RegHomCorr} with
	$n=2$] \label{C:commute-2} 	Let $X$ be a smooth projective variety
	over a perfect field~$K$.
	If $\operatorname{char}(K)>0$, suppose that there exists a prime $\ell_0\neq \operatorname{char}(K)$ such that $V_{\ell_0} \lambda^2: V_{\ell_0} \A^2(X_{\bar K}) \to H^3(X_{\bar K}, \rat_{\ell_0}(2))$ is surjective.

          If $X$ satisfies the standard assumption at $l$ (Definition~\ref{D:StAs}), then for any smooth projective variety
	$T$ and for any $Z\in \mathscr A^2_{X/K}(T)$, the diagram
	\eqref{E:RegHomCorr} with $n=2$ commutes.
	 	 	 \end{cor}
\begin{proof}
	The case where $\operatorname{char}(K)=0$ was explained in \S \ref{SS:commdiag} (and the standard assumption is satisfied for all $l$).  The case where $\operatorname{char}(K)>0$ is Proposition~\ref{P:commute} in case $d_X>2$ and Corollary~\ref{C:commute-d} in case $d_X=2$ (in which case the assumptions on $T_\ell \phi^2_{X_{\bar K}}$ and $V_\ell \lambda^2$ are superfluous).
\end{proof}

\begin{rem}\label{R:co2RatCon}
Note that, due
 to Proposition~\ref{P:lambdacoho} and 
Proposition \ref{P:BlSr}, the assumptions of Corollary~\ref{C:commute-2} are met for smooth projective varieties $X$ over a perfect field $K$ whose $\chow_0(X)_\rat$ is universally supported on a curve, \emph{e.g.}, for smooth projective geometrically rationally chain connected varieties, or for smooth projective varieties with geometric MRC quotient of dimension $\leq 1$.
\end{rem}

\subsection{Commutativity in the case of an abelian variety}
     In case the parameter space $T=B$ is an abelian variety, we can rephrase the commutativity of \eqref{E:RegHomCorr} under less restrictive assumptions.  The point is that in this case \eqref{E:RegHomCorr}  becomes
$$
 \xymatrix@C=.5em@R=1em{
T_lB \ar[rrrrd]_{Z_*}
  \ar[rrr]^<>(0.7){(\psi_Z)_* } 
  \ar@/^2.5pc/[rrrrr]_{T_\ell w_{Z}}
  && & T_l {\Ab}^{n}_{X/K}
  \ar[rr]^{(T_l \phi_X^n)^{-1}}_{\simeq} && T_l\A^n(X_{\bar K})
  \ar@{->}[rr]^<>(0.5){T_l\lambda^n} && H^{2n-1}(X_{\bar K},\integ_l(n))_\tau \\
 & &&&H^{2n-1}(X_{\bar K},\mathbb Z_{l}(n))\ar@{->>}[rrru]& \\
 }
$$
and therefore,  ignoring the existence of the algebraic representative, and whether $T_l\phi^n_X$ is an isomorphism if the algebraic representative exists, we can rephrase commutativity as\,:

\begin{pro}  \label{P:commute2}
 	Let $X$ be a smooth projective variety
	over a perfect field~$K$ and let $n$ be a natural number.
Suppose there exists a prime $\ell_0 \neq \operatorname{char}(K)$
	such that:
	\begin{itemize}
		\item $V_{\ell_0}\lambda^m : V_{\ell_0}\operatorname{A}^m(X_{\bar K})
		\longrightarrow H^{2m-1}(X_{\bar K},\mathbb Q_{\ell_0}(m))$ is surjective for $m:=\min \{n, d-n+1\}$.
	\end{itemize}
	Then, for any abelian variety $B$ over $K$, any $Z\in \mathscr A^n_{X/K}(B)$ and any prime $l$, the diagram
\begin{equation}\label{E:commute2}
	\xymatrix{ T_l B \ar[r]^{w_Z \quad } \ar[rd]^{Z_*\ } & T_l \operatorname{A}^n(X_{\bar K}) \ar[d]^{T_l\lambda^n}\\ &H^{2n-1}(X_{\bar K},\integ_l(n))_\tau
}
\end{equation}
is commutative.
\end{pro}
\begin{proof}
Recall from Lemma~\ref{L:ablambda} that in our situation $w_Z$ induces a map on $l$-adic Tate modules.
Since $T_l B$ and $H^{2n-1}(X_{\bar K},\integ_l(n))_\tau$ are torsion-free, it suffices to prove commutativity of
\eqref{E:RegHomCorr}  with $\mathbb Q_l$-coefficients.
The proof consists then in using Lemma~\ref{L:factorM} to reduce to the case of codimension-1 cycles, in which case the
commutativity is proven in Lemma~\ref{L:cod1}.

Using $\ell_0$, choose  $A$, $\Gamma$ and $\Gamma'$ as in Lemma~\ref{L:factorM}. We obtain a commutative diagram

\begin{equation}
 \xymatrix@R=1.5em{T_l B \ar[rr]^{ w_{Z}}
	\ar[rrd]_{ w_{\Gamma\circ Z}} 
	 	&&  T_l\A^n(X_{\bar K}) \ar[d]^{\Gamma_*}
	\ar[rr]^{T_l\lambda_X^n\quad } && H^{2n-1}(X_{\bar K},\rat_l(n))_\tau \ar[d]^{\Gamma_*} \\
	&& 
	T_l\A^1(A_{\bar K}) \ar[d]^{\Gamma'_*} \ar[rr]^{T_l\lambda_A^1\quad }_{\simeq\quad } &&
	H^1(A_{\bar K},\rat_l(1))  \ar[d]^{\Gamma'_*} \\
	&&  T_l\A^n(X_{\bar K})
	\ar[rr]^{T_l\lambda_X^n\quad } && H^{2n-1}(X_{\bar K},\rat_l(n))_\tau.
}
\end{equation}
The  squares commute thanks to the naturality of the Bloch map and the left part of the diagram
commutes by the definition of the maps $w_Z$ and $w_{\Gamma\circ Z}$.
Therefore $\Gamma'_*\circ \Gamma_*\circ
T_l\lambda^n_X \circ w_Z = \Gamma'_* \circ T_l \lambda^1_A \circ w_{\Gamma\circ Z} = \Gamma'_*\circ (\Gamma \circ Z)_* $, where the second equality follows from Lemma~\ref{L:cod1}.
 Thus the diagram~\eqref{E:commute2} commutes after composing with $\Gamma'_*\circ \Gamma_*$. Since
$\Gamma'_*\circ \Gamma_*: H^{2n-1}(X_{\bar K},\rat_l(n)) \to H^{2n-1}(X_{\bar K},\rat_l(n))$ is an isomorphism, we conclude the diagram~\eqref{E:commute2}  commutes with $\mathbb Q_l$-coefficients.
\end{proof}

\section{Cohomological decomposition and self-duality of the algebraic representative}\label{S:DiagPolCoh}

Let $X$ be a smooth projective threefold over a field $K$ and assume that $X_{\bar K}$ admits a universal codimension-2 cycle $Z$.
The aim of this section is to study the symmetric $\bar K$-homomorphism $$\Theta_X: \operatorname{Ab}^2_{X_{\bar K}/\bar K} \to \widehat{\operatorname{Ab}}\,^2_{X_{\bar K}/\bar K}$$ induced (see \S \ref{SS:RegHomCycle}) by the cycle $-({}^tZ\circ Z)\in \chow^1(\operatorname{Ab}^2_{X_{\bar K}/\bar K}  \times_{\bar K} \operatorname{Ab}^2_{X_{\bar K}/\bar K} )$.  
Under the assumption that $V_\ell \lambda^2$ is surjective for some prime $\ell \neq \operatorname{char}(K)$, we show in Theorem~\ref{T:CanPol} that $\Theta_X$ is an isogeny that descends to $K$ and is independent of the choice of the universal cycle $Z$. 
Moreover, in characteristic~$0$, in the case where $X$ is geometrically rationally connected, 
we show that $\Theta_X$ is the Hodge-theoretic polarization induced via the intersection pairing in cohomology (see Remark~\ref{R:PPolDM-Z}), 
and that a similar statement holds in positive characteristic for the symmetric $K$-isogeny $\Theta_X$ induced by $-({}^tZ\circ Z)$ (the precise meaning of this is explained in Definition~\ref{D:DistMorph}).  In particular, this shows that the isomorphism in Theorem~\ref{T:main-pol}, in characteristic~$0$,  agrees with the Hodge-theoretic polarization induced via the intersection pairing in cohomology.
In addition, let us already mention that, in positive characteristic, we will show in Proposition~\ref{P:ThetaSpecMini} that $\Theta_X$ is  a polarization under some liftability conditions to characteristic zero.  For instance, in Corollary \ref{C:ThetaSpecMini}, we will show that if $X$ is a geometrically stably rational threefold, and if $X$ lifts to characteristic $0$ as a geometrically rationally connected threefold, then $\Theta_X$ is a principal polarization.  

\subsection{Motivation from Hodge theory: morphisms of abelian varieties induced by cup product in cohomology}\label{S:DistMorph}
Let $X$ be a complex projective manifold, let $H$ be an ample divisor on $X$, let $n$ be a nonnegative integer with  $1\le  2n-1 \le d_{X}$, and assume
$\coniveau^{n-1}H^{2n-1}(X,\mathbb Q)= H^{2n-1}(X,\mathbb Q)$, which implies 
 $$H^{2n-1}(X,\mathbb C)=H^{n+1,n}(X)\oplus H^{n,n-1}(X).$$
Then the Hodge--Riemann bilinear pairing $(\alpha,\beta)\mapsto i(-1)^{n}\int_X \alpha \wedge \bar \beta \wedge [H]^{d_X-2n-1}$ gives a Hermitian form $h$ on  $H^{n,n+1}(X)$ so  that $-\operatorname{Im}2h$ is the cup product in cohomology (up to a sign)\,:
$$
H^{2n-1}(X,\mathbb Z)_\tau \otimes H^{2n-1}(X,\mathbb Z)_\tau\longrightarrow H^{2d_X}(X,\mathbb Z)= \mathbb Z
$$
$$
(\alpha ,\beta)\mapsto (-1)^{n-1}\alpha \cup \beta \cup [H]^{d_X-2n-1}
$$ 

The associated linear map  $2h:H^{n-1,n}(X)\to \overline {H^{n-1,n}(X)}$ therefore  induces a symmetric isogeny
$$\Theta_{X}:J^{2n-1}(X) \to \widehat J^{2n-1}(X)$$
 on the  intermediate Jacobian.   Note that under the assumption $\coniveau^{n-1}H^{2n-1}(X,\mathbb Q)=H^{2n-1}(X,\mathbb Q)$, we have that  $J^{2n-1}(X)$ is equal to $J^{2n-1}_a(X)$,   \emph{i.e.}, the image of the Abel--Jacobi map on algebraically trivial cycle classes, which is an abelian variety.    
 
 The above discussion  can be rephrased as saying that $\Theta_X$ induces  a commutative diagram
$$
\xymatrix@C=5em@R=1.5em{
	H_1(J^{2n-1}(X),\mathbb Z) \ar[d]_= \ar[r]^{\Theta_{X}}& H_1(\widehat J^{2n-1}(X), \mathbb Z)\\
	H^{2n-1}(X,\mathbb Z)_\tau \ar[r]^{\cup (-1)^{n-1}[H]^{d_X-2n+1}} & H^{2d_X-2n+1}(X,\mathbb Z)_\tau\ . \ar[u]_=
}
$$
The left vertical arrow is the canonical identification coming from the construction of the intermediate Jacobian, while the right vertical arrow is the dual identification, where we identify $H_1(J^{2n-1}(X),\mathbb Z)^\vee = H_1(\widehat J^{2n-1}(X), \mathbb Z)$ via the Weil pairing, and we identify $H^{2n-1}(X,\mathbb Z)^\vee_\tau = H^{2d_X-2n+1}(X,\mathbb Z)_\tau$ via the cup product.  We review this standard Hodge theory in \S \ref{S:A-HdgTh}.

Taking the Tate module of the Abel--Jacobi map gives us two equivalent maps \cite[Prop.~3.7]{bloch79}, namely the maps $T_\ell AJ: T_\ell \operatorname{A}^n(X)\to T_\ell J^{2n-1}_a(X)$ and $T_\ell \lambda^n: T_\ell \operatorname{A}^n(X)\to H^{2n-1}(X,\mathbb Z_\ell)_\tau$, and we can rephrase the diagram above  $\ell$-adically as saying the following diagram commutes\,:
$$
\xymatrix@C=5em@R=1.5em{
	H_1(J^{2n-1}(X),\mathbb Z_\ell) \ar[r]^{T_\ell \Theta_{X}}& H_1(\widehat J^{2n-1}(X), \mathbb Z_\ell)
	\ar[d]^{(T_\ell AJ)^\vee} \\
	T_\ell \operatorname{A}^n(X) \ar[d]_{T_\ell \lambda^n}  \ar[u]^{T_\ell AJ}  & T_\ell \operatorname{A}^n(X)^\vee  \\
	H^{2n-1}(X,\mathbb Z_\ell)_\tau \ar[r]^{\cup (-1)^{n-1} [H]^{d_X-2n+1}} & H^{2d_X-2n+1}(X,\mathbb Z_\ell)_\tau \ar[u]_{(T_\ell \lambda^n)^\vee}
}
$$

The isogeny $\Theta_{X}$ is in fact the only morphism $\Theta :J^{2n-1}(X)\to \widehat J^{2n-1}(X)$ making the above diagram commute (for any $\ell$).  Indeed, 
since for abelian varieties $A, B$, the natural map $\operatorname{Hom}(A,B)\to \operatorname{Hom}(V_\ell A,V_\ell B)$ is an inclusion, it suffices to show that $V_\ell AJ: V_\ell \operatorname{A}^n(X)\to V_\ell J^{2n-1}(X)$ is surjective\,; this follows from the Proposition \ref{P:surj} using the fact that the Abel--Jacobi map is a surjective regular homomorphism ($T_\ell AJ$ is in fact an isomorphism for $n=1,2,d_X$).

Note that if the Hodge coniveau filtration satisfies
$\coniveau^{i-1}_HH^{2i-1}(X,\mathbb Q) = 0$ for all $i<n$\,; \emph{i.e.}, the middle two terms of the Hodge decomposition are zero for all odd cohomology in degree less than $2n-1$, which holds  
for instance  if $n=1$, or if $d_X\ge 3$, $n=2$, and $h^{1,0}=0$, 
then the Hodge--Riemann bilinear pairing  is positive definite, so that $\Theta_X$ is a polarization.

\subsection{Distinguished homomorphisms}
The discussion above motivates the following
definition\,:

\begin{dfn}[Distinguished homomorphism]\label{D:DistMorph}
Let $X$ be a smooth projective variety over a perfect field $K$,  let $H\in \operatorname{CH}^1(X)$ be the class of an ample divisor, let $n$ be a natural number such that $1\le 2n-1\le d_X$, let $\Omega/K$ be an algebraically closed field extension, let $l$ be a prime, and let $\Phi:\mathscr A^n_{X/K}\to A$ be a \emph{surjective} regular homomorphism.   We say that a homomorphism  $$\Lambda: A_\Omega \longrightarrow  \widehat{A}_\Omega$$ is \emph{$l$-distinguished} (with respect to $H$) if it is induced by cup product in cohomology in the sense that the following diagram commutes\,:
\begin{equation}\label{E:DistMorph}
			\xymatrix@C=5em@R=1.5em{
				T_l A_\Omega\ar[r]^{T_l \Lambda }& T_l \widehat A_\Omega \ar[d]^{(T_l\phi)^\vee} \\
				T_l \operatorname{A}^n(X_\Omega )  \ar[u]^{T_l \phi}  \ar[d]_{T_l \lambda^n} & T_l \operatorname{A}^n(X_\Omega )^\vee(1)  \\
				H^{2n-1}(X_\Omega ,\mathbb Z_l(n))_\tau \ar[r]^<>(0.5){\cup [H]^{d_X-2n+1}} & H^{2d_X-2n+1}(X_\Omega,\mathbb Z_l(d_X-n+1))_\tau \ar[u]_{(T_l \lambda^n)^\vee}
			}
\end{equation}
Here $(T_l A_\Omega)^\vee(1)$ is identified with $T_l \widehat A_\Omega$
via the Weil pairing (see \ref{S:WeilPair}), and $H^{2n-1}(X_\Omega,\mathbb Z)^\vee_\tau$ is
identified with $H^{2d_X-2n+1}(X_\Omega,\mathbb Z)_\tau$ via the cup
product.  We say that $\Lambda$ is \emph{distinguished} if it is
distinguished for all primes $l$.
		\end{dfn}
	
\begin{rem}
Note that in comparison to the motivation in \S \ref{S:DistMorph}, we have removed the factor of $(-1)^{n-1}$ in the bottom row of \eqref{E:DistMorph} to simplify some of the diagrams, and to make the presentation more clearly motivic.  When we are interested in questions of positivity, we will always replace the 
($\ell$-)distinguished homomorphism $\Lambda$ with the homomorphism $\Theta=(-1)^{n-1}\Lambda$.  
\end{rem}		

                \begin{lem}[Uniqueness and descent]\label{L:DM-desc}
                  Let $\ell \not= \operatorname{char}(K)$ be prime.
In the notation of Definition \ref{D:DistMorph},  
there is at most one $\ell$-distinguished homomorphism $\Lambda:A_\Omega \to \widehat A_\Omega$, and this morphism descends to a $K$-homomorphism $\underline \Lambda :A\to \widehat A$.  Moreover, $\Theta=(-1)^{n-1}\Lambda$ is a (principal) polarization if and only if $\underline \Theta=(-1)^{n-1}\underline \Lambda$ is a (principal) polarization. 
\end{lem}

\begin{proof}

	As in the case of the Abel--Jacobi map (\S \ref{S:DistMorph}), by virtue of the fact that  the natural map  $\operatorname{Hom}(A,B)\to \operatorname{Hom}(V_\ell A,V_\ell B)$ is an inclusion,  and the fact that  $V_\ell\phi$ is surjective for all $\ell$ (Proposition~\ref{P:surj}(3)), a diagram chase in \eqref{E:DistMorph}
	 shows that an $\ell$-distinguished homomorphism, if it exists, is
	 unique. 
	To show an $\ell$-distinguished homomorphism $\Lambda$ descends to $K$ it suffices to show that it is $\operatorname{Aut}(\Omega/K)$-equivariant on Tate modules. In fact, it suffices to show that $V_\ell\Lambda$ is $\operatorname{Aut}(\Omega/K)$-equivariant\,; since $V_\ell\phi$ is surjective, this follows again from a diagram chase in \eqref{E:DistMorph}.
	Finally, it is standard that $\Theta=(-1)^{n-1}\Lambda$ is a (principal) polarization if and only if  $\underline \Theta = (-1)^{n-1}\underline \Lambda$ is a (principal) polarization (see \S \ref{S:ThetaLB}). 
\end{proof}

In light of the uniqueness and descent of Lemma \ref{L:DM-desc}, if  $\Phi^n_{X/K}:\mathscr A^n_{X/K}\to \operatorname{Ab}^n_{X/K}$ is an algebraic representative, we use the notation  $$\Lambda_X:\operatorname{Ab}^n_{X/K}\to \widehat {\operatorname{Ab}}^n_{X/K}$$
for a distinguished symmetric $K$-isogeny, if it exists.   Note that unless $d=2n-1$, we have that $\Lambda_X$ depends \emph{a priori} on $H$ as well as $X$.

\begin{rem}
	Note that since the  inclusion  $\operatorname{Hom}(A,B)\to \operatorname{Hom}_{\gal(K)}(T_\ell A,T_\ell B)$ is bijective if and only if $\operatorname{Hom}(A,B)= 0$, the existence of an $\ell$-distinguished morphism $\Lambda$ is not formal from the rest of the diagram \eqref{E:DistMorph}.  
\end{rem}

In summary, with Hodge theory as our inspiration, our goal is to investigate when algebraic representatives admit distinguished polarizations.  Obviously, motivated by the case of characteristic $0$, we have the following examples (recall that the distinguished model of the algebraic intermediate Jacobian agrees with the algebraic representative in characteristic $0$ for $n=1,2,d_X$)\,:

\begin{exa}[{Distinguished polarizations for  distinguished models in characteristic $0$}]
	\label{E:PPolDM} 
As in \S \ref{S:DistMorph}, let $X$ be a complex projective manifold, let $H$ be an ample divisor on $X$, let $n$ be a nonnegative integer with  $1\le  2n-1 \le d_{X}$, and assume
$\coniveau^{n-1}H^{2n-1}(X,\mathbb Q)= H^{2n-1}(X,\mathbb Q)$, which implies 
 $$H^{2n-1}(X,\mathbb C)=H^{n+1,n}(X)\oplus H^{n,n-1}(X).$$
	Then the  symmetric isogeny  $\Lambda_{X}:=(-1)^{n-1}\Theta_{X}:J^{2n-1}(X)\to \widehat J^{2n-1}(X) $ induced by $H$
	and the cup product in cohomology is $\ell$-distinguished, and therefore descends to symmetric $K$-isogeny 
	$\Lambda_X$ on $J^{2n-1}_{a,X/K}$.
If the  Hodge coniveau filtration satisfies
	$\coniveau^{i-1}_HH^{2i-1}(X,\mathbb Q) = 0$ for all $i<n$, then $\Theta_X=(-1)^{n-1}\Lambda_X$ is a  polarization on $J^{2n-1}_{a,X/K}$.  
\end{exa}

In positive characteristic, we will use miniversal cycles to 
construct distinguished homomorphisms.

\subsection{Distinguished homomorphisms and miniversal cycles}
\label{S:PolAb2}

Let $X$ be a smooth projective variety over a perfect field $K$, let $H \in \chow^1(X)$ be the class of an ample divisor,  and let $n$ be a natural number such that $1\le 2n-1 \le d_X$.   
Let $A$ be an abelian variety over $K$, and let $\Omega/K$ be an algebraically closed field extension.
Then for any cycle  $Z \in \chow^n(X_{\Omega}\times_{\Omega}A_{\Omega})$, 
the cycle ${}^tZ\circ [\cup H^{d_X-2n+1}] \circ Z\in \chow^1(A_{\Omega} \times_{\Omega}  A_{\Omega})$ induces, via \S \ref{SS:RegHomCycle}, a symmetric $\Omega$-homomorphism 
\begin{equation}\label{E:ThetaZ}
\Lambda_Z:A_\Omega\longrightarrow \widehat A_\Omega.
\end{equation}
Our goal is to investigate when this construction gives a distinguished homomorphism on $A$, in the sense of Definition~\ref{D:DistMorph}.

\subsubsection{}\label{SSS:Z}
The first observation is that, for each prime $l$, we have by construction a  commutative diagram
\begin{equation}\label{E:SSS:Z}
\xymatrix@C=5em{
	T_l A \ar[r]^{T_l \Lambda _Z}
	\ar[d]^{Z_*}&T_l \widehat A
	\\
	H^{2n-1}(X_{\Omega},\mathbb Z_l(n))_\tau  \ar[r]^<>(0.5){\cup  [H]^{d_X-2n+1}}&
	H^{2d_X-2n+1}(X_{\Omega},\mathbb Z_{l}(d_X-n+1))_\tau \ar[u]^<>(0.5){Z^*}
}
\end{equation}
where $(T_l A)^\vee(1)$ is identified with $T_l \widehat A$ via
the Weil pairing, and we identify $T_l A_\Omega =
T_l  A$ by rigidity.  When $l = \ell \ne \operatorname{char}(K)$,
\eqref{E:SSS:Z} follows by taking the $\ell$-adic realization of an
equality of cycle classes.
If $\operatorname{char}(K) = p>0$, then taking cycle classes in crystalline
cohomology yields a diagram
\[\xymatrixcolsep{5pc}
\xymatrix{ H_1^\cris(A/\ww(K)) \ar[r]^{\Lambda_{Z,\cris}} \ar[d]^{Z_{*,\cris}} & H_1^\cris(\widehat A/\ww(K))  \\
    H^{2n-1}_\cris(X/\ww(K)(n))_\tau \ar[r]^<>(0.5){
      [H]^{d_X-2n+1}} & H_\cris^{2d-2n+1}(X/\ww(K)(d-n+1))_\tau \ar[u]^<>(0.5){Z^*_\cris}
}
\]
and then taking $F$-invariants gives \eqref{E:SSS:Z} with $l=p$.

\subsubsection{} \label{SSS:main}
Suppose now that $\Phi:\mathscr A^n_{X/K}\to A$ is  a surjective regular homomorphism and assume that $V_{\ell}\lambda^n : V_{\ell}\operatorname{A}^n(X_{\bar K})
\longrightarrow H^{2n-1}(X_{\bar K},\mathbb Q_{\ell}(n))$ is surjective for some prime $\ell\neq \operatorname{char} K$.
Then, by combining \eqref{E:DistMorph} and \eqref{E:SSS:Z} together with Lemma~\ref{L:factorM} and Proposition~\ref{P:commute2}, we obtain for all primes $l$  a commutative diagram

\begin{equation}
\label{E:CanPol}
\xymatrix@C=3em@R=.8em{
	T_l A_\Omega
	\ar[dd]^{w_Z} \ar@/_3pc/[dddd]_{Z_*} \ar[rrr]^{T_{l}\Lambda_Z \qquad } \ar[rd]^{T_l\psi_Z}& &&
	T_{l}\widehat{A}_\Omega\\ 
	& T_l A_\Omega \ar@{-->}[r]^{\exists ? \ T_l \Lambda}& T_l\widehat{A}_\Omega \ar[ru]^{(T_l\psi_Z)^\vee} \ar[dr]_{(T_l\phi)^\vee}
	\\
	T_{l} \operatorname{A}^n(X_\Omega) \ar[ur]_{T_l\phi}
	\ar@{->}[dd]^<>(0.5){T_l \lambda^n } &&&	T_{l} \operatorname{A}^n(X_\Omega)^\vee  \ar[uu]^{(w_Z)^\vee}\\
	\\
	H^{2n-1}(X_{\Omega},\mathbb Z_{l}(n))_\tau  \ar[rrr]^<>(0.5){\cup [H]^{d_X-2n+1}}&&&
	H^{2d_X-2n+1}(X_{\Omega},\mathbb Z_{l}(d_X-n+1))_\tau.
	\ar@{->}[uu]^<>(0.5){(T_l \lambda^n )^\vee}  \ar@/_3pc/[uuuu]_{Z^*}
}
\end{equation}
and we are asking whether there exists an $\Omega$-homomorphism $\Lambda : A_\Omega \to \widehat{A}_\Omega$ making the diagram commute for all $l$. Note from Lemma~\ref{L:DM-desc}
that $\Lambda$, if it exists, is unique and descends to~$K$, and note also that the homomorphism $\Lambda_Z:A_\Omega\to \widehat{A}_\Omega$ is uniquely determined (since $\operatorname{Hom}(A,B)\to \operatorname{Hom}(T_\ell A,T_\ell B)$ is injective for $\ell \neq \operatorname{char} K$).

We fix  a miniversal cycle $Z\in \mathscr{A}^n_{X_{\Omega}/\Omega}(A_{\Omega})$ of degree, say, $r$, meaning that $\psi_Z: A\to A$ is multiplication by $r$. By considering the diagram \eqref{E:CanPol} at a prime $\ell \neq \operatorname{char}(K)$ and arguing as in the proof of Lemma~\ref{L:DM-desc}, 
we see that $\Lambda_Z$ descends to $K$ and only depends on $r$
(\emph{i.e.}, does not depend on the choice of a miniversal cycle of
degree $r$).

If $V_{\ell}\lambda^n$ is bijective,
then the symmetric $K$-homomorphism $\Lambda_Z$  is an  isogeny (for instance, replace $T_\ell$ by $V_\ell$ in the above diagram~\eqref{E:CanPol} and use that $r$ is invertible in $\mathbb Q_\ell$).

In case $Z\in \mathscr{A}^n_{X_{\Omega}/\Omega}(A_{\Omega})$ is  a universal cycle (\emph{i.e.}, $\psi_Z = \mathrm{id}_A$),  then there exists a distinguished homomorphism  $\Lambda$, namely, $\Lambda= \Lambda_Z$, since $\psi_Z = \mathrm{id}_A$. 
\medskip

The following theorem summarizes the above discussion\,:

\begin{teo}[Distinguished morphisms motivically]\label{T:CanPol}
	Let $X$ be a smooth projective variety over a perfect field $K$, let $H$ be an ample divisor on $X$, let  $n$ be a natural number such that $1\le 2n-1 \le d_X$, and let $\Omega/K$ be an algebraically closed field extension. Further, let  $\Phi: \mathscr A^n_{X/K}\to A$ be a surjective regular homomorphism and let 
		$Z\in \mathscr{A}^n_{X_{\Omega}/\Omega}(A_{\Omega})$ be  a miniversal cycle of degree $r$.  We denote $\Lambda_Z: 	A_\Omega \to \widehat{A}_\Omega  $ the symmetric $\Omega$-homomorphism 
	 	induced by the cycle ${}^tZ\circ H^{d_X-2n+1} \circ Z \in \chow^1(A_\Omega \times_{\Omega} A_\Omega )$.
	 	
	Assume that for some prime $\ell_0\ne \operatorname{char}(K)$:
	\begin{itemize}
		\item $
		V_{\ell_0}\lambda^n : V_{\ell_0}\operatorname{A}^n(X_{\bar K})
		\longrightarrow H^{2n-1}(X_{\bar K},\mathbb Q_{\ell_0}(n))$ is  surjective.
	\end{itemize}
	Then the symmetric $\Omega$-homomorphism 
	$\Lambda_Z: 	A_\Omega \to \widehat{A}_\Omega  $ descends to $K$ and 	depends only on $H$ and $r$ (and not on the choice of the miniversal cycle $Z$ of degree $r$).   
	Moreover, 
	\begin{enumerate}
			\item 
		If $Z$ is universal (\emph{i.e.}, $r=1$),
		 then 	
		$\Lambda_Z$ is a distinguished symmetric $K$-homomorphism.

	\item
	If $T_{l_i}\phi$  are isomorphisms for a given set of primes $\{l_i\}_{i\in I}$, 

	then 	
	there exists  a symmetric $K$-homomorphism $\Lambda' : A \to  \widehat{A}$ such that 
	\begin{equation}\label{E:rlambda}
	\Lambda_Z = \Big(\prod_{l \in \{l_i\}_{i\in I}} l^{v_l(r)}\Big)^2  \Lambda'.
	\end{equation}
	\item If $T_{l}\phi$ is an isomorphism at the primes $l$ dividing $r$, 
	then 	
	there exists a distinguished symmetric $K$-homomorphism $\Lambda$,  which makes  \eqref{E:CanPol} commute at all primes $l$, and $\Lambda_Z=r^2\Lambda$.
	\end{enumerate}	
Finally, if in addition $V_{\ell_0}\lambda^n$ is bijective, then $\Lambda_Z$ is a symmetric $K$-isogeny (and hence so are $\Lambda'$ and $\Lambda$ in (2) and (3), respectively).  
\end{teo}

\begin{proof}
	Everything except (2) and (3) follows immediately from the commutativity of \eqref{E:CanPol} and from the discussion above. Clearly (2) $\implies$ (3).	Concerning (2), at the primes $l_i$ for which $T_{l_i}\phi$ are isomorphisms, we obtain from \eqref{E:CanPol} a Galois-equivariant homomorphism $\alpha_i : T_{l_i}A \to T_{l_i}\widehat A$ such that $r^2\alpha_i = T_{l_i}\Lambda_Z$. 
We next recall the following elementary fact: Given a free $\integ$-module $M$ (of finite rank), an element $m\in M$ and a prime $l$, denoting by $\pi : M \hookrightarrow M\otimes \integ_{l}$ the canonical map, if $\pi(m)$ is divisible by an integer $N$ (\emph{i.e.}, in the image of the multiplication by $N$ map), then $m$ is divisible by $l^{v_l(N)}$.  Applying this to the abelian group of homomorphisms from $A_\Omega$ to $\widehat A_\Omega$,
 this implies that $\Lambda_Z$ is divisible by $l_i^{2v_{l_i}(r)}$.
	Hence, there exists a symmetric $K$-homomorphism $\Lambda' : A \to \widehat A$
	 such that \eqref{E:rlambda} holds.  
	In particular, we see that $\Lambda_Z$ becomes divisible by $r^2$ after inverting the primes $l$ dividing $r$ but distinct from the~$l_i$. 
\end{proof}

\begin{rem}[Characteristic $0$]
	\label{R:PPolDM-Z}  
	Let $X$ be a smooth projective variety over a field $K\subseteq \mathbb C$, as in Example \ref{E:PPolDM}.  	 
	Assuming 
 that the bullet point condition in Theorem \ref{T:CanPol}  holds, and that there is a universal codimension-$n$ cycle class for $X_{\mathbb C}$ 
 (resp.~$T_{\ell}\phi$ and $T_{\ell} \lambda^n$ are isomorphisms for all primes~$\ell$), then the symmetric $K$-isogeny $\Theta=(-1)^{n-1}\Lambda$ of Theorem \ref{T:CanPol}(1) (resp.~(3)) agrees with the  polarization
	$\Theta_X$ on $J^{2n-1}_{a,X/K}$ from Example~\ref{E:PPolDM}.  Indeed, both are $\ell$-distinguished, and so we may employ Lemma \ref{L:DM-desc}.  
\end{rem}

\begin{rem}[$\Theta_Z=(-1)^{n-1}\Lambda_Z$]
We emphasize that in light of the Hodge theory, it is $\Theta_Z:= (-1)^{n-1}\Lambda_Z$ in Theorem \ref{T:CanPol} that one might hope is a polarization on $A$.  We discuss this further in~\S\ref{SS:AlgRepPol}.  
\end{rem}

\begin{cor}[Distinguished homomorphisms for codimension-1 cycles]\label{C:CanPol-cod1}
Let  $\Phi^1_{X/K}: \mathscr A^1_{X/K}\to (\operatorname{Pic}^0_X)_{\mathrm{red}}=\operatorname{Ab}^1_{X/K}$ be the Abel--Jacobi map (the first algebraic representative). 
Then there exists a distinguished the symmetric $K$-isogeny $\Lambda_X : (\operatorname{Pic}^0_X)_{\mathrm{red}} \to (\operatorname{Pic}^0_X)_{\mathrm{red}}^\vee$, which, in case $K\subseteq \cx$,  agrees with the  polarization $\Theta_X$ induced by Hodge theory (see Example~\ref{E:PPolDM}). 
\end{cor}
\begin{proof}
	Note that $V_l\lambda^1$ is an isomorphism for all primes $l$ and that the Abel--Jacobi map $\Phi: \mathscr A^1_{X/K}\to (\operatorname{Pic}^0_X)_{\mathrm{red}}$ always admits a universal cycle $Z\in \mathscr{A}^n_{X_{\Omega}/\Omega}(A_{\Omega})$ so that Theorem~\ref{T:CanPol}(1) applies.
The agreement of $\Lambda_X$ with the polarization $\Theta_X$  induced by Hodge theory comes from the fact that both are distinguished homomorphisms.	
\end{proof}

\begin{rem}[Curves]
 In Corollary~\ref{C:CanPol-cod1}, if $X$ is a curve, then $\Lambda_X$ is independent of the choice of $H$, and agrees with the canonical principal polarization on the Jacobian of the curve, since the canonical principal polarization is known to be distinguished.
\end{rem}

\begin{cor}[Distinguished homomorphisms for codimension-2 cycles]
	\label{C:CanPol}
Let
	$\Phi^2_{X/K} : \mathscr A^2_{X/K}\to \operatorname{Ab}^2_{X/K}$ be the second algebraic representative and let $N$ be a natural number. 
	Suppose that, for a prime $\ell \neq \operatorname{char}(K)$, we have that  $N\Delta_{X_{\Omega}}$ admits a  cohomological $\integ_\ell$-decomposition of type $(d_X-1,1)$ with respect to $H^\bullet (-,\mathbb Z_\ell)$, 	\emph{e.g.}, $X$ is geometrically rationally connected. Then
	\begin{enumerate}
\item There is a distinguished symmetric $K$-isogeny $\Lambda_X :\operatorname{Ab}^2_{X/K}\to \widehat{\operatorname{Ab}}^2_{X/K}$\,;
\item 	If in addition $\dim X \leq 4$, then $X$ admits a codimension-2 miniversal cycle	$Z\in \mathscr{A}^n_{X_{\Omega}/\Omega}(A_{\Omega})$ of degree $N$ (Corollary~\ref{C:UnivCyc}) and $\Lambda_Z = N^2\Lambda_X$\,;
\item  If further $N=1$,  \emph{e.g.}, if $X$ is geometrically stably rational of dimension $\leq 4$, then $X$ admits a codimension-2 universal cycle $Z$ and $\Lambda_Z  = \Lambda_X$.
	\end{enumerate}
Moreover, in case $K\subseteq \cx$, the symmetric $K$-isogeny $-\Lambda_X$ agrees with the  polarization $\Theta_X$ induced by Hodge theory (see Example \ref{E:PPolDM}). 
\end{cor}
 \begin{proof}
The assumption on $N\Delta_{X_\Omega}$ implies that $V_\ell \lambda^2$ is bijective (Proposition~\ref{P:lambdacoho}) and that the $T_l\phi$  are isomorphisms for all primes $l$ (Proposition~\ref{P:BlSrTors}). Items (1)--(3) then follow from Theorem~\ref{T:CanPol}.
In case $K\subseteq \cx$, the agreement of $-\Lambda_X$ with the polarization $\Theta_X$  induced by Hodge theory comes from the fact that both $\Lambda_X$ and $-\Theta_X$ are distinguished homomorphisms.	
 \end{proof}

As another consequence of Theorem \ref{T:CanPol} we obtain a cohomological analogue to Theorem~\ref{T:main-pol}\,; this result will in fact show that the isomorphism $\Theta_X$ in Theorem~\ref{T:main-pol} is distinguished, and therefore, via Lemma \ref{L:DM-desc},  in characteristic $0$, agrees with the principal polarization coming from Hodge  theory\,:

\begin{teo}[Threefolds] \label{T:CanPol3-fold} 
	Let $X$ be a smooth projective threefold over a perfect field $K$, and let $\Omega/K$ be an algebraically closed field extension. 
	Let $Z\in \mathscr A^2_{X_\Omega/\Omega}(\operatorname{Ab}^2_{X_\Omega/\Omega})$ be a miniversal cycle class. 
	
	\begin{enumerate}
		\item If $V_\ell\lambda^2$ is an isomorphism for some prime $\ell \neq \operatorname{char}(K)$ (\emph{e.g.}, if $X$ is geometrically uniruled), then  the  $\Omega$-homomorphism $\Lambda_Z: 	\operatorname{Ab}^2_{X_\Omega/\Omega}\to \widehat {\operatorname{Ab}}\,^2_{X_\Omega/\Omega}  $
		induced by the cycle class  ${}^tZ \circ Z \in \chow^1(\operatorname{Ab}^2_{X_{\Omega}/\Omega}\times_{\Omega} \operatorname{Ab}^2_{X_{\Omega}/\Omega})$  descends to a  symmetric $K$-isogeny 
		$$\Lambda_Z :\ \operatorname{Ab}^2_{X/{K}} \to {\widehat{\operatorname{Ab}}\,^2_{X/{K}}}$$ 
		depending only on the degree of $Z$ as a miniversal cycle. Moreover, if $Z$ is universal, then $\Lambda_Z$ is distinguished and we denote $\Lambda_Z$ by $\Lambda_X$.

		\item If $\operatorname{CH}_0(X_{\bar K})\otimes {\mathbb Z[\frac{1}{N}]}$ is universally trivial for some natural number $N$ (\emph{e.g.},  $X$ is geometrically rationally chain connected),
then there exists a distinguished purely inseparable symmetric $K$-isogeny $\Lambda_X : \ \operatorname{Ab}^2_{X/{K}} \to {\widehat{\operatorname{Ab}}\,^2_{X/{K}}}$. Moreover, 
if $\operatorname{char}(K) \nmid N$, then $\Lambda_X$ is an isomorphism.
		
		\item If $\operatorname{CH}_0(X_{\bar K})$ is universally trivial (\emph{e.g.}, $X$ is geometrically stably rational), then the distinguished symmetric $K$-isogeny $\Lambda_X$ is an isomorphism.
	\end{enumerate}
	 Moreover, if $K\subseteq \mathbb C$, then the  symmetric $K$-isogeny $-\Lambda_X$ of (2) and (3) agrees with the principal   polarization $\Theta_X$ induced by Hodge theory (see Example \ref{E:PPolDM}). 
\end{teo}

\begin{proof}
	Item (1) follows immediately from Theorem~\ref{T:CanPol}.  In cases (2) and (3), the diagonal of $X_{\bar K}$ admits in particular a Chow $\rat$-decomposition of type $(2,2)$ and it follows from Proposition \ref{P:lambdacohoBW} that $T_l \lambda^2$ (and hence $V_l\lambda^2$) is bijective for all primes~$l$. 
	Concerning (2), the diagonal of $X_{\bar K}$ admits in particular a Chow $\rat$-decomposition of type $(2,1)$, and hence $T_l\phi^2$ is an isomorphism for all primes~$l$ by Proposition~\ref{P:BlSrTors}. Combined with the fact that $T_l \lambda^2$ is bijective for all primes $l$, there exists from Theorem~\ref{T:CanPol}(3) a distinguished symmetric $K$-isogeny $\Lambda_X$ and it follows from diagram~\eqref{E:CanPol} that $T_l\Lambda_X$ is an isomorphism for all primes $l$ and hence
	that $\Lambda_X$ is a purely inseparable isogeny. On the other hand, by Theorem~\ref{T:main-pol}(1), the universal triviality of $\operatorname{CH}_0(X_{\bar K})\otimes {\mathbb Z[\frac{1}{N}]}$ yields that the degree of the isogeny $\Lambda_X$ divides a power of $N$\,; it follows that, if $\operatorname{char}(K) \nmid N$, $\Lambda_X$ is an isomorphism.
For (3), if $\operatorname{CH}_0(X_{\bar K})$ is universally trivial, then by Corollary~\ref{C:UnivCyc} $X_{\bar K}$ admits a universal codimension-$2$ cycle $Z$ and $\Lambda_Z$ is the distinguished symmetric $K$-isogeny (\emph{i.e.}, it coincides with $\Lambda_X$) by Theorem~\ref{T:CanPol}(1). Now, the fact that $\Lambda_Z$ is an isomorphism follows by noting that by construction $\Lambda_Z$ coincides with the symmetric $K$-isomorphism $\Lambda_X$ of Theorem~\ref{T:main-pol}(2).
\end{proof}

\begin{rem}\label{R:T:main-polHdg}
	Note that by construction, the symmetric isogeny $\Lambda_X$ in  Theorem \ref{T:CanPol3-fold} agrees with that in 
	Theorem \ref{T:main-pol} when there is a (Chow) decomposition of the diagonal, so that the symmetric $K$-isomorphism  $\Lambda_X$ of Theorem \ref{T:main-pol}(2) is distinguished. 
\end{rem}

\section{Specialization and polarization on the algebraic representative}\label{SS:AlgRepPol}

 Regarding whether the symmetric $K$-isogeny $\Theta_X=(-1)^{n-1}\Lambda_X$ of Theorem \ref{T:CanPol}(2) is a polarization,  in this section we present results for threefolds 
 that essentially say that if a  geometrically stably rational threefold     can be lifted to a geometrically rationally chain connected threefold  in characteristic $0$, 
 then $\Theta_X$ is a principal polarization.  While many of the examples we have in mind
 are liftable to characteristic $0$ (see \emph{e.g.}, \cite[Thms.~22.1, 22.3]{HartDef}), recall, of course, that there are smooth, projective,  even
 rational varieties, over perfect fields of characteristic $p>0$ that do not lift
 (see \emph{e.g.}, \cite{AZ17lift}).

\subsection{Inducing polarizations on the algebraic representative via specialization}

 \begin{pro}[Polarizations and specialization]\label{P:ThetaSpecMini}
  Suppose that $S$ is the spectrum of a DVR with generic point $\eta$ and
  special point $s$ both with  perfect residue fields. 
  Suppose $f:{\mathcal X}\to S$ is a smooth projective morphism, and let $H$ be a relatively ample divisor.  Fix a natural number~$n$ such that $2n-1\le d= \dim_S{\mathcal X}$, and 
  let $\ell$ be a prime invertible in  $\kappa(s)$.

  Assume that:
  \begin{itemize}
  \item There exist algebraic representatives $\Phi^n_{{\mathcal X}_\eta /\eta}: \mathscr A^n_{{\mathcal X}_\eta/\eta}\to \operatorname{Ab}^n_{{\mathcal X}_\eta/\eta}$ and $\Phi^n_{{\mathcal X}_s/s}: \mathscr A^n_{{\mathcal X}_s/s}\to \operatorname{Ab}^n_{{\mathcal X}_s/s}$\,;

  	\item $	T_\ell\phi^n_{{\mathcal X}_\eta}: T_\ell \operatorname{A}^n({\mathcal X}_{\bar
  		\eta})  \to T_\ell \operatorname{Ab}^n_{{\mathcal X}_\eta}$ and  $	T_\ell\phi^n_{{\mathcal X}_s}: T_\ell \operatorname{A}^n({\mathcal X}_{\bar
  		s})  \to T_\ell \operatorname{Ab}^n_{{\mathcal X}_s}$ are isomorphisms\,;
  	\item $V_\ell\lambda^n_{{\mathcal X}_{\bar \eta}} : V_\ell\operatorname{A}^n({\mathcal X}_{\bar \eta})
  	\to H^{2n-1}({\mathcal X}_{\bar \eta},\mathbb Q_\ell(n))$ and $V_\ell\lambda^n_{{\mathcal X}_{\bar s}} : V_\ell\operatorname{A}^n({\mathcal X}_{\bar s})
  	\to H^{2n-1}({\mathcal X}_{\bar s},\mathbb Q_\ell(n))$ are isomorphisms.
	   \end{itemize}
Let 
$Z_\eta \in \mathscr{A}^n_{\mathcal X_{ \eta}/ \eta}(\operatorname{Ab}^n_{\mathcal X_{\eta}/ \eta})$  be a miniversal cycle of degree $r_\eta$,  and let $\zeta \in \mathscr{A}^n_{\mathcal X_{s}/ s}(\operatorname{Ab}^n_{\mathcal X_{s}/ s})$  be a miniversal cycle of degree $r_s$.  	   Let $\Lambda_{Z_\eta}$ and $\Lambda_{\zeta}$ be the symmetric $K$-isogenies of Theorem \ref{T:CanPol}.  
	   
 Then $(\operatorname{Ab}^n_{{\mathcal X}_{ \eta}},\Lambda_{Z_\eta})$ extends to an abelian scheme $(\operatorname{Ab}^n_{{\mathcal X}/S},\Lambda)$ over $S$, and   $Z_\eta$ induces  an isogeny
\begin{equation}\label{E:ThetaSpecMini}
h : ( \operatorname{Ab}^n_{{\mathcal X}/S} |_s,  r_s^2 \Lambda  |_s) \to ( \operatorname{Ab}^n_{{\mathcal X}_s/s} ,   \Lambda_{\zeta})\,;
\end{equation}
 the notation above using pairs, consisting of an abelian variety and a morphism to the dual abelian variety,  indicates that the indicated extensions and morphisms make the associated diagrams commute.  
    In particular,  $\Theta_{Z_\eta}:=(-1)^{n-1}\Lambda_{Z_\eta}$ is a polarization if and only if   $\Theta_{\zeta}:=(-1)^{n-1}\Lambda_\zeta $ is.
    
    Moreover, if $\ell\nmid r$, $d=2n-1$,  and $T_\ell\lambda^n_{{\mathcal X}_{\bar \eta}}$ and 
$T_\ell\lambda^n_{{\mathcal X}_{\bar s}}$ are isomorphisms, 
    then $T_\ell h$ is an isomorphism.
 \end{pro}

\begin{rem}
With the view to lifting to characteristic $0$, we  will want to employ Proposition \ref{P:ThetaSpecMini} in the case where 
$\operatorname{Ab}^n_{{\mathcal X}_{\eta}/\eta}$ admits in addition  an  $\ell$-distinguished symmetric $K$-isogeny $\Lambda'_\eta$.
In that case, $\Lambda'_\eta$ extends to a symmetric $K$-isogeny $\Lambda'$ on $\operatorname{Ab}^n_{\mathcal X/S}$, and we have $\Lambda'=r_\eta^2\Lambda$, so that we have in that case an isogeny  $h : ( \operatorname{Ab}^n_{{\mathcal X}/S} |_s,  r_s^2 r_\eta^2\Lambda'  |_s) \to ( \operatorname{Ab}^n_{{\mathcal X}_s/s} ,   \Lambda_{\zeta})$.
In particular, $\Theta'_\eta :=(-1)^{n-1}\Lambda'_\eta$ is a polarization if and only if $\Theta_{\zeta}:=(-1)^{n-1}\Lambda_\zeta $ is a polarization.    If  $\operatorname{Ab}^n_{\mathcal X_s/s}$ also admits an  $\ell$-distinguished symmetric $K$-isogeny $\Lambda_s$, then we have the isogeny $ 
h : ( \operatorname{Ab}^n_{{\mathcal X}/S} |_s,   r_\eta^2\Lambda'  |_s) \to ( \operatorname{Ab}^n_{{\mathcal X}_s/s} ,   \Lambda_s)$ .
\end{rem}

 \begin{proof}
 	First, from the bullet point assumptions, we have a Galois-equivariant isomorphism $V_\ell\lambda^n_{{\mathcal X}_{\bar \eta}}  \circ (V_\ell\phi^n_{{\mathcal X}_{\eta}})^{-1} : V_\ell \operatorname{Ab}^n_{{\mathcal X}_{\eta}/\eta} \to 
 	H^{2n-1}({\mathcal X}_{\bar \eta},\rat_{\ell}(n))$, showing by the Ogg--N\'eron--Shafarevich criterion that $\operatorname{Ab}^n_{{\mathcal X}_{ \eta}/\eta}$ extends to an abelian scheme $\operatorname{Ab}^n_{{\mathcal X}/S}$ over $S$.
	The fact that $\Lambda_{Z_\eta}$ then extends to a morphism  $\Theta$ over $S$ is standard (see \emph{e.g.}, \cite[Prop.~4.5]{ACMVfunctor} and \cite[Prop.~6.1]{mumfordGIT}).   Just as in Lemma~\ref{L:PolSpec}, we have that the following are equivalent: $\Theta$ is a polarization\,; $\Theta_\eta$ is a polarization\,;  $\Theta|_s$ is a polarization.  Thus we have reduced to showing the existence of the isogeny $h$ in the statement in the theorem, as well as the assertion about $T_\ell h$.

  Let now $Z_\eta \in \operatorname{CH}^n(\operatorname{Ab}^n_{{\mathcal X}_{ \eta}/ \eta} \times_{ \eta} {\mathcal X}_{\eta})$ and $\zeta\in  \operatorname{CH}^n(\operatorname{Ab}^n_{{\mathcal X}_{s}/ s} \times_{ s} {\mathcal X}_{s})$  be  miniversal  codimension-$n$  cycle classes. From our assumptions we have  the commutative diagrams \eqref{E:CanPol}\,:
  \begin{equation}\label{E:Thetax2}
  \xymatrix@C=1.5em@R=1.5em{
  T_\ell \operatorname{Ab}^n_{{\mathcal X}_\eta/\eta} \ar[r]^{T_\ell\Lambda_{{Z}_{\eta}}} \ar[d]^{T_\ell \psi_{Z_{ \eta}}} 
	\ar@/_2.5pc/[dd]_{(Z_{\eta})_*}
	& T_\ell \widehat{\operatorname{Ab}}\,^n_{{\mathcal X}_\eta/\eta}  &  T_\ell \operatorname{Ab}^n_{{\mathcal X}_s/s} \ar[r]^{T_\ell\Lambda_{\zeta}} \ar[d]^{T_\ell \psi_{\zeta}} 
	\ar@/_2.5pc/[dd]_{\zeta_*}
	& T_\ell \widehat{\operatorname{Ab}}\,^n_{{\mathcal X}_\eta/\eta} \\
 T_\ell \operatorname{Ab}^n_{{\mathcal X}_\eta/\eta} \ar@{-->}[r]^{\alpha_\eta } \ar[d]^{\iota} 
	& T_\ell \widehat{\operatorname{Ab}}\,^n_{{\mathcal X}_\eta/\eta}    \ar[u]^{T_\ell \psi_{Z_{\eta}}^\vee}
	 	& T_\ell \operatorname{Ab}^n_{{\mathcal X}_s/s} \ar@{-->}[r]^{\alpha_{s}} \ar[d]^{\iota } 
	& T_\ell \widehat{\operatorname{Ab}}\,^n_{{\mathcal X}_s/s} 
	\ar[u]^{T_\ell \psi^\vee_{\zeta}}   \\
 	H^{2n-1}({\mathcal X}_{\bar \eta},\integ_{\ell})_\tau \ar[r]^{\cup [H_\eta]^{d-2n+1}} 
	&	H^{2d-2n+1}({\mathcal X}_{\bar \eta},\integ_{\ell})_\tau  \ar[u]^{\iota^\vee}  \ar@/_2.5pc/[uu]_{(Z_{\eta })^*}
	 	&  	H^{2n-1}({\mathcal X}_{\bar s},\integ_{\ell})_\tau \ar[r]^{\cup [H_s]^{d-2n+1}}  
	&	H^{2d-2n+1}({\mathcal X}_{\bar s},\integ_{\ell})_\tau  \ar[u]^{\iota^\vee}  \ar@/_2.5pc/[uu]_{\zeta_*}
}
  \end{equation} 
  whose arrows are all isomorphisms after tensoring with $\rat_{\ell}$ and
where $\iota := T_\ell \lambda^n \circ (T_\ell \phi^n)^{-1}$, with  $\phi^n$ denoting the universal regular homomorphism. We have omitted the Tate twists   in the diagram above for space.  
Note that while Bloch maps are stable under specialization, regular homomorphisms need not be, and so  we may not simply take the specialization of the bottom half of the diagram on the left, and expect to obtain the bottom of the diagram on the right.  In any case, we can already draw the conclusion that $T_\ell \Lambda_{Z_\eta}=r_\eta^2\alpha_\eta $ and $T_\ell \Lambda_\zeta = r_s^2\alpha_s$.  

Let us now consider  $Z_{s} \in \operatorname{A}^n(\operatorname{Ab}^n_{{\mathcal X}_{S}/S}|_{ s} \times_{s} {\mathcal X}_{s})$, the specialization of the cycle $Z_\eta$. This cycle induces a homomorphism $h=\psi_{Z_{s}} : \operatorname{Ab}^n_{{\mathcal X}_{S}/S}|_{s} \to \operatorname{Ab}^n_{{\mathcal X}_{s}/s}$, and we consider  the diagram
 \begin{equation}\label{E:specialization}
\xymatrix@C=1.5em@R=1.5em{T_\ell \operatorname{Ab}^n_{{\mathcal X}/S}|_{s} \ar[r]^{T_\ell\Lambda_{Z_s}} \ar[d]^{T_\ell \psi_{Z_{ s}}=T_\ell h} \ar@/_2.5pc/[dd]_{(Z_{ s})_*}
	& T_\ell \widehat{\operatorname{Ab}}\,^n_{{\mathcal X}/S} |_s \\
    T_\ell \operatorname{Ab}^n_{{\mathcal X}_s/s} \ar[r]^{\alpha_s} \ar[d]^{\iota} 
	& T_\ell \widehat{\operatorname{Ab}}\,^n_{{\mathcal X}_s/s}  \ar[u]^{T_\ell \hat h=T_\ell\psi_{Z_{ s}}^\vee} \\
	 	H^{2n-1}({\mathcal X}_{\bar s},\integ_{\ell}(n))_\tau\ar[r]^{\cup [H_s]^{d-2n+1}}  
	&	H^{2d-2n+1}({\mathcal X}_{\bar s},\integ_{\ell}(n))_\tau.  \ar[u]^{\iota^\vee}  \ar@/_2.5pc/[uu]_{(Z_{s})^*}
}
\end{equation}
Thanks to Proposition~\ref{P:commute}, the vertical arrows form commutative diagrams.  
We observe that the outer rectangle is nothing but the specialization of the outer rectangle on the left-hand side of diagram~\eqref{E:Thetax2} and is hence commutative. 
In addition, the bottom square, being  the right-hand side of~\eqref{E:Thetax2}, is commutative.  Thus the diagram \eqref{E:specialization} is commutative.  

Tensoring \eqref{E:specialization}  with $\mathbb Q_\ell$, the morphisms $\iota$ and $(Z_s)_*$ are isomorphisms, and therefore $V_\ell h$ is an isomorphism, implying that $h$ is an isogeny.  On the other hand, the commutativity of the top square of  diagram \eqref{E:specialization} implies that $T_\ell \Lambda_{Z_s}=T_\ell \hat h \circ \alpha_s \circ T_\ell  h$.  Thus 
 $T_\ell (r_s^2\Lambda_{Z_s})=T_\ell \hat h \circ r_s^2\alpha_s \circ T_\ell  h =T_\ell \hat h \circ T_\ell \Lambda_\zeta \circ T_\ell  h = T_\ell (\hat h \circ \Lambda_\zeta \circ h)$, 
 establishing \eqref{E:ThetaSpecMini}.

 Finally, if we assume that $T_\ell\lambda^n_{{\mathcal X}_{\bar \eta}} : T_\ell\operatorname{A}^n({\mathcal X}_{\bar \eta})
 \to H^{2n-1}({\mathcal X}_{\bar \eta},\mathbb Z_\ell(n))$ and $T_\ell\lambda^n_{{\mathcal X}_{\bar s}} : T_\ell\operatorname{A}^n({\mathcal X}_{\bar s})
 \to H^{2n-1}({\mathcal X}_{\bar s},\mathbb Z_\ell(n))$ are isomorphisms, that  $\ell \nmid r$,  and that $d=2n-1$, then all of the morphisms in \eqref{E:specialization} are isomorphisms, and therefore  the isogeny $h$ induces an isomorphism $T_\ell h:T_\ell \operatorname{Ab}^n_{{\mathcal X}/S}|_{ s} \to T_\ell \operatorname{Ab}^n_{{\mathcal X}_{s}}$.
 \end{proof}

 We can apply this to the  case of a threefold   liftable to characteristic $0$, which essentially says that for a smooth projective geometrically rationally chain connected threefold $X$ over a perfect field $K$ that lifts to a geometrically rationally connected  threefold in characteristic $0$, then $\Theta_X$ is a  principal polarization.

\begin{cor}[Threefolds]\label{C:ThetaSpecMini} 
  Suppose that ${\mathcal X}/S$ is a smooth projective threefold over the spectrum $S$ of a
  DVR, such that the fraction field  $\kappa(\eta)$ has characteristic $0$ and the residue field $\kappa(s)$ is  perfect.   
  With  $\mathcal X_{\mathbb C}$ the base change of the generic fiber $\mathcal X_\eta$ to $\mathbb C$, assume that  $H^1(\mathcal X_{\mathbb C},\mathbb Z)=0$ and $AJ:\operatorname{A}^2(\mathcal X_{\mathbb C})\to J^3(\mathcal X_{\mathbb C})$ is surjective 
    (\emph{e.g.}, $\mathcal X_{\mathbb C}$ is rationally connected).

If for some natural number $N$, the multiple    $N \Delta_{{\mathcal X}_{\bar s}}\in \operatorname{CH}^{3}({\mathcal X}_{\bar s}\times_{\bar s} {\mathcal X}_{\bar
    s})$  admits a cohomological $\mathbb Z$-decomposition of type $(2,1)$ with respect to $H^\bullet (-,\mathbb Z_{\ell_s})$ for all primes $\ell\ne\operatorname{char}(\kappa(s))$  (\emph{e.g.}, ${\mathcal X}_s$ is geometrically rationally chain connected), 
     then $\Theta_{\mathcal X_s}$, the negative of the $\ell$-distinguished symmetric isogeny $\Lambda_{\mathcal X_s}$ of  Theorem~\ref{T:CanPol3-fold}, is a polarization on $\operatorname{Ab}^2_{\mathcal X_s/s}$.  
     
     If moreover,  $\Delta_{{\mathcal X}_{\bar s}}\in \operatorname{CH}^{3}({\mathcal X}_{\bar s}\times_{\bar s} {\mathcal X}_{\bar
    s})_{\mathbb Z[\frac{1}{N}]}$ admits a strict Chow decomposition for some natural number $N$ coprime to $\operatorname{char}(K)$ (\emph{e.g.}, $\mathcal X_s$ is geometrically stably rational), then $\Theta_{{\mathcal X}_s}$ is a principal polarization. 
  \end{cor}

\begin{proof}
On the generic fiber, we know that in characteristic $0$ both $T_\ell\phi^2$ and $T_\ell \lambda^2$ are isomorphisms so long as the Abel--Jacobi map is surjective (\emph{e.g.}, $\mathcal X_{\mathbb C}$ is uniruled \cite[Thm.~12.22]{voisinI}).  
   Thus the last two bullet points of  Proposition \ref{P:ThetaSpecMini} are satisfied for the generic fiber.  Moreover, in characteristic $0$, the Hodge-theoretic  polarization $\Theta_{\mathcal X_\eta}$ induces a distinguished symmetric isogeny $\Lambda_\eta=-\Theta_{\mathcal X_\eta}$ on $\operatorname{Ab}^2_{X_\eta/\eta}$,  since we have assumed that $H^1(\mathcal X_{\mathbb C}, \mathbb Z)=0$ (Theorem~\ref{T:CanPol}(2) and Remark \ref{R:PPolDM-Z}), which holds if we assume $\mathcal X_{\mathbb C}$ is rationally connected.

  	 	For the special fiber, the cohomological decomposition of a multiple of the diagonal  implies
that $\operatorname{Ab}^2_{\mathcal X_s/s}$ admits a distinguished symmetric $K$-isogeny (Corollary~\ref{C:CanPol}),  
and that $T_\ell\phi^2$ and $T_\ell \lambda^2$ are isomorphisms (Propositions~\ref{P:BlSr} and~\ref{P:lambdacohoBW}),
where $\phi^2$ indicates the universal regular homomorphism. 
	 Thus the last two bullet points of  Proposition \ref{P:ThetaSpecMini} are satisfied for  the special fiber.

 Thus we can conclude from  Proposition \ref{P:ThetaSpecMini} that $\Theta_{{\mathcal X}_s}$ is a polarization. 
\end{proof}

\subsection{Rationally chain connected threefolds and the work of Benoist--Wittenberg}\label{S:BWrat+que}

	Let $X$ be a smooth projective threefold over a perfect field $K$, and let  $\Lambda:\operatorname{Ab}^2_{X/K}\to \widehat{\operatorname{Ab}}\,^2_{X/K}$ be a symmetric $K$-isogeny.  Benoist--Wittenberg introduced the following terminology\,:
\begin{enumerate}
\item[(i)] $\Lambda$ satisfies  \cite[Property~2.4(i)]{BenWittClGr} if $\Lambda$ is distinguished.

\item[(ii)] $\Lambda$ satisfies  \cite[Property~2.4(ii)]{BenWittClGr} if $\Theta=-\Lambda$ is a principal polarization.  
\end{enumerate}
Benoist--Wittenberg give an equivalent formulation via the first Chern class $[\Lambda]$ (see \S\ref{S:ThetaLB}). 
However we prefer to work in the setting of symmetric isogenies.  
  They ask  \cite[p.6]{BenWittClGr}\,: 
  \begin{que}[{\cite{BenWittClGr} }]
Let $X$ be a smooth projective threefold over a perfect field $K$ of positive characteristic, with $\operatorname{CH}_0(X_{\bar K})_{\mathbb Q}$ universally trivial.  Does  
  there exist a symmetric $K$-isogeny $\Lambda_X$ on $\operatorname{Ab}^2_{X/K}$ such that\,:

\begin{enumerate}
\item $\Lambda_X$  is distinguished (\emph{i.e.}, satisfies  \cite[Property~2.4(i)]{BenWittClGr})\,?

\item $\Lambda_X$ is distinguished and principal (\emph{i.e.}, $\Lambda_X$ is an isomorphism)\,?
 
\item $\Lambda_X$ is distinguished and $\Theta_X=-\Lambda_X$ is a  polarization\,?

\item $\Lambda_X$ is distinguished and $\Theta_X=-\Lambda_X$ is a principal polarization (\emph{i.e.}, 
satisfies  \cite[Property~2.4(i) and (ii)]{BenWittClGr})\,?
\end{enumerate}
    
\end{que}

     Recall that if the characteristic of $K$ is zero, then the Hodge theoretic principal polarization $\Theta_X$ has the property that $\Lambda_X=-\Theta_X$ is distinguished (\emph{i.e.}, $\Lambda_X$ 
satisfies  \cite[Property~2.4(i) and (ii)]{BenWittClGr}).
 
  When $X$ is a geometrically \emph{rational} threefold over a perfect field $K$, Benoist--Wittenberg 
 provide an affirmative answer to all parts of their question  by constructing a principal polarization $\Theta_X$ on $\operatorname{Ab}^2_{X/K}$ such that $\Lambda_X=-\Theta_X$ is distinguished \cite[Cor.~2.8]{BenWittClGr}.  
 In fact, they  extend the Clemens--Griffiths condition over $\mathbb C$, showing that  if $X$ is rational \emph{over $K$}, then 
  $(\operatorname{Ab}^2_{X/K},\Theta_X)$ is the product of principally
 polarized Jacobians of curves.  
 Note that since their symmetric isogeny $\Lambda_X$ is distinguished, their result shows that  in this case  
  the  symmetric $K$-isomorphism $\Theta_X=-\Lambda_X:\operatorname{Ab}^2_{X/K}\to \widehat {\operatorname{Ab}}\,^2_{X/K}$  of Theorem \ref{T:main-pol}(2) and 
 Theorem~\ref{T:CanPol3-fold} is a principal polarization.

  Our results provide an affirmative answer  to part (1) of their question, and provide some further evidence for an affirmative answer to the other parts of their question.  More precisely, given a smooth projective threefold $X$ over a perfect field $K$ of positive characteristic, with $\operatorname{CH}_0(X_{\bar K})_{\mathbb Q}$ 
  universally trivial,  
  there exists a purely inseparable symmetric $K$-isogeny $\Lambda_X$ on $\operatorname{Ab}^2_{X/K}$ such that\,:
\begin{enumerate}
\item $\Lambda_X$ is distinguished (\emph{i.e.}, satisfies \cite[Property~2.4(i)]{BenWittClGr})\,;

\item $\Lambda_X$ is distinguished and principal if  $\operatorname{CH}_0(X_{\bar K})_{\mathbb Z[\frac{1}{N}]}$ is universally trivial for some natural number $N$ coprime to $\operatorname{char}(K)$.

\item  $\Lambda_X$ is distinguished and $\Theta_X=-\Lambda_X$ is a polarization  if $X$ lifts to characteristic $0$ to a geometrically  universally  $(\operatorname{CH}_0)_{\mathbb Q}$-trivial smooth projective threefold.

\item  $\Lambda_X$ is distinguished and $\Theta_X=-\Lambda_X$ is a principal polarization (\emph{i.e.}, satisfies \cite[Property~2.4(i) and (ii)]{BenWittClGr}) if  $\operatorname{CH}_0(X_{\bar K})_{\mathbb Z[\frac{1}{N}]}$ is universally trivial for some natural number $N$ coprime to $\operatorname{char}(K)$, and $X$ lifts to characteristic $0$ to a geometrically  universally  $(\operatorname{CH}_0)_{\mathbb Q}$-trivial smooth projective threefold.
\end{enumerate} 

The existence of the purely inseparable symmetric $K$-isogeny $\Lambda_X$, as well as (1) and (2), are shown in  Theorem \ref{T:CanPol3-fold}.  (3) and (4) are shown in  Corollary \ref{C:ThetaSpecMini}.   Note that in the language here, Conjecture \ref{conjecture} asserts that if we assume further that $\operatorname{CH}_0(X_{\bar K})$ is universally trivial, then the distinguished $K$-isomorphism $\Lambda_X$ of (2) above should have  $\Theta_X=-\Lambda_X$ being a (principal) polarization\,; \emph{i.e.}, it should also satisfy condition (4), without requiring the further hypothesis of lifting to characteristic $0$.

\section[Cohomological decomposition of the diagonal and minimal cohomology classes]{\for{toc}{Cohomological decomposition and minimal cohomology classes}\except{toc} {Cohomological decomposition of the diagonal, algebraic representatives, 
		and minimal cohomology classes}}
	\label{S:MinClass}
	
In this section, we relate cohomological decomposition of the diagonal to minimal cohomology classes.  	This follows Voisin's work, with the key addition being that we must use replacements for the canonical principal polarization coming from Hodge theory, as well as the canonical identification of the cohomology of a rationally connected threefold with the first homology of the intermediate Jacobian.  The starting point is a technical condition \eqref{E:tcond}, which is central to the discussion.

\subsection{The technical condition \eqref{E:tcond} for a
	cohomological decomposition of the diagonal}
\label{S:tcond}

The following technical theorem has a number of interesting
applications.  The key point is the condition~\eqref{E:tcond} for a
smooth projective variety $X$ over a field $K$, with respect to a
fixed Weil cohomology theory~$\mathcal H$ and a ring homomorphism $R\to R_{\mathcal H}$\,:
  \begin{equation}
\label{E:tcond}
\tag{$\ast_R$}
\parbox{\dimexpr\linewidth-4em}{
	\strut
	\emph{There exist finitely many (not necessarily distinct) smooth
		projective varieties $Y_i$ over $K$ of dimension $d_X- 2$, and correspondences
		$\Gamma_i\in \operatorname{CH}^{d_X-1}(Y_i \times_K X)_{R}$, such
		that for any $\alpha,\beta\in \mathcal H^{d_X}(X)$,}
	\begin{equation*}
	\langle\alpha,\beta \rangle_X=\sum_i \langle \Gamma_i^*\alpha ,\Gamma_i^*\beta
	\rangle_{Y_i}.
	\end{equation*}
	\strut
}
\end{equation}
This technical condition was introduced by Voisin~\cite[(35)]{voisinCubicCH} in the case $K=\cx$, $R=\integ$, and $\mathcal H^\bullet$ is Betti cohomology.
We will see in \S \ref{S:StarMinCl}, following Voisin,  that this condition can be related to
universal cycle classes and to minimal cohomology classes on abelian varieties.
\begin{rem}\label{R:algtriv}
If $K$ is either separably closed or finite and if $d_X> 2$, it is equivalent in~\eqref{E:tcond} to require the correspondences 	$\Gamma_i$ to belong to $\mathscr A_{X/K}^{d_X-1}(Y_i)_{R}$.
 To see this one replaces $\Gamma_i$ with $\Gamma_i - Y_i\times \Gamma_i|_{y_{i,0}}$ for any choice of zero-cycle $y_{i,0} \in \chow_0(Y_i)$ of degree 1, and one notes that $(Y_i\times \Gamma_i|_{y_{i,0}})^*\alpha = 0 $ for all $\alpha \in \mathcal{H}^{d_X}(X)$ for dimension reasons.
\end{rem}

\begin{teo}[{\cite[Thm.~3.1]{voisinCubicCH},
		\cite[Thm.~3.1]{mboro}}]\label{T:Mb-T3.1}
	Let $X$ be a smooth projective variety over a field~$K$, and let $R\to R_{\mathcal H}$ be a ring homomorphism.
  	
	\begin{enumerate}
		
		\item[(A)]  If  $\Delta_X\in \mathcal \chow^{d_X}(X\times_KX)$ admits a strict
		cohomological
		$R$-decomposition such that $D\subseteq X$ (in the notation of Definition \ref{D:StrCoDec}) admits an embedded
		resolution to a normal crossing divisor (\emph{e.g.}, $K$ is perfect and $d_X=3$
		\cite{CPRes2}), then \eqref{E:tcond} is satisfied.
		
		\item[(B)] As a partial converse, if \eqref{E:tcond} is satisfied, and the additional
		criteria are met:
		\begin{enumerate}
			
			\item $\mathcal H^\bullet (X)$ has no torsion,
			
			\item $\mathcal H^{2i}(X)(i)$ is $R$-algebraic for $2i\ne d_X$,
			
			\item $\mathcal H^{2i+1}(X) = 0$ for $2i +1\ne  d_X$,
		\end{enumerate}
		then $\Delta_X\in \chow^{d_X}(X\times_KX)$ admits a strict cohomological
		$R_{\mathcal H}$-decomposition.
	\end{enumerate}
\end{teo}

\begin{proof} The proof of
	\cite[Thm.~3.1]{voisinCubicCH} carries over directly to this case.  The case at hand here,
	where $K=\bar K$ and $\mathcal H^\bullet$ is $\ell$-adic cohomology ($\ell\ne
	\operatorname{char}(K)$), is \cite[Thm.~3.1]{mboro}.
  	 \end{proof}

\subsection{A first result concerning minimal cohomology classes}
For use in the proof of Theorem \ref{T:Intro-QDS-1}, we record  the following consequence of Corollary~\ref{C:ThetaSpecMini} regarding minimal cohomology classes\,:

\begin{pro}[Minimal cohomology classes]\label{C:MinCohSpec}
	With the same notation and assumptions as in Corollary~\ref{C:ThetaSpecMini},   
	if $g=\dim \operatorname{Ab}^2_{{\mathcal X}_{\bar {\eta}}/\bar {\eta}}= \dim \operatorname{Ab}^2_{{\mathcal X}_{\bar {s}}/\bar {s}}\le 3$, then  $[\Theta_{{\mathcal X}_{s}}]^{g-1}/(g-1)!$ is $\mathbb Z$-algebraic.  
	 	 \end{pro}

\begin{proof}
	We know from Corollary \ref{C:ThetaSpecMini} that $\Theta_{\mathcal X_s}$ is a polarization.   
	Thus there is an isogeny of polarized abelian varieties 
	$$
	h:(\operatorname{Ab}^2_{\mathcal X_s/s},\Theta_{\mathcal X_s})\to (A,\Theta)
	$$  
	with target a principally polarized abelian variety of dimension $g$ (see \emph{e.g.}, \cite[Prop.~11.25, Cor.~11.26]{EvdGM}, or \cite[Cor.~1, p.234]{mumfordAV}).  For dimension reasons, $(A,\Theta)$ is the Jacobian of a (possibly reducible) curve $C$, and therefore $[\Theta]^{g-1}/(g-1)!$ is the class of the Abel--Jacobi embedded curve~$[C]$. Consequently,  $h^*[C]=h^*[\Theta]^{g-1}/(g-1)!=[\Theta_{{\mathcal X}_{s}}]^{g-1}/(g-1)!$ is $\mathbb Z$-algebraic.
\end{proof}

 \subsection{Cohomological decompositions, universal codimension-$2$ cycles,  and
  minimal cohomology classes} \label{S:StarMinCl}

 We now revisit the technical Theorem~\ref{T:Mb-T3.1}, and convert the  condition
 \eqref{E:tcond} to a condition on universal codimension-$2$ cycles and minimal
 cohomology classes.

 \begin{pro}\label{P:StarMinCl}  Let $X$ be a smooth projective threefold over a perfect field $K$ and let $\ell\ne \operatorname{char}(K)$ be a prime. Assume that
   \begin{itemize}
   	\item $H^{3}(X_{\bar K},\mathbb Z_\ell(2))$ is torsion-free\,;
  	\item $
  	T_\ell\phi^2_{X}: T_\ell \operatorname{A}^2(X_{\bar
  		K})  \longrightarrow T_\ell \operatorname{Ab}^2_{X/K}$ is an isomorphism (standard assumption at $\ell$)\,;
   	\item $
  	T_\ell\lambda^2 : T_\ell\operatorname{A}^2(X_{\bar K})
  	\longrightarrow H^{3}(X_{\bar K},\mathbb Z_\ell(2))$ is an isomorphism.
   \end{itemize}
       Assume further (\emph{e.g.}, $K$ is finite or algebraically closed and  $X$ admits a universal codimension-$2$ cycle class\,; see  Theorem~\ref{T:CanPol}) that
 there exists
 an 
 $\ell$-distinguished symmetric $K$-isogeny  
 $\Lambda_X:\operatorname{Ab}^2_{X/K}\to \widehat{\operatorname{Ab}}\,^2_{X/K}$
  (Definition \ref{D:DistMorph}) such that $T_\ell\Lambda_X$ is an isomorphism, and set $[\Lambda_X]$ to be the first Chern class (\S\ref{S:ThetaLB}).  
  \begin{enumerate}
   \item[(A)]
   If \eqref{E:tcond} holds, and if $K$ is either finite or algebraically closed,
      then $$\frac{[\Lambda_X]^{g-1}}{(g-1)!}\in
  H^{2g-2}((\operatorname{Ab}^2_{X / K})_{\bar K},\mathbb Z_\ell(g-1))$$ is $R$-algebraic,  where $g :=
   \dim \operatorname{Ab}^2_{X/K}$.

   \item[(B)] If the class $\frac{[\Lambda_X]^{g-1}}{(g-1)!}\in
   H^{2g-2}((\operatorname{Ab}^2_{X / K})_{\bar K},\mathbb Z_\ell(g-1))$ is $R$-algebraic,  and
   $\operatorname{Ab}^2_{X/K}$ admits a universal codimension-$2$ cycle class, then
   \eqref{E:tcond} holds. 
  \end{enumerate}

 \end{pro}

 \begin{rem}
  The $R$-algebraicity of $[\Lambda_X]^{g-1}/(g-1)!$ is a tautology if $(g-1)!$ is a
  unit in the coefficient ring $R$.
 \end{rem}

 \begin{proof}
    A version  where $K=\mathbb C$, $R=\integ$,  and one uses Betti
  cohomology is \cite[Pf.~of Thm.~4.1]{voisinCubicCH} in case (A) and \cite[Pf.~of Thm.~4.9]{voisinAJ13} in case (B).
    The case where $X$ is a
  cubic threefold, $K=\bar K$, $\operatorname{char}(K)\ne 2$, 
  $\ell=2$, and 
  $R= \integ_2$  is \cite[Thm.~3.2]{mboro}.
To be precise, we note that our assumptions are slightly different from those of Voisin or Mboro. In Voisin's version, the bullet point conditions in the statement of the proposition are replaced with the  canonical identification  $H^3(X,\mathbb Z) = H_1(J^3(X),\mathbb Z)\cong H^1(J^3(X),\mathbb Z)$ induced via the principal polarization $\Theta_X=-\Lambda_X$ coming from the intersection product on $H^3(X,\mathbb Z)$. Similarly, Mboro uses the principally polarized Prym variety $(P,\Xi)$ of the cubic threefold as a replacement for the intermediate Jacobian, which Beauville has shown is in fact the algebraic representative and is  isomorphic to the group of algebraically trivial codimension-$2$ cycle classes, as well as Beauville's identification $H^3(X,\mathbb Z_\ell) =  H^1(P,\mathbb Z_\ell)$ (see~\cite{beauville77}).  

While our conditions essentially reduce to Voisin's and Mboro's  in these special cases, our proposition applies more generally. 
  Nevertheless, essentially the same argument as in \cite[Pf.~of Thm.~4.1]{voisinCubicCH} and
  \cite[Thm.~3.2]{mboro} carries over to this situation:
  a key point is that one can replace \cite[Lem~3.3]{mboro} with
  Corollary~\ref{C:commute-2}.    Note also that in the work of both Voisin and Mboro,
  $\Theta_X$ is a principal polarization, but the arguments in the setting of $H^\bullet (-,\mathbb Z_\ell)$ only require that it be an $\ell$-distinguished  symmetric isogeny that induces an isomorphism on $\ell$-adic Tate modules.

  Since our assumptions are more general, we provide a proof. For brevity, we denote $\mathcal H^\bullet(-)= H^\bullet (-,\mathbb Z_\ell)$. First, by our assumptions that $\Lambda_X$ is $\ell$-distinguished and $T_\ell\Lambda_X$ is an isomorphism, then by Lemma~\ref{L:linebundle}, we have a commutative diagram (see \eqref{E:DistMorph})
  \begin{equation*}
  \xymatrix@C=5em{
  	T_\ell \operatorname{Ab}^2_{X / K} 
  	\ar[d]^<>(0.5)\iota_\simeq& \mathcal H^1(\operatorname{Ab}^2_{X / K})(1) \ar[l]_{\cup \frac{[\Lambda_X]^{g-1}}{(g-1)!}}
  	\\
  	\mathcal H^3(X)(2)  \ar@{=}[r]&
    	\mathcal H^3(X)(2) \ar[u]^<>(0.5){\iota^\vee}_\simeq
  }
  \end{equation*}
 where $\iota = T_\ell \lambda^2\circ (T_\ell \phi^2)^{-1}$. 
  Since $	\mathcal H^3(X)(2)$ is identified with 
  $ \mathcal H^3(X)^\vee(-1)$ via the intersection pairing, we get 
  \begin{align}\label{E:inters}
  \langle \alpha,\beta \rangle_X & = \Big \langle \iota^\vee (\alpha), \iota^\vee(\beta) \cup  \frac{[\Lambda_X]^{g-1}}{(g-1)!} \Big \rangle_{\operatorname{Ab}^2_{X / K}} \quad \mbox{for all}\ \alpha,\beta \in \mathcal H^3(X)(2).
  \end{align}
  
  We proceed to prove (A).   Assume that \eqref{E:tcond} holds. Since we are assuming $K$ finite or algebraically closed, we may assume by Remark~\ref{R:algtriv} that the cycles $\Gamma_i \in \chow^{d_X-1}(Y_i\times_K X)_R$ in fact sit in $\mathscr{A}^{d_X-1}_{X/K}(Y_i)_R$. We can thus consider the $K$-morphisms $$\gamma_i:= \Phi^2_{X/K}(Y_i)(\Gamma_i) : Y_i \to \operatorname{Ab}^2_{X / K}.$$
  Corollary~\ref{C:commute-2} implies that 
  $$(\Gamma_i)_* = \iota \circ (\gamma_i)_* : \mathcal H_1(Y_i) \to \mathcal{H}^3(X)(2).$$
  By dualizing we get 
   $$\Gamma_i^*= \gamma_i^*\circ \iota^\vee :  \mathcal{H}^3(X)(2) \to \mathcal H^1(Y_i)(1) .$$
  By combining \eqref{E:tcond} with \eqref{E:inters}, we get for all $\alpha,\beta \in \mathcal H^3(X)(2)$
    \begin{align*}
  \Big \langle \iota^\vee (\alpha), \iota^\vee(\beta) \cup  \frac{[\Lambda_X]^{g-1}}{(g-1)!} \Big \rangle_{\operatorname{Ab}^2_{X / K}} 
  & = \sum_i n_i \langle \Gamma_i^*\alpha , \Gamma_i^*\beta \rangle_{Y_i} \\
  & = \sum_i n_i \gamma_i^* \langle \iota^\vee(\alpha), \iota^\vee(\beta) \rangle_{\operatorname{Ab}^2_{X / K}}.
  \end{align*}
 It follows that $$\frac{[\Lambda_X]^{g-1}}{(g-1)!} = \sum_i n_i (\gamma_i)_*\gamma_i^*[\operatorname{Ab}^2_{X/K}] = \sum_i n_i (\gamma_i)_*[Y_i],$$ thereby establishing the $R$-algebraicity of $\frac{[\Lambda_X]^{g-1}}{(g-1)!}$.

  We now prove (B). By assumption, there are finitely many connected curves $C_i$ in $\operatorname{Ab}^2_{X / K}$ and constants $n_i\in R$ such that $\frac{[\Lambda_X]^{g-1}}{(g-1)!} = \sum_i n_i[C_i]$. Let $\widetilde C_i \to C_i$ be the normalization morphisms and denote $j_i : \widetilde C_i \to \operatorname{Ab}^2_{X / K}$ the natural morphism. The map $\cup \frac{[\Lambda_X]^{g-1}}{(g-1)!} : H^1(\operatorname{Ab}^2_{X / K})(1) \to  	T_\ell \operatorname{Ab}^2_{X / K}$ is then given by $\sum_i n_i (j_i)_* j_i^*$. Let now $Z\in \mathscr{A}^2_{X/K}(\operatorname{Ab}^2_{X / K})$ be a universal cycle. By Corollary~\ref{C:commute-2}, $\iota$~coincides with $Z_*$ and $\iota^\vee$ coincides with $Z^*$\,; see~\eqref{E:CanPol}. We get
 \begin{align*}
 	\langle \alpha,\beta \rangle_X & = \Big \langle Z^*\alpha, Z^*\beta \cup  \frac{[\Lambda_X]^{g-1}}{(g-1)!} \Big \rangle_{\operatorname{Ab}^2_{X / K}}\\
 	&= \langle \alpha, \sum_i n_i (j_i)_* j_i^* Z^*\beta \rangle_{\operatorname{Ab}^2_{X / K}}\\
 		&= \sum_i \langle j_i^*Z^* \alpha,  j_i^* Z^*\beta \rangle_{\operatorname{Ab}^2_{X / K}},
 \end{align*} 
      where the first equality is~\eqref{E:inters} and the last equality is obtained using the projection formula. This establishes \eqref{E:tcond}.
    \end{proof}

 We next translate this into a statement about cohomological decompositions of
 the diagonal:

 \begin{cor}\label{C:StarMinCl}
 	With the assumptions of Proposition~\ref{P:StarMinCl}, we have\,:
 	 \begin{enumerate}
 		\item[(A)]
 		If $K$ is finite or algebraically closed and if $\Delta_X\in \chow^{d_X}(X\times_KX)$ admits a strict cohomological
 		$R$-decomposition,  then  the cohomology class $\frac{[\Lambda_X]^{g-1}}{(g-1)!}\in
 		 H^{2g-2}((\operatorname{Ab}^2_{X / K})_{\bar K},\mathbb Z_\ell(g-1))$ is $R$-algebraic,  where $g :=
 		\dim \operatorname{Ab}^2_{X/K}$.
 	\end{enumerate}
 	
  Assume further (to  the assumptions of Proposition~\ref{P:StarMinCl}) that\,:

  \begin{itemize}

   \item $H^\bullet(X_{\bar K},\mathbb Z_\ell)$ has no torsion,

   \item $H^{2i}(X_{\bar K}, \mathbb Z_\ell(i))$   are $\mathbb Z_\ell$-algebraic for all $i$,

   \item $ H^1(X_{\bar K},\mathbb Z_\ell) = 0$.
    \end{itemize}

  Then we have\,:

  \begin{enumerate}
   \item[(B)] If the class $\frac{[\Lambda_X]^{g-1}}{(g-1)!}\in
   H^{2g-2}((\operatorname{Ab}^2_{X / K})_{\bar K},\mathbb Z_\ell(g-1))$ is $\mathbb Z_\ell$-algebraic,   and
   $\operatorname{Ab}^2_{X/K}$ admits a universal codimension-$2$ cycle,
    then
   $\Delta_X\in \mathcal \chow^{d_X}(X\times_KX)$ admits a strict cohomological $\mathbb Z_\ell$-decomposition.
  \end{enumerate}
 \end{cor}

 \begin{proof}
  This is immediate from Theorem~\ref{T:Mb-T3.1} and
  Proposition~\ref{P:StarMinCl}.
 \end{proof}

\section{Proof of Theorem~\ref{T:Intro-M-StabInv}}\label{S:proof2}

We are now in a position to prove Theorem~\ref{T:Intro-M-StabInv}.  The main result of this section is Theorem \ref{T:intcohodec}, which is more general.
We also provide necessary and sufficient conditions for a threefold over an algebraically closed field to admit a cohomological $\mathbb Z_\ell$-decomposition in Theorem \ref{T:ZZell-iff}.

\subsection{Proof of Theorem \ref{T:Intro-M-StabInv}}

 \begin{teo}\label{T:intcohodec}
	Let  $X$ be a smooth projective threefold over a perfect field
$K$.

If $K=\bar K$ is algebraically closed and  the diagonal 
$\Delta_{X}\in\chow^{3}(X\times_{K} X)$ admits a strict cohomological \emph{$\mathbb Z$-decomposition}
 with respect to $H^\bullet (-,\mathbb Z_\ell)$, then
 for all prime numbers   $\ell\ne \operatorname{char}(K)$\,:

\begin{enumerate}
\item $H^1(X_{\bar K},\mathbb Z_\ell)=0$\,;

\item $H^{2i}(X_{\bar K},\mathbb Z_\ell(i))$  is $\mathbb Z$-algebraic for all $i$\,;

\item  The $\ell$-adic  Bloch map    $T_\ell \lambda^2:T_\ell \operatorname{A}^2(X_{\bar K})\to H^3(X_{\bar K},\mathbb Z_\ell(2))_\tau$ is an isomorphism\,;

\item[(3')] The $\ell$-adic map $T_\ell \phi^2_{X/K} :T_\ell \operatorname{A}^2(X_{\bar K}) \to T_\ell \operatorname{Ab}^2_{X/K}$
 is an isomorphism\,;

\item  $\operatorname{Tors}H^\bullet (X_{\bar K},\mathbb Z_\ell )=0$\,;
		
\item    $\operatorname{Ab}^2_{X/K}$ admits a universal codimension-$2$ cycle class\,; 
		
\item   Assuming (3) and (5), and setting $\Theta_X:\operatorname{Ab}^2_{X/K}\to \widehat {\operatorname{Ab}}\,^2_{X/K}$ to be the symmetric isogeny of Theorem \ref{T:Intro-CanPol0}(2), we have that $T_\ell\Theta_X$ is an isomorphism, and the class  $\frac{[\Theta_{X}]^{g-1}}{(g-1)!}\in H^{2g-2}((\operatorname{Ab}^2_{X/K})_{\bar K},\mathbb Z_\ell(g-1))$ is a $\mathbb Z$-algebraic class,
	where $g = \dim \operatorname{Ab}^2_{X/K}$ and $[\Theta_X]$ is the first Chern class of the line bundle associated to $\Theta_X$.

\end{enumerate}

As a partial converse, let  $K$ be  any perfect field and assume that (1)--(6)  hold (including (3')) for some prime number   $\ell\ne \operatorname{char}(K)$, where in (6) we define $[\Theta_X]$ to be the first Chern class of the line bundle associated to $(\Theta_X)_{\bar K}$ (\S \ref{S:c1Lambda}).  
Then 
the diagonal  $\Delta_{X} \in \chow^{3}(X\times_{K} X)$
admits a strict cohomological \emph{$\mathbb Z_\ell$-decomposition}.
 \end{teo}

\begin{rem}\label{R:Decomp21}
 In the case where $K=\bar K$, if we just assume that $\Delta_{X}\in\chow^{3}(X\times_{K} X)$ admits a  cohomological $\mathbb Z$-decomposition of type $(2,1)$, then (2)--(5) hold.
 \end{rem}

\begin{proof}
Assume  $K=\bar K$ and    that the diagonal	$\Delta_{X}\in\chow^{3}(X\times_{K} X)$   admits a strict cohomological $\mathbb Z$-decomposition.  Item~(1) is Corollary \ref{C:vanCoh}  (and in fact we get  this  using just a strict $\mathbb Z_\ell$-decomposition).   Item~(2)  is
	Corollary~\ref{C:AlgCyc} (and in fact  
	 we get $\mathbb Z_\ell$-algebraicity  using just a strict cohomological $\mathbb Z_\ell$-decomposition of type $(2,1)$, and using just a cohomological $\mathbb Z$-decomposition of type $(2,1)$ we 
	get \emph{$\mathbb Z$-algebraicity}).
	Note that this is the only place where we use that the field is algebraically closed\,; if one could prove Corollary~\ref{C:AlgCyc} over finite fields, then Theorem \ref{T:intcohodec} would hold over finite fields, as well.
 Item~(3) is Proposition~\ref{P:lambdacoho}  (and in fact we get this using just a $\mathbb Z_\ell$-decomposition of type $(2,2)$).  
Item~(3')  
 is Proposition \ref{P:BlSr} (and in fact we get this using just a $\mathbb Z_\ell$-decomposition of type~$(2,1)$). 
Item~(4) is Corollary~\ref{C:V-AJ4.4}  (and in fact we get this using just a $\mathbb Z_\ell$-decomposition of type~$(2,1)$). 
Item~(5) is Corollary~\ref{C:UnivCyc}  (and in fact we get this using just a $\mathbb Z$-decomposition of type~$(2,1)$).    
 Item~(6) is Corollary~\ref{C:StarMinCl}(A).

The converse statement is Corollary~\ref{C:StarMinCl}(B).  As this is the easier direction of Corollary~\ref{C:StarMinCl}, for convenience, we provide a brief proof in the notation  of Theorem \ref{T:intcohodec}.
  Working component-wise, we may and do assume $X$ is connected. By assumption (4), the intersection pairing $\mathcal H^2(X) \times \mathcal H^4(X) \to R_{\mathcal H}(-3)$ is perfect. By assumption (1) in case $i=1$ and $i=2$, it follows that the K\"unneth projectors $\pi^2_X$ and $\pi^4_X$ belong to the image of $\chow^3(X\times_K)_{R_{\mathcal H}} \to \mathcal H^6(X\times_K X)(3)$. Moreover, assumption (1) in case $i=3$ provides a zero-cycle $x\in \chow_0(X)_{R_\mathcal H}$ of degree $1$. We then consider the cycle 
$$\pi^3_X := \Delta_X - x\times_K X - \pi^2_X - \pi_X^4 - X\times_K x \quad \in \chow^3(X\times_K X)_{R_\mathcal H},$$
whose cohomology class defines, by assumptions (1) and (4), the K\"unneth projector on $\mathcal H^3(X)$. 
We aim to show that $[\pi^3_X]$ is supported on $D\times_K X$ for some divisor $D$. By assumptions (3), (3'), and (5), together with Lemma~\ref{L:linebundle} and Corollary~\ref{C:commute-2}, we have a commutative diagram (where as usual $\Lambda_X=-\Theta_X$):
 \begin{equation*}
\xymatrix@C=5em{
	T_\ell \operatorname{Ab}^2_{X / K} 
	\ar[d]^<>(0.5){Z_*}_\simeq& \mathcal H^1(\operatorname{Ab}^2_{X / K})(1) \ar[l]_{\cup \frac{[\Lambda_X]^{g-1}}{(g-1)!}}
	\\
	\mathcal H^3(X)(2)  \ar@{=}[r]&
	\mathcal H^3(X)(2)  \ar[u]^<>(0.5){Z^*}_\simeq
}
\end{equation*}
where $Z\in \mathscr{A}^2_{X/K}(\operatorname{Ab}^2_{X / K} )$
 is any universal codimension-2 cycle. Therefore, $$[\pi_X^3] =  [\pi_X^3]  \circ Z_* \circ \Big(-\cup \frac{[\Lambda_X]^{g-1}}{(g-1)!}\Big) \circ Z^* \circ [\pi_X^3].$$
By assumption (6),  there are finitely many connected curves $C_i$ in $\operatorname{Ab}^2_{X / K}$ and constants $n_i\in R_{\mathcal H}$ such that $\frac{[\Lambda_X]^{g-1}}{(g-1)!} = \sum_i n_i[C_i]$. Let $\widetilde C_i \to C_i$ be the normalizations morphisms and denote $j_i : \widetilde C_i \to \operatorname{Ab}^2_{X / K}$ the natural morphisms. The map $\cup \frac{[\Lambda_X]^{g-1}}{(g-1)!} : H^1(\operatorname{Ab}^2_{X / K})(1) \to  	T_\ell \operatorname{Ab}^2_{X / K}$ is then given by $\sum_i n_i (j_i)_* j_i^*$. Therefore, the cohomological correspondence $ [\pi_X^3]$ factors through $\bigoplus_i \mathcal H^1(\widetilde C_i)(-1)$, thereby establishing it is supported on $D\times_K X$ for some divisor $D$ in $X$.
\end{proof}

\subsection{Regarding conditions (1)--(4) of Theorem \ref{T:intcohodec} in positive characteristic}
As explained in the introduction, for any complex projective rationally connected threefold, conditions (1)--(3') of Theorem \ref{T:intcohodec} hold.  We show here that any smooth projective geometrically rationally chain connected threefold that lifts to a smooth projective geometrically rationally connected threefold in characteristic zero with no torsion in cohomology also satisfies conditions (1)--(3'), as well as (4) 
 (Corollary~\ref{C:RC->(0)}).  

We start with a  preliminary result that follows directly from a result of  Voisin \cite[Thm.~2]{voisinIntHodgeUni}:

\begin{teo}[Voisin]\label{T:V-Unirule}
 Suppose that ${\mathcal X}/S$ is a smooth projective threefold over the spectrum $S$ of a
  DVR $R$, such that the fraction field  $K := \kappa(\eta)\subseteq \mathbb C$ is characteristic $0$, and the residue field $k := \kappa(s)$ is algebraically closed.   
Let $\mathcal X_{\mathbb C}$ be the base change of $\mathcal X_\eta$, and fix a prime $\ell\ne \operatorname{char}\kappa(s)$.
\begin{enumerate}
\item
\label{enum:tf}
 If  $H^\bullet (\mathcal X_{\mathbb C}^{\an},\mathbb Z)$ has no $\ell$-torsion, then  $H^\bullet (\mathcal X_{\bar s},\mathbb Z_\ell)$ has no torsion.
\item
\label{enum:algebraic}
 If in addition $\mathcal X_{\mathbb C}$ is rationally connected or satisfies $H^2(\mathcal X_{\mathbb C},\mathcal O_{\mathcal X_{\mathbb C}})=0$ (resp.~$\mathcal X_{\mathbb C}$ is uniruled or satisfies $K_{\mathcal X_{\mathbb C}}\cong \mathcal O_{\mathcal X_{\mathbb C}}$ and $H^2({\mathcal X_{\mathbb C}},\mathcal O_{\mathcal X_{\mathbb C}})=0$), then
 $H^{2}({\mathcal X_{\bar s}},\mathbb Z_\ell(1))$ (resp.~$H^4({\mathcal X_{\bar s}},\mathbb Z_\ell(2))$)
 is algebraic.

\item
\label{enum:cris}
 If $R = \ww(k)$, the ring of Witt vectors of $k$, and if $p \ge 5$, then
  under the hypothesis that $H^\bullet (\mathcal X_{\mathbb C}^{\an},\mathbb Z)$  has no torsion, the crystalline cohomology $H^\bullet_\cris({\mathcal X_{\bar s}}/\ww(k))$
  has no torsion\,; and under the hypotheses of \eqref{enum:algebraic},
  $H^2_\cris({\mathcal X_{\bar s}}/\ww(k))$ and $H^4_\cris({\mathcal X_{\bar s}}/\ww(k))$, respectively, are
  algebraic.
\end{enumerate}
\end{teo}

\begin{proof}
If $H^\bullet (\mathcal X^{\an}_{\mathbb C},\mathbb Z)$ has no $\ell$-torsion, then  $H^\bullet({\mathcal X_{\mathbb C}},\mathbb Z_\ell)=H^\bullet(\mathcal X^{\an}_{\mathbb C},\mathbb Z_\ell)=H^\bullet(\mathcal X^{\an}_{\mathbb C},\mathbb Z)\otimes_{\mathbb Z}\mathbb Z_\ell$, which also has no torsion.  By proper base change, we also have that $H^\bullet({\mathcal X_{\mathbb C}},\mathbb Z_\ell)= H^\bullet (\mathcal X_{\bar K},\mathbb Z_\ell)= H^\bullet ({\mathcal X_{\bar s}},\mathbb Z_\ell)$, and so we have completed the argument for \eqref{enum:tf}.

For claim \eqref{enum:algebraic} on algebraicity, we argue as follows.  First, as $H^2({\mathcal X_{\mathbb C}},\mathcal O_{\mathcal X_{\mathbb C}})=0$, either  by assumption or else by using the assumption that  ${\mathcal X_{\mathbb C}}$ is rationally connected, then we can conclude that $H^2(\mathcal X^{\an}_{\mathbb C},\mathbb C)=H^{1,1}(\mathcal X^{\an}_{\mathbb C})$, so that algebraicity of $H^2(\mathcal X^{\an}_{\mathbb C},\mathbb Z)$ follows from the Lefschetz-$(1,1)$ theorem.  In the case where ${\mathcal X_{\mathbb C}}$ is uniruled, or $K_{\mathcal X_{\mathbb C}}\cong \mathcal O_{\mathcal X_{\mathbb C}}$ and $H^2({\mathcal X_{\mathbb C}},\mathcal O_{\mathcal X_{\mathbb C}})=0$, the  algebraicity of $H^4(\mathcal X^{\an}_{\mathbb C},\mathbb Z)$ is Voisin's result \cite[Thm.~2]{voisinIntHodgeUni}.

Since we are assuming \eqref{enum:tf}, we have $H^\bullet ({\mathcal X_{\mathbb C}},\mathbb Z_\ell)=H^\bullet (\mathcal X^{\an}_{\mathbb C},\mathbb Z_\ell)=H^\bullet (\mathcal X^{\an}_{\mathbb C},\mathbb Z)\otimes_{\mathbb Z}\mathbb Z_\ell$ and it follows from Chow's theorem that $H^{2i}({\mathcal X_{\mathbb C}},\mathbb Z_\ell(i))$ is algebraic for $i=1,2$ so long as $H^{2i}(\mathcal X^{\an}_{\mathbb C},\mathbb Z(i))$ is. 
 We next claim that this implies  $H^{2i}(\mathcal X_{\bar K},\mathbb Z_\ell(i))$ is algebraic.  Indeed, given any  $\alpha \in H^{2i}(\mathcal X_{\bar K},\mathbb Z_\ell(i))$, by proper base change, $\alpha$ is
identified with a class in $H^{2i}({\mathcal X_{\mathbb C}},\mathbb Z_\ell(i))$.  By algebraicity, we can write $\alpha=\sum a_i[Z_i]$ for some cycles $Z_i$ on ${\mathcal X_{\mathbb C}}$, and some coefficients
 $a_i\in \mathbb Z_\ell$. 
Each cycle $Z_i$ lies on some component of the Hilbert scheme for ${\mathcal X_{\mathbb C}}$, which is the base change to $\mathbb C$ of the Hilbert scheme for $\mathcal X_{\bar K}$.  Since the cohomology class of a
cycle is the same for any cycle in the same component of the Hilbert scheme (we can use resolution of singularities, for instance, to get a
smooth curve interpolating), we can replace $Z_i$ with a cycle $Z_i'$
defined over $\bar K$ (corresponding to any $\bar K$-point of the
corresponding Hilbert scheme), and we have $\alpha=\sum a_i[Z_i]=\sum
a_i[Z_i']$. Thus $H^{2i}(\mathcal X_{\bar K},\mathbb Z_\ell(i))$ is algebraic.
Now using the identification $H^{2i}(\mathcal X_{\bar K},\mathbb
Z_\ell(i))=H^{2i}({\mathcal X_{\bar s}},\mathbb Z_\ell(i))$ via proper base change, and
the fact that specialization of cycles in Chow respects the cycle
class map \cite[Exa.~20.3.5]{fulton} (\emph{e.g.},  \eqref{E:FultSpec}), we have that
$H^{2i}({\mathcal X_{\bar s}},\mathbb Z_\ell(i))$ is algebraic.

For \eqref{enum:cris}, the hypothesis that ${\mathcal X_{\bar s}}$ lifts to an
unramified mixed characteristic DVR guarantees that the freeness of
$H\udot(\mathcal X_\cx,\integ)$, and thus that of $H\udot(\mathcal X_{\bar K},\integ_p)$, implies that of 
$H\udot_\cris({\mathcal X_{\bar s}}/\ww(k))$ \cite[Thm.~1.1]{caruso08}.
   Moreover, this
freeness allows one to define an integral crystalline cycle class
map.  Now suppose that $H^{2i}(\mathcal X_\cx,\integ(i))$ is algebraic.
The comparison isomorphism between Betti and de Rham cohomology
is compatible with cycle class maps, and thus
$H^{2i}_{dR}(\mathcal X_\cx^\an)$ is algebraic.  Spreading and specializing shows
that there is a finite extension $L/\mathbb B(k)$ such that
$H^{2i}_{dR}(\mathcal X_L)$ is algebraic. (Note that the residue field of $L$
is again $k$.)  The comparison isomorphism between the de Rham
cohomology of $\mathcal X_L$ and the crystalline cohomology of ${\mathcal X_{\bar s}}$ is
compatible with cycle class maps \cite[App.~B]{gilletmessing87}, and
thus $H^{2i}_\cris({\mathcal X_{\bar s}}/\ww(k))$ is algebraic as well.
\end{proof}

We can now show that conditions (1)--(4) of Theorem \ref{T:intcohodec} hold for all threefolds liftable to rationally connected threefolds in characteristic $0$ with no torsion in cohomology.   

\begin{cor}\label{C:RC->(0)}
Suppose that ${\mathcal X}/S$ is a smooth projective threefold over the spectrum $S$ of a
  DVR $R$, such that the fraction field  $\kappa(\eta)\subseteq \mathbb C$ is characteristic $0$, and the residue field $\kappa(s)$ is algebraically closed.  
If $\mathcal X_s$ and $\mathcal X_\eta$ are geometrically rationally chain connected,   and $H^\bullet (\mathcal X_{\bar \eta},\mathbb Z_\ell)$ is torsion-free, then $\mathcal X_{ s}$ satisfies conditions  (1)--(4) of Theorem \ref{T:intcohodec}.  
 \end{cor}

\begin{proof}
Since $\mathcal X_s$ is rationally chain connected, a multiple of the diagonal admits a strict decomposition.  Then condition (1)  is Corollary \ref{C:vanCoh}.   (2) and (4) follow from Theorem \ref{T:V-Unirule}.
 (3) and (3') are \cite[Cor.~7.4]{ACMVBlochMap}.
\end{proof}

\subsection{Necessary and sufficient conditions for a cohomological $\mathbb Z_\ell$-decomposition}
We will see below, in Remark \ref{R:T-2-S+N}, that there are examples of smooth projective threefolds over an algebraically closed field $k=K=\bar K$  such that the diagonal admits a cohomological $\mathbb Z_\ell$-decomposition for some prime $\ell\ne \operatorname{char}(k)$, but where condition (5) in Theorem \ref{T:intcohodec} fails. 
In other words, while conditions (1)--(6) are sufficient for a  cohomological $\mathbb Z_\ell$-decomposition of the diagonal, they are not necessary. 
The following theorem gives necessary and sufficient conditions for a cohomological $\mathbb Z_\ell$-decomposition\,:

  \begin{teo}[Cohomological $\mathbb Z_\ell$-decomposition of the diagonal]\label{T:ZZell-iff}
	Let  $X$ be a smooth projective threefold over an algebraically closed field $k$, and fix a prime number   $\ell\ne \operatorname{char}(k)$.
Then 
$\Delta_{X}\in\chow^{3}(X\times_{k} X)$ admits a  strict cohomological \emph{$\mathbb Z_\ell$-decomposition}  with respect to $H^\bullet (- , \mathbb Z_\ell)$ if and only if 

\begin{enumerate}
\item $H^1(X,\mathbb Z_\ell)=0$\,;

\item[($2_\ell$)] $H^{2i}(X,\mathbb Z_\ell(i))$  is $\mathbb Z_\ell$-algebraic for all $i$\,;

\item[(3)]  The $\ell$-adic  Bloch map    $T_\ell \lambda^2:T_\ell \operatorname{A}^2(X)\to H^3(X,\mathbb Z_\ell(2))_\tau$ is an isomorphism\,;

\item[(3')] The $\ell$-adic map $T_\ell \phi^2_{X/k} :T_\ell \operatorname{A}^2(X) \to T_\ell \operatorname{Ab}^2_{X/k}$
 is an isomorphism\,;

\item[(4)]  $\operatorname{Tors}H^\bullet (X,\mathbb Z_\ell )=0$\,;
		
\item[($5_\ell$)] $\operatorname{Ab}^2_{X/k}$ admits a miniversal codimension-$2$ cycle class $Z$ of degree coprime to $\ell$\,;
   
\item[($6_\ell$)]   Assuming (3)  and ($5_\ell$), and setting $\Theta_Z:\operatorname{Ab}^2_{X/k}\to \widehat {\operatorname{Ab}}\,^2_{X/k}$ to be the morphism induced by the cycle class $-({}^tZ\circ Z)$ (see \eqref{E:ThetaZ} and Theorem \ref{T:CanPol}), 
we have that $\Theta_Z$ is a 
symmetric isogeny of degree coprime to $\ell$, 
    and   $\frac{[\Theta_{Z}]^{g-1}}{(g-1)!}\in H^{2g-2}(\operatorname{Ab}^2_{X/k},\mathbb Z_\ell(g-1))$ is a $\mathbb Z_\ell$-algebraic class,
	where $g = \dim \operatorname{Ab}^2_{X/k}$ and $[\Theta_Z]$ is the first Chern class of the line bundle associated to $\Theta_Z$.

\end{enumerate} 
 \end{teo}
 
 \begin{proof}
The necessity of the conditions above  is given in the proof of Theorem \ref{T:intcohodec}:  the results cited in that proof  show that the weaker conditions in Theorem \ref{T:ZZell-iff} hold under the weaker hypothesis of a cohomological  $\mathbb Z_\ell$-decomposition.  Note that in going from (3)  and ($5_\ell$) to the symmetric isogeny $\Theta_X$ of ($6_\ell$), we are using Theorem \ref{T:CanPol} in the case of a miniversal cycle, rather than the case of a universal cycle.

The proof of sufficiency is again  Corollary~\ref{C:StarMinCl}(B).
Technically, if $Z$ is a miniversal codimension-$2$ cycle class of degree $r$ coprime to $\ell$, then the symmetric isogeny $\Lambda_Z=-\Theta_Z$ of Theorem \ref{T:CanPol} is not $\ell$-distinguished, but rather, satisfies $T_\ell \Lambda_Z =r^2\circ \iota^\vee \circ \iota$, where $\iota = T_\ell \lambda^2 \circ (T_\ell \phi^2)^{-1}$.
It is easy to see that Proposition \ref{P:StarMinCl} holds under this hypothesis, and therefore that Corollary~\ref{C:StarMinCl}(B) does, as well.  One notationally easy way to express that is to say that $\frac{1}{r^2}\Lambda_Z\in \operatorname{Hom}(\operatorname{Ab}^2_{X/K},\widehat {\operatorname{Ab}}^2_{X/})_{\mathbb Q}$ is $\ell$-distinguished, in the sense that $\frac{1}{r^2}T_\ell \Lambda_Z = \iota^\vee \circ \iota$, and one can check that Proposition \ref{P:StarMinCl} and Corollary~\ref{C:StarMinCl}(B) hold for symmetric $K$-isogenies $\Lambda:\operatorname{Ab}^2_{X/K}\to \widehat {\operatorname{Ab}}^2_{X/K}$ such that $T_\ell \Lambda$ is an isomorphism, and such that  there is an integer $N$ invertible in $R$, such that $\frac{1}{N}\Lambda$ is $\ell$-distinguished.
\end{proof}

\begin{rem}
Since conditions (1)--(4) of Theorem \ref{T:intcohodec} imply conditions (1)--(4) of Theorem \ref{T:ZZell-iff}, one can apply Corollary \ref{C:RC->(0)} to Theorem \ref{T:ZZell-iff}, as well.
\end{rem}

\newpage
\part{Stable rationality and quartic double solids in positive characteristic}\label{P:positive}

\section{\for{toc}{Homological decomposition, resolution of singularities, and
	specialization}\except{toc}{Homological decomposition of the diagonal, resolution of singularities,\\ and
specialization}}\label{S:ResSing-Def}

The goal of this section is to show that cohomological decomposition of the
diagonal is stable under specialization, and under resolution of singularities
for
nodal projective varieties.  The results generalize Voisin's results over
$\mathbb C$.  Since those  arguments are made with cycle classes in Betti
homology, the central point is to convert elements of the arguments to work in
Borel--Moore homology, or more precisely,   in the algebraic setting of
$\ell$-adic homology.  For this we need a few basic results on $\ell$-adic homology, which
unfortunately we could not find in the literature.

In this section, for uniformity, we fix a coefficient ring $\Lambda=\mathbb Z$ or $\mathbb Z_\ell$.  For an algebraically closed field $k$ and a
scheme  $\pi:X\to \operatorname{Spec}k$ of finite type over $k$, we denote by
$\omega_X:=R\pi^!\Lambda$ the dualizing sheaf in the derived category of
constructible $\Lambda$-sheaves on $X$.  For a scheme $\pi:X\to \operatorname{Spec}K$ of finite type over a field $K$, we  denote by $\mathcal H_i(X):=\mathbb
H^{-i}(X_{\bar K},\omega_{X_{\bar K}})$ the $\ell$-adic homology, or alternatively, the Borel--Moore
homology, when working over $K=\mathbb C$ and with the analytic topology\,; we refer the reader to \cite{laumon76} for details on $\ell$-adic homology, and to \cite{BorelMoore} for details on the Borel--Moore homology.

\subsection{Homological decomposition of the diagonal}
 As before, let $R_{\mathcal H}$ denote the coefficient ring of $\mathcal H_\bullet$\,; \emph{i.e.}, $\mathbb Z_\ell$ for $\ell$-adic homology, and $\mathbb Z$ for Borel--Moore homology.
We adapt Definition \ref{D:CoDec} to the setting of homology\,:

\begin{dfn}[Homological decomposition of a cycle class]\label{D:HomDec}
	Let $R\to R_{\mathcal H}$ be a 
	homomorphism of commutative rings.
	Let $X$ be a scheme of finite type over  a field $K$, 
	and
	let
	$$
	\xymatrix{j_i:W_i\ar@{^{(}->}[r]^{\quad \ne}& X}, \ \ i=1,2
	$$
	be reduced closed subschemes not equal to $X$.
	A \emph{homological $R$-decomposition   of  type $(W_1,W_2)$ of  a cycle class $Z\in
		\operatorname{CH}^{d_X}(X\times_KX)_R$} (with respect to $\mathcal H_\bullet$)
	is an equality
	\begin{equation}\label{E:DefDcp3}
	[Z] =[Z_1]+[Z_2]\in \mathcal H_{2d_X}(X\times_K X)(-d_X),
	\end{equation}
	where $Z_1\in \operatorname{CH}^{d_X}(X\times _K X)_R$ is supported on
	$W_1\times_KX$ and $Z_2\in \operatorname{CH}^{d_X}(X\times _K X)_R$ is supported
	on $X\times_K W_2$ (see \eqref{E:ChowSupp} for the support of a cycle).   When $R=\mathbb Z$,  we call this a  \emph{homological decomposition
		of  type $(W_1,W_2)$} (with respect to $\mathcal H_\bullet$). 	We say that $Z\in
	\operatorname{CH}^{d_X}(X\times_KX)_R$ has a \emph{homological $R$-decomposition of type $(d_1,d_2)$} (with respect to $\mathcal H_\bullet$) if it admits a homological $R$-decomposition of type $(W_1,W_2)$ with $\dim W_1 \leq d_1$ and $\dim W_2 \leq d_2$.
\end{dfn}

\begin{rem}
Recall that in the case where $X$ is smooth projective and equidimensional, a homological decomposition is the same as a cohomological decomposition.   Indeed, setting $\mathcal H^\bullet$ to be $\ell$-adic cohomology with $\mathbb Z_\ell$-coefficients in the case of $\ell$-adic homology, or Betti cohomology in the case of Borel--Moore homology, 
the cap product with the fundamental class of $X$ induces for all $i$ (\emph{e.g.}, \cite[p.173]{laumon76}) an isomorphism $-\cap[X]: \mathcal H^{2d_X-i}(X)(d_X)\stackrel{\sim}{\to}\mathcal H_{i}(X)$, and the cycle class maps are compatible \cite[Rem.~6.4]{laumon76}. 
\end{rem}

\subsection{Homological decomposition of the diagonal and resolution of singularities}\label{P:pos}

\subsubsection{Long exact sequences in $\ell$-adic homology}

In this section we recall some long exact sequences in $\ell$-adic homology.
These are standard in Borel--Moore homology, going back to the original paper~\cite{BorelMoore}.  Unfortunately, these do not seem to appear in the literature in the
analogous theory of $\ell$-adic homology.    Here we present the  arguments of
\cite{BorelMoore} in the modern language of the $6$-functor formalism,
establishing the results in either setting.

First, assume that $i:Z\subseteq X$ is a closed subvariety, and let $U=X-Z$ be
the complement, with inclusion $j:U\subseteq X$. Then there is a long exact
sequence (\emph{e.g.}, \cite[Thm.~3.8]{BorelMoore} in Borel--Moore homology)
\begin{equation}\label{E:U=X-Zell}
\xymatrix{
	\cdots \ar[r]& \mathcal H_i(Z) \ar[r]^{i_*} & \mathcal H_i(X)\ar[r]^{j^*} &
	\mathcal H_i(U)\ar[r]& \cdots
}
\end{equation}
The  key observation to establish this is that there is an exact triangle in the derived category of
constructible sheaves
$$
i_*\omega_Z\to \omega_X\to j_*\omega_U\to i_*\omega_Z[1].
$$
Taking the long exact sequence in hyper-cohomology gives \eqref{E:U=X-Zell}.
To obtain this exact triangle, we replace $\omega_X$ with a quasi-isomorphic
complex of injectives, $I^\bullet$, and then consider the short exact sequence
(\emph{e.g.}, \cite[Exe.~II.1.20]{hartshorne})
$$
0\to \mathscr H_Z^0(I^\bullet)\to I^\bullet \to j_*(I^\bullet|_U)\to 0,
$$
which is exact on the right since injectives are flasque.
This gives the exact triangle
$$
\mathscr H^0_Z(\omega_X) \to \omega_X \to j_*(\omega_X|_U) \to \mathscr
H^0_Z(\omega_X)[1].
$$
Now we recall that  $\omega_X:=R\pi^!\Lambda $, where $\pi:X\to
\operatorname{Spec}K$ is the structure morphism.  For closed immersions, we have
$i_*Ri^!=\mathscr H^0_Z$ \cite[(0.3.2)(a)]{laumon76}, so that on the left we
have $\mathscr H^0_Z(\omega_X)=i_*i^!\omega_X=i_*i^!\pi^!\Lambda=i_*\omega_Z$.
For open immersions, we have $Rj^!=j^*$ so that
$\omega_X|_U=j^*\omega_X=Rj^!R\pi^!\Lambda = \omega_U$.

\medskip 
Similarly, suppose that we have $X=X_1\cup X_2$ a decomposition of a variety
into two irreducible components.
Let $i_j:X_j\hookrightarrow X$, $j=1,2$, be the closed immersions of the
components, and let $i_{12}:X_1\cap X_2\hookrightarrow X$ be the closed
immersion of the intersection.
Then we have a long exact sequence (\emph{e.g.}, \cite[Thm.~3.10]{BorelMoore} in
Borel--Moore homology)
\begin{equation}\label{E:X1X2ell}
\xymatrix{
	\cdots \ar[r]& \mathcal H_i(X_1\cap X_2) \ar[r]^{} & \mathcal H_i(X_1)\oplus
	\mathcal H_i(X_2)\ar[r]^{} & \mathcal H_i(X)\ar[r]& \cdots
}
\end{equation}
where the maps are the obvious inclusion and difference maps.
Again, the point is that we have an exact triangle
$$
i_{12*}\omega_{X_1\cap X_2}\to i_{1*}\omega_{X_1}\oplus i_{2*}\omega_{X_2}\to
\omega_X\to i_{12*}\omega_{X_1\cap X_2}[1].
$$
The argument to obtain this exact triangle is similar.   We replace $\omega_X$
with a quasi-isomorphic complex of injectives, $I^\bullet$, and then consider
the short exact sequence
$$
0\to \mathscr H_{X_1\cap X_2}^0(I^\bullet)\to \mathscr H_{X_1}^0(I^\bullet)
\oplus \mathscr H_{X_2}^0(I^\bullet) \to I^\bullet \to 0,
$$
where again we get exactness on the right using that injectives are flasque.
This gives the exact triangle
$$
\mathscr H_{X_1\cap X_2}^0(\omega_X)\to \mathscr H_{X_1}^0(\omega_X)\to
\omega_X\to \mathscr H_{X_1\cap X_2}^0(\omega_X)[1].
$$
Now we use again the description of the extraordinary pull back for closed
immersions to conclude $\mathscr H_{X_1\cap
	X_2}^0(\omega_X)=i_{12*}\omega_{X_1\cap X_2}$ and $\mathscr
H_{X_j}^0(\omega_X)=i_{j*}\omega_{X_j}$.

\subsubsection{Applications to decomposition of the diagonal and resolution of
	singularities}

Recall that we say that a variety $X$ over a perfect field $K$
       has at worst
ordinary double point singularities if it is smooth over $K$, or has isolated
singular points, each of which is a $K$-point with a projective tangent cone that is a
smooth quadric over $K$.

\begin{pro}\label{P:Dec-Res}
	Let $X$ be a projective variety over an algebraically closed field $k$, 
	having only ordinary double point singularities, let   $\epsilon:\widetilde X \to
	X$ be the standard resolution obtained by blowing up the singular points of $X$,
	and assume that the even degree cohomology of $\widetilde X$ is algebraic and
	without torsion.  Then $\Delta_X\in \operatorname{CH}^{d_X}(X\times_kX)$  admits  a
	strict homological $R$-decomposition  if and and only if $\Delta_{\widetilde X}\in \operatorname{CH}^{d_X}(\widetilde X\times_k\widetilde X)$  admits  a strict homological $R$-decomposition. 
\end{pro}

\begin{rem}
	A similar result holds for Chow groups  \cite[Thm.~2.1]{voisinUniv}
	\cite[Thm.~1.14]{CTP16} \cite[Prop.~8]{HKT16}).
\end{rem}

\begin{proof}
	The case where $K=\mathbb C$ and cycles are taken in Betti homology is
	\cite[Thm.~2.1]{voisinUniv}.
	Since we have  \eqref{E:U=X-Zell} and \eqref{E:X1X2ell} in $\ell$-adic (and
	Borel--Moore)  homology, and the even degree cohomology of a smooth  quadric over an
	algebraically closed field is algebraic \cite[Exp.~XII,Thm.~3.3]{SGA7II},
        	Voisin's proof carries over essentially without
	change.  For convenience, we include the proof here. 
	 	
	We start with the easy direction.
Assume that $\Delta_{\widetilde X}\in\operatorname{CH}^{d_X}(\widetilde
X\times_k\widetilde X)$  has  a strict $R$-decomposition, $[\Delta_{\widetilde
 X}]=[\widetilde Z_1]+[\widetilde Z_2]$ with $\widetilde Z_1$ supported on
$\widetilde D\times_k \widetilde X$ and $\widetilde Z_2=\tilde
{\operatorname{pr}}_2^*[\tilde \alpha]$ for some zero-cycle class $\tilde
\alpha\in \operatorname{CH}_0(\widetilde X)$.   Now since $\epsilon$ is proper,
we have push forward in homology that is compatible with the cycle class maps
\cite[\S 6]{laumon76}.    Then we simply push-forward via the proper map
$\epsilon\times \epsilon$. More precisely, we  have $\Delta_X=(\epsilon\times
\epsilon)_*\Delta_{\widetilde X}$,  $Z_1:=(\epsilon\times\epsilon)_*\widetilde
Z_1$ supported on $D\times _kX$ where $D$ is the image of $\widetilde D$, and
$Z_2:=(\epsilon\times\epsilon)_*\tilde Z_2=(\epsilon\times
\epsilon)_*\tilde{\operatorname{pr}}_2^*\tilde \alpha =
{\operatorname{pr}}_2^*\epsilon_*\tilde \alpha$.  This provides the strict
decomposition for~$\Delta_X$.

Conversely,  assume we have a strict $R$-decomposition of the diagonal for $X$\,:
\begin{equation}\label{E:A-V-2.1-Eq(13)}
[\Delta_{X}]-[Z_1]-[Z_2]=0\in \mathcal H_{2d_X}(X\times_k X)(-d_X),
\end{equation}
where $Z_2=\operatorname{pr}_2^*\alpha$ for some zero-cycle class $\alpha\in
\operatorname{CH}_0(X)_R$.  For space, we will leave off all of the Tate twists in what follows.  

We denote by $x_i$  the singular points of $X$, $x$ the union of the singular points,
$Q_i$ the exceptional divisors (smooth quadrics)  for the resolution
$\epsilon:\widetilde X\to X$, and $Q$ the union of these quadrics.

There is a diagram of maps of homology groups\,:

$$
\xymatrix{
 \mathcal H_{2d_X}(X\times_k X)  \ar[r] &\mathcal H_{2d_X}(X\times_k X-(X\times_k
 x\cup X\times_k x))  \ar@{=}[d]\\
 \mathcal H_{2d_X}(\widetilde X\times_k \widetilde X) \ar[r]& \mathcal
 H_{2d_X}(\widetilde X\times_k \widetilde X-(\widetilde X\times_k Q\cup Q\times_k
 \widetilde X)) 
}
$$

Since $\epsilon:\widetilde X\to X$ is lci ($\widetilde X$ is smooth), we
may take the Gysin pull back $\widetilde Z_1$ of $\widetilde Z$,  and set
$\widetilde Z_2=\tilde{\operatorname{pr}}_2^* \epsilon^*\alpha$. We find that the classes 
$$
[\Delta_{X}]-[Z_1]-[ Z_2] \ \  \ \left(=0\in \mathcal H_{2d_X}(X\times_k X) \right)
$$
$$
[\Delta_{\widetilde X}]-[\widetilde Z_1]-[\widetilde Z_2] \ \ \ \left(\in \mathcal H_{2d_X}(\widetilde X\times_k
\widetilde X) \right)
$$
have the same image in $\mathcal H_{2d_X}(X\times_k X-(X\times_k x\cup X\times_k
x)) = \mathcal H_{2d_X}(\widetilde X\times_k \widetilde X-(\widetilde X\times_k
Q\cup Q\times_k \widetilde X)) 
$.

From the long exact sequence \eqref{E:U=X-Zell},
one obtains that
$$
[\Delta_{\widetilde X}]-[Z_1]-[Z_2]\in \mathcal H_{2d_X}(\widetilde X\times_k
\widetilde X) 
$$
comes from a homology class
$$
\beta\in \mathcal H_{2d_X}(\widetilde X\times_k Q\cup Q\times_k\widetilde X) .
$$
The next observation is that the closed subset $\widetilde X\times_k Q\cup
Q\times_k \widetilde X\subseteq \widetilde X\times_k \widetilde X$ is the union
of $\widetilde X\times_k Q$ and $Q\times_k \widetilde X$ glued along $Q\times_k
Q$, so that we have from \eqref{E:X1X2ell}
$$
\cdots\to  \mathcal H_{2d_X}(\widetilde X\times_k Q) \oplus \mathcal
H_{2d_X}(Q\times_k \widetilde X)\to \mathcal H_{2d_X}(\widetilde X\times_k Q\cup
Q\times_k \widetilde X) 
\to \mathcal H_{2d_X-1}(Q\times_k Q) \to \cdots
$$
As $Q\times Q=\coprod_{i,j}Q_i\times Q_j$, and $Q_i\times Q_j$ has trivial
homology in odd degree \cite[Exp.~XII, Thm.~3.3]{SGA7II}, 
 we conclude that $\mathcal H_{2d_X-1}(Q\times_k Q)=0$, so
that $\beta$ comes from a homology class
$$
\gamma=(\gamma_1,\gamma_2)\in \mathcal H_{2d_X}(\widetilde X\times_k Q) \oplus
\mathcal H_{2d_X}(Q\times_k \widetilde X).
$$
    
We now use the assumption made on $\widetilde X$, namely that its even degree
cohomology is algebraic.
As the cohomology of $Q$ has no torsion and is algebraic
\cite[Exp.~XII,Thm.~3.3]{SGA7II}
 (this is the only place  we are using that $k$ is algebraically closed),
we get by the K\"unneth decomposition that
$$
\mathcal H_{2d_X}(Q\times_k \widetilde X) = \mathcal H^{2d_X-2}(Q\times_k \widetilde X)=\bigoplus_{0\le 2i\le 2d_X-2}\mathcal
H^{2i}(Q)\otimes \mathcal H^{2d_X-2-2i}(\widetilde X)
$$
is generated by classes of algebraic cycles $z_j\times_k z_j'\subseteq Q\times_k
\widetilde X$, and similarly for $\widetilde X\times_k Q$.

Putting everything together, we get an equality
$$
\Delta_{\widetilde X}-[Z_1]-[Z_2]=\sum n_j[z_j\times_k z_j']
+\sum n_j'[z_j'\times_k z_j] \in \mathcal H_{2d_X}(\widetilde X\times_k \widetilde
X).
$$
This provides us with an integral cohomological decomposition of the diagonal.
Indeed, all the cycle classes of the form $[\widetilde {\mathcal X_{\bar s}}\times_k pt]$ are
cohomologous and they have to sum-up to zero, while all the other terms
$[z_k'\times_k z_k]$ with $\dim z_k'<n$ are supported on $D\times_k \widetilde X$
for some closed algebraic subset $D\subsetneq \widetilde X$.
\end{proof}

\subsection{Homological decomposition of the diagonal and specialization}

\subsubsection{Specialization and cycle class maps in $\ell$-adic homology}

Let  $B$ be a smooth variety of dimension~$1$ over a
field $K$,  let  $f: \mathcal X \to  B$ be a smooth morphism.  Setting $X_{\bar
	\eta}$ to be the geometric generic fiber, and $X_{b}$ to be the fiber over a
$\bar K$-point $b\in B(\bar K)$, one has a commutative diagram
\cite[Exa.~20.3.5]{fulton} (see \cite[p.65]{fultonIHES75} and \cite[SGA6,
Exp.~X, 7.13-7.16]{SGA6})
\begin{equation}\label{E:FultSpec}
\xymatrix{
	\operatorname{CH}_{n}(\mathcal X_{\bar \eta}) \ar[r]^{\mathrm{sp}} \ar[d]^{[-]}&
	\operatorname{CH}_{n}(X_{b}) \ar[d]_{[-]}\\
	\mathcal H^{2n}(\mathcal X_{\bar \eta}) \ar@{=}[r]& \mathcal H^{2n}(X_{b})
}
\end{equation}
where the top arrow is the specialization map of \cite[\S 20.3]{fulton}, and the
bottom equality comes from proper base change.
If $K=\mathbb C$, and we identify $\bar{\kappa(\eta)}=\mathbb C$, then we have
the same result in Borel--Moore homology.

If more generally we want to consider  a morphism
$f : \mathcal X \to  B$ of finite type, then there is a specialization map in
$\ell$-adic homology making the following diagram commute\,:

\begin{equation}\label{E:SpCycEll}
\xymatrix{
	\operatorname{CH}_{n}(\mathcal X_{\bar \eta}) \ar[r]^{\mathrm{sp}} \ar@{->>}[d]^{[-]}&
	\operatorname{CH}_{n}(X_{b}) \ar[d]_{[-]}\\
	\mathcal H_{2n}(\mathcal X_{\bar \eta})' \ar[r]^{\mathrm{sp}}& \mathcal H_{2n}(X_{b})
}
\end{equation}
where $\mathcal H_{2n}(X_{\bar \eta})'\subseteq \mathcal H_{2n}(X_{\bar \eta})$
is  the image of the cycle class map.  Again, if $K=\mathbb C$, and we identify
$\bar{\kappa(\eta)}=\mathbb C$, then we have the same result in Borel--Moore
homology.
Since there does not appear to be a reference in the literature, we explain this
now.
Setting $B^\circ=B-\{b\}$ and $f^\circ :\mathcal X^\circ =B^\circ\times
_B\mathcal X\to B^\circ$ to be the restriction,  we obtain a commutative diagram
$$
\xymatrix@C=.5em@R=.5em{
	&&&&&\operatorname{CH}_n(X_{b}) \ar@{-}[d]&\\
	\operatorname{CH}_n(X_{b})\ar[rr]_{i_*}
	\ar[dd]^{[-]}&&\operatorname{CH}_n(\mathcal X/B) \ar[rr]_{j^*}\ar[dd]^{[-]}
	\ar[rrru]^{i^!}&&\operatorname{CH}_n(\mathcal
	X^\circ/B^\circ)\ar[rr]\ar[dd]^{[-]} \ar@{-->}[ru]_{\mathrm{sp}} & \ar[d]^{[-]}&0\\
	&&&&&\mathcal H_{2n}(X_{b})&\\
	\mathcal H_{2n}(X_{b})\ar[rr]_{i_*}&&\mathcal H_{2n}(\mathcal X) \ar[rr]_{j^*}
	\ar[rrru]^{i^!}&&\mathcal H_{2n}(\mathcal X^\circ) &&\\
}
$$
where the horizontal sequences are exact (\cite[Prop.~1.8]{fulton} and
\eqref{E:U=X-Zell}). Since  $i^!i_*=0$ for Chow groups \cite[\S 20.3]{fulton},
the top of this diagram gives the definition of the specialization map \cite[\S
20.3]{fulton}.  The compatibility of the cycle class maps with proper push
forward and flat pull back is standard \cite[\S 6]{laumon76}.

On the images of the cycle class maps in homology, by commutativity, we have
$i^!i_*=0$, so that from the diagram above,  we can define the specialization
map in homology\,:
$$
\xymatrix{
	\operatorname{CH}_n(\mathcal X^\circ/B^\circ) \ar[r]^{\mathrm{sp}} \ar@{->>}[d]^{[-]}&
	\operatorname{CH}_n(X_{b}) \ar[d]^{[-]}\\
	\mathcal H_{2n}(\mathcal X^\circ)'  \ar[r]^{\mathrm{sp}}& \mathcal H_{2n}(X_{b})
}
$$
We obtain \eqref{E:SpCycEll} by spreading cycle classes after finite base
changes, as in \cite[Exa.~20.3.8]{fulton}.

\begin{rem}\label{R:DimB>1}
	Let $f:\mathcal X\to B$ be a smooth morphism of smooth
	varieties of finite type over a field~$K$, and let $B'\subseteq B$ be a closed
	regular embedding of codimension~$1$ with trivial normal bundle. Let $\eta$
	(resp.~$\eta'$) be the generic point of $B$ (resp.~$B'$).  The arguments above
	generalize to this setting (see \cite[Exa.~20.3.8]{fulton}) to give a
	specialization map
	\begin{equation*}
	\xymatrix{
		\operatorname{CH}_{n}(\mathcal X_{\bar \eta}) \ar[r]^{\mathrm{sp}} \ar@{->>}[d]^{[-]}&
		\operatorname{CH}_{n}(\mathcal X_{\bar \eta '}) \ar[d]_{[-]}\\
		\mathcal H_{2n}(\mathcal X_{\bar \eta})' \ar[r]^{\mathrm{sp}}& \mathcal H_{2n}(\mathcal
		X_{\bar \eta '})'
	}
	\end{equation*}
	For any $b\in B(\bar K)$, if we iteratively take smooth subvarieties $b\in
	B'\subseteq B$, and restrict to Zariski open subsets to trivialize the normal
	bundle, we obtain a restriction map \eqref{E:SpCycEll} even if $\dim B>1$.
\end{rem}

\subsubsection{Application: strict decomposition of the diagonal and specialization}

\begin{teo}\label{T:DecDiagDeg}
	Let $B$ be a smooth integral variety over a  field $K$,  let  $\pi : \mathcal X
	\to  B$ be a   flat projective morphism of relative dimension $d$.
	If there exists a $\bar K$-point $b\in B(\bar K)$  such that for the fiber
	$ X_{b}$ the class of the diagonal
	$\Delta_{X_b}\in \operatorname{CH}^{d}(X_b\times_{\bar K}X_b)$
	does not admit a strict homological $R$-decomposition,
	 then the same is true for the geometric
	generic  fiber $X_{\bar {\eta}}$.
\end{teo}

\begin{rem}
	A similar result holds for Chow groups  \cite[Thm.~2.1]{voisinUniv}
	\cite[Thm.~1.14]{CTP16} \cite[Thm.~9]{HKT16}.
\end{rem}

\begin{proof}
	Using Remark~\ref{R:DimB>1},  it suffices to prove the theorem for $\dim B=1$.
	The case where $K=\mathbb C$ is a special case of  \cite[Thm.~2.1]{voisinUniv}.
	We proceed to give a similar proof in the setting of $\ell$-adic homology.

	We start with a few simplifying assumptions.
	Since the algebraic closure of the function field of $B$ is isomorphic to the
	algebraic closure of the function field of $B_{\bar K}$,
	 	we may assume that
	$K=\bar K$.
	Also, since the geometric generic fiber does not change after finite base
	change, we are free to make finite base changes.

	We will now prove the contrapositive of the theorem, and therefore we  start by
	assuming
	$\Delta_{X_{\bar {\eta}}}\in \operatorname{CH}^{d}(X_{\bar {\eta}}\times_{\bar
		{\eta}}X_{\bar {\eta}})
	$
	admits a strict homological $R$-decomposition.  By the finite type hypotheses,
	there is some finite extension of the function field $\kappa(\eta)$ of $B$
	over which the decomposition is defined. Therefore, after spreading,  we may assume
	that after a finite surjective base change $B'\to B$ there exist a divisor
	$\mathcal D'\subseteq \mathcal X':=B'\times _B\mathcal X$, which we may assume
	to contain no fiber of $\mathcal X'\to B'$, a cycle $\mathcal Z_1\in
	\operatorname{CH}^{d}(\mathcal X'\times_{B'}\mathcal X')$ supported on $\mathcal
	D'\times_{B'}\mathcal X'$, and a zero cycle $\alpha'\in
	\operatorname{CH}_0(\mathcal X')$ of relative degree-$1$ such that setting
	$\mathcal Z_2=\operatorname{pr}_2^*\alpha'$, we have
	$$[\Delta_{X'_{\bar \eta}}]-[(\mathcal Z_1)_{\bar \eta}]-[(\mathcal Z_2)_{\bar
		\eta}]=0\in \mathcal H_{2d}(X_{\bar \eta}'\times_{\bar {\kappa(\eta)}} X_{\bar
		\eta}')(-d).$$
	Since we are free to make finite base changes, to simplify the notation we relabel $B'$ as $B$ and continue the proof.
	  	We claim that at every $\bar K$-point $b\in B(\bar K)$, we have
	$$[\Delta_{X_{b}}]-[(\mathcal Z_1)_{b}]-[(\mathcal Z_2)_{b}]=0\in \mathcal
	H_{2d}(X_b\times_{\bar K}X_b)(-d).$$
	But this follows from the compatibility of the cycle class map with
	specialization given in \eqref{E:SpCycEll}.
\end{proof}

 \section{Quartic double solids and the proof of Theorem \ref{T:Intro-QDS-1}}
 \label{S:QDS}

 We now use the results developed so far to show that there exist unirational
 threefolds in positive characteristic that  have no universal
 codimension-$2$ cycle class.

 The starting point is the result of Artin and Mumford, that over an algebraically
 closed field $k$ of characteristic $\ne 2$ there exists a quartic double solid
 $X$ with exactly $10$ singular points,  in special position, all of which are nodes\,; i.e, a
 double cover of $\mathbb P^3_k$ branched along a quartic with $10$ nodes in
 special position, such that the standard resolution of singularities
 $\epsilon:\widetilde X\to X$ has non-trivial  torsion in cohomology: for
 $\mathcal H^\bullet (-)$ given by Betti cohomology $H^\bullet(-,\mathbb Z)$, or
 $2$-adic cohomology $H^\bullet(-,\mathbb Z_2)$, we have
 $\operatorname{Tors}\mathcal H^4(\widetilde X)\ne 0$ \cite[Prop.~3, \S
 4]{ArtMum72}.
 It follows from Corollary~\ref{C:V-AJ4.4} that the class of the diagonal
 $[\Delta_{\widetilde X}]\in \mathcal H^{6}(\widetilde X\times_k \widetilde
 X)(3)$ does not admit an  $R_{\mathcal H}$-decomposition.

We use this as a starting point for the investigation of
desingularizations of quartic double solids with at most $9$ nodes.
In order to connect such threefolds to the Artin--Mumford example we take
a brief excursion into the moduli spaces of (lattice-polarized) K3
surfaces in Section  \ref{S:K3}.

\subsection{Nodal quartic surfaces}
\label{S:K3}

Let $k$ be an algebraically closed field with $\operatorname{char}(k)\neq 2$.  
Let $Y \subset \proj^3_k$ be a quartic surface smooth away from
 rational double points $P_1, \ldots, P_n \in Y(k)$, with $n \ge 1$.  Let $\varpi:
 \widetilde Y \to Y$ be the minimal resolution of $Y$, obtained by blowing
 up the $n$ nodes.  Then $\widetilde Y$ is a smooth K3 surface.   Let
 $\lambda = \varpi^* \mathcal O_{\proj^3}(1) \in \pic(\widetilde Y) \iso
 \operatorname{NS}(\widetilde Y)$, and for $1 \le i \le n$
 let $\epsilon_i$ be the class of $\varpi^{-1}(P_i)$ in $\pic(\widetilde
 Y)$.  Under the intersection pairing, we have $(\lambda,\lambda) =
 4$\,; $(\epsilon_i,\epsilon_i) = -2$\,; $(\lambda,\epsilon_i) = 0$\,; and,
 if $i\not = j$, then $(\epsilon_i,\epsilon_j) = 0$.  Therefore, the $\integ$-span of $\left\{\lambda, \epsilon_1, \ldots, \epsilon_n\right\}$ is a primitive sub-lattice of $\pic(\widetilde Y)$ of rank $n+1$.
Moreover,
 for $N\gg 0$, $N\lambda - \sum \epsilon_i$ is (very) ample.  In particular, the lattice spanned by  $\left\{\lambda, \epsilon_1, \ldots, \epsilon_n\right\}$ contains the class of a polarization.

With this notation, it is not hard to show:

\begin{lem}
\label{L:smoothnodes}
Let $k$ be an algebraically closed field.  Suppose that either $\operatorname{char}(k) \not = 2$ and $R = k\powser T$, or that $\operatorname{char}(k) = p>2$ and that $R$ is the ring of Witt vectors $\ww(k)$.  Let $B =\operatorname{Frac}(R)$.

Let $Y\subset \proj^3_k$ be a quartic
surface which has $n$ rational double points and is smooth elsewhere. Then for each $0 \le m \le n$, there exists a deformation $\mathcal Y/R$ of $Y/k$ such that $\mathcal Y_B$ has exactly $m$ nodes.
\end{lem}

\begin{proof}
  This follows from the deformation theory for K3 surfaces worked out
  in \cite{delignelift}\,; see, \emph{e.g.}, \cite[Prop. 3.8]{achteroccult} for
  details.  We assume $\operatorname{char}(k) = p>2$, since the deformation theory in \cite{delignelift} is built in analogy to the well-known classical case over the complex numbers.
  Let $R_\univ = \ww(k)\powser{T_1, \ldots, T_n}$, and choose an isomorphism $\defo(\widetilde Y) \iso \spf R_\univ$.  We have seen that the collection 
  $\left\{\lambda, \epsilon_1, \ldots, \epsilon_n\right\}$ is a linearly independent set of primitive elements of $\pic(\widetilde Y)$.
  Consequently, there exist $f_0, \ldots, f_n\in R_\univ$ such that, for $0 \le m \le n$, $R_\univ/(f_0, \ldots, f_m)$ is smooth over $\ww(k)$ of relative dimension $19-m$\,; and if $A$ is an Artinian algebra, and if $\mu:R_\univ \to A$ is a deformation of $\widetilde Y$ to $A$, then
   $\epsilon_i$ (resp.~$\lambda$) extends to $\widetilde{\mathcal
     Y}_A$ if and only if $\mu(f_i)=0$ (resp. $\mu(f_0)=0$).  In
   particular, $\defo(\widetilde Y, \left\{\lambda, \epsilon_1, \ldots, \epsilon_m\right\}) \iso \spf R_\univ/(f_0, \ldots, f_m)$.

  So, let $\mathfrak p$ be the maximal ideal of $R$, and fix $0 \le m \le n$.
  Choose a compatible family of surjections $\mu_j: R_\univ \to R/(\mathfrak p)^j$ such that $\mu_j(f_i) =0$ if and only if $i \le m$. We obtain a formal deformation of $\widetilde Y$ to $\spf R$.  Moreover, since (for $N\gg 0$ ) $N\lambda- \sum_{i=1}^m\epsilon_i$ is ample on the generic fiber, the formal deformation algebraizes to yield an  algebraic deformation $\widetilde{\mathscr Y}/R$ of $\widetilde Y$ over $R$.  The only $(-2)$-curves on $\widetilde{\mathscr Y}_B$ are the curves representing $\epsilon_1, \ldots, \epsilon_m$\,; contracting these -- equivalently, mapping $\widetilde{\mathscr Y}$ to $\mathbb P^3_R$ using the quasi-ample line bundle $\lambda$ -- gives the desired deformation of $Y$.
  \end{proof}

Already, this is adequate for producing examples of Theorem \ref{T:Intro-QDS-1}.  To show that for $m \le n$ an arbitrary $m$-nodal quartic surface degenerates to an arbitrary $n$-nodal quartic requires a brief detour into the moduli theory of K3 surfaces.

Let $\mathsf R_{4,\ge n}\ /\ \integ[1/2]$ be the moduli space of quartic surfaces with at least $n$ rational double points.  Our goal is to show\,:

\begin{pro}
  \label{P:nnodesirred}
  If $n< 10$, then each fiber of $\mathsf R_{4, \ge n} \to \spec \integ[1/2]$ is geometrically irreducible.
\end{pro}

\begin{proof}
Before proceeding, it may be worth recalling Madapusi Pera's strategy for
showing that away from characteristics dividing $2d$, the moduli space $\mathsf
R_{2d}$ of quasipolarized K3 surfaces of degree
$2d$, is irreducible \cite[Cor.~5.16]{perak3}.  The well-known period map for K3 surfaces
realizes $\mathsf R_{2d}(\cx)$ as an arithmetic quotient of a
Hermitian symmetric domain.  This quotient admits a canonical model
over $\integ[1/2d]$\,; the existence of a good arithmetic
compactification for such an orthogonal Shimura variety implies, by
Zariski's main theorem, that the space stays irreducible upon
reduction modulo a prime.  We adopt a similar strategy here, using the
notion of a lattice-polarized K3 surface.

Initially, since the maximal number of nodes on a quartic surface is achieved by the $16$ nodes of a Kummer surface, we merely assume that $1 \le n \le 16$.
Let $L_n$ be the free $\integ$-module generated by symbols $\ell, e_1,
\ldots, e_n$, equipped with the pairing $(\ell,\ell) = 4$\,; $(e_i,e_i)
= -2$\,; and all other pairings are zero.  Then $L_n$ is a lattice of
rank $n+1$ and signature $(1,n)$.

Now let $Y/k$ be a quartic surface with $n$ rational double points.  A labeling of the these double
points -- equivalently, a labeling of $n$ exceptional curves in the
minimal resolution $\widetilde Y \to Y$ -- induces a primitive embedding
of lattices
\[
\xymatrix@R=0.5em{
L_n \ar[r]^-\alpha & \pic(\widetilde Y) \\
\ell \ar@{|->}[r] & \lambda \\
e_i \ar@{|->}[r] & \epsilon_i.
}
\]
Moreover, $\lambda$ is a quasi-polarization on $\widetilde
Y$\,; and if $Y$ is smooth away from the $n$ double points, then $\alpha(L_n)$ contains the ample classes
$N\lambda - \sum \epsilon_i$  for $N\gg 0$.

In short, $(\widetilde Y,\alpha)$ is an element of
$\mathsf R_{L_n}(k)$, where $\mathsf R_{L_n}$ is the moduli space of
K3 surfaces (quasi-) polarized by the lattice $L_n$ (\cite{achteroccult,
  dolgachevlattice}). Let $\mathsf R_{4,\ge n}$ be the moduli space of
quartic K3 surfaces with at least $n$ rational double points.
Contracting the classes represented by $\alpha(e_1), \ldots,
\alpha(e_n)$ defines a morphism
\[
\xymatrix{
\beta:\mathsf R_{L_n} \ar[r] & \mathsf R_{4,\ge n}
}
\]
which restricts to an isomorphism
\[
\xymatrix{
\mathsf R^\circ_{L_n} \ar[r] & \mathsf R_{4,=n},
}
\]
where the source is the moduli space of K3 surfaces with an ample
lattice polarization by $L_n$, and $\mathsf R_{4,=n}$ is the space of
quartic surfaces in $\proj^3$ with exactly $n$ rational double
points.  By Lemma \ref{L:smoothnodes}, $\mathsf R_{4,=n}$ is fiberwise
(over $\integ[1/2]$)
dense in $\mathsf R_{4, \ge n}$, and thus $\mathsf R^\circ_{L_n}$ is
fiberwise dense in $\mathsf R_{L_n}$.

It thus suffices to show that each fiber of $\mathsf R_{L_n}$ is
geometrically irreducible\,; it is now that we start assuming $n \le 9$.
Then there is a unique primitive embedding of $L_n$ into the standard K3 lattice (\emph{e.g.}, \cite[Thm.~14.1.12]{huybrechtsk3}).
Consequently, over $\cx$, there exist a Hermitian
symmetric domain $\dd^{L_n}$ of type IV, and an arithmetic group of
automorphisms $\Gamma^{L_n}$ of $\dd^{L_n}$, such that the complex
period map yields an isomorphism~\cite[Thm.~10.1]{dolgachevkondo}
\[
\xymatrix{
\mathsf R_{L_n}(\cx) \ar[r]^{\tau_{L_n,\cx}} & \Gamma^{L_n}\backslash \dd^{L_n}.
}
\]
In particular, $\mathsf R_{L_n,\cx}$ is irreducible.   The theory of integral canonical
models of Shimura varieties provides a canonical stack $\shim^{L_n}$
over $\integ[1/2]$ with $\shim^{L_n}_\cx = \Gamma^{L_n}\backslash
\dd^{L_n}$  \cite{kisin_intcan}, and $\tau_\cx$ is the complex fiber
of a morphism
\[
\xymatrix{
\mathsf R_{L_n} \ar[r]^{\tau_{L_n}} & \shim^{L_n}
}
\]
of stacks over $\integ[1/2]$ \cite[Lem~6.4]{achteroccult}. It is known
that $\shim^{L_n}$ is fiberwise geometrically irreducible
\cite[Cor.~4.1.11]{peratoroidal}.  Because fibers of $\mathsf R_{L_n}$
and of $\shim^{L_n}$ have the same dimension, it suffices to show
that $\tau_{L_n}$ is an immersion.

Let $p$ be an odd prime, and choose $N$ so that $\mu:=N\lambda - \sum
\epsilon_i$ is ample and $d := \frac 12 (\mu,\mu)$ is relatively prime
to $p$.  We have a morphism $\phi:\mathsf R_{L_n} \to   \mathsf R_{2d}$
which, on $S$-points, is given by $(\widetilde Y \to S, \alpha)
\mapsto (\widetilde Y \to S, \alpha(N\ell - \sum e_i))$.

As in \cite[\S 6]{achteroccult}, we have a commuting diagram of stacks over $\integ_{(p)}$\,:
\[
\xymatrix{
\mathsf R^\circ_{L_n} \ar[d]^\phi \ar[r]^{\tau_{L_n}} & \shim^{L_n}
\ar[d] \\
\mathsf R^\circ_{2d} \ar@{^{(}->}[r]^{\tau_{2d}} & \shim^{\langle 2d \rangle}\ .
}
\]
Since the minimal resolution of a K3
surface with exactly $n$ nodes admits a unique  $L_n$-polarization,
$\phi$ is an immersion.  Since $\tau_{2d}$ is an immersion
\cite[Cor.~5.15]{perak3}, $\tau_{L_n}$ is an immersion, too.
\end{proof}

\begin{cor}
\label{C:degen10}
  Let $k$ be an algebraically closed field with $\operatorname{char}(k) \not = 2$.  Suppose $Y_1$ and $Y_2$ are quartic surfaces with, respectively, exactly $m$ and $n<10$ nodes, with $m \le \min(n,9)$.  Then there exist a twice-pointed curve $(T,t_1,t_2)$ over $k$, and a relative quartic surface $\mathcal Y \to T$, such that $\mathcal Y_{t_1} \iso Y_1$ and $\mathcal Y_{t_2}\iso Y_2$.
\end{cor}

\begin{proof}
  By Lemma \ref{L:smoothnodes}, the closure of $\mathsf R_{4,=m}$ in $\mathsf R_{4, \ge m}$ contains $\mathsf R_{4,\ge n}$. Now use the fact (Proposition~\ref{P:nnodesirred}) that $\mathsf R_{4, \ge m, k}$ is irreducible.
\end{proof}

\begin{rem}\label{R:VoisinK3}
While we show above that, for $n\le 9$,  the locus $\mathsf R_{4,\ge n}$ of degree-$4$ polarized  K3 surfaces with greater than or equal to $n$ nodes is irreducible, and therefore contains the Artin--Mumford example with $n=10$ nodes (since one can deform the nodes independently for $n\le 10$), the situation for the locus  $\hilb^4_{\proj^3_k,\ge n}$ of quartic surfaces with greater than or equal to $n$ nodes in the Hilbert scheme
$\hilb^4_{\mathbb P^3_k}$ of quartic surfaces is different.  For contrast, we recall the situation over~$\mathbb C$.  It is known that for $n=6,7,8,9$, the locus  $\hilb^4_{\proj_{\mathbb C}^3,\ge n}$ is reducible (see \cite[Rem.~1.2]{voisinUniv}), while for $n \le 7$, there is a unique irreducible component of $\hilb^4_{\proj_{\mathbb C}^3,\ge n}$ dominating $(\mathbb P_{\mathbb C}^3)^{(n)}$ by the map sending  an $m$-nodal quartic to its  set of nodes, and this component contains the Artin--Mumford examples \cite[\S 2]{voisinUniv}.
  For quartics with exactly $n=8,9$ nodes, it is a classical result that the nodes must be in special position in $\mathbb P^3_{\mathbb C}$ (there is no component of $\hilb^4_{\proj_{\mathbb C}^3,\ge n}$ dominating $(\mathbb P^3)^n$)\,; we direct the reader to the MathSciNet review of \cite{voisinUniv} for references.
    \end{rem}

\subsection{Quartic double solids and the proof of Theorem \ref{T:Intro-QDS-1}}\label{S:Pf-QDS}
The previous section essentially says that quartic double solids are liftable to quartic double solids in characteristic $0$, and that every quartic double solid with at most $9$ nodes (and no other singularities) degenerates to the Artin--Mumford example with $10$ nodes (since the nodes deform independently).   We use this to give a proof of Theorem \ref{T:Intro-QDS-1}\,:

 \begin{proof}[Proof of Theorem \ref{T:Intro-QDS-1}] In the case where  $\operatorname{char} k=0$, this is \cite[Thm.~1.9]{voisinUniv}.   We now proceed with the proof under the hypothesis that $\operatorname{char}k> 2$.

To investigate conditions (1)--(4) of   Theorem~\ref{T:Intro-M-StabInv}, we will first want to use that $\widetilde X$ lifts to characteristic $0$.  More precisely, let us fix notation, and suppose that $Y$ is the $n$-nodal quartic surface over $k$ defining $\widetilde X$.  In other words, we have $\widetilde X\to X\to \mathbb P^3_k$, and  the second morphism is the double cover of projective space  branched over $Y$.    For concreteness, suppose that $Y$ is defined by $q(x_0,x_1,x_2, x_3)=0$ for some homogeneous quartic polynomial $q(x_0,x_1,x_2, x_3)$, so that $X$ is defined by $q(x_0,x_1,x_2,x_3)+x_4^2=0$ in the weighted projective space $\mathbb P^4_k(1,1,1,1,2)$.   Now we may take the lift $\mathcal Y/S$ of $Y$ to characteristic $0$ of Lemma \ref{L:smoothnodes}, and we define $\mathcal X/S$ in $\mathbb P^4_{S}(1,1,1,1,2)$ by taking the double cover of $\mathbb P^3_S$ branched along $\mathcal Y$.  We then define $\widetilde {\mathcal X}/S$ by blowing up the locus in $\mathcal Y$ of nodes in the fibers\,; \emph{i.e.}, the singular locus of the map $\mathcal Y\to S$.  In other words, $\widetilde X$ lifts to characteristic $0$ as the standard resolution of singularities of a quartic double solid with exactly $n$ nodes.  

With this lift $\widetilde {\mathcal X}/S$ of $\widetilde X$ to characteristic $0$, then using the fact that in characteristic $0$ the standard resolution of singularities of a quartic double solid with exactly $n$ nodes is rationally connected and  has no $2$-torsion in cohomology    for $n\le 9$ \cite{endrass},  we can conclude from Corollary \ref{C:RC->(0)} that  $\widetilde X$ satisfies conditions (1)--(4) of Theorem~\ref{T:Intro-M-StabInv} with $\ell=2$ for $n\le 9$.

 Thus we can focus on conditions (5) and (6).  For this we will use a new deformation of $X$, namely a deformation to the Artin--Mumford example.  More precisely, with $X$ and $Y$ as above, according to Corollary \ref{C:degen10}, we can find a second family of quartic surfaces $\mathcal Y'\to S'$ over a $k$-curve ~$S'$, with points $s_1',s_2'\in S'(k)$ such that $\mathcal Y'_{s'_1}\cong Y$ and $\mathcal Y'_{s'_2}$ is isomorphic to the quartic surface of the Artin--Mumford example.  We take $\mathcal X'/S'$ to be the double cover of $\mathbb P^3_{S'}$ branched along~$\mathcal Y'$.  Then $\mathcal X'_{s'_1}\cong X$, and $\mathcal X'_{s'_2}$ is the Artin--Mumford quartic double solid.   We denote by $\widetilde {\mathcal X}'_{s'_2}\to \mathcal X_{s'_2}'$ the standard resolution of singularities of the Artin--Mumford quartic double solid\,; \emph{i.e.}, the Artin--Mumford example.   As  $\operatorname{Tors}H^4(\widetilde {\mathcal X}'_{s'_2},\mathbb Z_2)\ne 0$ \cite[Prop.~3, \S
 4]{ArtMum72}, we see that the Artin--Mumford example $\widetilde {\mathcal X}'_{s'_2}$ does not admit a strict cohomological $\mathbb Z_2$-decomposition of the diagonal (Corollary~\ref{C:V-AJ4.4}).  Using  Proposition~\ref{P:Dec-Res} twice 
 and   Theorem~\ref{T:DecDiagDeg} once, we can conclude that $\widetilde X=\widetilde {\mathcal X}'_{s'_1}$ does not admit a  
 a strict cohomological $\mathbb Z_2$-decomposition of the diagonal.  
 
 Therefore, from Theorem \ref{T:Intro-M-StabInv}, one of the conditions (5) or (6) must fail.    
Let us focus on (6).  Using that (3) and (3') hold for $\widetilde X$, it follows that  $\dim \operatorname{Ab}^2_{\widetilde X/K}=\dim H^3(\widetilde X_{\bar K},\mathbb Q_\ell)$.  By proper base change, we know that $ H^3(\widetilde X_{\bar K},\mathbb Q_\ell)\cong H^3(\widetilde {\mathcal X}_{\bar \eta},\mathbb Q_\ell)$, and it is well-known that $\dim H^3(\widetilde {\mathcal X}_{\bar \eta},\mathbb Q_\ell)=10-n$ (for $n=0$, this is the standard computation of the Betti numbers of a double cover, and for $n\ge 1$ this is \cite[4.10.4]{beauville77}).
 Using again the lift $\widetilde {\mathcal X}/S$ of $\widetilde X$ to characteristic $0$, then assuming (5) holds,  we can use 
Proposition~\ref{C:MinCohSpec} to conclude that for $n=7,8,9$, we have that $[\Theta_{\widetilde X}]^{g-1}/(g-1)!$ is $\mathbb Z$-algebraic, so that condition (6) holds.   Therefore, for $n=7,8,9$, we must have had that condition (5) fails, since otherwise conditions (1)--(6) would hold and $\widetilde X$ would admit a cohomological $\mathbb Z_2$-decomposition of the diagonal, which we know is not the case.   
 \end{proof}

\begin{rem}\label{R:2DeltaCH}
Here we recall that the standard desingularization $\widetilde X$ of a nodal quartic double solid is unirational.
 For the case of $n=0$ nodes, we direct the reader to \cite[p.10]{welters81}, where a unirational parameterization of $\widetilde X$ is given\,; Welters works over $\mathbb C$ but the argument holds for $\operatorname{char}(k)\ne 2$.   
For $n\ge 1$, we have moreover that  $\widetilde X$ is separably rationally connected, there is a degree $2$ dominant rational map $\mathbb P^3_k\dashrightarrow \widetilde X$, and 
   the class $2\Delta_{\widetilde X}\in
  \operatorname{CH}^3(\widetilde X\times _k \widetilde X)$ admits a strict
 decomposition. 
 This can be found in \cite[4.5.4]{beauville77}, \cite[Exa.~3, p.25]{beauville83rational}.  For convenience, we sketch the argument here.  
  The key point is to show is that there
 is  a degree-$2$  dominant rational map to  $\widetilde X$ from a rational threefold.  Projecting from a chosen node exhibits the blow-up $X'$ of  $X$ at that node as a singular fibration in quadrics.   The exceptional divisor $Q\subseteq X'$ is a quadric surface, which, under the structure map $ X'\to \mathbb P^2_k$, gives a double cover of $\mathbb P^2_k$.  The base change of $ X'$ to~$Q$ under this double cover admits a section, and therefore is rational (see \emph{e.g.}, \cite[Prop.~4.1]{beauville77}), completing the proof.
 If $\mathbb P^3_k\dashrightarrow \widetilde X$ is the associated degree-$2$
dominant rational map from projective space, then since $\operatorname{char}k \ne 2$, we must have that this
 is a separable rational map, 
   so that  $\widetilde X$ is
 separably rationally connected \cite[Exa.~IV.3.2.6.2]{kollar}.
 Finally, the degree-$2$ dominant rational map  $\mathbb P^3_k\dashrightarrow \widetilde X$ also implies that
$2\Delta_X$ admits a strict decomposition (Remark \ref{R:UniRDecDiag}).
\end{rem}

\begin{rem}\label{R:T-2-S+N}
In Theorem \ref{T:Intro-QDS-1} we showed that the standard desingularization of a very general quartic double solid with exactly $n\le 9$ nodes does not admit a cohomological $\mathbb Z_2$-decomposition of the diagonal.  The previous remark implies that for $1\le n\le 9$,  the diagonal admits a cohomological $\mathbb Z_\ell$-decomposition for all  $\ell\ne 2,\operatorname{char}(k)$, since $2$ is invertible in $\mathbb Z_\ell$.  
This shows in particular that while conditions (1)--(6) in Theorem~\ref{T:Intro-M-StabInv} are sufficient to imply the existence of a cohomological $\mathbb Z_\ell$-decomposition, they are not necessary.   See Theorem \ref{T:ZZell-iff} for necessary and sufficient conditions for a cohomological $\mathbb Z_\ell$-decomposition.  
\end{rem}

\newpage
\appendix

\addtocontents{toc}{\protect\setcounter{tocdepth}{-2}}
\section{Some facts about abelian varieties}
    
\subsection{The Weil pairing}\label{S:WeilPair}
Let $A/K$ be an abelian variety over a field, with dual abelian
variety $\widehat A$. For any integer $N$, the
group scheme $\widehat A[N]$ is canonically isomorphic to the Cartier
dual $A[N]^D := \hom(A[N],\mathbb G_M)$ of $A[N]$, and thus there is a canonical Weil pairing
\begin{equation}
  \label{E:weilfinite}
\xymatrix{
  A[N] \times \widehat A[N] \ar[r] & \mmu_N.}
\end{equation}
Let $l$ be any prime.  The $l$-adic Tate module of $A$ is $T_l A :=
\varprojlim A[l^n](\bar K)$. If $l = \ell \ne \operatorname{char}(K)$,
then $T_\ell A$ is abstractly isomorphic to $\integ_\ell^{2\dim A}$\,;
but if $p = \operatorname{char}(K)>0$, then $T_pA$ is free over
$\integ_p$ of rank at most $\dim A$.  For any $l$, we have a canonical
isomorphism
\begin{equation}
  \label{E:tatedual}
T_l(\widehat A) = T_l(A)^\vee(1)
\end{equation}
(the notation is recalled in our Conventions \ref{conventions}.)
To see this in the case where $l = p = \operatorname{char}(K) = p>0$,
use \eqref{E:weilfinite} to obtain that  $\widehat A[p^\infty]$ is isomorphic to the Serre dual
$(A[p^\infty])^D = \hom(A[p^\infty], \mmu_{p^\infty})$\,; then use the
fact that $T_pA = T_p(A[p^\infty])$, and that for any $p$-divisible
group $G$ there are canonical isomorphisms $T_p(G^D) \iso
\hom_{\integ_p}(T_p(G)(-1),\integ_p) = T_p(G)^\vee(1)$.

If $\ell \ne \operatorname{char}(K)$, then the intersection pairing,
Poincar\'e duality and the Weil pairing yield canonical isomorphisms
\begin{equation}
  \label{E:dualityl}
  H^{2g-1}(A,\integ_{\ell}(g)) \stackrel{\cup}{=} H^1({A},\integ_\ell)^\vee  \stackrel{PD}{=} T_\ell A   \stackrel{\text{Weil}}{=} (T_\ell\widehat{A})^\vee(1) \stackrel{PD}{=}  H^1(\widehat{A},\integ_\ell(1))
\end{equation}
and $$H^{2g-1}(A,\rat_{\ell}(g)) \stackrel{\cup}{=}  H^1({A},\rat_\ell)^\vee  \stackrel{PD}{=} V_\ell A   \stackrel{Weil}{=} (V_\ell\widehat{A})^\vee(1) \stackrel{PD}{=}  H^1(\widehat{A},\rat_\ell(1)).$$

If $K$ is perfect of characteristic $p>0$, there are canonical isomorphisms of $F$-crystals
\[
H^{2g-1}_\cris(A/\ww)(g) = H^1_\cris(A/\ww)^\vee = H^1(\widehat A/\ww)(1)\,;
\]
taking invariants under $F$ yields \eqref{E:dualityl} at $p$.

\subsection{Line bundles and symmetric isogenies}\label{S:ThetaLB}

We review the link between symmetric isogenies $A\to \widehat A$ and line bundles on $A_{\bar K}$.   We also prove Lemma \ref{L:linebundle}, which we will use later.

Let $A$ be an abelian variety over a field $K$.
Recall that, to  an isomorphism class of a line  bundle $L$ on $A$, one associates the (symmetric, \emph{i.e.}, self-dual) homomorphism  $\varphi_L:A\to \widehat A$, which on points
$\bar a\in A(\bar K)$ is given by  $\varphi_L(\bar a)=t_{\bar a}^*L\otimes L^{-1}$.
When $K$ is finite or algebraically closed, the assignment $L\mapsto \varphi_L$  surjects onto the set of symmetric isogenies (\emph{e.g.}, \cite[Thm.~2.6]{conrad-pol} if $K$ is finite and   \cite[\S 20]{mumfordAV} if $K=\bar K$).
We say that $L$ is \emph{nondegenerate} if
$\varphi_L$ is an isogeny.  For nondegenerate line bundles $L,L'$ on $A$, we have
$\varphi_L=\varphi_{L'}$ if and only if $L$ and $L'$ are algebraically
equivalent, \emph{i.e.}, if and only if $L$ and $L'$ differ by translation.

A symmetric isogeny $\Lambda:A\to \widehat A$ is \emph{principal} if it is an isomorphism.   If $L$ is a line bundle on $A_{\bar K}$ such that $\Lambda_{\bar K}=\varphi_L$, we have the equalities 
$L^g/g! = \chi(L) = \pm \sqrt{\deg \varphi_L}$, where $g=\dim A$.  In other words, $\Lambda$ is principal if and only if $\chi(L)=\pm 1$.  
We observe that if $\Omega/K$ is any algebraically closed field, then $\Lambda_\Omega:A_\Omega\to \widehat A_\Omega$ is principal if and only if $\Lambda$ is.

	Recall that a symmetric isogeny $\Lambda:A\to \widehat A $ is a polarization if $\Lambda _{\bar K}:A_{\bar K}\to \widehat A_{\bar K}$ is induced by an ample line bundle\,; \emph{i.e.}, for any line bundle $L$ on $A_{\bar K}$ such that   
	$\Lambda_{\bar K}=\varphi_L$, we have that $L$ is ample.
Recall that a nondegenerate line bundle $L$ is ample if
and only if $h^0(A,L)>0$ (\emph{e.g.}, \cite[\S 17]{mumfordAV}), and so $\Lambda$ is a polarization (resp.~principal polarization) if and only if $h^0(A,L)>0$ (resp.~$h^0(A,L)=1$). 
	We observe that if $\Omega/K$ is any algebraically closed field, then $\Lambda_\Omega:A_\Omega\to \widehat A_\Omega$ is a (principal) polarization if and only if $\Lambda$ is.  Indeed, clearly if $\Lambda$ is a polarization, then $\Lambda_\Omega $ is as well.  Conversely, suppose that $\Lambda_\Omega$ is a polarization, and let $L$ be a line bundle on $A_{\bar K}$ such that $\Lambda_{\bar K}=\varphi_L$.  
	 	 As $L$ is by definition non-degenerate, and  a non-degenerate line bundle on an abelian variety is ample if and only if it is effective, it suffices to show that $L$ is effective.  For this, note that $\Lambda_\Omega = (\varphi_L)_\Omega=\varphi_{L_\Omega}$.
		Since $\Lambda_\Omega$ is a polarization, it follows  that $L_\Omega$ is ample, and therefore effective. 		   Using cohomology and base-change (over the affine base fields), one concludes that  $L$ is effective.

\subsection{First Chern class of a symmetric isogeny}\label{S:c1Lambda}
Given a symmetric isogeny $\Lambda:A\to \widehat A$, we
denote by $[\Lambda]\in H^2(A_{\bar K},\mathbb Z_\ell(1))$ the first Chern class of the unique line bundle (up to translation) on $A_{\bar K}$ inducing the base change of $\Lambda$ to $\bar K$. In other words, if $L$ is any line bundle on $A_{\bar K}$ such that $\Lambda_{\bar K}=\varphi_{L}:A_{\bar K}\to \widehat A_{\bar K}$, then we set $[\Lambda ]:=c_1(L)$. 
One can realize the first Chern class of $\Lambda$ more directly in the following way.  The symmetric isogeny $\Lambda$ induces by Tensor-Hom adjunction a morphism $T_\ell \Lambda: T_\ell A \otimes T_\ell \widehat A^\vee \to \mathbb Z_\ell$.  The Weil pairing $e : T_\ell A \times T_\ell \widehat A \to \integ_{\ell}(1)$  then provides an isomorphism $e: T_\ell \widehat A^\vee \stackrel {\sim}{\to} T_\ell A(-1)$, which all together gives  the pairing
	$$E^\Lambda: T_\ell A \times T_\ell A \to \integ_{\ell}(1), \quad
	(x,y) \mapsto e(x,T_\ell\Lambda ( y)).$$ 
	 It is classical that the pairing $E^\Lambda$ 
	is alternating and that, seen as an element of $(\bigwedge^2T_\ell A)^\vee(1) = H^2(A_{\bar K},\integ_{\ell}(1))$, it coincides with $c_1(L)$ 
	for any line bundle $L$ such that $\Lambda = \varphi_L$\,; see \emph{e.g.}~\cite[(11.23)]{EvdGM}.

In case a symmetric isogeny $\Lambda$ induces an isomorphism $T_\ell\Lambda: T_\ell A \stackrel{\simeq}{\longrightarrow} T_\ell \widehat{A}$,  the inverse to  $T_\ell\Lambda$ has an explicit description in terms of the first Chern class $[\Lambda]$\,:

\begin{lem}\label{L:linebundle}
	Let $\Lambda :A\to \widehat A$ be a symmetric isogeny of a $g$-dimensional  abelian variety over a field $K$.
	If $T_\ell\Lambda : T_\ell A \stackrel{\simeq}{\longrightarrow} T_\ell \widehat{A}$ is an isomorphism for some prime $\ell$, then the inverse is given by the map
$$
(T_\ell \Lambda)^{-1}=\frac{[\Lambda]^{g-1}}{(g-1)!}\cup - : \  H^1(A_{\bar K},\integ_{\ell}(1)) \to H^{2g-1}(A_{\bar K},\integ_{\ell}(g)),
$$ 
	where we have identified $T_\ell A$ with $H^{2g-1}(A_{\bar K},\integ_{\ell}(g))$  and $T_\ell \widehat{A}$ with $H^1(A_{\bar K},\integ_{\ell}(1))$ as laid out in \S \ref{S:WeilPair}.
\end{lem}
\begin{proof}
Identifying $[\Lambda]$ with the pairing $E^\Lambda$ above, 
	one concludes by using the fact that the cohomology algebra $H^\bullet(A_{\bar K},\integ_{\ell})$ identifies, via the intersection pairing, with the alternating algebra on $H^1(A_{\bar K},\integ_{\ell})$. 
\end{proof}

\subsection{Some Hodge theoretic conventions}\label{S:A-HdgTh}

Suppose that $X$ is a  complex projective manifold of dimension $d_X$, with ample divisor $H$.  Setting $\omega= c_1(H)\in H^2(X,\mathbb Z)$, and using $\omega$ also for the associated $2$-form in $H^2_{dR}(X,\mathbb C)$ under the natural map $H^2(X,\mathbb Z)\to H^2(X,\mathbb C)=H^2_{dR}(X,\mathbb C)$, the Hodge--Riemann bilinear form on $H^{p,q}(X)$ is the Hermitian form given by 
$$
h(\alpha,\beta)= i^{p-q}(-1)^{\frac{k(k-1)}{2}}\int_X\alpha\wedge \bar \beta \wedge \omega^{n-k},
$$
where $k=p+q$.

Fixing an integer $n$ such that $1\le 2n-1 \le d_X $, assume that
$\coniveau^nH^{2n-1}(X,\mathbb Q)=H^{2n-1}(X,\mathbb Q)$, which implies that 
 $$H^{2n-1}_{dR}(X,\mathbb C)=H^{n,n-1}(X)\oplus H^{n-1,n}(X).$$
Via the composition 
$$
H^{2n-1}(X,\mathbb Z)\to H^{2n-1}_{dR}(X,\mathbb C)\twoheadrightarrow H^{n-1,n}(X)
$$
$$
\alpha \mapsto \alpha^{n,n-1} + \alpha^{n-1,n} \mapsto \alpha^{n-1,n}
$$
we obtain an inclusion $H^{2n+1}(X,\mathbb Z)_\tau \subseteq H^{n,n+1}(X)$, where in the above we have $\alpha^{n,n-1}=\overline{\alpha^{n-1,n}}$.  
The intermediate Jacobian is the complex torus
$$
J^{2n-1}(X)=H^{n-1,n}(X)/H^{2n-1}(X,\mathbb Z)_\tau.
$$

The  Hodge--Riemann bilinear form on $H^{n-1,n}(X)$ is in this case
$$
h(\alpha,\beta)=(i)^{(n-1)-n}(-1)^{(2n-1)(2n-2)/2}\int_X \alpha \wedge \bar \beta \wedge \omega^{d_X - 2n+1}=i (-1)^{n}\int_X \alpha \wedge \bar \beta \wedge \omega^{d_X - 2n+1}.
$$ 
For any $\alpha,\beta\in H^{2n-1}(X,\mathbb Z)$, after viewing them as  classes in $H^{2n-1}_{dR}(X,\mathbb C)$ under the natural map, we have that
\begin{align*}
\alpha \cup \beta \cup [H]^{d_X-2n+1}& = 
 \int_X \alpha \wedge \beta \wedge \omega^{d_X-2n+1} \\
&= \int_X  (\overline{\alpha^{n-1,n}} + \alpha^{n-1,n}) \wedge (\overline{\beta^{n-1,n}} + \beta^{n-1,n})\wedge \omega^{d_X-2n+1}\\
& = 2  \operatorname{Re}\int_X \alpha^{n-1,n}\wedge \overline{\beta^{n-1,n}}\wedge \omega^{d_X-2n+1} \\
& = 2 \operatorname{Re} (-i)(-1)^n  h(\alpha^{n-1,n},\beta^{n-1,n})\\
& = 2(-1)^{n}\left(\operatorname{Im}h(\alpha^{n-1,n},\beta^{n-1,n})\right).
\end{align*}
In other words, for the hermitian form $2h$ on $H^{n-1,n}(X)$, the associated alternating  form 
$$
E=-\operatorname{Im}2h:H^{n-1,n}(X)\times H^{n-1,n}(X)\to \mathbb R
$$
when restricted to the image $H^{2n-1}(X,\mathbb Z)_\tau\subseteq H^{n-1,n}(X)$ is given by $(-1)^{n-1}$ times the cup product in cohomology (via the morphism $H^{2n-1}(X,\mathbb Z)\to H^{n-1,n}(X)$).  We note the consequence that $E$ evaluated on 
$H^{2n-1}(X,\mathbb Z)_\tau\subseteq H^{n-1,n}(X)$ takes integral values.   Summarizing, we have the commutative diagram
\begin{equation}\label{E:Imhdiag}
\xymatrix@C=8em@R=1em{
H^{n-1,n}(X)\times H^{n-1,n}(X)\ar[r]^<>(0.5){E=-\operatorname{Im}2h}& \mathbb R\\
H^{2n-1}(X,\mathbb Z)_\tau \times  H^{2n-1}(X,\mathbb Z)_\tau \ar[r]^<>(0.5){E=-\operatorname{Im}2h } \ar@{^(->}[u] & \mathbb Z \ar@{^(->}[u]\\
H^{2n-1}(X,\mathbb Z) \times  H^{2n-1}(X,\mathbb Z)\ar[r]^<>(0.5){(-1)^{n-1}\alpha \cup\beta \cup \omega^{d_X-2n+1} } \ar[u] & \mathbb Z \ar@{=}[u]
}
\end{equation}

As in \cite[\S 2.2]{BL}, since the hermitian form $2h$ has
associated alternating form $E$ taking integral values on the integral lattice,  there is an
 induced line bundle $\Theta_X$ on $J^{2n-1}(X)$, such that $$c_1(\Theta_X)= E = (-1)^{n-1}\int_X (-)\wedge (-) \wedge\omega^{d_X-2n-1}$$ under the identification $H^2(J^{2n-1}(X),\mathbb Z)=\bigwedge^2H^1(J^{2n-1}(X),\mathbb Z) = \bigwedge^2H^{2n-1}(X,\mathbb Z)_\tau^\vee$.  Note that in \cite[\S 2.2]{BL}, they associate to the hermitian form $2h$ the alternating form $\operatorname{Im}2h$, and therefore their alternating form is the negative of $c_1(\Theta_X)$.

We now translate this discussion using Tensor-Hom adjunction.  First, we see that $2h$ induces an isomorphism
$$
\xymatrix@R=1em@C=4em{
H^{n-1,n}(X) \ar[r]^{\alpha\mapsto 2h(\alpha,-)}_\sim&  \overline {H^{n-1,n}(X)}^\vee
}
$$
Now, given a complex vector space $V$ there is an identification
\begin{equation}\label{E:HomRV}
\operatorname{Hom}_{\mathbb C}(\overline V,\mathbb C)= \operatorname{Hom}_{\mathbb R}(V,\mathbb R)
\end{equation}
$$
\ell\mapsto -\operatorname{Im}\ell.
$$
The inverse map is given by $\phi\mapsto \phi(i(-))-i\phi(-)$. 
 Note that in \cite[\S 2.2]{BL} the opposite identification is made by the assignment $\ell\mapsto \operatorname{Im}\ell$.  
 Using the identification \eqref{E:HomRV} above, we have from \eqref{E:Imhdiag} a commutative diagram
$$
\xymatrix@R=1em@C=6em{
H^{n-1,n}(X) \ar[r]^{\alpha\mapsto 2h(\alpha,-)}_\sim&  \overline {H^{n-1,n}(X)}^\vee\\
H^{n-1,n}(X) \ar[r]_<>(0.5)\sim^<>(0.5){\alpha \mapsto -\operatorname{Im}2h(\alpha ,-)} \ar@{=}[u]&  \operatorname{Hom}_{\mathbb R}(H^{n-1,n}(X),\mathbb R)  \ar@{=}[u]
}
$$

As a consequence of the discussion above, the Hodge--Rieman bilinear form induces a commutative diagram 
$$
\xymatrix@R=1em{
H^{n-1,n}(X) \ar[r]^{\alpha\mapsto 2h(\alpha,-)}_\sim&  \overline {H^{n-1,n}(X)}^\vee\\
H^{2n-1}(X,\mathbb Z)_\tau  \ar[r] \ar@{^(->}[u]& (H^{2n-1}(X,\mathbb Z)_\tau)^{\widehat \ }  \ar@{^(->}[u]
}
$$
where by definition 
\begin{align}
(H^{2n-1}(X,\mathbb Z)_\tau)^{\widehat \ } &= \{\varphi \in  \overline {H^{n-1,n}(X)}^\vee: \varphi(\alpha )\in \mathbb Z \ \text{ for all } \ \alpha\in H^{2n-1}(X,\mathbb Z)_\tau\}\\
\label{E:HomRVhat}& =  \{\phi \in  \operatorname{Hom}_{\mathbb R}(H^{n-1,n}(X),\mathbb R) : \phi(\alpha )\in \mathbb Z \ \text{ for all } \ \alpha\in H^{2n-1}(X,\mathbb Z)_\tau\}
\end{align}
This by definition induces an isogeny of complex tori
$$
\Theta_X: J^{2n-1}(X)\longrightarrow \widehat J^{2n-1}(X).
$$
As in \cite[\S 2.2]{BL}, one has that $\Theta_X=\varphi_{\Theta_X}$\,; note that despite our difference in conventions from  \cite[\S 2.2]{BL} regarding alternating forms, the map $\Theta_X$ is determined by its induced morphism on complex vector spaces, which is given by the hermitian form $2h$, which is the same in our conventions and those of \cite[\S 2.2]{BL}.

Taking the induced map in homology for $\Theta_X$, and combining with the discussion above, we obtain a commutative diagram 
$$
\xymatrix@C=5em@R=1.5em{
	H_1(J^{2n-1}(X),\mathbb Z) \ar[d]_= \ar[r]^{\Theta_{X}}& H_1(\widehat J^{2n-1}(X), \mathbb Z)\\
	H^{2n-1}(X,\mathbb Z)_\tau \ar[r]^{\cup (-1)^{n-1}[H]^{d_X-2n+1}} & H^{2d_X-2n+1}(X,\mathbb Z)_\tau\ . \ar[u]_=
}
$$
The left vertical arrow is the canonical identification coming from the construction of the intermediate Jacobian, while the right vertical arrow is the dual identification, where we identify $ H_1(\widehat J^{2n-1}(X), \mathbb Z)=H_1(J^{2n-1}(X),\mathbb Z)^\vee $ via the Weil pairing, and we identify $H^{2n-1}(X,\mathbb Z)^\vee_\tau = H^{2d_X-2n+1}(X,\mathbb Z)_\tau$ via the cup product.   Note that the Weil pairing is the composition of the identifications $H_1(\widehat J^{2n-1}(X), \mathbb Z)=(H^{2n-1}(X,\mathbb Z)_\tau)^{\widehat {\ }}=H_1(J^{2n-1}(X),\mathbb Z)^\vee $, where the first is the canonical identification from the definition, and the second comes from the evaluation pairing from the definition \eqref{E:HomRVhat}.  This second identification uses the convention \eqref{E:HomRV}\,; using the opposite convention,  \emph{i.e.}, taking the imaginary part of a Hermitian form, rather than its negative,  may lead one naturally in the analysis above to include an extra factor of $(-1)$ in the bottom row of the diagram, in which case one would have to use the negative of the Weil pairing to make the diagram commute.

Finally, we note that taking Tate modules, this gives a commutative diagram 
$$
\xymatrix@C=5em@R=1.5em{
	T_\ell J^{2n-1}(X) \ar[d]_= \ar[r]^{T_\ell \Theta_{X}}_{c_1(\Theta_X)}& T_\ell \widehat J^{2n-1}(X)\\
	H^{2n-1}(X,\mathbb Z_\ell )_\tau \ar[r]^{\cup (-1)^{n-1}[H]^{d_X-2n+1}} & H^{2d_X-2n+1}(X,\mathbb Z_\ell)_\tau\  \ar[u]_=
}
$$
where in the top row we have identified $T_\ell\Theta_X$ with $c_1(\Theta_X)$ as elements of $H^2(J^{2n-1}(X),\mathbb Z_\ell)$ as in~\S \ref{S:c1Lambda}.  

\subsubsection{Polarizations}  If we assume that $\coniveau_H^mH^{2m-1}(X,\mathbb Q)= 0$ 
for all $1\le m<n$, \emph{i.e.}, if $H^{m,m-1}(X)=0$ for all $1\le m<n$, then $H^{n-1,n}(X)$ is primitive, and the Hodge--Riemann bilinear form is positive definite on $H^{n-1,n}(X)$.  In this case, $\Theta_X$ is ample, and gives a polarization on $J^{2n-1}(X)$. 

\begin{rem}[Weight-$1$ Hodge structure]
The category of polarized abelian varieties is equivalent to the category of polarized weight-$1$ $	\mathbb Z$-Hodge structures.  In the case where $\coniveau_H^mH^{2m-1}(X,\mathbb Q)= 0$ 
for all $1\le m<n$, the polarized abelian variety $(J^{2n-1}(X),\Theta_X)$ corresponds to the polarized weight-$1$ $\mathbb Z$-Hodge structure $(H^{2n-1}(X,\mathbb Z)_\tau, Q)$, where $Q$ is the alternating form
$$
Q:H^{2n-1}(X,\mathbb Z)_\tau \times H^{2n-1}(X,\mathbb Z)_\tau\longrightarrow \mathbb Z
$$
$$
Q(\alpha,\beta)=(-1)^{n-1}\int_X \alpha\wedge\beta \wedge \omega^{d_X-2n+1}.  
$$
To be clear, the Hodge decomposition is given by $(H^{2n-1}(X,\mathbb Z)_\tau \otimes_{\mathbb Z}\mathbb C)^{1,0}=H^{n,n-1}(X)$. Note that under these identifications, the polarizations satisfy $Q=E=c_1(\Theta_X)$.   The associated Hermitian forms also agree, as $i^{0-1}Q(\alpha,\bar \beta)=i^{-1}Q(\alpha,\bar \beta)=i(-1)^n\int_X\alpha \wedge \bar \beta \wedge \omega^{d_X-2n+1}=h(\alpha,\beta)$.   
\end{rem}

\newpage

\addtocontents{toc}{\protect\setcounter{tocdepth}{-2}}
 \bibliographystyle{hamsalpha}
 \bibliography{DCG}

\end{document}